\documentclass[11pt,reqno]{amsart}

\usepackage[letterpaper,margin=0.90in]{geometry}
\usepackage{amsmath,amssymb,mathtools,mathrsfs,esint}
\usepackage{enumitem}
\usepackage{microtype}
\usepackage[colorlinks=true,linkcolor=blue,citecolor=blue,urlcolor=blue]{hyperref}
\usepackage[nameinlink,capitalise,noabbrev]{cleveref}

\allowdisplaybreaks
\numberwithin{equation}{section}

\newtheorem{theorem}{Theorem}[section]
\newtheorem{proposition}[theorem]{Proposition}
\newtheorem{lemma}[theorem]{Lemma}
\newtheorem{corollary}[theorem]{Corollary}

\newtheorem{example}[theorem]{Example}
\theoremstyle{definition}
\newtheorem{definition}[theorem]{Definition}

\newcommand{\C}{\mathbb C}
\newcommand{\N}{\mathbb N}
\newcommand{\Nzero}{\mathbb N_0}
\newcommand{\dd}{\,d}
\newcommand{\Hol}{\operatorname{Hol}}
\newcommand{\loc}{\mathrm{loc}}
\newcommand{\dist}{\operatorname{dist}}
\newcommand{\supp}{\operatorname{supp}}
\newcommand{\avg}{\mathop{\fint}}
\newcommand{\Piabs}{\boldsymbol\Pi}
\newcommand{\dideal}{\mathfrak d_{p,q}^{\,r}}
\newcommand{\V}{\mathcal V}
\newcommand{\Q}{\mathcal Q}
\newcommand{\M}{\mathcal M}
\newcommand{\G}{\mathcal G}
\newcommand{\R}{\mathcal R}

\title[Absolutely summing operators on finite-type domains]
{Absolutely summing operators on Bergman spaces over convex domains of finite type}

\author[Jianxiang Dong]{Jianxiang Dong$^{1}$}
\address{School of Mathematics and Statistics, Tianshui Normal University, Tianshui 741000, P.R. China}
\email{jianxd@tsnu.edu.cn}

\author[Chunxu Xu]{Chunxu Xu$^{2,*}$}
\address{School of Science, Nanjing Forestry University,
Nanjing 210037, P.R. China}
\email{1968385450@qq.com}
\thanks{$^{*}$Corresponding author.}


\hypersetup{
  pdftitle={Absolutely summing operators on Bergman spaces over convex domains of finite type},
  pdfauthor={Jianxiang Dong and Chunxu Xu},
  pdfsubject={Operator ideals on Bergman spaces},
  pdfkeywords={absolutely summing operator, Bergman space, Hankel operator, composition operator, Volterra operator, finite type}
}

\begin{document}

\begin{abstract}
Let $\Omega\subset\mathbb C^n$ be a smoothly bounded convex domain
of finite type. We characterize absolute summability on its Bergman
spaces and estimate approximation errors in the summing norm.
For $1<p<\infty$ and
$1\le q,r<\infty$, we characterize absolutely $r$-summing Carleson
embeddings $A^p(\Omega)\to L^q(\mu)$ by diagonal operators formed
from normalized masses on McNeal balls. For $1<p,q<\infty$ and
$1\le r<\infty$, we obtain a corresponding big Hankel criterion
using local $L^q$ distances to holomorphic functions, with the
appropriate volume weights. Both criteria give upper and lower norm
estimates. The proofs combine uniform local nuclear estimates,
kernel synthesis, and an analytic decomposition at two fixed scales.
Local Taylor polynomials give explicit rank and error bounds for
the embeddings. These bounds transfer to general Hankel operators
when $p=2$ or $q=2$. Suitable atomic embeddings have the same best
approximation errors as their diagonal models, up to constants.
A family of Hankel operators contains a complemented copy of the
diagonal ideal and has matching approximation rates for power
sequences in the Hilbert case.
For positive Toeplitz operators on finite-type ellipsoids, we compute
exact spectral asymptotics in three regimes, including the critical
logarithmic factor and the leading constants for geometric symbols.
These asymptotics give optimal approximation rates in the summing norm.
Applications to composition and Volterra operators include the
target endpoint $q=1$.
\end{abstract}

\maketitle

\setcounter{tocdepth}{1}
\tableofcontents

\section{Introduction}

\subsection{Background and the problem}

Let $X,Y$ be Banach spaces and $1\le r<\infty$.
For a finite family $(x_k)_{k=1}^N\subset X$, define its weak
$r$-norm by
\begin{equation}\label{eq:weak-r-norm}
 w_r(x_1,\ldots,x_N)
 =\sup_{x^*\in B_{X^*}}
 \left(\sum_{k=1}^N|x^*(x_k)|^r\right)^{1/r},
\end{equation}
where $X^*$ is the continuous dual of $X$ and $B_{X^*}$ is its
closed unit ball. A bounded linear operator $T:X\to Y$ is
absolutely $r$-summing if
\begin{equation}\label{eq:intro-summing}
 \left(\sum_{k=1}^N\|Tx_k\|_Y^r\right)^{1/r}
 \le Cw_r(x_1,\ldots,x_N)
\end{equation}
for every finite family in $X$. The least such $C$ is the summing
norm $\pi_r(T)$, and $\Piabs_r(X,Y)$ denotes the class of these
operators \cite{Pietsch1967,Pietsch1980,DiestelJarchowTonge1995}.
Thus absolute summability compares the norms of the images $Tx_k$
with the weak $r$-norm of the original family. For Bergman spaces,
we seek criteria in terms of local data of a measure or a symbol.
We also study approximation in $\pi_r$ by operators of finite rank,
where $\operatorname{rank}T$ is the dimension of the range of $T$.
The error estimate must hold uniformly over all finite families.

Throughout, let $\Omega\subset\mathbb C^n$ be a bounded convex domain
with $C^\infty$ boundary and finite D'Angelo type. This means that
\[
 \Delta_1(\partial\Omega)
 =\sup_{\zeta\in\partial\Omega}\sup_{\varphi}
 \frac{\nu_0(\rho_\Omega\circ\varphi)}
 {\nu_0(\varphi-\zeta)}<\infty.
\]
Here $\rho_\Omega$ is a smooth defining function, so that
$\Omega=\{\rho_\Omega<0\}$ and $\nabla\rho_\Omega\ne0$ on
$\partial\Omega$; $\nu_0$ is the vanishing order at $0$, and
$\varphi:(\C,0)\to(\C^n,\zeta)$ ranges over nonconstant
holomorphic germs. For a vector-valued germ, its vanishing order
is the minimum of the orders of its components.
Write $v$ for Euclidean volume, $\dist$ for Euclidean distance, and
$\Hol(U)$ for the holomorphic functions on an open set $U$.
The Bergman space and its norm are
\[
 A^s(\Omega)=L^s(\Omega,v)\cap\Hol(\Omega),\qquad
 \|F\|_{A^s}=\left(\int_\Omega|F|^s\dd v\right)^{1/s},
 \qquad 1\le s<\infty.
\]
Put $\delta(z)=\dist(z,\partial\Omega)$.
Let $K$ be the Bergman kernel, characterized by
$F(z)=\int_\Omega K(z,w)F(w)\dd v(w)$ for $F\in A^2(\Omega)$,
with $K(\cdot,w)\in A^2(\Omega)$ and
$K(z,w)=\overline{K(w,z)}$.
The boundary scales on
these domains depend on both the point and the direction. Local
restriction estimates and analytic approximation must therefore be
uniform under nonisotropic changes of scale.

Fan, He, Wang, and Zeng completed the classification of absolutely
summing diagonal operators. They also proved two general transference
principles and Carleson and Hankel criteria on weighted Bergman spaces
of the ball \cite{FanHeWangZeng2026}. Their ball results include the
source endpoint $p=1$. We use their classification and transference
principles, together with Garling's diagonal theory
\cite{Garling1974}, on unweighted finite-type spaces with $1<p<\infty$.
Xiao, Yang, and Yuan studied Carleson embeddings and positive Toeplitz
operators on the same class of finite-type domains
\cite{XiaoYangYuan2026}. Their criteria concern boundedness,
compactness, and Schatten membership. For a compact operator $T$
between Hilbert spaces, its singular values $s_N(T)$ are the
eigenvalues of $|T|=(T^*T)^{1/2}$, listed in nonincreasing order
with multiplicity and completed by zeros when necessary.
For $0<s<\infty$, the Schatten class $\mathcal S_s$ consists of
such operators with $\sum_Ns_N(T)^s<\infty$; write
$\|T\|_{\mathcal S_s}=(\sum_Ns_N(T)^s)^{1/s}$.
The case $s=2$ is the Hilbert--Schmidt class.
We pass from these properties
to absolutely summing ideals between Banach spaces. The quantitative
problem is to control approximation in the summing norm, with an
explicit bound on the rank.

For the embedding criterion, we prove uniform nuclear estimates for
local restriction maps after nonisotropic rescaling. A map $T:X\to Y$
is nuclear if it admits a representation
$Tx=\sum_\ell x_\ell^*(x)y_\ell$ with $x_\ell^*\in X^*$,
$y_\ell\in Y$, and $\sum_\ell\|x_\ell^*\|\,\|y_\ell\|<\infty$.
Its nuclear norm $\nu_1(T)$ is the infimum of these sums over all
such representations. We write
$(x^*\otimes y)x=x^*(x)y$ for a rank-one term.
Together with
kernel synthesis, these estimates verify the hypotheses of the known
transference principles on $\Omega$. The local Taylor expansions also
give explicit rank and error bounds; see \cref{thm:intro-approximation}.

For the Hankel criterion, local holomorphic approximants must agree
up to controlled errors on overlapping cells. We obtain this control
by using two fixed scales and a partition of unity. We then compare
Ahn's form norm with the dual McNeal norm. This comparison allows us to
apply Ahn's weighted $\bar\partial$ estimate \cite{Ahn2004} for every
$1<q<\infty$; see \cref{prop:FTI-all,lem:ida-decomposition}.

The approximation problem requires both an error estimate and a
rank bound. We obtain them from the cells retained and the local
Taylor degrees. Pietsch domination gives convergence of the
diagonal tails directly, without using their explicit classification.
For general Hankel operators, integral distance to analytic functions
(IDA), the local $L^q$ distance to holomorphic functions defined in
\eqref{eq:intro-ida}, gives a corresponding sequence and an
upper bound when $p=2$ or $q=2$. An orthogonal projection transfers
the embedding approximation without increasing its rank; see
\cref{cor:Hankel-summing-approximation}.

For suitable atomic measures, we obtain factorizations in both
directions between the embedding and a diagonal operator. These
factorizations preserve every rank bound. Hence the best errors
agree up to constants. For a family of Hankel operators, fixed
bounded maps recover the whole diagonal operator. This gives a
complemented copy of the diagonal ideal, meaning an isomorphic image
that is the range of a bounded projection. It also gives lower bounds against
arbitrary approximants of finite rank; see
\cref{thm:intro-diagonal-models}.
For this family, we also choose a separate Taylor degree on each
symbol cell. This gives matching upper and lower approximation rates
for power sequences in the Hilbert case. The rank remains proportional
to the number of cells; see
\cref{prop:Hankel-adaptive-approximation,cor:Hankel-Hilbert-approximation-lower}.
An explicit Toeplitz family on ellipsoids allows us to compute the
leading spectral constants and a critical logarithmic factor.
A further calculation gives exact constants for symbols defined by
the finite-type boundary scales; see
\cref{thm:intro-spectral-model,prop:ellipsoid-profile}.

The geometric background includes the Carleson criteria in
\cite{Hastings1975,CimaWogen1982,Luecking1983} and their finite-type
forms in \cite{Jasiczak2010,KrantzLi1995,LiLiuWang2024}.
More general weakly pseudoconvex settings are studied in
\cite{Zhang2025,Zhang2026}. Criteria for Toeplitz operators in trace ideals on classical
and strongly pseudoconvex Bergman spaces appear in
\cite{Luecking1987,PauZhao2015,AbateRaissySaracco2012,
CuckovicMcNeal2006,HuLvZhu2016}.
For Hankel operators, local analytic approximation and oscillation estimates
were developed in
\cite{Zhu1991,Li1992,Li1993,LiLuecking1995,FangXia2018,
Isralowitz2013,Pau2016,PelaezPeralaRattya2020}.
The IDA method on Fock spaces is treated in \cite{HuVirtanen2023};
its relation to singular value asymptotics on weighted Bergman spaces
is studied in \cite{FanWangZeng2026}.

Absolute summability of Carleson embeddings on Hardy and Bergman
spaces is studied in
\cite{LefevreRodriguezPiazza2018,HeJreisLefevreLou2024}.
See \cite{FaresLefevre2022} for Bloch spaces and
\cite{ChenDongWang2025,Lefevre2026} for results at the source endpoint.
Those results at $p=1$ differ from the target endpoint $q=1$
considered here. The operator applications below use the pullback
method for composition operators
\cite{CowenMacCluer1995,CuckovicZhao2004}, the radial identity for
Volterra operators \cite{AlemanSiskakis1997,JreisLefevre2024}, and
projection decompositions for Toeplitz and little Hankel operators.
Related summability results on the ball and on Fock spaces appear in
\cite{HuWang2025,ChenHeWang2026,HuLv2026,XuDong2026}.
The section introductions give the comparisons needed for each
application. General results on trace ideals are given in
\cite{Simon2005,Zhu2007}.

\subsection{Main results}

The Bergman projection is
\[
 Pg(z)=\int_\Omega K(z,w)g(w)\dd v(w)
\]
initially on $L^2(\Omega)$. It satisfies
\begin{equation}\label{eq:intro-Pq}
 P:L^s(\Omega)\longrightarrow A^s(\Omega)
 \quad\text{is bounded},\qquad 1<s<\infty.
\end{equation}
See \cite{McNealStein1994,McNealSIO1994}.

For $1\le q<\infty$, define the symbol class
\begin{equation}\label{eq:symbol-class}
 \mathcal D_{\Omega,q}
 =\{f\in L^q_{\loc}(\Omega):
 fK(\cdot,z)\in L^q(\Omega)\text{ for every }z\in\Omega\}.
\end{equation}
The notation $L^q_{\loc}$ means $L^q$ on every compact subset.
Let $\Gamma$ be the algebraic span of the kernel sections
$K(\cdot,z)$. For $f\in\mathcal D_{\Omega,q}$ and $1<q<\infty$, put
\begin{equation}\label{eq:intro-operators}
 M_fg=fg,\qquad T_fg=P(fg),\qquad H_fg=(I-P)(fg),
 \qquad g\in\Gamma.
\end{equation}
These are the multiplication operator $M_f$, the Toeplitz operator
$T_f$, and the big Hankel operator $H_f$. Here $I$ is the identity on $L^q$.
At $q=1$, define only $M_fg=fg$ on $\Gamma$.
The space $\Gamma$ is dense in $A^p$ for
$1<p<\infty$ by \cref{lem:kernel-span-density}.
Membership of an initial operator in $\Piabs_r$ always means
that it has an absolutely $r$-summing extension. Such an extension
is unique. We put $\pi_r(T)=+\infty$ if no such extension exists.
For a positive locally finite Borel measure $\mu$, let $J_\mu$
denote the inclusion $A^p(\Omega)\to L^q(\Omega,\mu)$ whenever
this inclusion is bounded. Such an inclusion is called a Carleson
embedding, and $\mu$ is then a Carleson measure for these exponents.

We write $A\lesssim B$ for $A\le CB$ with a constant $C>0$, and
$A\asymp B$ when both $A\lesssim B$ and $B\lesssim A$ hold.
For positive quantities, $A\sim B$ means that $A/B$ tends to $1$
in the indicated limit. We write $A=O(B)$ when $|A|\le C|B|$
and $A=o(B)$ when $A/B\to0$, in the stated limit.
Unless stated otherwise, the constants may
depend on $\Omega$, the displayed exponents, and the fixed
geometric radii. They do not depend on lattice indices, symbols,
functions, or measures. Put $\N=\{1,2,\ldots\}$ and
$\Nzero=\{0,1,2,\ldots\}$. We write $\mathbf1_E$ for the
characteristic function of $E$ and $\#F$ for the size of a finite set $F$.

For every sufficiently small $s>0$, let $\Q_s(z)$ be the McNeal ball,
a nonisotropic polydisc centered at $z$ whose radii follow the
extremal complex directions at scale $s\delta(z)$.
Its construction is given in \cref{sec:geometry}.
Let $\V_s(z)=v(\Q_s(z))$.
For $f\in L^q_{\loc}(\Omega)$, set
\begin{align}
 \M_{q,s}f(z)
 &=\left(\avg_{\Q_s(z)}|f(w)|^q\dd v(w)\right)^{1/q},
 \notag\\
 \G_{q,s}f(z)
 &=\inf_{h\in\Hol(\Q_s^*(z))}
 \left(\avg_{\Q_s(z)}|f(w)-h(w)|^q\dd v(w)\right)^{1/q}.
 \label{eq:intro-ida}
\end{align}
Here $\Q_s^*(z)$ is one fixed enlargement of $\Q_s(z)$, and
$\avg_Eu\dd v=v(E)^{-1}\int_Eu\dd v$ is the normalized integral.

Fix an admissible pair of radii
\begin{equation}\label{eq:admissible-radius-pair}
 0<\rho<\rho_+<s_0,
 \qquad C_*\rho<\rho_+.
\end{equation}
The constants $s_0$ and $C_*>1$ depend only on the finite-type
geometry.  They are chosen in \cref{lem:ida-decomposition}.  The small
radius $\rho$ is used for the lattice.  The larger radius $\rho_+$ is
used for analytic approximation.  This convention with two fixed scales is
used throughout.

All sequence spaces are complex and indexed by $\N$.
For $1\le s<\infty$, $\ell^s$ consists of sequences $c=(c_j)$
with $\|c\|_{\ell^s}=(\sum_j|c_j|^s)^{1/s}<\infty$;
$\ell^\infty$ has norm $\|c\|_{\ell^\infty}=\sup_j|c_j|$.
For a scalar sequence $b=(b_j)$, define
\begin{equation*}
 D_b:\ell^p\to\ell^q,
 \qquad D_b(c_j)=(b_jc_j),
 \qquad
 \dideal=\{b:D_b\in\Piabs_r(\ell^p,\ell^q)\},
 \qquad
 \|b\|_{\dideal}=\pi_r(D_b).
\end{equation*}
Let $\Lambda=\{a_j\}$ be an admissible finite-type $\rho$-lattice:
the balls $\Q_\rho(a_j)$ cover $\Omega$, fixed smaller balls are
pairwise disjoint, and each fixed admissible enlargement has bounded
overlap, as specified in \eqref{eq:admissible-lattice}. Put
$V_j=\V_\rho(a_j)$.
Define the coefficient sequences
\begin{align}
 \mathbf M_{p,q,\rho}f
 &=\{V_j^{1/q-1/p}\M_{q,\rho}f(a_j)\}_j,
 \label{eq:lattice-M}\\
 \mathbf G_{p,q,\rho}f
 &=\{V_j^{1/q-1/p}\G_{q,\rho_+}f(a_j)\}_j.
 \label{eq:lattice-G}
\end{align}
Then
\begin{align}
 \|f\|_{\mathcal L_{p,q}^{r}(\Omega)}
 &=\|\mathbf M_{p,q,\rho}f\|_{\dideal},
 \notag\\
 \|f\|_{\mathcal I_{p,q}^{r}(\Omega)}
 &=\|\mathbf G_{p,q,\rho}f\|_{\dideal}.
 \label{eq:intro-I-space}
\end{align}
Set
\begin{equation*}
 \mathcal L_{p,q}^{r}(\Omega)
 =\{f\in L^q_{\loc}(\Omega):
 \|f\|_{\mathcal L_{p,q}^{r}(\Omega)}<\infty\},
 \qquad
 \mathcal I_{p,q}^{r}(\Omega)
 =\{f\in L^q_{\loc}(\Omega):
 \|f\|_{\mathcal I_{p,q}^{r}(\Omega)}<\infty\}.
\end{equation*}
The $\mathcal I$-quantity is a seminorm. Its notation omits
the fixed pair $(\rho,\rho_+)$ and the $\rho$-lattice.
Its null space is $\Hol(\Omega)$, by \cref{prop:ida-calculus}.
It therefore defines a norm on
$\mathcal I_{p,q}^r(\Omega)/\Hol(\Omega)$.
The $\mathcal L$-quantity is independent, up to equivalent norms, of the
admissible lattice and radius. For $1<q<\infty$, the same is true
of the $\mathcal I$-seminorm on $\mathcal D_{\Omega,q}$, by
\cref{thm:intro-hankel}. At $q=1$, the notation in
\eqref{eq:intro-I-space} refers to the fixed pair of scales.

The Carleson criterion is used throughout the paper.

\begin{theorem}\label{thm:intro-carleson}
Let $1<p<\infty$, $1\le q,r<\infty$, and let $\mu$ be a positive
locally finite Borel measure on $\Omega$. Put
\[
 \beta_j(\mu)
 =\frac{\mu(\Q_\rho(a_j))^{1/q}}{V_j^{1/p}}.
\]
Then
\[
 J_\mu\in\Piabs_r(A^p(\Omega),L^q(\Omega,\mu))
 \quad\Longleftrightarrow\quad
 \{\beta_j(\mu)\}_j\in\mathfrak d_{p,q}^{\,r}.
\]
Moreover,
\[
 \pi_r(J_\mu)\asymp
 \|\{\beta_j(\mu)\}_j\|_{\mathfrak d_{p,q}^{\,r}}.
\]
The constants are independent of $\mu$.
\end{theorem}

The proof is given in \cref{thm:carleson}. The corresponding
Hankel criterion uses the local distance to holomorphic functions.

\begin{theorem}\label{thm:intro-hankel}
Let $1<p,q<\infty$, $1\le r<\infty$, and
$f\in\mathcal D_{\Omega,q}$. Then
\begin{equation}\label{eq:intro-hankel-theorem}
 H_f\in\Piabs_r(A^p(\Omega),L^q(\Omega))
 \quad\Longleftrightarrow\quad
 f\in\mathcal I_{p,q}^{r}(\Omega),
\end{equation}
and
\begin{equation}\label{eq:intro-hankel-norm}
 \pi_r(H_f)\asymp\|f\|_{\mathcal I_{p,q}^{r}(\Omega)}.
\end{equation}
The constants depend only on $p,q,r$, $\Omega$, and the fixed
admissible geometric parameters. No additional solution hypothesis
is needed on $\Omega$.
\end{theorem}

The local restriction proof also controls approximation by operators of finite rank.
The next statement is \cref{thm:summing-approximation}.

\begin{theorem}\label{thm:intro-approximation}
Let $1<p<\infty$, $1\le q,r<\infty$, and let $\mu$ be a positive
locally finite Borel measure on $\Omega$. Assume
$J_\mu\in\Piabs_r(A^p(\Omega),L^q(\Omega,\mu))$, and put
\[
 b_j=\mu(\Q_{C\rho}(a_j))^{1/q}V_j^{-1/p},
\]
where $C$ is the fixed enlargement constant in
\cref{thm:summing-approximation}.
There are constants $A>0$ and $0<\vartheta<1$, independent of $\mu$,
such that for every finite $F\subset\N$ and every $L\in\Nzero$
there is an operator of finite rank $R_{F,L}:A^p(\Omega)\to L^q(\Omega,\mu)$
with
\begin{align*}
 \operatorname{rank}R_{F,L}&\le\#F\binom{L+n}{n},\\
 \pi_r(J_\mu-R_{F,L})
 &\le A\left(\|b\mathbf1_{F^c}\|_{\dideal}
 +(L+1)^{n-1}\vartheta^{L+1}\|b\mathbf1_F\|_{\dideal}\right).
\end{align*}
In particular, $J_\mu$ is a limit of operators of finite rank in $\pi_r$.
\end{theorem}

The first term measures the omitted cells. The second term measures
the local polynomial error. The estimate gives a rank bound for each
choice of cells and polynomial degree, including when $q=1$.

We next compare best errors at a prescribed rank. For
$T\in\Piabs_r(X,Y)$, put
\[
 a_N^{(r)}(T)=\inf\{\pi_r(T-A):A\in\mathcal L(X,Y),\
                                  \operatorname{rank}A<N\},
 \qquad N\ge1,
\]
where $\mathcal L(X,Y)$ is the space of bounded linear maps.

There is also an upper bound for general Hankel operators when
$p=2$ or $q=2$. Let $1<p,q<\infty$, $1\le r<\infty$, and
$f\in\mathcal D_{\Omega,q}\cap\mathcal I_{p,q}^r(\Omega)$.
Put $\gamma_j=V_j^{1/q-1/p}\G_{q,\rho_+}f(a_j)$.
For fixed constants $C_0\in\N$ and $0<\vartheta<1$, set
$\eta_L=(L+1)^{n-1}\vartheta^{L+1}$. Then
\[
 a_N^{(r)}(H_f)
 \lesssim\inf_{\substack{F\subset\N\text{ finite},\ L\in\Nzero\\
                    C_0\#F\binom{L+n}{n}<N}}
 \left(\|\gamma\mathbf1_{F^c}\|_{\dideal}
       +\eta_L\|\gamma\|_{\dideal}\right).
\]
The constants do not depend on $f$ or $N$.
\Cref{cor:Hankel-summing-approximation} proves this estimate.
The Hilbert assumption permits the use of an orthogonal projection
while preserving the rank bound.

\begin{theorem}\label{thm:intro-diagonal-models}
Let $1<p<\infty$ and $1\le q<\infty$.
There is an infinite sublattice, relabeled as $\{a_j\}_{j\ge1}$,
with the following properties for every $1\le r<\infty$.
\begin{enumerate}[label=\textup{(\roman*)}]
\item For $b\in\mathfrak d_{p,q}^{\,r}$, define
\[
 \mu_b=\sum_{j\ge1}|b_j|^qK(a_j,a_j)^{-q/p}\delta_{a_j},
\]
where $\delta_{a_j}$ is the unit point mass at $a_j$.
Then $\mu_b$ is finite, $J_{\mu_b}\in\Piabs_r(A^p,L^q(\mu_b))$, and
\[
 a_N^{(r)}(J_{\mu_b})\asymp a_N^{(r)}(D_b),\qquad N\ge1.
\]
\item If $q>1$, there are bounded maps
$T_p:\ell^p\to A^p$ and $Q:L^q\to\ell^q$, independent of $b$ and $r$,
and a bounded complex-linear map
$\mathcal E:\mathfrak d_{p,q}^{\,r}\to\Piabs_r(A^p,L^q)$.
For each $b\in\mathfrak d_{p,q}^{\,r}$, write $\mathcal E b=H_{f_b}$,
where $f_b\in\mathcal D_{\Omega,q}$. Then
\[
 QH_{f_b}T_p=D_b,\qquad
 a_N^{(r)}(H_{f_b})\ge
 \frac{a_N^{(r)}(D_b)}{\|Q\|\,\|T_p\|}.
\]
The map $\mathcal E$ is an isomorphism onto a complemented subspace
of $\Piabs_r(A^p,L^q)$ consisting of big Hankel operators.
\end{enumerate}
The comparison constants are independent of $b$ and $N$.
\end{theorem}

This combines \cref{thm:complemented-Hankel-model,cor:atomic-sharp-approximation}.
The estimates for the atomic embeddings hold at each rank.
The recovery maps for the Hankel model are fixed throughout the
symbol family and for every summing exponent $r$.
The lower bound therefore applies to every approximant of finite rank,
regardless of its operator form.
For $p=q=2$ and $b_j=j^{-\alpha}$, $\alpha>1/2$, the same bounded
symbol satisfies the upper and lower bounds
\[
 a_N^{(r)}(H_{f_b})\asymp_{r,\alpha}N^{1/2-\alpha},
 \qquad N\ge1,\quad 1\le r<\infty.
\]
The upper bound uses different Taylor degrees on different cells.
For arbitrary $b$, this lower bound need not have a matching upper
bound in terms of the diagonal tail.

The distinction between strong and weak boundary regions is explicit
on $\Omega_m=\{(z,w)\in\C^2:|z|^{2m}+|w|^2<1\}$, where
$m\ge2$ is an integer. Put
$t=1-|z|^{2m}-|w|^2$, $R=|z|^2+t^{1/m}$, and
$f=t^\gamma R^{-\eta/2}$, where $\gamma\ge\eta/(2m)\ge0$.
In a power range $\mathfrak d_{p,q}^{\,r}=\ell^\kappa$, where
$\kappa=\kappa(p,q,r)$ is the exponent in \cref{prop:garling}, write
$\Delta=1/q-1/p$. Then
\[
 M_f\in\Piabs_r
 \quad\Longleftrightarrow\quad
 \kappa(\gamma+3\Delta)>2
 \quad\hbox{and}\quad
 \kappa\left(\gamma+\left(2+\frac1m\right)\Delta
                    -\frac\eta{2m}\right)>1.
\]
For $q>1$, the same condition characterizes $T_f$. The proof integrates
over the whole boundary layer. It also treats logarithmic modifications
of these symbols, including the equality cases; see
\cref{thm:ellipsoid-two-thresholds,cor:ellipsoid-Orlicz-border}.

The same ellipsoids give a direct spectral check of the symbol
criteria. The next statement is part of \cref{thm:ellipsoid-counting}.

\begin{theorem}\label{thm:intro-spectral-model}
Let $m\ge2$ be an integer, let $\gamma>\beta\ge0$, and set
\[
 g_{\gamma,\beta}=t^\gamma(1-|w|^2)^{-\beta},\qquad
 \alpha=\max\left\{\frac2\gamma,\frac1{\gamma-\beta}\right\},
 \qquad \varepsilon=\mathbf1_{\{\gamma=2\beta\}}.
\]
The operator $T_{g_{\gamma,\beta}}$ on $A^2(\Omega_m)$ is positive
and compact. Its singular values satisfy
\[
 s_N(T_{g_{\gamma,\beta}})
 \sim A_*N^{-1/\alpha}(\log N)^{\varepsilon/\alpha},
 \qquad N\to\infty,
\]
with an explicit $A_*>0$.
If $\alpha<2$, then for every $1\le r<\infty$,
\[
 a_N^{(r)}(T_{g_{\gamma,\beta}})
 \asymp_{m,\gamma,\beta,r}
 N^{1/2-1/\alpha}(\log N)^{\varepsilon/\alpha},
 \qquad N\to\infty.
\]
\end{theorem}

The two exponents describe different parts of the monomial count.
When $\gamma>2\beta$, the two-dimensional count gives the leading
term. When $\gamma<2\beta$, each fixed degree in the first variable
contributes to the leading constant. At $\gamma=2\beta$, the sum
over these degrees produces a logarithmic factor.
\Cref{thm:ellipsoid-counting} gives the constants in all three
cases and the exact approximation constant for $r=2$.
The passage to the other summing norms uses the classical
comparison on Hilbert spaces in \cref{prop:Hilbert-calibration}.
\Cref{prop:ellipsoid-profile} also treats symbols
$t^\gamma(1-|w|^2)^{-\beta}\psi(|z|^{2m}/(1-|w|^2))$ with a bounded
measurable function $\psi:[0,1]\to[0,\infty)$ having a positive
limit at $1$.
When $\gamma\ge2\beta$, the leading constant depends only on that
limit. When $\gamma<2\beta$, it depends on the full function $\psi$
through a convergent series of beta integrals. In particular, this
gives exact constants for the geometric symbol
$t^\gamma(|z|^2+t^{1/m})^{-m\beta}$.

\subsection{Applications and method}

The remaining operator results are applications of the two criteria.
For Toeplitz operators, diagonal extraction gives a necessary condition
in terms of the Berezin transform
$\widetilde f(z)=K(z,z)^{-1}\int_\Omega f(w)|K(w,z)|^2\dd v(w)$.
A converse holds for positive symbols satisfying the local
reverse-H\"older estimate $\M_{q,\rho}f\le C\M_{1,\rho}f$.
For complex
symbols satisfying the matching IDA condition, Toeplitz summability
is equivalent to multiplication summability and to the local mean
condition; see
\cref{thm:positive-Toeplitz,cor:matching-ida}.
For little Hankel operators, we prove a corresponding scalar
necessary condition in \cref{thm:little-necessity}.
The joint graph criteria follow from bounded projection decompositions
of multiplication. They do not give converses for the individual
components without the stated symbol hypotheses.

Weighted composition and Volterra composition operators are treated
in \cref{sec:composition,sec:volterra}. Their norms on finite families
are comparable to those of explicit pullback Carleson embeddings.
The Volterra argument uses the radial Littlewood--Paley formula in
\cref{lem:radial-LP}. These applications include $q=1$.
At this endpoint, the multiplication criterion also remains valid;
see \eqref{eq:intro-q1-multiplication}. The projected graph arguments
require $q>1$. Analytic Toeplitz symbols need no projection and are
treated at $q=1$ in \cref{cor:analytic-symbols}.

The classification in \cref{prop:garling} identifies the diagonal
ideal with a power sequence space, except for one logarithmic Orlicz
case. Section~\ref{sec:common} writes the resulting parameter conditions
explicitly. Separated symbols show that the sequence conditions are
sharp; see \cref{thm:sharp-Hankel-models,thm:infinite-separated-realization,thm:sharp-models}.

The proofs use three local estimates. Nonisotropic Taylor expansions
give uniformly nuclear restriction maps. Kernel synthesis and
Rademacher averages give the converse diagonal bounds. Finally,
analytic approximation at two fixed scales and Ahn's weighted $\bar\partial$
estimate give the Hankel upper bound. The comparison of form norms
is proved in \cref{prop:FTI-all}. The Hilbert case has a second proof
in \cref{prop:FTI-two}, based on \cite{Zimmer2021,Zimmer2023}.
Other weighted and nonisotropic solution estimates appear in
\cite{Cumenge2001,AhnCho2003,Alexandre2006,
CharpentierDupain2018,CharpentierDupain2024,Yao2024}.

Section~\ref{sec:geometry} gives the geometric tools and facts about operator ideals.
Sections~\ref{sec:carleson} and~\ref{sec:ida} prove the main criteria,
the quantitative approximation estimates, and the diagonal model
results. Sections~\ref{sec:toeplitz} to~\ref{sec:volterra} give the
projection and pullback applications. Section~\ref{sec:common}
contains the explicit parameter conditions, finite-type examples,
exact spectral asymptotics, and endpoint statements.

\section{Geometry and operator ideals}\label{sec:geometry}

Write $c_{00}$ for finitely supported scalar sequences, $e_j$ for
the $j$th coordinate vector, and $e_j^*(c)=c_j$.
For $1<s<\infty$, write $s'=s/(s-1)$ for the conjugate exponent.
For vectors in $\C^n$, use the Hermitian inner product
$\langle z,w\rangle=\sum_{k=1}^nz_k\overline{w_k}$.
The symbol $\supp$ denotes support. We use standard terminology
from several complex variables as in \cite{FollandKohn1972,Krantz2001}.
For later integral criteria, write
\[
 \dd\lambda_\Omega(z)=\frac{\dd v(z)}{\V_\rho(z)}.
\]

\subsection{McNeal polydiscs}

McNeal's nonisotropic polydiscs describe the Bergman kernel and
the boundary geometry \cite{McNeal1992,McNeal1994}; see also
\cite{Catlin1987,NagelSteinWainger1985,NagelRosaySteinWainger1989,
NikolovPflug2003}. The dyadic approach in \cite{GanHuKhan2022}
gives another proof of \eqref{eq:intro-Pq}; see Corollary~4.5 there.
It also describes the structure as a space of homogeneous type.

For $\zeta\in\partial\Omega$, let $\nu_\Omega(\zeta)$ be the outer
unit normal to $\partial\Omega$ at $\zeta$.

Choose a $C^\infty$ defining function $\rho_\Omega$, convex near
$\partial\Omega$, such that
\[
 \Omega=\{z:\rho_\Omega(z)<0\},
 \qquad |\nabla\rho_\Omega|>0\quad\text{on }\partial\Omega.
\]
Recall that $\delta(z)=\dist(z,\partial\Omega)$.  For $z$ near the boundary, a
unit vector $u\in\C^n$, and $\varepsilon>0$, define
\begin{equation*}
 \tau(z,u,\varepsilon)
 =\sup\{t>0:
 |\rho_\Omega(z+\lambda u)-\rho_\Omega(z)|<\varepsilon
 \text{ whenever }|\lambda|<t\}.
\end{equation*}
Choose a minimal McNeal extremal basis
$u_1(z,\varepsilon),\ldots,u_n(z,\varepsilon)$ as follows.
The first vector is the complex normal to the level set of
$\rho_\Omega$ through $z$.  After $u_1,\ldots,u_{k-1}$ have been
chosen, choose a unit vector $u_k$ that minimizes
$\tau(z,u,\varepsilon)$ in their complex orthogonal complement.
Choose the minimizing vectors Borel measurably in $(z,\varepsilon)$.
Such choices exist: the directional radius is continuous, each
minimization is over a compact sphere, and measurable selection
applies successively in the orthogonal complements. Put
\[
 \tau_k(z,\varepsilon)
 =\tau(z,u_k(z,\varepsilon),\varepsilon).
\]
We use the comparisons between minimal and maximal extremal bases in
\cite[Propositions~3.1 and~3.2]{NikolovPflugThomas2013}; see also
\cite[Section~2, Remark~2.3]{GanHuKhan2022}.
These comparisons give equivalent directional norms and polydisc scales.
We do not use tangency at the successive boundary contact points of a
maximal basis.  Define
\begin{equation*}
 \Q(z,\varepsilon)
 =\left\{z+\sum_{k=1}^n\zeta_ku_k(z,\varepsilon):
 |\zeta_k|<c_0\tau_k(z,\varepsilon)\right\},
\end{equation*}
where $c_0>0$ is fixed and small. First fix a geometric cutoff
$\rho_{\rm geom}>0$ for the boundary polydiscs, scale comparisons,
and engulfing estimates. A smaller working cutoff $\rho_0$ will be
chosen in \cref{lem:geometry-package}. Auxiliary polydiscs remain
defined below $\rho_{\rm geom}$ when $\rho_0$ is reduced.
For $0<\rho<\rho_{\rm geom}$, set
\begin{equation*}
 \Q_\rho(z)=\Q(z,\rho\delta(z)),
 \qquad
 \V_\rho(z)=v(\Q_\rho(z)).
\end{equation*}
Changing the scale may change the extremal basis.  We therefore
treat a change of scale and an affine dilation separately. For $L>0$, put
\begin{equation}\label{eq:frozen-polydisc-dilation}
 \Q_s^{[L]}(z)=z+L\bigl(\Q_s(z)-z\bigr).
\end{equation}
This dilation keeps the basis and the ratios of side lengths fixed.
Let $m_\Omega$ be an integer not smaller than the finite type, and put
\[
 N_{z,s}(X)=\max_{1\le k\le n}
 \frac{|\langle X,u_k(z,s\delta(z))\rangle|}
 {c_0\tau_k(z,s\delta(z))}.
\]
For centers in the boundary collar, directional norm comparison and
the finite-type scale estimates give
a constant $C_{\rm sc}\ge1$ such that
\begin{equation}\label{eq:polydisc-scale-gauge}
 N_{z,t}(X)\le C_{\rm sc}(s/t)^{1/m_\Omega}N_{z,s}(X),
 \qquad 0<s\le t<\rho_{\rm geom}.
\end{equation}
Indeed, the gauge in extremal coordinates is uniformly equivalent to
$|X|/\tau(z,X/|X|,s\delta(z))$ for $X\ne0$.  The directional radius
at scale $t\delta(z)$ is at least a constant times
$(t/s)^{1/m_\Omega}$ times its value at scale $s\delta(z)$.
These two estimates imply \eqref{eq:polydisc-scale-gauge}; see
\cite[Section~2]{Jasiczak2010} and
\cite[Propositions~3.1 and~3.2]{NikolovPflugThomas2013}.

Choose $A_*>1$ sufficiently large that
$4C_{\rm sc}A_*^{-1/m_\Omega}<1/2$.  For $A_*s<\rho_0$, define
\begin{equation*}
 \Q_s^*(z)=\Q_{A_*s}(z).
\end{equation*}
Then \eqref{eq:polydisc-scale-gauge} gives the uniform inclusion
\begin{equation}\label{eq:polydisc-uniform-enlargement}
 \overline{\Q_s^{[4]}(z)}
 \subset \Q_{A_*s}^{[1/2]}(z)
 \subset \Q_s^*(z)\Subset\Omega.
\end{equation}
In particular, every function used in the infimum defining $\G_{q,s}$ is
defined on a neighborhood of $\overline{\Q_s(z)}$.
No inclusion between $\Q_s$ and $\Q_{2s}$ is assumed.
An admissible $\rho$-lattice is a sequence $\Lambda=\{a_j\}$ for which
some fixed $0<c<1$ with $C_{\rm sc}c^{1/m_\Omega}<1/2$ satisfies
\begin{equation}\label{eq:admissible-lattice}
 \Omega=\bigcup_j\Q_\rho(a_j),\qquad
 \Q_{c\rho}(a_j)\cap\Q_{c\rho}(a_k)=\varnothing\quad(j\ne k),
 \qquad
 \sum_j\mathbf1_{\Q_{C\rho}(a_j)}\le N_C
\end{equation}
for every fixed $C\ge1$ satisfying $C\rho<\rho_0$.  The overlap bound
$N_C<\infty$ is independent of the lattice indices.
The condition on $c$ and \eqref{eq:polydisc-scale-gauge} ensure
$\overline{\Q_{c\rho}(z)}\subset\Q_\rho^{[1/2]}(z)$ in the boundary
collar.  We impose the same inclusion on the core charts below.
All radii below are chosen so small that every fixed enlargement used in
the argument is still below $\rho_0$.
Complete the boundary covering on a compact core by finitely many
Euclidean coordinate polydiscs.  Choose each inner polydisc together
with concentric affine dilates of factors $2,3,4$, all compactly
contained in $\Omega$, and a larger coordinate polydisc containing
their closures.  Use these nested polydiscs for the core lattice
cells and their enlargements.  The prescribed scales are fixed before
this finite completion.  To define $\Q_s(z)$ also at nonlattice core
points, take a finite Borel partition subordinate to the inner charts.
For a center $z$ in one part, use the Euclidean polydisc centered at
$z$ with the prescribed nested radii of that chart.  The outer chart
contains all these polydiscs.  Thus $\Q_s(z)$ and $\V_s(z)$ are defined
at every core point, and the lattice construction selects finitely
many of these cells.  The geometric enlargement constants are
first fixed in the boundary collar.  Shrink the inner core cells if necessary so
that the inclusions in \eqref{eq:polydisc-uniform-enlargement} and the
engulfing inclusions used below also hold on the overlap with the boundary
covering.  This overlap lies in a compact subset.  Thus only finitely
many cells and constants are added.  All local affine maps, submean
estimates, and restriction estimates below use polydiscs also on the
core.  When a formula uses $u_k$ and $c_0\tau_k$ at a core cell, these
mean its fixed coordinate directions and polydisc side lengths.

For a fixed $C\ge1$, a sublattice
$\Lambda_0=\{a_j:j\in J_0\}$ is called $C$-separated if the sets
$\Q_{C\rho}(a_j)$, $j\in J_0$, are pairwise disjoint.  The phrase
\emph{sufficiently separated} means $C$-separated for one fixed $C$
large enough to contain every enlargement used in the stated result.
The value of $C$ is independent of the lattice indices.

The diagonal kernel estimates
\cite{McNeal1994,NikolovPflug2003} and the basis comparisons above give
\begin{equation}\label{eq:volume-product}
 \V_\rho(z)\asymp
 \prod_{k=1}^n\tau_k(z,\rho\delta(z))^2,
 \qquad
 K(z,z)\asymp\V_\rho(z)^{-1}.
\end{equation}
For the finitely many core cells, the side lengths in the volume
product are those of their Euclidean coordinate polydiscs.  The kernel
comparison there follows by compactness, with the fixed radii included
in the constants.
The engulfing and scale comparisons for these polydiscs are given in
\cite[Section~2]{Jasiczak2010} and
\cite[Section~2]{GanHuKhan2022}.  If $w\in\Q_\rho(z)$, then
\begin{equation*}
 \Q_\rho(w)\subset\Q_{C\rho}(z),
 \qquad
 \V_\rho(w)\asymp\V_\rho(z).
\end{equation*}
The constants are uniform in $z,w$.

For a $(0,1)$-form $\omega$, put
\[
 \omega_k(z)=
 \omega_z\bigl(\overline{u_k(z,\delta(z))}\bigr).
\]
Define its dual McNeal norm by
\begin{equation}\label{eq:dual-McNeal}
 |\omega(z)|_{\Omega}
 =\left(\sum_{k=1}^n
 |\omega_k(z)|^2\tau_k(z,\delta(z))^2\right)^{1/2}.
\end{equation}
The primal directional norms for the admissible extremal bases are
uniformly equivalent by
\cite[Proposition~3.2]{NikolovPflugThomas2013}.  Duality and the
equivalence of the finite-dimensional $\ell^1$ and $\ell^2$ norms give
the same conclusion for \eqref{eq:dual-McNeal}.
On a fixed compact core of $\Omega$, we use the Euclidean form norm.
This gives a norm on all of $\Omega$, equivalent under any two such
choices.

The measurable basis choice also makes the local quantities used
below measurable. Indeed, the set of pairs $(z,w)$ with
$w\in\Q_s(z)$ is Borel. Thus $\V_s$ and $\M_{q,s}f$ are measurable
by integration. For $\G_{q,s}f$, the infimum in
\eqref{eq:intro-ida} may be taken over polynomials whose coefficients
have rational real and imaginary parts. To see this, expand an
admissible holomorphic function in the affine coordinates of
$\Q_s^*(z)$. Its Taylor polynomials converge uniformly on
$\overline{\Q_s(z)}$. Rewrite these polynomials in the original
coordinates on $\C^n$. Then approximate their coefficients by
numbers in $\mathbb Q+i\mathbb Q$. Conversely, every polynomial is admissible.
Hence the same infimum is countable and is measurable.
The form norm in \eqref{eq:dual-McNeal} is measurable as well.

For $u\in L^p(\Omega)$ and $v\in L^{p'}(\Omega)$, use the pairing
$\langle u,v\rangle=\int_\Omega u\overline v\dd v$.
Let $\mathcal P(\Omega)$ denote the restrictions to
$\Omega$ of holomorphic polynomials on $\mathbb C^n$.  Its density in
$A^p(\Omega)$ is proved in \cref{prop:domain-ledger}.

\begin{lemma}\label{lem:kernel-span-density}
Let $1<p<\infty$. Every kernel section $K(\cdot,z)$ belongs to
$A^p(\Omega)$, and their algebraic span $\Gamma$ is dense in $A^p(\Omega)$.
The Bergman projection satisfies the pairing identity
\eqref{eq:projection-pairing} on $L^p\times L^{p'}$.
\end{lemma}

\begin{proof}
Fix $a\in\Omega$. Choose a smooth nonnegative radial function $\chi_a$
supported in a Euclidean ball centered at $a$, with
$\int_\Omega\chi_a\dd v=1$. The antiholomorphic mean-value identity gives
\[
 P\chi_a(z)=\int_\Omega K(z,w)\chi_a(w)\dd v(w)=K(z,a).
\]
Since $P$ is bounded on $L^p$, this proves $K(\cdot,a)\in A^p$.
The projection is also bounded on $L^{p'}$. Let $L\in(A^p)^*$.  Hahn--Banach and
$L^p$ duality give $h\in L^{p'}$ such that
\[
 L(g)=\int_\Omega g\overline h\dd v.
\]
To justify the projection pairing at these exponents, choose
$u_m\in L^2\cap L^p$ and $v_m\in L^2\cap L^{p'}$ such that
$u_m\to u$ in $L^p$ and $v_m\to v$ in $L^{p'}$.
Self-adjointness on $L^2$ gives
\[
 \int_\Omega Pu_m\,\overline{v_m}\dd v
 =\int_\Omega u_m\,\overline{Pv_m}\dd v.
\]
For the left side, H\"older's inequality gives
\[
 \left|\int_\Omega Pu_m\,\overline{v_m}\dd v
       -\int_\Omega Pu\,\overline v\dd v\right|
 \le\|P(u_m-u)\|_p\|v_m\|_{p'}
    +\|Pu\|_p\|v_m-v\|_{p'}\longrightarrow0.
\]
The right side is handled in the same way. Thus
\begin{equation}\label{eq:projection-pairing}
 \langle Pu,v\rangle=\langle u,Pv\rangle,
 \qquad u\in L^p,\quad v\in L^{p'}.
\end{equation}
Since $Pg=g$ for $g\in A^p$, this replaces $h$ by $Ph\in A^{p'}$
in the representation of $L$. Write $h$ for this holomorphic representative.
This pairing identifies $(A^p)^*$ with $A^{p'}$, with equivalent norms.
For every $z\in\Omega$, the function $\chi_z$ chosen above gives
\[
 \int_\Omega K(w,z)\overline{h(w)}\dd v(w)
 =\langle P\chi_z,h\rangle
 =\langle\chi_z,Ph\rangle=\overline{h(z)}.
\]
If $L$ vanishes on every $K(\cdot,z)$, then
\[
 0=L(K(\cdot,z))=\overline{h(z)}
 \qquad(z\in\Omega).
\]
Thus $L=0$.  The annihilator of $\Gamma$ is trivial, so $\Gamma$ is
dense.

\end{proof}

\begin{lemma}\label{lem:geometry-package}
Let $1<p<\infty$. Fix an admissible radius $\rho>0$ such that
$A_*\rho<\rho_0$. There is a sequence
$\Lambda=\{a_j\}_{j\ge1}\subset\Omega$ such that:
\begin{enumerate}[label=\textup{(\roman*)}]
\item $\Omega=\bigcup_j\Q_\rho(a_j)$;
\item the balls $\Q_{c\rho}(a_j)$ are pairwise disjoint;
\item every fixed enlargement $\{\Q_{C\rho}(a_j)\}$ with
$C\rho<\rho_0$ has finite overlap;
\item $V_j:=\V_\rho(a_j)\asymp K(a_j,a_j)^{-1}$;
\item for every sufficiently separated sublattice,
\begin{equation}\label{eq:kernel-synthesis}
 S_p(c)=\sum_jc_j
 \frac{K(\cdot,a_j)}{K(a_j,a_j)^{1-1/p}}
 \quad\text{satisfies}\quad
 \|S_p(c)\|_{A^p}\lesssim\|c\|_{\ell^p};
\end{equation}
\item for every fixed $C\ge1$ with $C\rho<\rho_0$ and every
$z\in\Q_{C\rho}(a_j)$,
\begin{equation}\label{eq:kernel-lower-finite}
 |K(z,a_j)|\asymp K(a_j,a_j)\asymp V_j^{-1}.
\end{equation}
\end{enumerate}
Every lattice is a finite union of sublattices for which the balls
$\Q_\rho^*(a_j)$ are pairwise disjoint and
\eqref{eq:kernel-synthesis} holds.
\end{lemma}

\begin{proof}
Take a maximal family of pairwise disjoint balls
$\Q_{c\rho}(a_j)$. It is countable. Each ball contains a point with
rational real and imaginary coordinates, and the balls are disjoint.
For every
$x\in\Omega$, maximality gives an index $j$ such that
\[
 \Q_{c\rho}(x)\cap\Q_{c\rho}(a_j)\ne\varnothing.
\]
Engulfing then gives $x\in\Q_{C_0c\rho}(a_j)$, where $C_0$ is fixed.
Decrease $c$ at the construction stage so that
$C_{\rm sc}(C_0c)^{1/m_\Omega}<1$.  The scale-gauge estimate gives
\[
 x\in\Q_{C_0c\rho}(a_j)\subset\Q_\rho(a_j),
 \qquad \Omega=\bigcup_j\Q_\rho(a_j).
\]
Thus (i) and (ii) hold.

Fix $C\ge1$ with $C\rho<\rho_0$. In the boundary collar, let
$L\ge1$ be a uniform engulfing constant. Require
$L\rho_0<\rho_{\rm geom}$. If $z\in\Q_{C\rho}(a_j)$, then
\[
 \Q_{c\rho}(a_j)\subset\Q_{LC\rho}(z),
 \qquad
 \V_\rho(a_j)\asymp\V_\rho(z),
 \qquad LC\rho<\rho_{\rm geom}.
\]
The small balls on the left are disjoint. If $I$ is any finite set of
such indices, scale comparison within the geometric cutoff
gives
\begin{align*}
 \#I\,\V_\rho(z)
 &\lesssim\sum_{j\in I}v(\Q_{c\rho}(a_j))
 =v\left(\bigsqcup_{j\in I}\Q_{c\rho}(a_j)\right)\\
 &\le v(\Q_{LC\rho}(z))\lesssim\V_\rho(z).
\end{align*}
The constants are independent of $I$ and $z$.  Taking the supremum
over all finite $I$ proves (iii) in the collar. The finitely many
core charts satisfy the same assertion by their prescribed nested
polydiscs, with constants for the fixed scales. Formula
\eqref{eq:volume-product} gives (iv).

For (vi), fix a sufficiently small reference scale
$0<s_*<\rho_{\rm geom}$ before
the final choice of $\rho_0$.  Write $\mathbb D$ for the unit disc and
put
\[
 D_a\zeta=c_0\sum_{k=1}^n
 \tau_k(a,s_*\delta(a))\zeta_ku_k(a,s_*\delta(a)),
 \qquad F_a(\zeta)=K(a+D_a\zeta,a).
\]
The kernel Cauchy--Schwarz inequality and volume comparison give
\[
 \sup_{\zeta\in\mathbb D^n}|F_a(\zeta)|
 \le\sup_{w\in\Q_{s_*}(a)}K(w,w)^{1/2}K(a,a)^{1/2}
 \lesssim K(a,a).
\]
Cauchy's estimate in each coordinate gives
\[
 \sup_{\zeta\in\frac12\mathbb D^n}
 |\partial_{\zeta_k}F_a(\zeta)|\lesssim K(a,a),
 \qquad 1\le k\le n.
\]
By \eqref{eq:polydisc-scale-gauge}, if $z\in\Q_\theta(a)$, then
\[
 \|D_a^{-1}(z-a)\|_\infty
 \le C_{\rm sc}(\theta/s_*)^{1/m_\Omega}.
\]
Take $\theta$ so small that the right side is at most $1/2$.
Integrate the derivatives of $F_a$ along the segment from $0$ to
$D_a^{-1}(z-a)$.  This yields
\[
 |K(z,a)-K(a,a)|
 \le\sum_{k=1}^n |(D_a^{-1}(z-a))_k|
 \sup_{\frac12\mathbb D^n}|\partial_{\zeta_k}F_a|
 \le C_0\theta^{1/m_\Omega}K(a,a).
\]
Choose $\rho_0$ so that
\[
 0<\rho_0<s_*,\qquad L\rho_0<\rho_{\rm geom},\qquad
 C_{\rm sc}(\rho_0/s_*)^{1/m_\Omega}<\tfrac12,\qquad
 C_0\rho_0^{1/m_\Omega}<\tfrac12.
\]
Thus the scale comparison with the fixed reference $s_*$ remains
valid after the working cutoff has been reduced. All preceding
conditions involving $\rho_0$ use this final choice.
Taking $\theta=C\rho<\rho_0$ proves (vi) in the boundary collar.
On the compact core, continuity of $K$ near the diagonal and
$K(a,a)>0$ give the same inequality after the outer coordinate
polydiscs are chosen sufficiently small.  Compactness makes this
choice uniform for the finite core family.  Thus (vi) holds everywhere.

We prove (v).  Put
\[
 e_{j,p}(z)=\frac{K(z,a_j)}{K(a_j,a_j)^{1-1/p}}.
\]
By (iv) and (vi), for every fixed admissible $C$,
\begin{equation}\label{eq:atom-normalization}
 |e_{j,p}(z)|\asymp V_j^{-1/p}
 \quad(z\in\Q_{C\rho}(a_j)).
\end{equation}
The kernels in this formula belong to $A^p$.  Indeed, for each fixed
$a\in\Omega$, choose a smooth nonnegative radial function $\chi_a$
supported in a Euclidean ball centered at $a$, with
$\int\chi_a\dd v=1$.  The antiholomorphic mean-value identity gives
\[
 P\chi_a(z)=\int K(z,w)\chi_a(w)\dd v(w)=K(z,a).
\]
Boundedness of $P$ on $L^p$ proves the assertion.  No uniform bound on
$\chi_a$ is used below.

Let $c=(c_j)$ have finite support.  For $\phi\in L^{p'}(\Omega)$,
$\|\phi\|_{p'}\le1$, set $h=P\phi$.  Since $P$ is bounded on $L^{p'}$,
\[
 \|h\|_{A^{p'}}\lesssim1.
\]
The pairing identity \eqref{eq:projection-pairing} and the radial
mean-value identity give
\[
 \int K(z,a_j)\overline{\phi(z)}\dd v(z)
 =\int P\chi_{a_j}\,\overline\phi\dd v
 =\int\chi_{a_j}\,\overline h\dd v
 =\overline{h(a_j)}.
\]
Consequently,
\begin{align}
 \left|\int_\Omega S_p(c)\overline\phi\dd v\right|
 &=\left|\sum_j
 \frac{c_j\overline{h(a_j)}}{K(a_j,a_j)^{1-1/p}}\right|\sum_j|c_j|V_j^{1/p'}|h(a_j)|.                 \label{eq:synthesis-dual}
\end{align}
The balls $\Q_{c\rho}(a_j)$ are disjoint.  The holomorphic submean
inequality implies
\begin{equation*}
 \sum_jV_j|h(a_j)|^{p'}
 \lesssim\sum_j\int_{\Q_{c\rho}(a_j)}|h|^{p'}\dd v
 \le \|h\|_{A^{p'}}^{p'}.
\end{equation*}
Apply H\"older's inequality to \eqref{eq:synthesis-dual}.  Take the
supremum over $\phi$.  Then
\[
 \|S_p(c)\|_{A^p}\lesssim\|c\|_{\ell^p}.
\]
Let $P_m$ be the projection onto the first $m$ coordinates of $\ell^p$.
For $c\in\ell^p$,
\[
 \|S_p(P_m c)-S_p(P_kc)\|_{A^p}
 \lesssim\|P_mc-P_kc\|_{\ell^p}\longrightarrow0.
\]
Hence the finite sums converge in $A^p$.  Their limit defines the
bounded map in \eqref{eq:kernel-synthesis}.  Take $c$ to be the $j$th
unit coordinate vector.  Equation \eqref{eq:atom-normalization} gives
\[
 \|e_{j,p}\|_{A^p}\asymp1.
\]
Finally, let two indices be adjacent when
$\Q_\rho^*(a_j)\cap\Q_\rho^*(a_k)\ne\varnothing$.  Engulfing places all
neighbors of $a_j$ in one fixed enlargement of $\Q_\rho(a_j)$.
Packing and (iii) bound the vertex degree by a constant $N_*$.  A greedy
coloring uses at most $N_*+1$ colors.  Balls of one color are disjoint.
If a larger separation is required, replace $\Q_\rho^*$ by that fixed
enlargement before coloring.  The number of colors remains finite.  See
\cite{McNeal1994,McNealStein1994} for the kernel and projection estimates
used above.
\end{proof}

For later use, define the normalized kernel atom
\[
 e_{j,p}=\frac{K(\cdot,a_j)}{K(a_j,a_j)^{1-1/p}}.
\]
Then $\|e_{j,p}\|_{A^p}\asymp1$ and
\begin{align*}
 \|fe_{j,p}\|_{L^q(\Q_\rho(a_j))}
 &\asymp
 V_j^{1/q-1/p}\M_{q,\rho}f(a_j),
 \\
 \inf_{h\in\Hol(\Q_{\rho_+}^*(a_j))}
 \|(f-h)e_{j,p}\|_{L^q(\Q_{\rho_+}(a_j))}
 &\asymp
 V_j^{1/q-1/p}\G_{q,\rho_+}f(a_j).
\end{align*}
Thus multiplication and analytic distance have the same weight.
The formula applies to all the stated values of $p,q$, and $r$.
\subsection{Absolutely summing operators}

Recall the weak $r$-norm from \eqref{eq:weak-r-norm}.
The definition of absolute summability gives
\[
 \left(\sum_{k=1}^N\|Tx_k\|^r\right)^{1/r}
 \le\pi_r(T)w_r(x_1,\ldots,x_N).
\]

For Banach spaces $X_j$ and $1\le s<\infty$, define
\[
 \left(\bigoplus_jX_j\right)_{\ell^s}
 =\left\{(x_j):\sum_j\|x_j\|_{X_j}^s<\infty\right\},
 \qquad
 \|(x_j)\|=\left(\sum_j\|x_j\|_{X_j}^s\right)^{1/s}.
\]
The same notation applies to a finite family of spaces.
The following facts will be used throughout the paper.

\begin{lemma}\label{lem:summing-calculus}
Let $X,X_0,Y,Y_0$ be Banach spaces and $1\le r<\infty$.
\begin{enumerate}[label=\textup{(\roman*)}]
\item If $T\in\Piabs_r(X,Y)$, $A\in\mathcal L(X_0,X)$, and
$B\in\mathcal L(Y,Y_0)$, then
\begin{equation}\label{eq:summing-ideal-property}
 BTA\in\Piabs_r(X_0,Y_0),
 \qquad
 \pi_r(BTA)\le\|B\|\pi_r(T)\|A\|.
\end{equation}
\item Let $1\le s<\infty$ and
$T_\ell\in\mathcal L(X,Y_\ell)$ for $1\le\ell\le m$.  Define
\[
 \mathscr T x=(T_1x,\ldots,T_mx)
 \in\left(\bigoplus_{\ell=1}^mY_\ell\right)_{\ell^s}.
\]
Then $\mathscr T$ is absolutely $r$-summing if and only if every
$T_\ell$ is absolutely $r$-summing.  In this case,
\begin{equation}\label{eq:finite-graph-summing}
 \max_{1\le\ell\le m}\pi_r(T_\ell)
 \le\pi_r(\mathscr T)
 \le\sum_{\ell=1}^m\pi_r(T_\ell).
\end{equation}
\item Let $T\in\Piabs_r(X,Y)$.  Let
$(\mathcal U_t)_{t\in[0,1]}$ and
$(\mathcal V_t)_{t\in[0,1]}$ be strongly measurable families of
contractions on $Y$ and $X$, respectively.  If
\[
 Sx=\int_0^1\mathcal U_tT\mathcal V_tx\dd t,
\]
where the Bochner integral exists for every $x\in X$, then
$S\in\Piabs_r(X,Y)$ and
\begin{equation*}
 \pi_r(S)\le\pi_r(T).
\end{equation*}
\item Let $T\in\Piabs_r(X,Y)$.  If $(x_j)$ converges weakly to $x$ in $X$,
then
\begin{equation}\label{eq:summing-complete-continuity}
 \|Tx_j-Tx\|_Y\longrightarrow0.
\end{equation}
\end{enumerate}
Here $\mathcal L(X,Y)$ denotes the bounded linear operators from $X$
to $Y$.
\end{lemma}

\begin{proof}
For a finite family $(x_k)_{k=1}^N\subset X_0$,
\begin{align*}
 \left(\sum_k\|BTAx_k\|^r\right)^{1/r}
 &\le\|B\|\pi_r(T)w_r(Ax_1,\ldots,Ax_N).
\end{align*}
For $\phi\in B_{X^*}$, the functional
$\phi\circ A/\|A\|$ belongs to $B_{X_0^*}$ when $A\ne0$.  Hence
\[
 w_r(Ax_1,\ldots,Ax_N)
 \le\|A\|w_r(x_1,\ldots,x_N).
\]
This proves \eqref{eq:summing-ideal-property}.

Assume first that $\mathscr T$ is absolutely $r$-summing.  Let
$Q_\ell$ be the coordinate projection from the finite
$\ell^s$-sum onto $Y_\ell$.  Since $\|Q_\ell\|=1$,
\[
 \pi_r(T_\ell)=\pi_r(Q_\ell\mathscr T)\le\pi_r(\mathscr T).
\]
Thus every $T_\ell$ is absolutely $r$-summing.  Conversely, assume
that every $T_\ell$ is absolutely $r$-summing.  Minkowski's inequality
gives
\begin{align*}
 \left(\sum_k\|\mathscr Tx_k\|_{\ell^s(Y_\ell)}^r\right)^{1/r}
 &\le
 \sum_{\ell=1}^m
 \left(\sum_k\|T_\ell x_k\|_{Y_\ell}^r\right)^{1/r}\leq\left(\sum_{\ell=1}^m\pi_r(T_\ell)\right)
 w_r(x_1,\ldots,x_N).
\end{align*}
This proves \eqref{eq:finite-graph-summing}.

For part (iii), apply Minkowski's integral inequality:
\begin{align*}
 \left(\sum_k\|Sx_k\|^r\right)^{1/r}
 &\le
 \int_0^1
 \left(\sum_k
 \|\mathcal U_tT\mathcal V_tx_k\|^r\right)^{1/r}\dd t\\
 &\le
 \pi_r(T)\int_0^1
 w_r(\mathcal V_tx_1,\ldots,\mathcal V_tx_N)\dd t\\
 &\le\pi_r(T)w_r(x_1,\ldots,x_N).
\end{align*}
This proves part (iii).

For part (iv), Pietsch domination gives a probability measure $\nu$ on
$B_{X^*}$ such that
\begin{equation}\label{eq:intro-Pietsch-domination}
 \|Tx\|_Y^r
 \le\pi_r(T)^r\int_{B_{X^*}}|\phi(x)|^r\dd\nu(\phi).
\end{equation}
If $(x_j)$ converges weakly to zero, then it is bounded and
\[
 \phi(x_j)\longrightarrow0,
 \qquad
 |\phi(x_j)|^r\le\sup_j\|x_j\|^r.
\]
Dominated convergence in \eqref{eq:intro-Pietsch-domination} gives
$\|Tx_j\|\to0$.  Translation by $x$ proves
\eqref{eq:summing-complete-continuity}.
\end{proof}
\section{Diagonal ideals and absolutely summing Carleson embeddings}\label{sec:carleson}

\subsection{Block reduction and Carleson embeddings}

Let $\mu$ be a positive locally finite Borel measure on $\Omega$.  Define
\begin{equation*}
 \mathcal D(J_\mu)=A^p(\Omega)\cap L^q(\Omega,\mu),
 \qquad
 J_\mu g=g\quad(g\in\mathcal D(J_\mu)).
\end{equation*}
We write $J_\mu:A^p\to L^q(\mu)$ only when
$\mathcal D(J_\mu)=A^p$ and the inclusion is bounded.

For an admissible lattice $\Lambda=\{a_j\}$, put
\begin{equation*}
 \beta_j(\mu)
 =\frac{\mu(\Q_\rho(a_j))^{1/q}}{V_j^{1/p}}.
\end{equation*}

We use the block principle of Fan, He, Wang, and Zeng
\cite[Proposition~4.2]{FanHeWangZeng2026}.  We state it with normalized
block norms.  We include the proof by Pietsch domination to specify the
constants in the Carleson and Hankel arguments.
In the block argument, $p'=\infty$ when $p=1$.

\begin{lemma}\label{lem:block-principle}
Let $1\le p,q,r<\infty$.  Let $X_j,Y_j$ be Banach spaces.  Suppose
$B_j:X_j\to Y_j$ are absolutely $1$-summing and
$\sup_j\pi_1(B_j)\le1$.  For $b=(b_j)$ define
\[
 \mathbb B_b:(\oplus_jX_j)_{\ell^p}
 \longrightarrow(\oplus_jY_j)_{\ell^q},
 \qquad
 \mathbb B_b(x_j)=(b_jB_jx_j).
\]
If $D_b\in\Piabs_r(\ell^p,\ell^q)$, then
\begin{equation}\label{eq:block-principle}
 \pi_r(\mathbb B_b)
 \le \pi_r(D_b).
\end{equation}
If $\pi_1(B_j)\le C_0$, multiply the right side by $C_0$.
\end{lemma}

\begin{proof}
First suppose that there are only $m$ coordinates.  Pietsch domination
gives probability measures $\nu_j$ on $B_{X_j^*}$ such that
\begin{equation*}
 \|B_jx\|_{Y_j}
 \le\int_{B_{X_j^*}}|x^*(x)|\dd\nu_j(x^*).
\end{equation*}
Let $x^{(1)},\ldots,x^{(N)}\in(\oplus X_j)_{\ell^p}$ and write
$x^{(k)}=(x_j^{(k)})_j$.  By Minkowski's inequality,
\begin{align*}
 \|\mathbb B_bx^{(k)}\|_{\ell^q(Y_j)}
 &\le
 \left\|\left\{
 |b_j|\int|x^*(x_j^{(k)})|\dd\nu_j(x^*)
 \right\}_j\right\|_{\ell^q}\\
 &\le
 \int
 \|D_b( x_j^*(x_j^{(k)}) )_j\|_{\ell^q}
 \dd\nu((x_j^*)_j),
\end{align*}
where $\nu=\prod_{j=1}^m\nu_j$.  Apply Minkowski once more in $\ell^r_N$
and then the $r$-summing inequality for $D_b$:
\begin{align*}
 \left(\sum_{k=1}^N
 \|\mathbb B_bx^{(k)}\|^r\right)^{1/r}
 &\le\pi_r(D_b)
 \int
 \sup_{\|\alpha\|_{\ell^{p'}}\le1}
 \left(\sum_{k=1}^N
 \left|\sum_j\alpha_jx_j^*(x_j^{(k)})\right|^r
 \right)^{1/r}\dd\nu\\
 &\le\pi_r(D_b)
 \sup_{x^*\in B_{\left((\oplus_jX_j)_{\ell^p}\right)^*}}
 \left(\sum_{k=1}^N|x^*(x^{(k)})|^r\right)^{1/r}.
\end{align*}
For completeness, fix $(x_j^*)$ and
$\alpha\in B_{\ell^{p'}}$.  The functional
\[
 F_{\alpha,x^*}(x)
 =\sum_{j=1}^m\alpha_jx_j^*(x_j)
\]
satisfies
\[
 |F_{\alpha,x^*}(x)|
 \le\sum_{j=1}^m|\alpha_j|\|x_j\|_{X_j}
 \le\|\alpha\|_{p'}
 \bigl\|(\|x_j\|_{X_j})\bigr\|_p
 \le\|x\|_{(\oplus X_j)_{\ell^p}}.
\]
Thus $F_{\alpha,x^*}$ belongs to the dual unit ball.  This proves the
estimate on finitely many coordinates with constant one.

Now consider infinitely many coordinates.  Since
$\|B_j\|\le\pi_1(B_j)\le1$,
\begin{align*}
 \|\mathbb B_bx\|_{\ell^q(Y_j)}
 &\le\|(|b_j|\|x_j\|_{X_j})_j\|_{\ell^q}\\
 &\le\|D_b\|\|x\|_{\ell^p(X_j)}.
\end{align*}
Hence $\mathbb B_b$ is well defined.  Let $P_m$ and $Q_m$ denote the
first-$m$ coordinate projections.  The finite estimate gives
\[
 \pi_r(Q_m\mathbb B_bP_m)\le\pi_r(D_b)
 \qquad(m\in\N).
\]
For a fixed finite family $(x^{(k)})_{k=1}^N$, first replace every
$x^{(k)}$ by $P_mx^{(k)}$.  Since $\|P_m\|\le1$,
\[
 w_r(P_mx^{(1)},\ldots,P_mx^{(N)})
 \le w_r(x^{(1)},\ldots,x^{(N)}).
\]
Moreover,
\[
 Q_m\mathbb B_bP_mx^{(k)}
 =\mathbb B_bP_mx^{(k)}\longrightarrow\mathbb B_bx^{(k)}
\]
in the target norm. Let $m\to\infty$ in the inequality on finitely
many coordinates. This gives the $r$-summing inequality for $\mathbb B_b$.
This proves \eqref{eq:block-principle}.  Scaling gives the bound
with $C_0$.
\end{proof}

We specify the measurable cells used in the factorization.  With $c$
fixed as in \eqref{eq:admissible-lattice}, put
\begin{align*}
 U&=\bigcup_j\Q_{c\rho}(a_j),\qquad
 P_j=\Q_\rho(a_j)\setminus\bigcup_{k<j}\Q_\rho(a_k),\\
 E_j&=\Q_{c\rho}(a_j)\cup(P_j\setminus U).
\end{align*}
The small polydiscs are pairwise disjoint and are contained in their
corresponding $\rho$-polydiscs.  The $P_j$ form a partition of $\Omega$.
Hence the $E_j$ are measurable and satisfy
\begin{equation}\label{eq:partition-cells}
 \Omega=\bigsqcup_jE_j,
 \qquad
 \Q_{c\rho}(a_j)\subset E_j\subset\Q_{C\rho}(a_j),
 \qquad v(E_j)\asymp V_j.
\end{equation}
Each cell is relatively compact in $\Omega$.  Since $\mu$ is locally
finite, $\mu(E_j)<\infty$.
Choose $C$ so large that the cell inclusion in
\eqref{eq:partition-cells} holds and
$C_{\rm sc}C^{-1/m_\Omega}<1/2$. Then
\eqref{eq:polydisc-scale-gauge} gives
$\Q_\rho(a_j)\subset\Q_{C\rho}(a_j)$.
Choose $C_1>C$ large enough for the scaled inclusion in
\eqref{eq:local-scaled-inclusion} below.  As required by the convention
in \eqref{eq:admissible-lattice}, take the global lattice radius so that
$C_1\rho<\rho_0$.  Let
\[
 X_j=A^p(\Q_{C_1\rho}(a_j)),
 \qquad
 \|u\|_{X_j}=\|u\|_{L^p(\Q_{C_1\rho}(a_j))}.
\]
Here $A^p(U)=L^p(U,v)\cap\Hol(U)$ for an open set $U$.  Define
\[
 Rg=(g|_{\Q_{C_1\rho}(a_j)})_j.
\]
Finite overlap gives
\begin{equation*}
 \|Rg\|_{\ell^p(X_j)}^p
 =\sum_j
 \int_{\Q_{C_1\rho}(a_j)}|g|^p\dd v
 \lesssim\|g\|_{A^p}^p.
\end{equation*}

\begin{lemma}\label{lem:local-restriction}
Let $1\le p,q<\infty$.  Let
\[
 B_j:X_j\longrightarrow L^q(E_j,\mu),
 \qquad B_ju=u|_{E_j}.
\]
Then
\begin{equation}\label{eq:local-one-summing}
 \pi_1(B_j)
 \lesssim\frac{\mu(E_j)^{1/q}}{V_j^{1/p}}.
\end{equation}
The constant is independent of $j$ and $\mu$.
\end{lemma}

\begin{proof}
If $\mu(E_j)=0$, then $B_j=0$.  Assume $\mu(E_j)>0$.  Put
$\delta_j=\delta(a_j)$.  Use the extremal basis at scale
$C_1\rho\delta_j$.  Let $\mathbb D=\{\zeta\in\C:|\zeta|<1\}$.
The affine map
\[
 \Phi_j(\zeta)
 =a_j+c_0\sum_{k=1}^n
 \tau_k(a_j,C_1\rho\delta_j)\zeta_ku_k(a_j,C_1\rho\delta_j)
\]
maps $\mathbb D^n$ onto $\Q_{C_1\rho}(a_j)$.  For
$z\in\Q_{C\rho}(a_j)$, \eqref{eq:polydisc-scale-gauge} gives
\[
 \|\Phi_j^{-1}(z)\|_\infty
 =N_{a_j,C_1\rho}(z-a_j)
 \le C_{\rm sc}(C/C_1)^{1/m_\Omega}.
\]
Fix $0<\theta<\theta_1<1$ and choose $C_1/C$ so large that the last
bound is less than $\theta$.  On the core use the prescribed nested
coordinate polydiscs.  Thus
\begin{equation}\label{eq:local-scaled-inclusion}
 \Phi_j^{-1}(E_j)
 \subset \Phi_j^{-1}(\Q_{C\rho}(a_j))
 \subset\theta\mathbb D^n
 \Subset\theta_1\mathbb D^n\Subset\mathbb D^n
\end{equation}
for every $j$ \cite{McNeal1994,McNealStein1994}. The absolute value
of the real Jacobian determinant of $\Phi_j$, denoted by
$J_{\mathbb R}\Phi_j$, satisfies
\begin{equation*}
 J_{\mathbb R}\Phi_j
 =c_0^{2n}\prod_{k=1}^n
 \tau_k(a_j,C_1\rho\delta_j)^2
 \asymp V_j.
\end{equation*}
Write
\[
 u(\Phi_j(\zeta))
 =\sum_{\alpha\in\Nzero^n}c_\alpha\zeta^\alpha.
\]
For a multiindex $\alpha=(\alpha_1,\ldots,\alpha_n)$, write
$|\alpha|=\sum_k\alpha_k$ and
$\zeta^\alpha=\prod_k\zeta_k^{\alpha_k}$.  Let
$\ell_{j,\alpha}(u)=c_\alpha$ and
$m_{j,\alpha}(z)=(\Phi_j^{-1}z)^\alpha$ for
$z\in\Q_{C_1\rho}(a_j)$.
Cauchy's formula and the local $L^p$ submean estimate give
\[
 |c_\alpha|\le
 C V_j^{-1/p}\theta_1^{-|\alpha|}\|u\|_{X_j},
 \qquad \theta<\theta_1<1.
\]
More explicitly, Cauchy's formula on the torus
$|\zeta_k|=\theta_1$ gives
\[
 |c_\alpha|\theta_1^{|\alpha|}
 \le\sup_{\overline{\theta_1\mathbb D^n}}|u\circ\Phi_j|.
\]
The scaled submean inequality on finitely many polydiscs covering
$\overline{\theta_1\mathbb D^n}$ gives
\[
 \sup_{\overline{\theta_1\mathbb D^n}}|u\circ\Phi_j|
 \le C V_j^{-1/p}\|u\|_{X_j}.
\]
Thus
\begin{align*}
 \|\ell_{j,\alpha}\|_{X_j^*}
 &\le C V_j^{-1/p}\theta_1^{-|\alpha|},
 \\
 \|m_{j,\alpha}\|_{L^q(E_j,\mu)}
 &\le\theta^{|\alpha|}\mu(E_j)^{1/q}.
\end{align*}
On $E_j$,
\[
 B_ju=\sum_\alpha
 \ell_{j,\alpha}(u)\,m_{j,\alpha}|_{E_j},
\]
and
\[
 \|\ell_{j,\alpha}\|\,
 \|m_{j,\alpha}\|_{L^q(E_j,\mu)}
 \le C(\theta/\theta_1)^{|\alpha|}
 \frac{\mu(E_j)^{1/q}}{V_j^{1/p}}.
\]
Since
\[
 \#\{\alpha\in\Nzero^n:|\alpha|=m\}
 =\binom{m+n-1}{n-1},
\]
we obtain
\begin{align*}
 \sum_\alpha\|\ell_{j,\alpha}\|
 \|m_{j,\alpha}\|_{L^q(E_j,\mu)}
 &\le C\frac{\mu(E_j)^{1/q}}{V_j^{1/p}}
 \sum_{m=0}^{\infty}\binom{m+n-1}{n-1}
 \left(\frac{\theta}{\theta_1}\right)^m\\
 &=C\left(1-\frac{\theta}{\theta_1}\right)^{-n}
 \frac{\mu(E_j)^{1/q}}{V_j^{1/p}}.
\end{align*}
The Taylor series converges uniformly on $\theta\mathbb D^n$.  The
convergent sum of norms above also gives convergence in operator norm
from $X_j$ to $L^q(E_j,\mu)$.  Thus the identity
\[
 B_j=\sum_{\alpha\in\Nzero^n}
 \ell_{j,\alpha}\otimes(m_{j,\alpha}|_{E_j})
\]
holds in operator norm.  This is a nuclear decomposition.  The nuclear
norm of $B_j$ is bounded by the right side of
\eqref{eq:local-one-summing}.  To see the $1$-summing bound directly,
let $u_1,\ldots,u_N\in X_j$.  Then
\begin{align*}
 \sum_{k=1}^N\|B_ju_k\|_{L^q(E_j,\mu)}
 &\le\sum_\alpha\|m_{j,\alpha}\|_{L^q(E_j,\mu)}
       \sum_{k=1}^N|\ell_{j,\alpha}(u_k)|\\
 &\le\left(\sum_\alpha\|\ell_{j,\alpha}\|
                \|m_{j,\alpha}\|_{L^q(E_j,\mu)}\right)
       w_1(u_1,\ldots,u_N).
\end{align*}
This proves \eqref{eq:local-one-summing}, including $q=1$.
\end{proof}

\begin{theorem}\label{thm:carleson}
Let $1<p<\infty$ and $1\le q,r<\infty$.  Let $\mu$ be a positive locally
finite Borel measure on $\Omega$.  Then
\begin{equation*}
 J_\mu\in\Piabs_r(A^p(\Omega),L^q(\Omega,\mu))
 \quad\Longleftrightarrow\quad
 \{\beta_j(\mu)\}_j\in\mathfrak d_{p,q}^{\,r}.
\end{equation*}
Moreover,
\begin{equation*}
 \pi_r(J_\mu)
 \asymp
 \|\{\beta_j(\mu)\}\|_{\mathfrak d_{p,q}^{\,r}}.
\end{equation*}
The constants depend only on $p,q,r$, the finite-type data, and the fixed
geometric parameters.  They are independent of $\mu$.
\end{theorem}

\begin{proof}
We first prove sufficiency.  By \eqref{eq:partition-cells},
\[
 L^q(\Omega,\mu)
 =\left(\bigoplus_jL^q(E_j,\mu)\right)_{\ell^q}.
\]
Since $E_j\subset\Q_{C\rho}(a_j)$, first compare $\beta_j(\mu)$ with
\[
 \beta_j^C(\mu)=\frac{\mu(\Q_{C\rho}(a_j))^{1/q}}{V_j^{1/p}}.
\]
Set
\[
 \mathcal N(j)=
 \{k:\Q_\rho(a_k)\cap\Q_{C\rho}(a_j)\ne\varnothing\}.
\]
The lattice covering gives
\[
 \Q_{C\rho}(a_j)
 \subset\bigcup_{k\in\mathcal N(j)}\Q_\rho(a_k).
\]
Engulfing and packing by the disjoint polydiscs
$\Q_{c\rho}(a_k)$ give
\[
 \sup_j\#\mathcal N(j)
 +\sup_k\#\{j:k\in\mathcal N(j)\}\le N_C,
 \qquad V_k\asymp V_j\quad(k\in\mathcal N(j)).
\]
Therefore
\begin{equation}\label{eq:beta-radius}
\begin{aligned}
\beta_j^C(\mu)
&\le
C\left(
  \sum_{k\in\mathcal N(j)}
  \frac{\mu(\Q_\rho(a_k))}{V_k^{q/p}}
 \right)^{1/q}
\le
C\sum_{k\in\mathcal N(j)}\beta_k(\mu).
\end{aligned}
\end{equation}
\Cref{lem:finite-neighbor} gives
\[
 \|\beta^C(\mu)\|_{\mathfrak d_{p,q}^{\,r}}
 \lesssim\|\beta(\mu)\|_{\mathfrak d_{p,q}^{\,r}}.
\]
This upper bound is enough for the factorization.  Write
\[
 B_j=\beta_j^C(\mu)\widehat B_j,
 \qquad \pi_1(\widehat B_j)\lesssim1,
\]
using \cref{lem:local-restriction}.  Set $\widehat B_j=0$ when
$\beta_j^C(\mu)=0$.  Let
\[
 \mathcal B(u_j)=(B_ju_j)_j.
\]
With the isometric identification above, the embedding factors as
\begin{equation}\label{eq:carleson-factor}
 A^p(\Omega)
 \xrightarrow{R}
 (\oplus_jX_j)_{\ell^p}
 \xrightarrow{\mathcal B}
 (\oplus_jL^q(E_j,\mu))_{\ell^q}.
\end{equation}
The first map is bounded by finite overlap.  Since
$B_j=\beta_j^C(\mu)\widehat B_j$, \cref{lem:block-principle} gives
\[
 \pi_r(\mathcal B)
 \lesssim\|\beta^C(\mu)\|_{\mathfrak d_{p,q}^{\,r}}.
\]
Hence
\begin{equation}\label{eq:carleson-upper}
 \pi_r(J_\mu)
 \lesssim\|\{\beta_j(\mu)\}\|_{\mathfrak d_{p,q}^{\,r}}.
\end{equation}
The map $\mathcal B$ in \eqref{eq:carleson-factor} is defined on the whole
direct sum.  Hence, for every $g\in A^p$,
\[
 \sum_j\int_{E_j}|g|^q\dd\mu<\infty.
\]
Thus $g\in L^q(\mu)$ and the extended map is the inclusion $J_\mu$.

For necessity, divide $\Lambda$ into finitely many sublattices such that
the polydiscs $Q_j:=\Q_\rho(a_j)$ are pairwise disjoint within each
sublattice and \eqref{eq:kernel-synthesis} holds.  Fix one sublattice.
Let $S_p:\ell^p\to A^p(\Omega)$ be the synthesis map in
\eqref{eq:kernel-synthesis}.  Put
\[
 Y=\left(\bigoplus_jL^q(Q_j,\mu)\right)_{\ell^q}.
\]
Restriction to the disjoint union of the $Q_j$ defines a contraction
$R_\mu:L^q(\mu)\to Y$. Set $B=R_\mu J_\mu S_p$.
Use the Rademacher functions
$r_j(t)=(-1)^{\lfloor2^jt\rfloor}$ for $0\le t<1$ and $j\ge1$;
their values at $t=1$ may be chosen arbitrarily. Define
\[
 U_t(c_j)=(r_j(t)c_j),
 \qquad
 V_t(u_j)=(r_j(t)u_j).
\]
Define on finitely supported sequences
\[
 \mathbb Dc=\{c_je_{j,p}|_{Q_j}\}_j.
\]
Enumerate the fixed sublattice.  Let $P_m$ and $Q_m$ be the first-$m$
coordinate projections.  For $c\in c_{00}$ and $j\le m$,
\[
 \bigl(Q_mV_tBU_tP_mc\bigr)_j
 =\sum_{k=1}^m r_j(t)r_k(t)c_k e_{k,p}|_{Q_j}.
\]
Let $\delta_{jk}$ be the Kronecker delta, equal to $1$ when $j=k$
and to $0$ otherwise. Since
$\int_0^1r_j(t)r_k(t)\dd t=\delta_{jk}$,
\begin{equation*}
 Q_m\mathbb DP_m
 =\int_0^1Q_mV_tBU_tP_m\dd t.
\end{equation*}
The maps $U_t$ and $V_t$ are isometries.  The ideal property and
\cref{lem:summing-calculus}(iii) give
\[
 \pi_r(Q_m\mathbb DP_m)\le\pi_r(B).
\]
The same estimate holds with $\{1,\ldots,m\}$ replaced by any finite set
$F$.  Write $P_F$ and $Q_F$ for the corresponding coordinate
projections.  Hence, for $k<m$ and $c\in\ell^p$,
\[
 \|Q_{\{k+1,\ldots,m\}}\mathbb D
   P_{\{k+1,\ldots,m\}}c\|_Y
 \le \pi_r(B)\|P_{\{k+1,\ldots,m\}}c\|_{\ell^p}.
\]
Thus $\mathbb D$ extends from $c_{00}$ to a bounded map
$\ell^p\to Y$.  Apply the estimate on finitely many coordinates to a finite family
and let $m\to\infty$.  Then
\[
 \pi_r(\mathbb D)\le\pi_r(B)
 \le\|S_p\|\pi_r(J_\mu).
\]
Put $\gamma_j=\|e_{j,p}\|_{L^q(Q_j,\mu)}$, which is finite because
$J_\mu$ is bounded.  If $\gamma_j>0$, define
\[
 \varphi_j(u)=
 \begin{cases}
 \displaystyle\gamma_j^{1-q}\int_{Q_j}
 u\,\overline{e_{j,p}}\,|e_{j,p}|^{q-2}\dd\mu,&q>1,\\[6pt]
 \displaystyle\int_{Q_j}u\,
 \frac{\overline{e_{j,p}}}{|e_{j,p}|}\dd\mu,&q=1.
 \end{cases}
\]
The integrands are set to zero where $e_{j,p}=0$.
H\"older's inequality gives $\|\varphi_j\|\le1$, and
$\varphi_j(e_{j,p})=\gamma_j$ gives equality.  Put $\varphi_j=0$ if
$\gamma_j=0$.  This also covers $q=1$.  The target need not be reflexive.
The block map
\[
 \Phi:Y\longrightarrow\ell^q,
 \qquad \Phi(u_j)=(\varphi_j(u_j))_j,
\]
has norm at most one, since
\[
 \|\Phi(u_j)\|_q^q
 =\sum_j|\varphi_j(u_j)|^q
 \le\sum_j\|u_j\|_{L^q(Q_j,\mu)}^q.
\]
Then $D_\gamma=\Phi\mathbb D$.  By \eqref{eq:atom-normalization},
\[
 \gamma_j
 \gtrsim\frac{\mu(Q_j)^{1/q}}{V_j^{1/p}}=\beta_j(\mu).
\]
The ideal property gives
\[
 \|\gamma\|_{\mathfrak d_{p,q}^{\,r}}
 \le\pi_r(\mathbb D).
\]
Since $0\le\beta_j(\mu)\lesssim\gamma_j$, solidity yields
\[
 \pi_r(D_{\beta})\lesssim\pi_r(J_\mu)\|S_p\|.
\]
For each color, extend its coefficient sequence by zero.  The sum of
these sequences is $\beta$.  Solidity and the triangle inequality give
\begin{equation}\label{eq:carleson-lower}
 \|\{\beta_j(\mu)\}\|_{\mathfrak d_{p,q}^{\,r}}
 \lesssim\pi_r(J_\mu).
\end{equation}
Combining \eqref{eq:carleson-upper} and
\eqref{eq:carleson-lower} proves the theorem.
\end{proof}

For a measurable symbol $f$, put $\dd\mu_f=|f|^q\dd v$.  Then
\begin{equation}\label{eq:mu-f-coefficient}
 \beta_j(\mu_f)
 =V_j^{1/q-1/p}\M_{q,\rho}f(a_j).
\end{equation}

\begin{corollary}\label{cor:multiplication}
Let $1<p<\infty$, $1\le q<\infty$, $1\le r<\infty$, and
$f\in\mathcal D_{\Omega,q}$.  Then
\begin{equation*}
 M_f\in\Piabs_r(A^p(\Omega),L^q(\Omega))
 \quad\Longleftrightarrow\quad
 f\in\mathcal L_{p,q}^{r}(\Omega),
\end{equation*}
and
\begin{equation}\label{eq:multiplication-norm}
 \pi_r(M_f)
 \asymp\|f\|_{\mathcal L_{p,q}^{r}(\Omega)}.
\end{equation}
The constants are independent of $f$.
\end{corollary}

\begin{proof}
The measure $\mu_f$ is locally finite because $f\in L^q_{\loc}$.  For
$g\in\Gamma$,
\[
 \|M_fg\|_{L^q(v)}=\|J_{\mu_f}g\|_{L^q(\mu_f)}.
\]
The same identity holds for every difference $g-h$ with $g,h\in\Gamma$.
Hence $M_f:\Gamma\to L^q(v)$ has a bounded extension if and only if
$J_{\mu_f}|_\Gamma:\Gamma\to L^q(\mu_f)$ has one.  Both extensions are
unique because $\Gamma$ is dense in $A^p$.

Let $g_m\in\Gamma$ and $g_m\to g$ in $A^p$.  Point evaluation estimates
give $g_m\to g$ uniformly on compact subsets of $\Omega$.  Since
$f\in L^q_{\loc}$,
\[
 f g_m\longrightarrow f g\quad\text{locally in }L^q(v).
\]
Thus the two extensions are the pointwise multiplication map and the
inclusion map, respectively.  Passing to the limit in the norm identity
gives, for every $g\in A^p$,
\[
 \|M_fg\|_{L^q(v)}=\|J_{\mu_f}g\|_{L^q(\mu_f)}.
\]
Apply this identity to every member of a finite family in $A^p$.  Their
weak $r$-norm is the same.  Therefore
\[
 \pi_r(M_f)=\pi_r(J_{\mu_f}).
\]
Now apply \cref{thm:carleson} and \eqref{eq:mu-f-coefficient}.
\end{proof}

\subsection{The diagonal ideal}

The classification in \cref{prop:garling} combines
\cite[Theorem~9]{Garling1974} with
\cite[Theorem~1.4]{FanHeWangZeng2026}.  We give the change of parameters
and the endpoint checks needed for our notation.

For $1\le p,q\le\infty$ and $1\le r<\infty$, recall
\begin{equation*}
 \mathfrak d_{p,q}^{\,r}
 =\{b:D_b\in\Piabs_r(\ell^p,\ell^q)\},
 \qquad
 \|b\|_{\mathfrak d_{p,q}^{\,r}}=\pi_r(D_b).
\end{equation*}

\begin{proposition}\label{prop:diagonal-structure}
Let $1\le p,q\le\infty$ and $1\le r<\infty$.
The space $\mathfrak d_{p,q}^{\,r}$ is a symmetric Banach sequence
lattice.  If $|c_j|\le|b_j|$, then
\begin{equation}\label{eq:solid}
 \|c\|_{\mathfrak d_{p,q}^{\,r}}
 \le\|b\|_{\mathfrak d_{p,q}^{\,r}}.
\end{equation}
If $\sigma:\N\to\N$ is a permutation, then
\[
 \|(b_{\sigma(j)})_j\|_{\mathfrak d_{p,q}^{\,r}}
 =\|b\|_{\mathfrak d_{p,q}^{\,r}}.
\]
If $\sigma:\N\to\N$ is only injective, then
\begin{equation}\label{eq:diagonal-subsequence}
 \|(b_{\sigma(j)})_j\|_{\mathfrak d_{p,q}^{\,r}}
 \le\|b\|_{\mathfrak d_{p,q}^{\,r}}.
\end{equation}
\end{proposition}

\begin{proof}
Put
\[
 \alpha_j=\begin{cases}c_j/b_j,&b_j\ne0,\\0,&b_j=0.\end{cases}
\]
Then $D_c=D_\alpha D_b$ and $\|D_\alpha:\ell^q\to\ell^q\|\le1$.
The ideal property proves \eqref{eq:solid}.  Let
$\sigma:\N\to\N$ be injective.  Define
\[
 J_\sigma^{(p)}:\ell^p\to\ell^p,
 \qquad (J_\sigma^{(p)}x)_{\sigma(j)}=x_j,
\]
with all other coordinates equal to zero.  Define
\[
 Q_\sigma^{(q)}:\ell^q\to\ell^q,
 \qquad (Q_\sigma^{(q)}y)_j=y_{\sigma(j)}.
\]
Both maps have norm one, and
\[
 D_{(b_{\sigma(j)})_j}
 =Q_\sigma^{(q)}D_bJ_\sigma^{(p)}.
\]
The ideal property proves \eqref{eq:diagonal-subsequence}.  If $\sigma$
is a permutation, apply the same argument to $\sigma^{-1}$.  The two
opposite inequalities give equality of norms.

Let $(b^{(m)})$ be Cauchy in $\mathfrak d_{p,q}^{\,r}$.  The absolutely
$r$-summing operators form a Banach operator ideal under $\pi_r$.  Hence
$(D_{b^{(m)}})$ is Cauchy in $\Piabs_r(\ell^p,\ell^q)$.  It converges
to some $T\in\Piabs_r(\ell^p,\ell^q)$.  Moreover,
\[
 |b_j^{(m)}-b_j^{(k)}|
 =\|(D_{b^{(m)}}-D_{b^{(k)}})e_j\|_q
 \le\pi_r(D_{b^{(m)}}-D_{b^{(k)}}).
\]
Thus $b_j^{(m)}\to b_j$ for every $j$.  For every $x\in\ell^p$ and
every coordinate $j$, continuity of coordinate evaluation gives
\[
 (Tx)_j=\lim_m(D_{b^{(m)}}x)_j=b_jx_j.
\]
Hence $T=D_b$ on $\ell^p$.  Finally,
\[
 \|b^{(m)}-b\|_{\mathfrak d_{p,q}^{\,r}}
 =\pi_r(D_{b^{(m)}}-T)\longrightarrow0.
\]
This proves completeness.
\end{proof}

We next prove that the coefficients vanish at the boundary and compare
different summing exponents.

\begin{proposition}
\label{prop:diagonal-compact-monotone}
Let $1<p<\infty$, $1\le q<\infty$, and
$1\le r\le s<\infty$.
Then
\begin{equation}\label{eq:diagonal-monotone}
 \mathfrak d_{p,q}^{\,r}\hookrightarrow
 \mathfrak d_{p,q}^{\,s},
 \qquad
 \|b\|_{\mathfrak d_{p,q}^{\,s}}
 \le \|b\|_{\mathfrak d_{p,q}^{\,r}}.
\end{equation}
Moreover,
\begin{equation}\label{eq:diagonal-c0}
 \mathfrak d_{p,q}^{\,r}\subset c_0.
\end{equation}
Here $c_0$ is the Banach space of scalar sequences converging to zero.
Every $D_b$ with $b\in\mathfrak d_{p,q}^{\,r}$ is compact.
\end{proposition}

\begin{proof}
Let $D_b\in\Piabs_r(\ell^p,\ell^q)$.  Pietsch domination gives a
probability measure $\nu$ on $B_{\ell^{p'}}$ such that
\begin{equation*}
 \|D_bx\|_q
 \le \pi_r(D_b)
 \left(\int_{B_{\ell^{p'}}}|\phi(x)|^r\dd\nu(\phi)\right)^{1/r}.
\end{equation*}
Since $\nu$ is a probability measure and $r\le s$,
\[
 \left(\int|\phi(x)|^r\dd\nu\right)^{1/r}
 \le
 \left(\int|\phi(x)|^s\dd\nu\right)^{1/s}.
\]
Thus the same measure is a Pietsch measure of order $s$.  Hence
\[
 D_b\in\Piabs_s(\ell^p,\ell^q),
 \qquad
 \pi_s(D_b)\le\pi_r(D_b).
\]
This proves \eqref{eq:diagonal-monotone}.

Every absolutely summing operator is completely continuous by
\cref{lem:summing-calculus}(iv).  The space $\ell^p$ is reflexive.
Thus every bounded sequence in $\ell^p$ has a weakly convergent
subsequence.  The image of this subsequence under $D_b$ converges in norm.
Thus $D_b$ is compact.

For \eqref{eq:diagonal-c0}, let $(e_j)$ be the canonical basis of
$\ell^p$.  Since $1<p<\infty$, the sequence $(e_j)$ converges weakly
to zero in $\ell^p$.
Complete continuity gives
\[
 |b_j|=\|D_be_j\|_{\ell^q}\longrightarrow0.
\]
Therefore $b\in c_0$.
\end{proof}

For later reference, let $j\to\partial\Omega$ mean
\begin{equation*}
 \delta(a_j)\longrightarrow0.
\end{equation*}
Only finitely many lattice points lie in a fixed compact subset of
$\Omega$.  Thus $j\to\partial\Omega$ means that the lattice points leave every
finite subset.

We use the exact diagonal ideal in all proofs.  Garling's
diagonal theorem gives the classification outside
\begin{equation*}
 \mathcal E
 =\{(p,q,r):1<p<2<q<\infty,\ r>\max\{p',q\}\}.
\end{equation*}
Fan, He, Wang, and Zeng completed the classification on
$\mathcal E$ \cite{FanHeWangZeng2026}.
We write $\log^+t=\max\{\log t,0\}$.
Recall that a Young function $\Psi:[0,\infty)\to[0,\infty)$ is
continuous, increasing, and convex.  It satisfies $\Psi(0)=0$ and
$\Psi(t)\to\infty$ as $t\to\infty$.
For $1<q<2$, put
\[
 M_q(0)=0,\qquad
 M_q(t)=t^q\bigl(1+\log^+(t^{-1})\bigr)\quad(t>0).
\]
For $0<t<1$,
\begin{align*}
 M_q'(t)
 &=t^{q-1}\left[q\left(1+\log\frac1t\right)-1\right],\\
 M_q''(t)
 &=t^{q-2}\left[q(q-1)\left(1+\log\frac1t\right)-(2q-1)\right].
\end{align*}
Choose $\varepsilon_q\in(0,1)$ so that both expressions are positive on
$(0,\varepsilon_q]$.  Fix the Young function
\begin{equation}\label{eq:qminus}
 \Psi_q(t)=
 \begin{cases}
  M_q(t),&0\le t\le\varepsilon_q,\\
  M_q(\varepsilon_q)+M_q'(\varepsilon_q)(t-\varepsilon_q)
  +(t-\varepsilon_q)^2,&t>\varepsilon_q.
 \end{cases}
\end{equation}
The two pieces have the same value and derivative at $\varepsilon_q$.
The derivative of the second piece is
$M_q'(\varepsilon_q)+2(t-\varepsilon_q)$.  Hence $\Psi_q$ is increasing
and convex.
We use $q^-$ as a label for this Orlicz scale.  It is not a numerical
exponent.  For each $A\ge1$, this function satisfies
\[
 \Psi_q(At)\le C_{A,q}\Psi_q(t),\qquad t\ge0.
\]
Indeed,
\[
 \lim_{t\downarrow0}\frac{\Psi_q(At)}{\Psi_q(t)}=A^q,
 \qquad
 \lim_{t\to\infty}\frac{\Psi_q(At)}{\Psi_q(t)}=A^2.
\]
The quotient is continuous on $(0,\infty)$.  It is therefore bounded.
The notation $\ell^{q^-}$ means the Orlicz sequence space equipped with
the Luxemburg norm
\begin{equation}\label{eq:luxemburg}
 \|b\|_{\ell^{q^-}}
 =\inf\left\{t>0:\sum_j\Psi_q\left(\frac{|b_j|}{t}\right)\le1\right\}.
\end{equation}

\begin{proposition}\label{prop:garling}
Let $1<p<\infty$, $1\le q<\infty$, and $1\le r<\infty$.  For
$(p,q,r)\in\mathcal E$, put
\begin{equation}\label{eq:exceptional-index}
 \theta=\frac{1/q-1/r}{1/2-1/r},
 \qquad
 \frac1{\kappa_{\mathcal E}}
 =\frac{1-\theta}{r}+\frac{\theta}{p'}.
\end{equation}
The exceptional inequalities imply $0<\theta<1$.
Then the following sequence spaces equal $\mathfrak d_{p,q}^{\,r}$,
with equivalent norms:
\begin{equation}\label{eq:kappa-table}
\mathfrak d_{p,q}^{\,r}=
\begin{cases}
 \ell^{(1/p'+1/q-1/2)^{-1}},
 &1<p\le2,\ 1\le q\le2,\\
 \ell^q,&2<p<\infty,\ 1\le q<p',\\
 \ell^{q^-},&2<p<\infty,\ q=p'<2,\ 1\le r<q,\\
 \ell^q,&2<p<\infty,\ q=p'<2,\ q\le r,\\
 \ell^{\max\{p',\min\{r,q\}\}},
 &2\le p<\infty,\ p'<q<\infty,\\
 \ell^{p'},&1<p<2<q<\infty,\ 1\le r\le p',\\
 \ell^r,&1<p<2,\ p'<q<\infty,\ p'<r\le q,\\
 \ell^{\kappa_{\mathcal E}},
 &1<p<2<q<\infty,\ r>\max\{p',q\}.
\end{cases}
\end{equation}
In the last line,
\begin{equation}\label{eq:exceptional-index-order}
 \max\{p',q\}<\kappa_{\mathcal E}<r.
\end{equation}
\end{proposition}

\begin{proof}
Garling writes the source space as $\ell^{\mathsf P'}$.  Our source is
$\ell^p$.  Therefore
\begin{equation*}
 \mathsf P'=p,
 \qquad
 \mathsf P=p'.
\end{equation*}
Put $u=p'$.  Thus Garling's diagonal maps from $\ell^{u'}=\ell^p$ into
$\ell^q$.  We apply \cite[Theorem~9]{Garling1974} with the pair $(u,q)$.

First suppose $1<p\le2$.  Then $2\le u<\infty$.  If $1\le q\le2$,
Garling's cases (iii) to (v), including their common boundary lines, give
\begin{equation*}
 b\in\ell^{\phi(u,q)},
 \qquad
 \frac1{\phi(u,q)}=\frac1u+\frac1q-\frac12.
\end{equation*}
Since $u=p'$, this is the first line of
\eqref{eq:kappa-table}.  At the two endpoints,
\[
 p=2:\quad \phi(2,q)=q,
 \qquad
 q=2:\quad \phi(p',2)=p'.
\]
In particular, the corner $p=q=2$ gives $\ell^2$ for every $r$.

Continue with $1<p<2$ and $2<q<\infty$.  Now $u=p'>2$.  Garling's
cases (vi) and (vii) give
\begin{align}
 1\le r\le u
 &\quad\Longrightarrow\quad b\in\ell^u,
 \label{eq:Garling-high-high-one}\\
 u<r\le q,\quad u<q
 &\quad\Longrightarrow\quad b\in\ell^r.
 \notag
\end{align}
Substituting $u=p'$ gives the sixth and seventh lines of
\eqref{eq:kappa-table}.  If $q\le p'$, only
\eqref{eq:Garling-high-high-one} applies before the remaining range.  If
$p'<q$, both lines apply.  The remaining range is exactly
\[
 r>\max\{p',q\}.
\]
Thus the remaining parameter range with $1<p<2$ and $q>2$ is exactly
$\mathcal E$.

Next suppose $2<p<\infty$.  Then $1<u=p'<2$.  For $1\le q<u$, Garling's
case (iv) gives
\begin{equation*}
 b\in\ell^q
 \qquad(1\le r<\infty).
\end{equation*}
This is the second line of \eqref{eq:kappa-table}.  On the diagonal
$q=u=p'<2$, case (ii) gives
\begin{equation*}
 b\in
 \begin{cases}
  \ell^{q^-},&1\le r<q,\\
  \ell^q,&q\le r<\infty.
 \end{cases}
\end{equation*}
Garling defines $\ell^{q^-}$ by a modular equivalent near zero to
\[
 \sum_j|b_j|^q\bigl(1+|\log|b_j||\bigr).
\]
Changing the defining function away from zero does not change the
sequence space or its topology.  Hence the Young function in
\eqref{eq:qminus} gives the same Orlicz space.  This proves the third
and fourth lines of \eqref{eq:kappa-table}.

For $1<u\le2$ and $u<q$, Garling's case (i) gives
\begin{equation}\label{eq:Garling-above-diagonal}
 b\in
 \begin{cases}
  \ell^u,&1\le r\le u,\\
  \ell^r,&u\le r\le q,\\
  \ell^q,&q\le r<\infty.
 \end{cases}
\end{equation}
At the common endpoints the two adjacent conditions agree.  Formula
\eqref{eq:Garling-above-diagonal} is equivalent to
\[
 b\in\ell^{\max\{u,\min\{r,q\}\}}.
\]
Substituting $u=p'$ gives the fifth line of
\eqref{eq:kappa-table}.  This completes the power and Orlicz regions.

In the remaining region $\mathcal E$, the theorem of Fan, He, Wang,
and Zeng gives
\[
 D_b\in\Piabs_r(\ell^p,\ell^q)
 \quad\Longleftrightarrow\quad
 b\in\ell^{\kappa_{\mathcal E}},
\]
with equivalence of norms
\cite[Theorem~1.4]{FanHeWangZeng2026}.  Their formula for the index is
\[
 \kappa_{\mathcal E}
 =\frac{p'q(r-2)}
 {(p'-2)(q-2)+2(r-2)}.
\]
To compare the exponents, put
\[
 D=(p'-2)(q-2)+2(r-2).
\]
Since $r>\max\{p',q\}$ and $p',q>2$,
\begin{align*}
 q(r-2)-D&=(r-p')(q-2)>0,\\
 p'(r-2)-D&=(r-q)(p'-2)>0,\\
 rD-p'q(r-2)&=2(r-p')(r-q)>0.
\end{align*}
Thus $\kappa_{\mathcal E}>p'$, $\kappa_{\mathcal E}>q$, and
$\kappa_{\mathcal E}<r$.  This proves
\eqref{eq:exceptional-index-order}.  Direct substitution of
\[
 \theta=\frac{2(r-q)}{q(r-2)}
\]
into \eqref{eq:exceptional-index} gives the same formula for
$\kappa_{\mathcal E}$.

Finally, coordinate evaluations are continuous for
$\mathfrak d_{p,q}^{\,r}$ by
\cref{prop:diagonal-structure}.  They are continuous for each listed
power or Orlicz space.  The identity map between two equal sequence
sets has a closed graph in both directions.  The closed graph theorem
therefore gives equivalence of norms in each case.
\end{proof}

\subsection{Lattice stability}

Recall the coefficient sequences in
\eqref{eq:lattice-M} and \eqref{eq:lattice-G}.

\begin{lemma}\label{lem:finite-neighbor}
Let $1\le p,q\le\infty$ and $1\le r<\infty$.
Suppose that $N_0\in\N$, that $\mathcal N(j)$ are finite sets, and
\[
 \sup_j\#\mathcal N(j)+\sup_k\#\{j:k\in\mathcal N(j)\}\le N_0.
\]
If $e_j=\sum_{k\in\mathcal N(j)}|b_k|$, then
\begin{equation*}
 \|e\|_{\mathfrak d_{p,q}^{\,r}}
 \le N_0\|b\|_{\mathfrak d_{p,q}^{\,r}}.
\end{equation*}
\end{lemma}

\begin{proof}
View the incidence relation
$\{(j,k):k\in\mathcal N(j)\}$ as a bipartite graph.  Every vertex has
finite degree.  Enumerate its edges.  When an edge $(j,k)$ is colored, at
most
\[
 \bigl(\#\mathcal N(j)-1\bigr)
 +\bigl(\#\{i:k\in\mathcal N(i)\}-1\bigr)
 \le N_0-2
\]
colors are forbidden at its endpoints.  Thus $N_0-1$ colors suffice if
the graph is nonempty.  In that case $N_0\ge2$.  Edges of one color
define a partial injection $\sigma$.  This is a one-to-one map defined
on a subset of $\N$.  If $b_j^\sigma=b_{\sigma(j)}$ on the domain of
$\sigma$ and is zero otherwise, then
\[
 (J_\sigma x)_{\sigma(j)}=x_j,
 \qquad
 (Q_\sigma y)_j=y_{\sigma(j)}
\]
on that domain; all other coordinates are zero.  Thus
\[
 D_{b^\sigma}=Q_\sigma D_bJ_\sigma,
 \qquad \|Q_\sigma\|,\|J_\sigma\|\le1.
\]
The ideal property gives
$\|b^\sigma\|_{\mathfrak d_{p,q}^{\,r}}\le
\|b\|_{\mathfrak d_{p,q}^{\,r}}$.
If the color classes are indexed by $1\le\nu\le m$, then $m<N_0$ and
\[
 e=\sum_{\nu=1}^m|b^{\sigma_\nu}|.
\]
The triangle inequality gives
\[
 \|e\|_{\mathfrak d_{p,q}^{\,r}}
 \le m\|b\|_{\mathfrak d_{p,q}^{\,r}}
 \le N_0\|b\|_{\mathfrak d_{p,q}^{\,r}}.
\]
\end{proof}

\begin{lemma}\label{lem:discretization}
Let $1\le p,q<\infty$, $1\le r<\infty$, and
$f\in L^q_{\loc}(\Omega)$.  The norm of
$\mathbf M_{p,q,\rho}f$ is unchanged, up to fixed constants,
when the admissible radius or lattice is changed.  Hence
$\mathcal L_{p,q}^{r}(\Omega)$ is well defined.

\end{lemma}

\begin{proof}
Let $(\rho,\Lambda)$ and $(\sigma,\Lambda')$ be two admissible choices,
where
\[
 \Lambda=\{a_j\},\qquad \Lambda'=\{a_k'\},\qquad
 V_j=v(\Q_\rho(a_j)),\qquad W_k=v(\Q_\sigma(a_k')).
\]
The superscript on $\mathbf M^{\Lambda}_{p,q,\rho}$ denotes the
sampling lattice.  The norm is the same as before.
For each $j$, set
\[
 \mathcal N(j)
 =\{k:\Q_\sigma(a_k')\cap\Q_\rho(a_j)\ne\varnothing\}.
\]
Since the $\sigma$-polydiscs cover $\Omega$,
\begin{equation}\label{eq:two-lattice-cover}
 \Q_\rho(a_j)
 \subset\bigcup_{k\in\mathcal N(j)}\Q_\sigma(a_k').
\end{equation}
Engulfing, volume doubling, and the disjoint small polydiscs in the two
lattices give
\begin{equation}\label{eq:two-lattice-incidence}
 \sup_j\#\mathcal N(j)
 +\sup_k\#\{j:k\in\mathcal N(j)\}\le N_{\rho,\sigma}\in\N,
 \qquad
 V_j\asymp W_k\quad(k\in\mathcal N(j)).
\end{equation}
The constants depend on the two fixed radii.  They do not depend on the
indices.

Put $\alpha=1/q-1/p$.  From
\eqref{eq:two-lattice-cover} and \eqref{eq:two-lattice-incidence},
\begin{align*}
 \M_{q,\rho}f(a_j)^q
 &=\frac1{V_j}\int_{\Q_\rho(a_j)}|f|^q\dd v\notag\le\sum_{k\in\mathcal N(j)}
 \frac{W_k}{V_j}\M_{q,\sigma}f(a_k')^q\notag\lesssim\sum_{k\in\mathcal N(j)}
 \M_{q,\sigma}f(a_k')^q.
\end{align*}
Thus
\begin{align*}
 V_j^\alpha\M_{q,\rho}f(a_j)
 &\lesssim
 \left(\sum_{k\in\mathcal N(j)}
 \bigl[W_k^\alpha\M_{q,\sigma}f(a_k')\bigr]^q\right)^{1/q}\notag\le\sum_{k\in\mathcal N(j)}
 W_k^\alpha\M_{q,\sigma}f(a_k').
\end{align*}
Solidity and \cref{lem:finite-neighbor} give
\begin{equation*}
 \|\mathbf M_{p,q,\rho}^{\Lambda}f\|_{\mathfrak d_{p,q}^{\,r}}
 \lesssim
 \|\mathbf M_{p,q,\sigma}^{\Lambda'}f\|_{\mathfrak d_{p,q}^{\,r}}.
\end{equation*}
Reverse the two admissible choices.  This proves the norm comparison
for $\mathbf M$.

\end{proof}

\subsection{Approximation in the summing norm}

The local Taylor decomposition also gives approximation by operators of finite rank in
the summing norm.  The estimate depends on the number of chosen lattice
cells and the Taylor degree on each cell.  It applies to
singular measures and includes $q=1$.

\begin{theorem}\label{thm:summing-approximation}
Let $1<p<\infty$, $1\le q,r<\infty$, and let $\mu$ be a positive
locally finite Borel measure on $\Omega$.  Assume
$J_\mu\in\Piabs_r(A^p(\Omega),L^q(\Omega,\mu))$.
Use the cells $E_j$ and the fixed enlargements from
\eqref{eq:partition-cells} through \eqref{eq:local-scaled-inclusion}, and put
\[
 b_j=\frac{\mu(\Q_{C\rho}(a_j))^{1/q}}{V_j^{1/p}}.
\]
There are constants $A>0$ and $0<\vartheta<1$, independent of $\mu$,
with the following property.  For every finite $F\subset\N$ and every
integer $L\ge0$, there is an operator
$R_{F,L}:A^p(\Omega)\to L^q(\Omega,\mu)$ with
\begin{equation}\label{eq:local-Taylor-rank}
 \operatorname{rank}R_{F,L}
 \le \#F\binom{L+n}{n}
\end{equation}
and
\begin{equation}\label{eq:summing-approximation-bound}
 \pi_r(J_\mu-R_{F,L})
 \le A\left(
 \|b\mathbf1_{F^c}\|_{\mathfrak d_{p,q}^{\,r}}
 +(L+1)^{n-1}\vartheta^{L+1}
 \|b\mathbf1_F\|_{\mathfrak d_{p,q}^{\,r}}
 \right).
\end{equation}
The constants depend only on the parameters and the fixed geometric
data.  In particular, $J_\mu$ is a limit of operators of finite rank in
the $r$-summing norm.
\end{theorem}

\begin{proof}
Let $X_j$, $\Phi_j$, $\ell_{j,\alpha}$, and $m_{j,\alpha}$ be as in
\cref{lem:local-restriction}.  For $u\in X_j$, define
\[
 T_{j,L}u
 =\sum_{|\alpha|\le L}
 \ell_{j,\alpha}(u)\,m_{j,\alpha}|_{E_j}.
\]
The rank of $T_{j,L}$ is at most
$\#\{\alpha\in\Nzero^n:|\alpha|\le L\}=\binom{L+n}{n}$.
Write $\vartheta=\theta/\theta_1\in(0,1)$, with $\theta,\theta_1$
from \eqref{eq:local-scaled-inclusion}. The cells, affine maps, and
the numbers $\theta,\theta_1$ are fixed before $\mu$, $F$, and $L$
are chosen. The nuclear decomposition in the proof of
\cref{lem:local-restriction} gives
\begin{align}
 \pi_1(B_j-T_{j,L})
 &\le C\frac{\mu(E_j)^{1/q}}{V_j^{1/p}}
 \sum_{m=L+1}^{\infty}
 \binom{m+n-1}{n-1}\vartheta^m\notag\\
 &\le C_1(L+1)^{n-1}\vartheta^{L+1}b_j.
 \label{eq:local-Taylor-tail}
\end{align}
For the last inequality, write $m=L+1+k$ and use
\[
 \binom{L+k+n}{n-1}
 \le C_n(L+1)^{n-1}(k+1)^{n-1},
 \qquad
 \sum_{k=0}^{\infty}(k+1)^{n-1}\vartheta^k<\infty.
\]
The same restriction lemma gives $\pi_1(B_j)\le C_2b_j$.
These estimates also hold if $b_j=0$, since all the corresponding
operators into $L^q(E_j,\mu)$ then vanish.

Under the direct-sum identification of $L^q(\Omega,\mu)$, set
\[
 (R_{F,L}h)|_{E_j}
 =
 \begin{cases}
 T_{j,L}(h|_{\Q_{C_1\rho}(a_j)}),&j\in F,\\
 0,&j\notin F.
 \end{cases}
\]
Its range is contained in the span of
\[
 \bigl\{\mathbf1_{E_j}m_{j,\alpha}:j\in F,\ |\alpha|\le L\bigr\}.
\]
Every function in this finite set belongs to $L^q(\mu)$ because
$\mu(E_j)<\infty$.  Each coefficient is a bounded functional on
$A^p$, by \cref{lem:local-restriction} and the restriction estimate.
This proves boundedness and \eqref{eq:local-Taylor-rank}.  Put
\[
 \eta_L=(L+1)^{n-1}\vartheta^{L+1},\qquad
 d_j=b_j\bigl(\mathbf1_{F^c}(j)+\eta_L\mathbf1_F(j)\bigr).
\]
Let $C_j=B_j-T_{j,L}$ for $j\in F$ and $C_j=B_j$ otherwise.
By \eqref{eq:local-Taylor-tail},
\[
 \pi_1(C_j)\le C_3d_j.
\]
The restriction map
$R:A^p\to(\oplus_jX_j)_{\ell^p}$ is bounded, and
\[
 J_\mu-R_{F,L}=\mathcal C R,
 \qquad \mathcal C(u_j)=(C_ju_j)_j.
\]
Apply \cref{lem:block-principle} to $C_j/d_j$, taking the normalized
block to be zero when $d_j=0$.  Solidity and the triangle inequality
give
\begin{align*}
 \pi_r(J_\mu-R_{F,L})
 &\le C_3\|R\|\|d\|_{\mathfrak d_{p,q}^{\,r}}\le C_3\|R\|\left(
 \|b\mathbf1_{F^c}\|_{\mathfrak d_{p,q}^{\,r}}
 +\eta_L\|b\mathbf1_F\|_{\mathfrak d_{p,q}^{\,r}}\right).
\end{align*}
This is \eqref{eq:summing-approximation-bound}.

By \cref{thm:carleson} and the radius comparison in its proof,
$b\in\mathfrak d_{p,q}^{\,r}$.  We verify the required tail convergence.
Let $P_m$ be the first-$m$ coordinate projection on $\ell^p$.
Pietsch domination gives a probability measure $\nu$ on
$B_{\ell^{p'}}$ such that, for every finite family
$x_1,\ldots,x_N\in\ell^p$,
\begin{align*}
 \sum_{k=1}^N\|D_b(I-P_m)x_k\|_q^r
 &\le\pi_r(D_b)^r\int_{B_{\ell^{p'}}}
       \sum_{k=1}^N|((I-P_m)^*\phi)(x_k)|^r\dd\nu(\phi)\\
 &\le\pi_r(D_b)^r w_r(x_1,\ldots,x_N)^r
       \int_{B_{\ell^{p'}}}
       \|(I-P_m)^*\phi\|_{p'}^r\dd\nu(\phi).
\end{align*}
Here $p'=p/(p-1)<\infty$. For each $\phi\in B_{\ell^{p'}}$,
\[
 \|(I-P_m)^*\phi\|_{p'}^{p'}
 =\sum_{j>m}|\phi_j|^{p'}\longrightarrow0,
 \qquad \|(I-P_m)^*\phi\|_{p'}\le1.
\]
Dominated convergence therefore gives
\begin{align*}
 \|b\mathbf1_{\{j>m\}}\|_{\mathfrak d_{p,q}^{\,r}}
 &=\pi_r(D_b(I-P_m))\\
 &\le\pi_r(D_b)
 \left(\int_{B_{\ell^{p'}}}
       \|(I-P_m)^*\phi\|_{p'}^r\dd\nu(\phi)\right)^{1/r}
 \longrightarrow0.
\end{align*}
This argument covers every parameter range without using
\cref{prop:garling}.
Given $\varepsilon>0$, first take $F=\{1,\ldots,m\}$ so that
$A\|b\mathbf1_{F^c}\|_{\mathfrak d_{p,q}^{\,r}}<\varepsilon/2$.
Then choose $L$ so large that
$A\eta_L\|b\|_{\mathfrak d_{p,q}^{\,r}}<\varepsilon/2$.
Solidity and \eqref{eq:summing-approximation-bound} give
\[
 \pi_r(J_\mu-R_{F,L})<\varepsilon.
\]
Thus the operators of finite rank converge to $J_\mu$ in the
$r$-summing norm.
\end{proof}

Recall that, for $T\in\Piabs_r(X,Y)$,
\[
 a_N^{(r)}(T)
 =\inf\{\pi_r(T-A):A\in\mathcal L(X,Y),\ \operatorname{rank}A<N\},
 \qquad N\ge1.
\]
Thus \cref{thm:summing-approximation} gives the rank bound
\begin{equation}\label{eq:summing-rank-budget}
 a_N^{(r)}(J_\mu)
 \le A\!\inf_{\substack{F\subset\N\text{ finite},\ L\ge0\\
                 \#F\binom{L+n}{n}<N}}
 \left(
 \|b\mathbf1_{F^c}\|_{\mathfrak d_{p,q}^{\,r}}
 +\eta_L\|b\mathbf1_F\|_{\mathfrak d_{p,q}^{\,r}}
 \right).
\end{equation}
Here $L$ is an integer.  The estimate is uniform in the measure.

\begin{corollary}\label{cor:multiplication-approximation}
Let $1<p<\infty$, $1\le q,r<\infty$, and
$f\in\mathcal L_{p,q}^{r}(\Omega)$.  Then
$M_f:A^p(\Omega)\to L^q(\Omega,v)$ is a limit of operators of finite rank in the $r$-summing norm.  The rank and error estimates
\eqref{eq:local-Taylor-rank} and \eqref{eq:summing-approximation-bound}
hold with
\[
 b_j=V_j^{-1/p}
 \left(\int_{\Q_{C\rho}(a_j)}|f|^q\dd v\right)^{1/q}.
\]
\end{corollary}

\begin{proof}
Set $\dd\mu_f=|f|^q\dd v$. Since
$f\in\mathcal L_{p,q}^{r}(\Omega)$, \cref{thm:carleson} gives
$J_{\mu_f}\in\Piabs_r(A^p,L^q(\mu_f))$. Hence
\[
 \|fg\|_{L^q(v)}
 =\|J_{\mu_f}g\|_{L^q(\mu_f)}<\infty
 \qquad(g\in A^p).
\]
Every kernel section belongs to $A^p$ by
\cref{lem:kernel-span-density}. Thus $f\in\mathcal D_{\Omega,q}$.
The map
\[
 U_f:L^q(\mu_f)\to L^q(v),\qquad U_fh=fh,
\]
is a well-defined linear isometry, and $M_f=U_fJ_{\mu_f}$.
Use the operators of finite rank $U_fR_{F,L}$ and the ideal property.
Their ranks do not exceed those of $R_{F,L}$.
\end{proof}
At $q=1$, \cref{cor:multiplication,prop:garling} give the following
explicit form. Let $1<p<\infty$, $1\le r<\infty$, and
$f\in\mathcal D_{\Omega,1}$. Then
\begin{equation}\label{eq:intro-q1-multiplication}
 M_f\in\Piabs_r(A^p,L^1)
 \Longleftrightarrow
 \{V_j^{1-1/p}\M_{1,\rho}f(a_j)\}_j\in\ell^{s_p},
 \qquad
 s_p=
 \begin{cases}
  \dfrac{2p}{3p-2},&1<p\le2,\\[3pt]
  1,&p>2.
 \end{cases}
\end{equation}
The exponent $s_p$ does not depend on $r$.

\section{Big Hankel operators}\label{sec:ida}

Earlier results for Bergman spaces and finite-type domains appear in
\cite{Li1992,Li1993,LiLuecking1995,Pau2016}.  We use a nonisotropic
IDA quantity in place of a Schatten symbol norm.

For $1<q<\infty$ and $f\in\mathcal D_{\Omega,q}$, recall that
\[
 H_fg=(I-P)(fg),\qquad g\in\Gamma.
\]
The local distance in \eqref{eq:intro-ida} is adapted to McNeal balls.
The enlargement in its definition is needed.  It leaves room for uniform
Cauchy estimates and local analytic continuation.  For every admissible
scale $s$,
\begin{equation}\label{eq:G-less-M-finite}
 \G_{q,s}f(z)\le\M_{q,s}f(z).
\end{equation}

\begin{proposition}
\label{prop:ida-calculus}
Let $1<p<\infty$, $1\le q,r<\infty$, and let
$f,g\in L^q_{\loc}(\Omega)$.  Let $s>0$ be an admissible scale.
For $h\in\Hol(\Omega)$,
\begin{align}
 \G_{q,s}(f+h)(z)&=\G_{q,s}f(z),
 \notag\\
 \G_{q,s}(f+g)(z)&\le
 \G_{q,s}f(z)+\G_{q,s}g(z),
 \label{eq:G-subadditive}\\
 |\G_{q,s}f(z)-\G_{q,s}g(z)|
 &\le\M_{q,s}(f-g)(z).
 \label{eq:G-Lipschitz}
\end{align}
Consequently,
\begin{equation}\label{eq:I-seminorm-properties}
 \|f+h\|_{\mathcal I_{p,q}^{r}}
 =\|f\|_{\mathcal I_{p,q}^{r}},
 \qquad
 \|f+g\|_{\mathcal I_{p,q}^{r}}
 \le\|f\|_{\mathcal I_{p,q}^{r}}
 +\|g\|_{\mathcal I_{p,q}^{r}}.
\end{equation}
If $\|f\|_{\mathcal I_{p,q}^{r}}=0$, then $f$ agrees almost
everywhere with a holomorphic function on $\Omega$.
\end{proposition}

\begin{proof}
Let $U=\Q_s(z)$ and $U^*=\Q_s^*(z)$.  Replacing each
competitor $a$ by $b=a-h$ on $U^*$ gives
\begin{align*}
 \G_{q,s}(f+h)(z)
 &=\inf_{a\in\Hol(U^*)}
 \left(\avg_U|f+h-a|^q\dd v\right)^{1/q}\\
 &=\inf_{b\in\Hol(U^*)}
 \left(\avg_U|f-b|^q\dd v\right)^{1/q}
 =\G_{q,s}f(z).
\end{align*}
Fix $\varepsilon>0$.  Choose $a,b\in\Hol(U^*)$ whose errors are within
$\varepsilon$ of the two infima.  Minkowski's inequality gives
\[
 \left(\avg_U|f+g-a-b|^q\dd v\right)^{1/q}
 \le
 \left(\avg_U|f-a|^q\dd v\right)^{1/q}
 +
 \left(\avg_U|g-b|^q\dd v\right)^{1/q}.
\]
Let $\varepsilon\downarrow0$.  This proves
\eqref{eq:G-subadditive}.  Apply it
to
\[
 f=g+(f-g),\qquad g=f+(g-f),
\]
and use \eqref{eq:G-less-M-finite}.  We obtain
\[
 \G_{q,s}f-\G_{q,s}g
 \le\M_{q,s}(f-g),
 \qquad
 \G_{q,s}g-\G_{q,s}f
 \le\M_{q,s}(f-g).
\]
This is \eqref{eq:G-Lipschitz}.

Take $s=\rho_+$ and $z=a_j$.  Multiply the pointwise inequalities by
$V_j^{1/q-1/p}$.  Use
solidity and the triangle inequality in
$\mathfrak d_{p,q}^{\,r}$.  This proves
\eqref{eq:I-seminorm-properties}.

Assume that the seminorm is zero.  Then
$\G_{q,\rho_+}f(a_j)=0$ for every $j$.  For each $j$, choose
$h_{j,m}\in\Hol(\Q_{\rho_+}^*(a_j))$ with
\[
 \|f-h_{j,m}\|_{L^q(\Q_{\rho_+}(a_j))}\longrightarrow0.
\]
Write $C_c^\infty(U)$ for the smooth compactly supported functions
on $U$, and $\mathcal D'(U)$ for their continuous linear dual,
the space of distributions on $U$. Restricting to
$\Q_{\rho_+}(a_j)$ and passing to distributions gives
\[
 \bar\partial f=0
 \quad\text{in }\mathcal D'(\Q_{\rho_+}(a_j)).
\]
Indeed, if $\varphi\in C_c^\infty(\Q_{\rho_+}(a_j))$, then for
$1\le k\le n$,
\[
 \left|\int (f-h_{j,m})
 \frac{\partial\varphi}{\partial\bar z_k}\dd v\right|
 \le \|f-h_{j,m}\|_{L^q(\Q_{\rho_+}(a_j))}
 \left\|\frac{\partial\varphi}{\partial\bar z_k}\right\|_{L^{q'}}
 \longrightarrow0.
\]
Here $q'=\infty$ when $q=1$.  The integral with $h_{j,m}$ is zero.
The sets $\Q_{\rho_+}(a_j)$ cover $\Omega$.  Thus
$\bar\partial f=0$ in $\mathcal D'(\Omega)$.  Weyl's lemma gives one
function $h\in\Hol(\Omega)$ with $f=h$ almost everywhere.
\end{proof}

The following property describes the solution estimate used in the
Hankel argument.  We verify it for every $1<q<\infty$ by applying
Ahn's weighted estimate.

\begin{definition}\label{def:FTI}
Let $1<q<\infty$.  We say that $\Omega$ has property
$\mathrm{FTI}_q$ if there is a linear operator
$\mathcal S_{\Omega,q}$ from the vector space of smooth
$\bar\partial$-closed $(0,1)$-forms with
$|\omega|_\Omega\in L^q(\Omega)$ into $L^q(\Omega)$ such that
\begin{equation}\label{eq:FTI-solution}
 \bar\partial\mathcal S_{\Omega,q}\omega=\omega
\end{equation}
in distributions, and
\begin{equation}\label{eq:FTI-basic-estimate}
 \|\mathcal S_{\Omega,q}\omega\|_{L^q(\Omega)}
 \le C_q\||\omega|_\Omega\|_{L^q(\Omega)}.
\end{equation}
\end{definition}

\begin{proposition}\label{prop:FTI-all}
Every bounded smooth convex domain of finite type has property
$\mathrm{FTI}_q$ for each $1<q<\infty$.  More precisely, for every
smooth $\bar\partial$-closed $(0,1)$-form $\omega$ with
$|\omega|_\Omega\in L^q(\Omega)$, there is a unique solution
$u\in L^q(\Omega)$ such that
\begin{equation}\label{eq:FTI-all-normalization}
 \bar\partial u=\omega,\qquad Pu=0.
\end{equation}
The map $\omega\mapsto u$ is linear, and
\begin{equation}\label{eq:FTI-all-estimate}
 \|u\|_q\le C_{\Omega,q}\||\omega|_\Omega\|_q.
\end{equation}
\end{proposition}

\begin{proof}
We apply Ahn's weighted solution estimate
\cite[Theorem~1.1(1) and Remark~\textup{(ii)}]{Ahn2004}.
We first compare its form norm with \eqref{eq:dual-McNeal}.
For a point $z$ in the boundary collar and $X\ne0$, put
\[
 \kappa_z(X)
 =\frac{\delta(z)|X|}
 {\tau(z,X/|X|,\delta(z)/2)},
 \qquad
 \|\omega(z)\|_{\mathrm A}
 =\sup_{X\ne0}
 \frac{|\omega_z(\overline X)|}{\kappa_z(X)}.
\]
These are the directional quantity and form norm in
\cite[Section~2.2.1]{Ahn2004}.  The distance $\delta$ is the same
Euclidean boundary distance used in this paper.
Write $X=\sum_kX_ku_k(z,\delta(z))$ and set
\[
 \mathcal N_z(X)
 =\left(\sum_{k=1}^n
 \frac{|X_k|^2}{\tau_k(z,\delta(z))^2}\right)^{1/2}.
\]
The comparison at fixed scales for $\tau$ and the directional norm
comparison in \cref{sec:geometry} give
\begin{equation}\label{eq:Ahn-directional-comparison}
 \frac{\kappa_z(X)}{\delta(z)}
 =\frac{|X|}{\tau(z,X/|X|,\delta(z)/2)}
 \asymp
 \frac{|X|}{\tau(z,X/|X|,\delta(z))}
 \asymp\mathcal N_z(X).
\end{equation}
The constants are independent of $z$ and $X$.  In the last comparison,
changing between extremal bases is justified by
\cite[Propositions~3.1 and~3.2]{NikolovPflugThomas2013}.
Taking dual norms in \eqref{eq:Ahn-directional-comparison} gives
\begin{align}
 \delta(z)\|\omega(z)\|_{\mathrm A}
 &\asymp\sup_{X\ne0}
 \frac{\left|\sum_{k=1}^n\omega_k(z)\overline{X_k}\right|}
 {\mathcal N_z(X)}\notag\\
 &=\left(\sum_{k=1}^n
 |\omega_k(z)|^2\tau_k(z,\delta(z))^2\right)^{1/2}
 =|\omega(z)|_\Omega.
 \label{eq:Ahn-McNeal-dual}
\end{align}
The equality follows from Cauchy--Schwarz.  If $\omega(z)\ne0$,
equality is attained by choosing
$X_k=\omega_k(z)\tau_k(z,\delta(z))^2$.
On the fixed compact core, both form norms in
\eqref{eq:Ahn-McNeal-dual} are uniformly equivalent to the Euclidean
form norm.  Thus this comparison holds throughout $\Omega$.

For each $\alpha>0$ for which the right-hand side below is finite,
Ahn's estimate gives a distributional solution $v$ of
$\bar\partial v=\omega$ satisfying
\begin{equation}\label{eq:Ahn-weighted-estimate}
 \int_\Omega |v(z)|^q\delta(z)^{\alpha-1}\dd v(z)
 \le C_{\alpha,q}
 \int_\Omega
 \|\omega(z)\|_{\mathrm A}^q
 \delta(z)^{\alpha-1+q}\dd v(z),
 \qquad \alpha>0.
\end{equation}
Both integrals use Euclidean volume. Ahn's solution formula is
stated for $C^1$ forms on $\Omega$; see
\cite[Section~2.1, formula~(3)]{Ahn2004}.
Thus it applies to the forms used below, which are smooth in
$\Omega$ and have finite weighted norm even when they grow at
the boundary. The cited theorem is stated for
$(n,1)$-forms.  Its Remark~\textup{(ii)} gives the same estimate for
$(0,1)$-forms.  Equivalently, one can multiply $\omega$ by the constant
holomorphic volume form and then identify the coefficient of the
resulting $(n,0)$-form.
Take $\alpha=1$ in \eqref{eq:Ahn-weighted-estimate}.  Equation
\eqref{eq:Ahn-McNeal-dual} gives
\begin{equation}\label{eq:Ahn-unweighted-estimate}
 \|v\|_q^q
 \le C_q\int_\Omega
 \bigl(\delta(z)\|\omega(z)\|_{\mathrm A}\bigr)^q\dd v(z)
 \lesssim\int_\Omega|\omega(z)|_\Omega^q\dd v(z).
\end{equation}

The Bergman projection is bounded on $L^q(\Omega)$.  Set
$u=(I-P)v$.  Then
\[
 \bar\partial u=\omega,\qquad Pu=0,
 \qquad
 \|u\|_q\le(1+\|P\|_{L^q\to L^q})\|v\|_q.
\]
This proves \eqref{eq:FTI-all-estimate}.
If $u_1,u_2\in L^q$ both satisfy
\eqref{eq:FTI-all-normalization}, then Weyl's lemma gives
$u_1-u_2\in A^q$.  Hence
\[
 u_1-u_2=P(u_1-u_2)=0.
\]
The normalized solution is therefore independent of the solution $v$
chosen in \eqref{eq:Ahn-weighted-estimate}.
Define $\mathcal S_{\Omega,q}\omega=u$.
For forms $\omega_1,\omega_2$ and scalars $a,b\in\C$, the function
$a\mathcal S_{\Omega,q}\omega_1+b\mathcal S_{\Omega,q}\omega_2$
has zero Bergman projection and solves the equation with right-hand
side $a\omega_1+b\omega_2$.  Uniqueness proves linearity.
\end{proof}

For $q=2$, the normalized solution above is the solution of smallest
$L^2$ norm.  The following proof gives an independent verification
using the Bergman metric and Zimmer's estimate.

\begin{proposition}\label{prop:FTI-two}
Every bounded smooth convex domain of finite type has property
$\mathrm{FTI}_2$. More precisely, if $\omega$ is a smooth
$\bar\partial$-closed $(0,1)$-form and
$|\omega|_\Omega\in L^2(\Omega)$, then there is a unique solution
$u\in L^2(\Omega)\ominus A^2(\Omega)$ of
$\bar\partial u=\omega$, and
\begin{equation}\label{eq:FTI-two-estimate}
 \|u\|_2^2\le C_\Omega
 \int_\Omega |\omega(z)|_\Omega^2\dd v(z).
\end{equation}
Here $L^2(\Omega)\ominus A^2(\Omega)$ is the orthogonal complement
of $A^2(\Omega)$ in $L^2(\Omega)$. The map $\omega\mapsto u$ is linear.
\end{proposition}

\begin{proof}
Define the positive Hermitian operator $B_\Omega(z)$ representing
the Bergman metric by its quadratic form in standard coordinates:
\[
 \langle B_\Omega(z)X,X\rangle
 =\sum_{j,k=1}^n
 \frac{\partial^2\log K(z,z)}{\partial z_j\partial\overline z_k}
 X_j\overline{X_k},\qquad X\in\C^n.
\]
For $X\in\C^n$, write $X=\sum_kX_ku_k(z,\delta(z))$.
The estimates for invariant metrics and the comparison of extremal bases
\cite{McNeal1994,NikolovPflug2003,NikolovPflugThomas2013} give
\begin{equation}\label{eq:bergman-mcneal-primal}
 \langle B_\Omega(z)X,X\rangle
 \asymp\sum_{k=1}^n
 \frac{|X_k|^2}{\tau_k(z,\delta(z))^2}.
\end{equation}
On the compact core this follows from positive definiteness and
continuity. Thus the comparison is uniform on $\Omega$.
Taking dual norms yields
\begin{equation}\label{eq:bergman-mcneal-dual}
 |\omega(z)|_{B_\Omega^{-1}}^2
 :=\sup_{X\ne0}
 \frac{|\omega_z(\overline X)|^2}
 {\langle B_\Omega(z)X,X\rangle}
 \asymp\sum_{k=1}^n
 |\omega_k(z)|^2\tau_k(z,\delta(z))^2
 =|\omega(z)|_\Omega^2.
\end{equation}
To see the direction of this comparison, let
$D(z)=\operatorname{diag}(\tau_1^{-2},\ldots,\tau_n^{-2})$
in the extremal coordinates.  Equation
\eqref{eq:bergman-mcneal-primal} says that
\[
 cD(z)\le B_\Omega(z)\le CD(z),
 \qquad
 C^{-1}D(z)^{-1}\le B_\Omega(z)^{-1}
 \le c^{-1}D(z)^{-1}.
\]
The constants $c,C>0$ are independent of $z$.

For bounded convex domains, the geometric estimates in
\cite[Section~4]{Zimmer2021} allow us to apply the unweighted case of
\cite[Corollary~3.4]{Zimmer2023}. It gives a distributional
solution $v\in L^2(\Omega)$ with
\begin{equation}\label{eq:zimmer-unweighted}
 \bar\partial v=\omega,
 \qquad
 \|v\|_2^2\le C
 \int_\Omega|\omega(z)|_{B_\Omega^{-1}}^2\dd v(z).
\end{equation}
Both integrals use Euclidean volume. In the notation of that
corollary, the plurisubharmonic weight is $\lambda_2=0$.
Put $u=(I-P)v$. Then
\[
 \bar\partial u=\omega,
 \qquad u\perp A^2(\Omega),
 \qquad \|u\|_2\le\|v\|_2.
\]
Every other $L^2$ solution has the form $u+h$, with $h\in A^2$.
Consequently,
\[
 \|u+h\|_2^2=\|u\|_2^2+\|h\|_2^2,
\]
so this choice is also the solution of smallest $L^2$ norm.
Equations \eqref{eq:bergman-mcneal-dual} and
\eqref{eq:zimmer-unweighted} prove
\eqref{eq:FTI-two-estimate}.

If $u_1$ and $u_2$ are two such solutions, then
$\bar\partial(u_1-u_2)=0$ in distributions. Weyl's lemma gives
$u_1-u_2\in A^2(\Omega)$. Orthogonality gives $u_1=u_2$.
For closed forms $\omega_1,\omega_2$, let $u_1,u_2$ be their normalized
solutions. For scalars $a,b\in\C$, $a u_1+b u_2$ is orthogonal to $A^2$
and solves the equation with right-hand side $a\omega_1+b\omega_2$.
Uniqueness proves linearity.
\end{proof}

Let $\omega$ be a smooth $\bar\partial$-closed $(0,1)$-form.
Its norm is defined in \eqref{eq:dual-McNeal}.  If
$g\in\Hol(\Omega)$ and $g|\omega|_\Omega\in L^q(\Omega)$, then
\[
 \bar\partial(g\omega)=g\bar\partial\omega=0,
 \qquad |g\omega|_\Omega=|g|\,|\omega|_\Omega.
\]
Applying \eqref{eq:FTI-basic-estimate} to $g\omega$ gives
\begin{equation}\label{eq:dbar-solution}
 \|\mathcal S_{\Omega,q}(g\omega)\|_{L^q(\Omega)}
 \le C_q\|g|\omega|_\Omega\|_{L^q(\Omega)}.
\end{equation}
Conversely, $g=1$ in \eqref{eq:dbar-solution} gives
\eqref{eq:FTI-basic-estimate}.  Thus the holomorphic factor introduces
no additional assumption.  No identity between
$\mathcal S(g\omega)$ and $g\mathcal S\omega$ is required.

The choice $\alpha=1$ in Ahn's estimate is essential here.  Its
right-hand side contains
\[
 \delta^q\|\omega\|_{\mathrm A}^q\asymp|\omega|_\Omega^q.
\]
Thus the powers of the directional radii match the norm required by
the Hankel argument.  This is the estimate from \cite{Ahn2004}; the
weighted estimates in \cite{CharpentierDupain2018} use a different
anisotropic quantity.

The convention with two fixed scales in \eqref{eq:admissible-radius-pair} gives
a fixed margin for the local holomorphic approximants.  The solution estimate in
\cref{prop:FTI-all} gives the Hankel upper estimate
\[
 \|H_ug\|_q\lesssim\|g|\bar\partial u|_\Omega\|_q
\]
for every $1<q<\infty$.  The IDA decomposition and the lower estimate
below do not use the solution operator.

The following decomposition is a main step in the proof.

\begin{lemma}\label{lem:ida-decomposition}
Let $1<p,q<\infty$, $1\le r<\infty$, and
$f\in L^q_{\loc}(\Omega)$.  For every $\varepsilon>0$, there are a
sequence $\eta=(\eta_j)\ge0$,
$f_1\in C^\infty(\Omega)$, and $f_2\in L^q_{\loc}(\Omega)$ such that
$f=f_1+f_2$,
\begin{equation}\label{eq:eta-small}
 \left\|\{V_j^{1/q-1/p}\eta_j\}_j\right\|_{\dideal}
 \le\varepsilon,
\end{equation}
and
\begin{equation}\label{eq:ida-local-decomposition}
 \M_{q,\rho}(|\bar\partial f_1|_\Omega)(a_j)
 +\M_{q,\rho}f_2(a_j)
 \le C\sum_{k\in\mathcal N(j)}
 \bigl(\G_{q,\rho_+}f(a_k)+\eta_k\bigr).
\end{equation}
The neighbor sets have uniformly bounded incidence multiplicities.
Moreover,
\begin{equation}\label{eq:ida-sequence-decomposition}
\begin{aligned}
 &\left\|\{V_j^{1/q-1/p}
 \M_{q,\rho}(|\bar\partial f_1|_\Omega)(a_j)\}\right\|_{\dideal}+\left\|\{V_j^{1/q-1/p}\M_{q,\rho}f_2(a_j)\}\right\|_{\dideal}
 \lesssim\|f\|_{\mathcal I_{p,q}^{r}(\Omega)}+\varepsilon.
\end{aligned}
\end{equation}
\end{lemma}

\begin{proof}
Put $\alpha=1/q-1/p$.  Normalize $(2^{-j})$ in
$\mathfrak d_{p,q}^{\,r}$ by setting
\[
 d_j=\frac{2^{-j}}
 {\|\{2^{-k}\}_{k\ge1}\|_{\mathfrak d_{p,q}^{\,r}}}.
\]
This is possible because $(2^{-j})\in\ell^1$ and every
$\ell^1$ diagonal is nuclear.  More precisely,
\[
 D_d=\sum_{j=1}^\infty d_j e_j^*\otimes e_j,
 \qquad
 \sum_j\|d_je_j^*\|_{(\ell^p)^*}\|e_j\|_{\ell^q}
 =\sum_j|d_j|.
\]
Here $e_j$ is the $j$th unit vector and $e_j^*(c)=c_j$.
Thus $D_d$ is $1$-summing and hence $r$-summing.  Define
\begin{equation*}
 \eta_j=\varepsilon d_jV_j^{-\alpha}.
\end{equation*}
Then \eqref{eq:eta-small} holds.

For each lattice point, freeze the extremal coordinates at scale
$\rho\delta(a_j)$ and put
\[
 P_j(L)=\Q_\rho^{[L]}(a_j),\qquad L=1,2,3,4.
\]
Thus $P_j(1)=\Q_\rho(a_j)$ and $v(P_j(L))=L^{2n}V_j$.
These are affine dilates in one coordinate system.  By
\eqref{eq:polydisc-uniform-enlargement}, their closures are contained
in a fixed scale enlargement.  Use the nested coordinate polydiscs
chosen in \cref{sec:geometry} for the finitely many core cells.

Pull back one fixed cutoff on $2\mathbb D^n$ that equals one on
$\mathbb D^n$.  This gives $\psi_j\in C_c^\infty(P_j(2))$ such that
\[
 0\le\psi_j\le1,\qquad
 \psi_j=1\quad\hbox{on }P_j(1).
\]
In the frozen coordinates its derivatives satisfy
\[
 |(\bar\partial\psi_j)_z
   (\overline{u_k(a_j,\rho\delta(a_j))})|
 \lesssim\tau_k(a_j,\rho\delta(a_j))^{-1}.
\]
On $P_j(2)$, stability of the directional norms at nearby centers,
followed by comparison of the fixed scales $\rho\delta(a_j)$ and
$\delta(z)$, gives
\[
 |\bar\partial\psi_j(z)|_\Omega
 \le C_\rho
 \left(\sum_{k=1}^n
 \tau_k(a_j,\rho\delta(a_j))^2
 |(\bar\partial\psi_j)_z
   (\overline{u_k(a_j,\rho\delta(a_j))})|^2\right)^{1/2}
 \le C_\rho.
\]
Here $C_\rho$ is independent of $j$.  On the core the same bound follows
from the finitely many affine maps.  Put
\[
 \Psi=\sum_j\psi_j,\qquad \chi_j=\psi_j/\Psi.
\]
The covering, engulfing, and finite overlap give
$1\le\Psi\le N$ and $|\bar\partial\Psi|_\Omega\le C_\rho N$.
The quotient rule gives a smooth partition of unity with
\begin{equation}\label{eq:partition-unity}
 \supp\chi_j\Subset P_j(2),\qquad
 \sum_j\chi_j=1,\qquad
 |\bar\partial\chi_j|_\Omega\le C_\rho.
\end{equation}
In fact,
\[
 \bar\partial\chi_j
 =\frac{\bar\partial\psi_j}{\Psi}
 -\frac{\psi_j\bar\partial\Psi}{\Psi^2},
 \qquad
 |\bar\partial\chi_j|_\Omega
 \le |\bar\partial\psi_j|_\Omega
 +\psi_j|\bar\partial\Psi|_\Omega.
\]
The sums are locally finite.  Indeed, a fixed compact set meets only
finitely many lattice cells, by packing and the positive lower bound
for their volumes on that set.

For fixed $j$, define
\[
 \mathcal N(j)=\{k:\supp\chi_k\cap P_j(1)\ne\varnothing\}.
\]
If $k\in\mathcal N(j)$, the two centers lie in a common fixed scale
enlargement.  Engulfing and packing of the disjoint small lattice balls
give
\begin{equation}\label{eq:neighbor-incidence}
 \sup_j\#\mathcal N(j)
 +\sup_k\#\{j:k\in\mathcal N(j)\}<\infty,
 \qquad V_k\asymp V_j\quad(k\in\mathcal N(j)).
\end{equation}
The column bound follows from the same argument with the centers
interchanged.  More precisely, there is a fixed $C_0\ge1$ such that
\[
 \overline{P_j(4)}\subset\Q_{C_0\rho}(a_k)
 \qquad(k\in\mathcal N(j)).
\]
By \eqref{eq:polydisc-uniform-enlargement}, $P_j(4)$ and $P_k(2)$
lie in fixed scale enlargements.  Apply engulfing at an intersection
point to obtain $C_0$.  Increase the constant once to
include the closures.  The finite core charts were chosen with the
same containment property.

Fix $C_*\ge A_*C_0$ in \eqref{eq:admissible-radius-pair}, and take
$s_0<\rho_0/(2A_*)$.  Choose the pair so that
$C_*\rho<\rho_+<s_0$.  Formula
\eqref{eq:polydisc-scale-gauge} and the core nesting convention give
\begin{equation}\label{eq:ida-common-outer-polydisc}
 \overline{P_j(4)}\subset\Q_{\rho_+}(a_k)
 \Subset\Q_{\rho_+}^*(a_k)
 \qquad(k\in\mathcal N(j)).
\end{equation}
In particular, the averaging region below lies in
$\Q_{\rho_+}(a_k)$ for every $k\in\mathcal N(j)$.
All these constants are fixed before choosing the local approximants.

For each $j$, choose
$h_j\in\Hol(\Q_{\rho_+}^*(a_j))$ such that
\[
 \left(\avg_{\Q_{\rho_+}(a_j)}|f-h_j|^q\dd v\right)^{1/q}
 \le\G_{q,\rho_+}f(a_j)+\eta_j.
\]
Set
\begin{equation}\label{eq:ida-f1-f2}
 f_1=\sum_j\chi_jh_j,\qquad
 f_2=f-f_1=\sum_j\chi_j(f-h_j).
\end{equation}
Each product is extended by zero outside $P_j(2)$.  Its support is
compactly contained in the domain of $h_j$, by
\eqref{eq:ida-common-outer-polydisc} with $k=j$.
Thus $f_1\in C^\infty(\Omega)$ and $f_2\in L^q_{\loc}(\Omega)$.

Let $k,\ell\in\mathcal N(j)$. The inclusion
\eqref{eq:ida-common-outer-polydisc} gives
\[
 \left(\avg_{P_j(3)}|f-h_k|^q\dd v\right)^{1/q}
 \le
 \left(\frac{\V_{\rho_+}(a_k)}{v(P_j(3))}\right)^{1/q}
 \bigl(\G_{q,\rho_+}f(a_k)+\eta_k\bigr).
\]
The volume ratio is bounded uniformly in $j$ and $k$.
The same estimate holds with $k$ replaced by $\ell$.
The function $h_k-h_\ell$ is holomorphic on a neighborhood of
$\overline{P_j(4)}$. By the triangle inequality,
\begin{align*}
 \left(\avg_{P_j(3)}|h_k-h_\ell|^q\dd v\right)^{1/q}
 &\le
 \left(\avg_{P_j(3)}|f-h_k|^q\dd v\right)^{1/q}
 +\left(\avg_{P_j(3)}|f-h_\ell|^q\dd v\right)^{1/q}\\
 &\lesssim\G_{q,\rho_+}f(a_k)+\eta_k
 +\G_{q,\rho_+}f(a_\ell)+\eta_\ell.
\end{align*}
The affine map from $P_j(3)$ onto $\mathbb D^n$ sends $P_j(2)$
onto $(2/3)\mathbb D^n$.  The ordinary submean estimate therefore has
a constant independent of $j$.  Scaling back gives
\begin{equation}\label{eq:h-difference}
 \sup_{P_j(2)}|h_k-h_\ell|
 \lesssim\G_{q,\rho_+}f(a_k)+\eta_k
 +\G_{q,\rho_+}f(a_\ell)+\eta_\ell.
\end{equation}

Choose $j_0\in\mathcal N(j)$.  Since
$\sum_k\bar\partial\chi_k=0$, on $\Q_\rho(a_j)$ we have
\[
 \bar\partial f_1
 =\sum_{k\in\mathcal N(j)}
 \bar\partial\chi_k(h_k-h_{j_0})
 \quad\text{on }\Q_\rho(a_j).
\]
Equations \eqref{eq:partition-unity} and \eqref{eq:h-difference} give
\[
 |\bar\partial f_1|_\Omega
 \lesssim\sum_{k\in\mathcal N(j)}
 \bigl(\G_{q,\rho_+}f(a_k)+\eta_k\bigr).
\]
Minkowski's inequality and the second identity in
\eqref{eq:ida-f1-f2} give
\begin{align*}
 \M_{q,\rho}f_2(a_j)
 &\le\sum_{k\in\mathcal N(j)}
 \left(\avg_{\Q_\rho(a_j)}
 |\chi_k(f-h_k)|^q\dd v\right)^{1/q}\\
 &\lesssim\sum_{k\in\mathcal N(j)}
 \bigl(\G_{q,\rho_+}f(a_k)+\eta_k\bigr).
\end{align*}
This proves \eqref{eq:ida-local-decomposition}.

Multiply \eqref{eq:ida-local-decomposition} by $V_j^\alpha$.  Neighboring
volumes are comparable.  Put
\[
 c_k=V_k^\alpha\G_{q,\rho_+}f(a_k).
\]
Each sequence on the left of \eqref{eq:ida-sequence-decomposition} is
pointwise bounded by
\[
 C\left\{\sum_{k\in\mathcal N(j)}
 (c_k+V_k^\alpha\eta_k)\right\}_j.
\]
By \cref{lem:finite-neighbor}, its $\mathfrak d_{p,q}^{\,r}$ norm is at
most $C(\|c\|_{\mathfrak d_{p,q}^{\,r}}+\varepsilon)$.  By
\eqref{eq:lattice-G},
\[
 \|c\|_{\mathfrak d_{p,q}^{\,r}}
 =\|\mathbf G_{p,q,\rho}f\|_{\mathfrak d_{p,q}^{\,r}}
 =\|f\|_{\mathcal I_{p,q}^{r}(\Omega)}.
\]
The constants do not depend on the errors $\eta_j$.  Since
\eqref{eq:eta-small} is exact, no limiting choice of the functions
$h_j$ is needed.
This is \eqref{eq:ida-sequence-decomposition}.
\end{proof}

\begin{lemma}\label{lem:H-two-estimates}
Let $1<q<\infty$.
Let $u\in C^\infty(\Omega)$ and
$g\in\Hol(\Omega)$.  If
\[
 ug\in L^q(\Omega),
 \qquad
 g|\bar\partial u|_\Omega\in L^q(\Omega),
\]
then
\begin{equation}\label{eq:H-smooth}
 \|H_ug\|_{L^q}
 \lesssim\|g|\bar\partial u|_\Omega\|_{L^q}.
\end{equation}
Here $H_ug=(I-P)(ug)$.
For any measurable $u$ with $ug\in L^q(\Omega)$, we have
\begin{equation}\label{eq:H-rough}
 \|H_ug\|_{L^q}
 \le(1+\|P\|_{L^q\to L^q})\|ug\|_{L^q}.
\end{equation}
\end{lemma}

\begin{proof}
Put $\vartheta=\bar\partial u$.  It is a smooth
$\bar\partial$-closed $(0,1)$-form.  Since $g$ is holomorphic,
$g\vartheta$ is also $\bar\partial$-closed.  Set
\[
 v=\mathcal S_{\Omega,q}(g\vartheta).
\]
Equation \eqref{eq:dbar-solution}, applied to $\vartheta$ with
multiplier $g$, gives
\[
 \|v\|_q\le C_q\|g|\bar\partial u|_\Omega\|_q<\infty.
\]
By \eqref{eq:FTI-solution},
\[
 \bar\partial v=g\bar\partial u,
 \qquad
 \bar\partial(ug-v)=0.
\]
Both $ug$ and $v$ belong to $L^q(\Omega)$.  Hence $ug-v$ belongs to
$L^q(\Omega)$ and is distributionally holomorphic.  The distributional
Weyl lemma gives $ug-v\in A^q(\Omega)$.  Since
$P(ug-v)=ug-v$,
\begin{equation*}
 H_ug=(I-P)(ug)=(I-P)v=v,
\end{equation*}
where $Pv=0$ by \eqref{eq:FTI-all-normalization}.  Therefore
\[
 \|H_ug\|_q=\|v\|_q
 \lesssim\|g|\bar\partial u|_\Omega\|_q.
\]
This proves \eqref{eq:H-smooth}.  The second estimate follows directly
from $H_u=(I-P)M_u$ on its natural domain.
\end{proof}

\begin{proof}[Proof of \cref{thm:intro-hankel}]
Assume first that $f\in\mathcal I_{p,q}^{r}(\Omega)$.  Fix
$\varepsilon>0$ and apply \cref{lem:ida-decomposition}.  Put
\[
 \dd\mu_1=|\bar\partial f_1|_\Omega^q\dd v,
 \qquad
 \dd\mu_2=|f_2|^q\dd v.
\]
By \eqref{eq:ida-sequence-decomposition} and
\cref{thm:carleson},
\begin{equation}\label{eq:Jmu12}
 \pi_r(J_{\mu_1})+\pi_r(J_{\mu_2})
 \lesssim\|f\|_{\mathcal I_{p,q}^{r}(\Omega)}+\varepsilon.
\end{equation}
For $g\in A^p$, the two embedding norms are
\begin{equation}\label{eq:H-embedding-identities}
 \|J_{\mu_1}g\|_{L^q(\mu_1)}
 =\|g|\bar\partial f_1|_\Omega\|_{L^q(v)},
 \qquad
 \|J_{\mu_2}g\|_{L^q(\mu_2)}
 =\|gf_2\|_{L^q(v)}.
\end{equation}
For $g\in\Gamma$, the definition of $\mathcal D_{\Omega,q}$ gives
$fg\in L^q$.  The boundedness of $J_{\mu_2}$ gives $f_2g\in L^q$.
Hence
\begin{equation*}
 f_1g=fg-f_2g\in L^q,
 \qquad
 g|\bar\partial f_1|_\Omega\in L^q.
\end{equation*}
Thus both operators below are defined on $\Gamma$.  Recall the weak
$r$-norm $w_r$ from \eqref{eq:weak-r-norm}.
For $g_1,\ldots,g_N\in\Gamma$, equations
\eqref{eq:H-smooth}, \eqref{eq:H-rough}, and
\eqref{eq:H-embedding-identities} imply
\begin{align*}
 \left(\sum_{k=1}^N\|H_{f_1}g_k\|_q^r\right)^{1/r}
 &\lesssim\pi_r(J_{\mu_1})w_r(g_1,\ldots,g_N),\\
 \left(\sum_{k=1}^N\|H_{f_2}g_k\|_q^r\right)^{1/r}
 &\lesssim\pi_r(J_{\mu_2})w_r(g_1,\ldots,g_N).
\end{align*}
On $\Gamma$, $H_f=H_{f_1}+H_{f_2}$.  Hence Minkowski's inequality and
\eqref{eq:Jmu12} give
\begin{align}
 \left(\sum_{k=1}^N\|H_fg_k\|_q^r\right)^{1/r}
 &\le
 \left(\sum_{k=1}^N\|H_{f_1}g_k\|_q^r\right)^{1/r}
 +\left(\sum_{k=1}^N\|H_{f_2}g_k\|_q^r\right)^{1/r}\notag\\
 &\lesssim
 \bigl(\|f\|_{\mathcal I_{p,q}^{r}(\Omega)}+\varepsilon\bigr)
 w_r(g_1,\ldots,g_N).                         \label{eq:H-upper-family}
\end{align}
The case $N=1$ gives a bounded extension from the dense space
$\Gamma$ to $A^p$.  The extension is unique.  Approximate each member
of a finite family in $A^p$ by elements of $\Gamma$.  Passing to the
limit in \eqref{eq:H-upper-family} gives
\begin{equation}\label{eq:H-upper-finite}
 \pi_r(H_f)
 \lesssim\|f\|_{\mathcal I_{p,q}^{r}(\Omega)}+\varepsilon.
\end{equation}
If $g_{k,m}\to g_k$ in $A^p$ for
$1\le k\le N$, then
\[
 |w_r(g_{1,m},\ldots,g_{N,m})-w_r(g_1,\ldots,g_N)|
 \le\left(\sum_{k=1}^N\|g_{k,m}-g_k\|_{A^p}^r\right)^{1/r}
 \longrightarrow0.
\]
The bounded extension gives convergence of each $H_fg_{k,m}$ in $L^q$.
Thus both sides of the inequality for finite families pass to the limit.
Let $\varepsilon\downarrow0$.  This proves the upper estimate.

Conversely, assume $H_f\in\Piabs_r(A^p,L^q)$.  Color the lattice so that,
within each color, the balls $\Q_{\rho_+}^*(a_j)$ are pairwise disjoint and
\eqref{eq:kernel-synthesis} holds.  Fix one color and denote its index
set by $J_\ell$.  All indices in the following construction belong to
$J_\ell$.  On
$\Q_{\rho_+}^*(a_j)$ the kernel $K(\cdot,a_j)$ has no zeros by
\eqref{eq:kernel-lower-finite}. Since $e_{j,p}$ is a scalar
multiple of $K(\cdot,a_j)$, the symbol condition gives
$fe_{j,p}\in L^q(\Omega)$. Hence $P(fe_{j,p})\in A^q(\Omega)$, and
\[
 h_j(z)=\frac{P(f e_{j,p})(z)}{e_{j,p}(z)}
\]
is holomorphic there.  It is admissible in the infimum defining
$\G_{q,\rho_+}f(a_j)$.  The ratio $\rho_+/\rho$ is fixed, so
$\V_{\rho_+}(a_j)\asymp V_j$.  By \eqref{eq:atom-normalization},
\begin{align}
 \|H_fe_{j,p}\|_{L^q(\Q_{\rho_+}(a_j))}
 &=\left(\int_{\Q_{\rho_+}(a_j)}
 |f-h_j|^q|e_{j,p}|^q\dd v\right)^{1/q}\notag\\
 &\gtrsim
 V_j^{1/q-1/p}\G_{q,\rho_+}f(a_j).
 \label{eq:H-local-lower}
\end{align}
Let
\[
 Y=\left(\bigoplus_jL^q(\Q_{\rho_+}(a_j))\right)_{\ell^q},
 \qquad
 Bc=\{(H_fS_pc)|_{\Q_{\rho_+}(a_j)}\}_j.
\]
The target balls are disjoint.  Hence the restriction map
\[
 R:L^q(\Omega)\longrightarrow Y,
 \qquad Ru=\{u|_{\Q_{\rho_+}(a_j)}\}_j,
\]
is contractive.  Since $B=RH_fS_p$, the ideal property gives
\begin{equation*}
 B\in\Piabs_r(\ell^p,Y),
 \qquad
 \pi_r(B)\le\|S_p\|\pi_r(H_f).
\end{equation*}
Let $r_j:[0,1]\to\{-1,1\}$ be the Rademacher functions.  Define
\[
 U_t(c_j)=(r_j(t)c_j),
 \qquad
 V_t(u_j)=(r_j(t)u_j).
\]
Both maps are isometries.  Rademacher averaging gives, first on finite
coordinate compressions, the diagonal blocks of $B$.  Define on
finitely supported sequences
\[
 \mathbb Hc=\{c_j(H_fe_{j,p})|_{\Q_{\rho_+}(a_j)}\}_j.
\]
Enumerate the chosen color.  Let $P_m$ and $Q_m$ be the first-$m$
coordinate projections in the source and target.  Then
\begin{equation}\label{eq:H-block-extraction}
 Q_m\mathbb HP_m
 =\int_0^1Q_mV_tBU_tP_m\dd t.
\end{equation}
The compressed integrand has finitely many blocks.
It is a simple operator-valued function.  Thus the integral is a Bochner integral in
operator norm.
Orthogonality,
\[
 \int_0^1r_j(t)r_k(t)\dd t=\delta_{jk},
\]
proves \eqref{eq:H-block-extraction}.  Convexity and the ideal property
give
\[
 \pi_r(Q_m\mathbb HP_m)\le\pi_r(B).
\]
Here $\delta_{jk}$ is the Kronecker delta.  If $c$ is finitely supported,
then $(Q_m\mathbb HP_m)c=\mathbb Hc$ for all sufficiently large $m$.
Apply this observation simultaneously to each member of a finite family.
The definition of $\pi_r$ gives
\[
 \pi_r(\mathbb H)\le\pi_r(B).
\]
More explicitly, the case of a single finitely supported vector gives
$\|\mathbb Hc\|_Y\le\pi_r(B)\|c\|_{\ell^p}$.
Since $c_{00}$ is dense in $\ell^p$, this defines a unique bounded
extension of $\mathbb H$.  Approximation of finite families gives the
displayed summing estimate for that extension.

Put $v_j=(H_fe_{j,p})|_{\Q_{\rho_+}(a_j)}$ and
$d_j=\|v_j\|_q$.  For $d_j>0$, define
\[
 \varphi_j(u)=d_j^{1-q}
 \int_{\Q_{\rho_+}(a_j)}u\,\overline{v_j}|v_j|^{q-2}\dd v,
\]
where the integrand is zero wherever $v_j=0$.
If $d_j=0$, put $\varphi_j=0$.  H\"older's inequality gives
$\|\varphi_j\|\le1$ and
\[
 \varphi_j(v_j)=d_j.
\]
The block map $\Phi:(u_j)\mapsto(\varphi_j(u_j))$ is contractive since
\[
 \|\Phi(u_j)\|_{\ell^q}^q
 =\sum_j|\varphi_j(u_j)|^q\le\sum_j\|u_j\|_q^q.
\]
Moreover, $\Phi\mathbb H=D_d$.  Therefore
\eqref{eq:H-local-lower}, solidity, and the ideal property give
\begin{equation*}
 \|(\mathbf G_{p,q,\rho}f)|_{J_\ell}\|_{\mathfrak d_{p,q}^{\,r}}
 \lesssim\pi_r(H_f).
\end{equation*}
Let $N_{\mathrm{col}}$ be the number of colors.  For the $\ell$th
color, let $b^{(\ell)}$ equal the corresponding
restriction of $\mathbf G_{p,q,\rho}f$ and vanish elsewhere.  Then
\[
 \mathbf G_{p,q,\rho}f=\sum_{\ell=1}^{N_{\mathrm{col}}}b^{(\ell)},
 \qquad
 \|b^{(\ell)}\|_{\mathfrak d_{p,q}^{\,r}}
 \lesssim\pi_r(H_f).
\]
The triangle inequality gives the estimate on the full lattice
\begin{equation}\label{eq:H-lower-finite}
 \|\mathbf G_{p,q,\rho}f\|_{\mathfrak d_{p,q}^{\,r}}
 \le\sum_{\ell=1}^{N_{\mathrm{col}}}
 \|b^{(\ell)}\|_{\mathfrak d_{p,q}^{\,r}}
 \lesssim\pi_r(H_f).
\end{equation}
This proves the necessity part of \cref{thm:intro-hankel}.
Combine \eqref{eq:H-upper-finite} with the lower estimate.  This proves
\eqref{eq:intro-hankel-theorem} and \eqref{eq:intro-hankel-norm}.
\end{proof}

For the target $L^2$, the criterion has an explicit form. Put
\[
 \kappa_2(p,r)=
 \begin{cases}
 p',&1<p\le2,\\
 \max\{p',\min\{r,2\}\},&2<p<\infty.
 \end{cases}
\]
The known diagonal classification gives
\begin{equation}\label{eq:intro-Hankel-two-explicit}
 \pi_r(H_f:A^p\to L^2)
 \asymp
 \left(\sum_j
   [V_j^{1/2-1/p}\G_{2,\rho_+}f(a_j)]^{\kappa_2(p,r)}
 \right)^{1/\kappa_2(p,r)},
 \qquad f\in\mathcal D_{\Omega,2},
\end{equation}
with both sides allowed to be infinite. At $p=2$ this is the
Hilbert--Schmidt condition, independently of $r$; see
\cref{prop:Hilbert-calibration}.

\subsection{Quantitative approximation}

The next bound uses the local distance sequence directly. The proof
first dominates the Hankel operator by a Carleson embedding. It then
transfers each approximant without increasing its rank.

\begin{corollary}\label{cor:Hankel-summing-approximation}
Let $1<p,q<\infty$ and $1\le r<\infty$. Assume that $p=2$ or $q=2$.
Let $f\in\mathcal D_{\Omega,q}\cap\mathcal I_{p,q}^{r}(\Omega)$, and put
\[
 \gamma_j=V_j^{1/q-1/p}\G_{q,\rho_+}f(a_j),\qquad
 \eta_L=(L+1)^{n-1}\vartheta^{L+1},\quad L\ge0.
\]
Here $0<\vartheta<1$ is a fixed geometric constant. There are constants
$A>0$ and $C_0\in\mathbb N$, independent of $f$, such that
\begin{equation}\label{eq:Hankel-summing-rank-budget}
 a_N^{(r)}(H_f)
 \le A\!\inf_{\substack{F\subset\mathbb N\text{ finite},\ L\ge0\\
              C_0\#F\binom{L+n}{n}<N}}
 \left(
 \|\gamma\mathbf1_{F^c}\|_{\mathfrak d_{p,q}^{\,r}}
 +\eta_L\|\gamma\|_{\mathfrak d_{p,q}^{\,r}}
 \right),\qquad N\ge1.
\end{equation}
The degree $L$ is an integer.
\end{corollary}

\begin{proof}
Fix $\varepsilon>0$. Apply \cref{lem:ida-decomposition}, and write
$f=f_1+f_2$. Put
\[
 w_1=|\bar\partial f_1|_\Omega,\qquad w_2=|f_2|,
 \qquad \mathrm d\mu=(w_1^q+w_2^q)\,\mathrm dv,
 \qquad e_j=V_j^{1/q-1/p}\eta_j.
\]
The measure $\mu$ is locally finite and $\|e\|_{\mathfrak d_{p,q}^{\,r}}
\le\varepsilon$. Let $C$ be the fixed enlargement constant in
\cref{thm:summing-approximation}, and set
\[
 b_j=V_j^{-1/p}\mu(\Q_{C\rho}(a_j))^{1/q},\qquad
 \mathcal M(j)=\{k:\Q_\rho(a_k)\cap\Q_{C\rho}(a_j)\ne\varnothing\}.
\]
The covering and packing estimates at fixed scales give
\[
 \Q_{C\rho}(a_j)\subset\bigcup_{k\in\mathcal M(j)}\Q_\rho(a_k),
 \qquad V_k\asymp V_j\quad(k\in\mathcal M(j)).
\]
Both the row and column incidence bounds for $\mathcal M$ are uniform.
By the triangle inequality and \eqref{eq:ida-local-decomposition},
\begin{align*}
 b_j
 &\le V_j^{-1/p}\sum_{k\in\mathcal M(j)}
       \mu(\Q_\rho(a_k))^{1/q}\\
 &\lesssim\sum_{k\in\mathcal M(j)}V_k^{1/q-1/p}
       \bigl(\M_{q,\rho}w_1(a_k)+\M_{q,\rho}w_2(a_k)\bigr)\\
 &\lesssim\sum_{k\in\mathcal M(j)}\sum_{\ell\in\mathcal N(k)}
       (\gamma_\ell+e_\ell).
\end{align*}
Set $\mathcal N_*(j)=\bigcup_{k\in\mathcal M(j)}\mathcal N(k)$.
The multiplicities in the last double sum are uniformly bounded.
Thus
\[
 b_j\lesssim\sum_{\ell\in\mathcal N_*(j)}(\gamma_\ell+e_\ell).
\]
The relation $\mathcal N_*$ has uniform row and column incidence bounds.
All these relations are fixed independently of $f$ and $\varepsilon$.
For a finite set $F$, define
\[
 F^+=\{j:\mathcal N_*(j)\cap F\ne\varnothing\}.
\]
There is a fixed integer $C_0$ such that $\#F^+\le C_0\#F$.
If $j\notin F^+$, every index in $\mathcal N_*(j)$ lies outside $F$.
Consequently, solidity and \cref{lem:finite-neighbor} give
\begin{align}
 \|b\mathbf1_{(F^+)^c}\|_{\mathfrak d_{p,q}^{\,r}}
 &\lesssim\|\gamma\mathbf1_{F^c}\|_{\mathfrak d_{p,q}^{\,r}}
            +\varepsilon,\label{eq:Hankel-approx-tail}\\
 \|b\mathbf1_{F^+}\|_{\mathfrak d_{p,q}^{\,r}}
 &\le\|b\|_{\mathfrak d_{p,q}^{\,r}}
 \lesssim\|\gamma\|_{\mathfrak d_{p,q}^{\,r}}+\varepsilon.
 \label{eq:Hankel-approx-full}
\end{align}
In particular, $J_\mu\in\Piabs_r(A^p,L^q(\mu))$ by
\cref{thm:carleson}.

For $g\in\Gamma$, the assumption $f\in\mathcal D_{\Omega,q}$ gives
$fg\in L^q$. Boundedness of $J_\mu$ gives $f_2g\in L^q$ and
$gw_1\in L^q$, so $f_1g\in L^q$ as well.
Apply \cref{lem:H-two-estimates} to the two terms. This gives
\[
 \|H_fg\|_q
 \lesssim\|gw_1\|_q+\|gw_2\|_q
 \lesssim\left(\int_\Omega|g|^q\,\mathrm d\mu\right)^{1/q}.
\]
The bounded extensions and the density of $\Gamma$ give
\begin{equation}\label{eq:Hankel-approx-domination}
 \|H_fg\|_q\le C_1\|J_\mu g\|_{L^q(\mu)}
 \qquad(g\in A^p).
\end{equation}
The constant $C_1$ does not depend on $f$ or $\varepsilon$.

Apply \cref{thm:summing-approximation} with the finite set $F^+$.
It gives $R:A^p\to L^q(\mu)$ with
\[
 \operatorname{rank}R\le C_0\#F\binom{L+n}{n}
\]
and, by \eqref{eq:Hankel-approx-tail} and \eqref{eq:Hankel-approx-full},
\begin{equation}\label{eq:Hankel-approx-embedding-error}
 \pi_r(J_\mu-R)
 \lesssim\|\gamma\mathbf1_{F^c}\|_{\mathfrak d_{p,q}^{\,r}}
       +\eta_L\|\gamma\|_{\mathfrak d_{p,q}^{\,r}}
       +\varepsilon.
\end{equation}
Here $\sup_{L\ge0}\eta_L<\infty$, so the last constant is uniform in $L$.

Suppose first that $p=2$. Let $E$ be the orthogonal projection of
$A^2$ onto $(\ker R)^\perp$, and put $B=H_fE$. Then
\[
 \operatorname{rank}B\le\operatorname{rank}E
 =\operatorname{rank}R,
 \qquad R(I-E)=0,\qquad \|I-E\|\le1.
\]
For every finite family $(g_k)\subset A^2$, domination gives
\begin{align*}
 \left(\sum_k\|(H_f-B)g_k\|_q^r\right)^{1/r}
 &\le C_1\left(\sum_k\|J_\mu(I-E)g_k\|_{L^q(\mu)}^r\right)^{1/r}\\
 &=C_1\left(\sum_k\|(J_\mu-R)(I-E)g_k\|_{L^q(\mu)}^r\right)^{1/r}\\
 &\le C_1\pi_r(J_\mu-R)w_r(g_1,\ldots,g_m).
\end{align*}
Hence $\pi_r(H_f-B)\le C_1\pi_r(J_\mu-R)$.

Now suppose that $q=2$. Define $U_0(J_\mu g)=H_fg$.
Equation \eqref{eq:Hankel-approx-domination} makes this a well-defined
bounded map on $J_\mu(A^p)$ with norm at most $C_1$.
Extend it continuously to
$Z=\overline{J_\mu(A^p)}\subset L^2(\mu)$.
Let $Q_Z$ be the orthogonal projection onto $Z$, and put
$U=U_0Q_Z:L^2(\mu)\to L^2(v)$.
Then $H_f=UJ_\mu$ and $\|U\|\le C_1$. With $B=UR$,
\[
 \operatorname{rank}B\le\operatorname{rank}R,
 \qquad
 \pi_r(H_f-B)\le C_1\pi_r(J_\mu-R).
\]
Thus the same rank and error bounds hold in both cases.
For every pair $F,L$ allowed in \eqref{eq:Hankel-summing-rank-budget},
the rank is less than $N$. Take the infimum over operators of finite rank
and then let $\varepsilon$ tend to zero in
\eqref{eq:Hankel-approx-embedding-error}.
Finally take the infimum over $F,L$.
This proves \eqref{eq:Hankel-summing-rank-budget}.
\end{proof}

\subsection{Separated symbols and diagonal models}

The lower estimate above is sharp already on finite separated families.
The next construction also explains why analytic distance, rather than
the local mean of the symbol, is the natural quantity for Hankel operators.
For a sublattice indexed by $J_0$, every sequence on $J_0$ is identified
with its zero extension to the ambient lattice.

\begin{theorem}
\label{thm:sharp-Hankel-models}
Let $1<p,q<\infty$ and $1\le r<\infty$.  Let
$\Lambda_0=\{a_j:j\in J_0\}$ be a sufficiently separated sublattice.
For every finitely supported sequence $b=(b_j)_{j\in J_0}\ge0$, there
is a bounded, compactly supported symbol
$f_b\in\mathcal D_{\Omega,q}$ such that
\begin{align}
 \|\mathbf G_{p,q,\rho}f_b\|_{\mathfrak d_{p,q}^{\,r}}
 &\asymp\|b\|_{\mathfrak d_{p,q}^{\,r}},
 \label{eq:sharp-Hankel-G}\\
 \|\mathbf M_{p,q,\rho}f_b\|_{\mathfrak d_{p,q}^{\,r}}
 &\asymp\|b\|_{\mathfrak d_{p,q}^{\,r}},
 \label{eq:sharp-Hankel-M}\\
 \pi_r(H_{f_b}:A^p\to L^q)
 &\asymp\|b\|_{\mathfrak d_{p,q}^{\,r}}.
 \label{eq:sharp-Hankel-operator}
\end{align}
The constants are independent of $b$ and of its finite support.
The construction uses the rough estimate \eqref{eq:H-rough} and
the lower estimate \eqref{eq:H-lower-finite}.
\end{theorem}

\begin{proof}
For $j\in J_0$, let
\[
 \Phi_j(\zeta)
 =a_j+c_0\sum_{k=1}^n
 \tau_k(a_j,A_*\rho\delta(a_j))
 \zeta_k u_k(a_j,A_*\rho\delta(a_j)).
\]
The comparison of the two fixed scales gives a constant $C_\Phi$,
independent of $j$, such that
\[
 N_{a_j,\rho}(\Phi_j(\zeta)-a_j)
 \le C_\Phi\|\zeta\|_\infty.
\]
Choose $0<\theta_1<\min\{1,(2C_\Phi)^{-1}\}$ and put
$\theta_0=\theta_1/2$.  Then the sets
\[
 U_j=\Phi_j(\theta_1\mathbb D^n),
 \qquad
 U_j^0=\Phi_j(\theta_0\mathbb D^n)
\]
satisfy
\begin{equation}\label{eq:sharp-Hankel-cells}
 U_j^0\Subset U_j\Subset\Q_\rho(a_j),
 \qquad
 v(U_j^0)\asymp v(U_j)\asymp V_j.
\end{equation}
The separation makes the sets $U_j$ pairwise disjoint.  Put
\[
 \zeta_j=\Phi_j^{-1},
 \qquad
 A_j=b_jV_j^{1/p-1/q},
\]
and define
\begin{equation*}
 f_b(z)=
 \sum_{j\in J_0}A_j
 \overline{(\zeta_j(z))_1}\,\mathbf1_{U_j}(z).
\end{equation*}
Only finitely many summands are nonzero.  Their supports are relatively
compact in $\Omega$.  Hence $f_b$ is bounded and compactly supported.
For each fixed $z\in\Omega$, the kernel section $K(\cdot,z)$ is bounded
on that support.  Therefore
\[
 f_bK(\cdot,z)\in L^q(\Omega),
 \qquad
 f_b\in\mathcal D_{\Omega,q}.
\]
We first prove the local lower estimate.  Let
$h\in\Hol(\Q_{\rho_+}^*(a_j))$.  On $U_j^0$, change variables by $\Phi_j$.
Its real Jacobian is comparable to $V_j$.  Thus
\begin{equation}\label{eq:sharp-Hankel-change-variable}
\begin{aligned}
&\left(
  \avg_{\Q_{\rho_+}(a_j)}
  |f_b-h|^q\dd v
 \right)^{1/q}
\gtrsim
 \left(
  \int_{\theta_0\mathbb D^n}
  |A_j\overline{\zeta_1}-h(\Phi_j(\zeta))|^q
  \dd v(\zeta)
 \right)^{1/q}.
\end{aligned}
\end{equation}
The distance from $\overline{\zeta_1}$ to holomorphic functions is
positive on the fixed polydisc.  Define
\[
 L(u)=\int_{\theta_0\mathbb D^n}u(\zeta)\zeta_1\dd v(\zeta).
\]
If $a$ is holomorphic on a neighborhood of the closed polydisc, angular
integration gives
\begin{equation*}
 L(a)=0.
\end{equation*}
On the other hand,
\[
 L(\overline{\zeta_1})
 =\int_{\theta_0\mathbb D^n}|\zeta_1|^2\dd v(\zeta)>0.
\]
H\"older's inequality gives
\begin{align*}
 A_jL(\overline{\zeta_1})
 &=|L(A_j\overline{\zeta_1}-a)|\le
 \|A_j\overline{\zeta_1}-a\|_{L^q(\theta_0\mathbb D^n)}
 \|\zeta_1\|_{L^{q'}(\theta_0\mathbb D^n)}.
\end{align*}
Hence
\begin{equation}\label{eq:sharp-Hankel-fixed-distance}
 \inf_{a\in\Hol(\theta_1\mathbb D^n)}
 \|A_j\overline{\zeta_1}-a\|_{L^q(\theta_0\mathbb D^n)}
 \ge c_qA_j.
\end{equation}
Apply this to $a=h\circ\Phi_j$ in
\eqref{eq:sharp-Hankel-change-variable}.  This function belongs to
$\Hol(\theta_1\mathbb D^n)$ because
$U_j\Subset\Q_\rho(a_j)\Subset\Q_{\rho_+}^*(a_j)$.  We obtain
\begin{equation}\label{eq:sharp-Hankel-local-lower}
 V_j^{1/q-1/p}\G_{q,\rho_+}f_b(a_j)\gtrsim b_j.
\end{equation}
The same change of variables on $U_j$ gives
\begin{align*}
 \M_{q,\rho}f_b(a_j)^q
 &\ge \frac{|A_j|^q}{V_j}
 \int_{U_j}|(\zeta_j(z))_1|^q\dd v(z)\notag\\
 &\asymp |A_j|^q
 \int_{\theta_1\mathbb D^n}|\zeta_1|^q\dd v(\zeta)
 \asymp |A_j|^q.
\end{align*}
Therefore
\begin{equation}\label{eq:sharp-Hankel-M-local-lower}
 V_j^{1/q-1/p}\M_{q,\rho}f_b(a_j)\gtrsim b_j.
\end{equation}

The reverse estimate follows from
$\G_{q,\rho_+}f_b\le\M_{q,\rho_+}f_b$.  More precisely, let
$s\in\{\rho,\rho_+\}$.  If $\Q_s(a_k)$ meets $U_j$, engulfing gives
$V_k\asymp V_j$.
Each ball meets at most a fixed number of the sets $U_j$.  Each $U_j$
meets at most a fixed number of lattice balls.  Consequently,
\begin{equation}\label{eq:sharp-Hankel-upper-neighbor}
 V_k^{1/q-1/p}\M_{q,s}f_b(a_k)
 \lesssim
 \sum_{j\in\mathcal N(k)}b_j,
 \qquad s\in\{\rho,\rho_+\},
\end{equation}
where the relation $\mathcal N$ has bounded row and column incidence.
Indeed, the supports $U_j$ are disjoint and
$|\zeta_{j,1}|\le\theta_1$ there.  Hence, for $s=\rho$ or $\rho_+$,
\begin{align*}
 \M_{q,s}f_b(a_k)^q
 &\le \frac{C}{V_k}
 \sum_{j\in\mathcal N(k)}|A_j|^qv(U_j)
 \le \frac{C}{V_k}
 \sum_{j\in\mathcal N(k)}b_j^qV_j^{q/p},\\
 V_k^{1/q-1/p}\M_{q,s}f_b(a_k)
 &\le C\left(\sum_{j\in\mathcal N(k)}b_j^q\right)^{1/q}
 \le C\sum_{j\in\mathcal N(k)}b_j.
\end{align*}
Extend $b$ by zero outside $J_0$.  Equations
\eqref{eq:sharp-Hankel-local-lower} and
\eqref{eq:sharp-Hankel-M-local-lower}, together with solidity, give the
two lower norm bounds.  Equation
\eqref{eq:sharp-Hankel-upper-neighbor} and
\cref{lem:finite-neighbor} give the two upper norm bounds.  This proves
\eqref{eq:sharp-Hankel-G} and \eqref{eq:sharp-Hankel-M}.

By \cref{cor:multiplication},
\[
 \pi_r(M_{f_b})
 \asymp\|\mathbf M_{p,q,\rho}f_b\|_{\mathfrak d_{p,q}^{\,r}}
 \asymp\|b\|_{\mathfrak d_{p,q}^{\,r}}.
\]
Since $H_{f_b}=(I-P)M_{f_b}$,
\begin{equation*}
 \pi_r(H_{f_b})
 \le(1+\|P\|_{L^q\to L^q})\pi_r(M_{f_b})
 \lesssim\|b\|_{\mathfrak d_{p,q}^{\,r}}.
\end{equation*}
The unconditional lower estimate on the full lattice
\eqref{eq:H-lower-finite} and \eqref{eq:sharp-Hankel-G} give
\[
 \|b\|_{\mathfrak d_{p,q}^{\,r}}
 \lesssim\pi_r(H_{f_b}).
\]
This proves \eqref{eq:sharp-Hankel-operator}.
\end{proof}

The restriction to finite support can be removed on the summing ideal.
Thus every nonnegative sequence in $\mathfrak d_{p,q}^{\,r}$
can be represented by a symbol on $\Omega$.

\begin{theorem}
\label{thm:infinite-separated-realization}
Let $1<p,q<\infty$ and $1\le r<\infty$.  Let
$\Lambda_0=\{a_j:j\in J_0\}$ be a sufficiently separated infinite
sublattice.  For every nonnegative
$b=(b_j)_{j\in J_0}\in\mathfrak d_{p,q}^{\,r}$, there are symbols
\[
 f_b^{\mathrm M},f_b^{\mathrm H}\in
 L^q(\Omega)\cap\mathcal D_{\Omega,q}
\]
such that
\begin{align}
 \|\mathbf M_{p,q,\rho}f_b^{\mathrm M}\|_{\mathfrak d_{p,q}^{\,r}}
 &\asymp\|b\|_{\mathfrak d_{p,q}^{\,r}},
 \label{eq:infinite-model-M-data}\\
 \pi_r(M_{f_b^{\mathrm M}})
 &\asymp\|b\|_{\mathfrak d_{p,q}^{\,r}},
 \label{eq:infinite-model-M-operator}\\
 \|\mathbf G_{p,q,\rho}f_b^{\mathrm H}\|_{\mathfrak d_{p,q}^{\,r}}
 &\asymp\|b\|_{\mathfrak d_{p,q}^{\,r}},
 \label{eq:infinite-model-H-data}\\
 \pi_r(H_{f_b^{\mathrm H}})
 &\asymp\|b\|_{\mathfrak d_{p,q}^{\,r}}.
 \label{eq:infinite-model-H-operator}
\end{align}
All constants are independent of $b$.
\end{theorem}

\begin{proof}
Enumerate the ambient lattice.  Start with the pairwise disjoint balls
$\Q_{c\rho}(a_j)$.  Assign every remaining point of $\Omega$ to the
first ball $\Q_\rho(a_j)$ that contains it.  This gives measurable,
pairwise disjoint cells with
\[
 \Q_{c\rho}(a_j)\subset E_j\subset\Q_\rho(a_j).
\]
Retain the cells with $j\in J_0$.  Choose
$0<\theta_0<\theta_1<1$ as in \eqref{eq:sharp-Hankel-cells}.  For these
indices, put
\[
 \Phi_j(\zeta)
 =a_j+c_0\sum_{k=1}^n
 \tau_k(a_j,A_*\rho\delta(a_j))
 \zeta_ku_k(a_j,A_*\rho\delta(a_j)),
\]
and set
\[
 U_j=\Phi_j(\theta_1\mathbb D^n),
 \qquad U_j^0=\Phi_j(\theta_0\mathbb D^n),
 \qquad \zeta_j=\Phi_j^{-1}.
\]
The separation makes the sets $U_j$
pairwise disjoint.  Moreover,
\begin{equation}\label{eq:infinite-model-volume}
 v(E_j)\asymp v(U_j)\asymp V_j.
\end{equation}
Put
\begin{align*}
 f_b^{\mathrm M}(z)
 &=\sum_{j\in J_0}
 b_jV_j^{1/p-1/q}\mathbf1_{E_j}(z),
 \\
 f_b^{\mathrm H}(z)
 &=\sum_{j\in J_0}
 b_jV_j^{1/p-1/q}
 \overline{(\zeta_j(z))_1}\mathbf1_{U_j}(z).
\end{align*}
The supports are disjoint, so both sums are well defined pointwise.

We first check global $L^q$ integrability.  Let
\[
 c_j=V_j^{1/p}.
\]
The balls $\Q_{c\rho}(a_j)$ are pairwise disjoint and have volume
comparable to $V_j$.  Hence
\begin{equation*}
 \|c\|_{\ell^p}^p
 =\sum_{j\in J_0}V_j
 \lesssim v(\Omega)<\infty.
\end{equation*}
Since $D_b\in\Piabs_r(\ell^p,\ell^q)$,
$\|D_b\|\le\pi_r(D_b)=\|b\|_{\mathfrak d_{p,q}^{\,r}}$.  Therefore
\begin{equation}\label{eq:weighted-bq-summable}
 \sum_{j\in J_0}b_j^qV_j^{q/p}
 =\|D_bc\|_{\ell^q}^q
 \le\|D_b\|^q\|c\|_{\ell^p}^q<\infty.
\end{equation}
Equations \eqref{eq:infinite-model-volume} through \eqref{eq:weighted-bq-summable} imply
\begin{align*}
 \|f_b^{\mathrm M}\|_{L^q}^q
 &\asymp\sum_jb_j^qV_j^{q/p},
 \\
 \|f_b^{\mathrm H}\|_{L^q}^q
 &\lesssim\sum_jb_j^qV_j^{q/p}.
\end{align*}
Their membership in $\mathcal D_{\Omega,q}$ will follow from the local
estimates below and the Carleson theorem.

We next prove the local lower estimates.  On
$\Q_{c\rho}(a_j)\subset E_j$, the first symbol has the constant value
$b_jV_j^{1/p-1/q}$.  Thus
\[
 \M_{q,\rho}f_b^{\mathrm M}(a_j)
 \gtrsim b_jV_j^{1/p-1/q}.
\]
For the second symbol, repeat
\eqref{eq:sharp-Hankel-change-variable} through \eqref{eq:sharp-Hankel-fixed-distance} on $U_j^0$.  Since the supports
are disjoint, the other summands vanish there.  Consequently, for
$j\in J_0$,
\begin{align*}
 V_j^{1/q-1/p}\M_{q,\rho}f_b^{\mathrm M}(a_j)
 &\gtrsim b_j,
 \\
 V_j^{1/q-1/p}\G_{q,\rho_+}f_b^{\mathrm H}(a_j)
 &\gtrsim b_j.
\end{align*}
For every lattice point $a_k$, finite overlap gives
\begin{align*}
 V_k^{1/q-1/p}\M_{q,\rho}f_b^{\mathrm M}(a_k)
 &\lesssim\sum_{j\in\mathcal N(k)}b_j,
 \\
 V_k^{1/q-1/p}\M_{q,\rho}f_b^{\mathrm H}(a_k)
 &\lesssim\sum_{j\in\mathcal N(k)}b_j.
\end{align*}
The last estimate also holds with $\M_{q,\rho_+}$ in place of
$\M_{q,\rho}$.  Indeed, $\rho_+/\rho$ is fixed and the larger balls
have uniformly bounded overlap and incidence.  Hence
\begin{equation*}
 V_k^{1/q-1/p}\M_{q,\rho_+}f_b^{\mathrm H}(a_k)
 \lesssim\sum_{j\in\mathcal N_+(k)}b_j,
\end{equation*}
where $\mathcal N_+$ has bounded row and column incidence.
The relation $\mathcal N$ has bounded row and column incidence.  The
proof uses only one ball at a time.  It does not require finite support.
Solidity and \cref{lem:finite-neighbor} therefore give
\begin{align}
 \|\mathbf M_{p,q,\rho}f_b^{\mathrm M}\|_{\mathfrak d_{p,q}^{\,r}}
 &\asymp\|b\|_{\mathfrak d_{p,q}^{\,r}},
 \label{eq:infinite-M-sequence-equivalence}\\
 \|\mathbf G_{p,q,\rho}f_b^{\mathrm H}\|_{\mathfrak d_{p,q}^{\,r}}
 &\asymp\|b\|_{\mathfrak d_{p,q}^{\,r}},
 \label{eq:infinite-H-sequence-equivalence}\\
 \|\mathbf M_{p,q,\rho}f_b^{\mathrm H}\|_{\mathfrak d_{p,q}^{\,r}}
 &\lesssim\|b\|_{\mathfrak d_{p,q}^{\,r}}.
 \label{eq:infinite-H-mean-upper}
\end{align}

Apply \cref{thm:carleson} to the locally finite measures
\[
 \dd\mu_{\mathrm M}=|f_b^{\mathrm M}|^q\dd v,
 \qquad
 \dd\mu_{\mathrm H}=|f_b^{\mathrm H}|^q\dd v.
\]
Equations \eqref{eq:infinite-M-sequence-equivalence} and
\eqref{eq:infinite-H-mean-upper} give bounded embeddings
$J_{\mu_{\mathrm X}}:A^p\to L^q(\mu_{\mathrm X})$ for
$\mathrm X\in\{\mathrm M,\mathrm H\}$.
Since $K(\cdot,z)\in A^p$ for each fixed $z\in\Omega$,
\[
 \|f_b^{\mathrm X}K(\cdot,z)\|_{L^q(v)}
 =\|J_{\mu_{\mathrm X}}K(\cdot,z)\|_{L^q(\mu_{\mathrm X})}
 \le\|J_{\mu_{\mathrm X}}\|\|K(\cdot,z)\|_{A^p}<\infty,
 \qquad \mathrm X\in\{\mathrm M,\mathrm H\}.
\]
Thus both symbols belong to $\mathcal D_{\Omega,q}$ before we apply
the multiplication and Hankel criteria.

The multiplication theorem and
\eqref{eq:infinite-M-sequence-equivalence} prove
\eqref{eq:infinite-model-M-data} and \eqref{eq:infinite-model-M-operator}.  For the Hankel symbol,
\[
 \pi_r(H_{f_b^{\mathrm H}})
 \le(1+\|P\|)\pi_r(M_{f_b^{\mathrm H}})
 \lesssim\|b\|_{\mathfrak d_{p,q}^{\,r}}
\]
by \eqref{eq:infinite-H-mean-upper}.  The unconditional lower Hankel
estimate and \eqref{eq:infinite-H-sequence-equivalence} give
\[
 \|b\|_{\mathfrak d_{p,q}^{\,r}}
 \lesssim\pi_r(H_{f_b^{\mathrm H}}).
\]
This proves \eqref{eq:infinite-model-H-data} and
\eqref{eq:infinite-model-H-operator}.
\end{proof}

Thus the sequence ideal also describes infinite sequences.
Each nonnegative sequence in the ideal gives local data with an equivalent
norm.  This holds for both multiplication and big Hankel operators.

The construction above can also preserve linear structure.
On a suitable infinite sublattice, a fixed local moment recovers the
whole diagonal operator. This gives lower bounds for approximation
by operators of arbitrary finite rank.

\begin{theorem}\label{thm:complemented-Hankel-model}
Let $1<p,q<\infty$. There is a sufficiently separated infinite
sublattice, relabeled as $\{a_j\}_{j\ge1}$, with the following
properties. For every $1\le r<\infty$, the formula
\begin{equation}\label{eq:linear-Hankel-model}
 f_b(z)=\sum_{j\ge1}b_jV_j^{1/p-1/q}
 \overline{(\zeta_j(z))_1}\mathbf1_{U_j}(z),
 \qquad
 \mathcal E b=H_{f_b},
\end{equation}
defines a bounded complex-linear map
\[
 \mathcal E:\mathfrak d_{p,q}^{\,r}\longrightarrow
 \Piabs_r(A^p(\Omega),L^q(\Omega)).
\]
Here $U_j$ and $\zeta_j$ are the cells and coordinates in
\cref{thm:infinite-separated-realization}.
For each $b\in\mathfrak d_{p,q}^{\,r}$, the symbol in
\eqref{eq:linear-Hankel-model} belongs to
$L^q(\Omega)\cap\mathcal D_{\Omega,q}$. There are bounded maps
\[
 T_p:\ell^p\longrightarrow A^p(\Omega),
 \qquad Q:L^q(\Omega)\longrightarrow\ell^q,
\]
independent of $b$ and $r$, such that
\begin{equation}\label{eq:Hankel-diagonal-factorization}
 QH_{f_b}T_p=D_b.
\end{equation}
There is also a bounded linear map
\[
 \mathcal C:\Piabs_r(A^p(\Omega),L^q(\Omega))
 \longrightarrow\mathfrak d_{p,q}^{\,r},
 \qquad \mathcal C\mathcal E=I.
\]
Thus $\mathcal E\mathfrak d_{p,q}^{\,r}$ is a complemented subspace of
$\Piabs_r(A^p(\Omega),L^q(\Omega))$, and its elements are big Hankel
operators. Moreover, for every $N\ge1$,
\begin{equation}\label{eq:Hankel-approximation-lower}
 a_N^{(r)}(H_{f_b})
 \ge \frac{a_N^{(r)}(D_b)}{\|Q\|\,\|T_p\|}.
\end{equation}
The construction uses only the Carleson theorem and boundedness of $P$.
\end{theorem}

\begin{proof}
We first select the sublattice. Write $K_j=K(a_j,a_j)$ and
$p'=p/(p-1)$. The sampling map
\[
 R_ph=\{K_j^{-1/p}h(a_j)\}_j
\]
is bounded from $A^p$ to $\ell^p$. Indeed, submean and the disjoint
small balls give
\[
 \sum_jK_j^{-1}|h(a_j)|^p
 \lesssim\sum_j\int_{\Q_{c\rho}(a_j)}|h|^p\dd v
 \le\|h\|_{A^p}^p.
\]
Every infinite separated lattice has only finitely many points in
each compact subset of $\Omega$. As its points approach the boundary,
$V_j\to0$. For $1<s<\infty$ and fixed $h\in A^s$, the same submean
estimate gives
\[
 K_j^{-1}|h(a_j)|^s
 \lesssim\int_{\Q_{c\rho}(a_j)}|h|^s\dd v\longrightarrow0.
\]
The last limit follows from absolute continuity of the integral,
since $v(\Q_{c\rho}(a_j))\asymp V_j\to0$.

Recall $e_{j,p}=K(\cdot,a_j)/K_j^{1/p'}$. The matrix of $R_pS_p$ is
\[
 M_{ij}=\frac{K(a_i,a_j)}{K_i^{1/p}K_j^{1/p'}},
 \qquad M_{jj}=1.
\]
For fixed $i$, both $M_{ij}$ and $M_{ji}$ tend to zero as $a_j$
approaches the boundary. For $M_{ji}$, apply the limit above
with $s=p$ to $K(\cdot,a_i)\in A^p$. For $M_{ij}$, use kernel
symmetry and apply it with $s=p'$ to
$K(\cdot,a_i)\in A^{p'}$. These memberships were proved in
\cref{lem:geometry-package}.
Start with one infinite sufficiently separated sublattice.
Choose its points recursively. After $a_1,\ldots,a_{j-1}$ have
been fixed, choose $a_j$ farther towards the boundary so that
\[
 |M_{ij}|\le2^{-i-j-2},\qquad
 |M_{ji}|\le2^{-i-j-2},\qquad 1\le i<j.
\]
The two limits above allow this choice. Each stage imposes only
finitely many conditions. Taking a subsequence preserves separation.
Put
$E=M-I$. Its row and column sums are at most $1/4$. For a finitely
supported $x$, H\"older's inequality gives
\begin{align*}
 \|Ex\|_p^p
 &\le\sum_i\left(\sum_j|E_{ij}|\right)^{p-1}
                  \sum_j|E_{ij}|\,|x_j|^p\le4^{-p}\sum_j|x_j|^p.
\end{align*}
Thus $\|E\|_{\ell^p\to\ell^p}\le1/4$. The Neumann series defines
$M^{-1}$ on $\ell^p$, with $\|M^{-1}\|\le4/3$. Set
\begin{equation}\label{eq:thin-lattice-interpolation}
 T_p=S_pM^{-1},
 \qquad R_pT_p=I.
\end{equation}
The identity $R_pS_p=M$ holds first on finitely supported sequences
and then on $\ell^p$ by boundedness.

The proof of \cref{thm:infinite-separated-realization} also applies
to complex sequences $b$. Each upper estimate there uses only
$|b_j|$. The selected points still belong to the original lattice.
Their small polydiscs remain pairwise disjoint, so
\[
 \sum_jV_j\lesssim v(\Omega)<\infty,
 \qquad c=(V_j^{1/p})_j\in\ell^p.
\]
Since $D_b:\ell^p\to\ell^q$ is bounded,
\[
 \sum_j|b_j|^qV_j^{q/p}
 =\|D_bc\|_q^q
 \le\|D_b\|^q\left(\sum_jV_j\right)^{q/p}<\infty.
\]
Disjointness of the symbol cells and the multiplication estimate give
\[
 \|f_b\|_q^q\lesssim\sum_j|b_j|^qV_j^{q/p}<\infty,
 \qquad
 \pi_r(M_{f_b})\lesssim\|b\|_{\mathfrak d_{p,q}^{\,r}}.
\]
The same argument verifies $f_b\in\mathcal D_{\Omega,q}$ before
forming $H_{f_b}$. Since $P$ is bounded on $L^q$,
\[
 \pi_r(H_{f_b})\le(1+\|P\|)\pi_r(M_{f_b})
 \lesssim\|b\|_{\mathfrak d_{p,q}^{\,r}}.
\]
The symbol formula and the identity $H_{f_b}=(I-P)M_{f_b}$ show that
$\mathcal E$ is complex-linear.

We now construct a fixed coefficient map. Put
\[
 m_j=\int_{U_j}|(\zeta_j(z))_1|^2\dd v(z)\asymp V_j,
 \qquad
 \lambda_j(u)=\frac{V_j^{1/q-1/p}}{m_j}
 \int_{U_j}\frac{u(z)}{e_{j,p}(z)}
                  (\zeta_j(z))_1\dd v(z).
\]
The denominator has no zeros on a neighborhood of
$\overline{U_j}$, by \eqref{eq:kernel-lower-finite}. Since
$|e_{j,p}|\asymp V_j^{-1/p}$ there, H\"older's inequality gives
\begin{align*}
 |\lambda_j(u)|
 &\lesssim V_j^{1/q-1/p-1}V_j^{1/p}V_j^{1/q'}
             \|u\|_{L^q(U_j)}\\
 &\lesssim\|u\|_{L^q(U_j)}.
\end{align*}
Define $Qu=\{\lambda_j(u)\}_j$. The sets $U_j$ are disjoint, so
\[
 \|Qu\|_{\ell^q}^q
 \lesssim\sum_j\|u\|_{L^q(U_j)}^q\le\|u\|_q^q.
\]

Fix $g\in A^p$. Multiplication by $f_b$ maps $A^p$ into $L^q$,
so $P(f_bg)$ is holomorphic. On $U_j$, the quotient
$P(f_bg)/e_{j,p}$ is holomorphic on a neighborhood of the closed
cell. Angular integration in the affine coordinates gives
\[
 \int_{U_j}\frac{P(f_bg)(z)}{e_{j,p}(z)}
                   (\zeta_j(z))_1\dd v(z)=0.
\]
For every holomorphic function $h$ on that neighborhood, the same
angular integration gives the weighted mean-value formula
\[
 \int_{U_j}h(z)|(\zeta_j(z))_1|^2\dd v(z)=m_jh(a_j).
\]
This follows by expanding $h\circ\Phi_j$ in its uniformly convergent
power series on the closed polydisc. Every nonconstant monomial has
zero angular mean, also after multiplication by $|\zeta_1|^2$.
Since the other symbol cells do not meet $U_j$, we obtain
\begin{align*}
 \lambda_j(H_{f_b}g)
 &=\frac{b_j}{m_j}\int_{U_j}
       \frac{g(z)}{e_{j,p}(z)}|(\zeta_j(z))_1|^2\dd v(z)\\
 &=b_j\frac{g(a_j)}{e_{j,p}(a_j)}
 =b_jK_j^{-1/p}g(a_j).
\end{align*}
Consequently,
\[
 QH_{f_b}=D_bR_p.
\]
Compose with \eqref{eq:thin-lattice-interpolation}. This proves
\eqref{eq:Hankel-diagonal-factorization}.

For $A\in\Piabs_r(A^p,L^q)$, let
\[
 \mathcal C(A)_j=(QAT_pe_j)_j,
\]
where $e_j$ is the $j$th coordinate vector of $\ell^p$.
Diagonal extraction by finite sign averages gives
\[
 \|\mathcal C(A)\|_{\mathfrak d_{p,q}^{\,r}}
 \le\pi_r(QAT_p)
 \le\|Q\|\,\|T_p\|\,\pi_r(A).
\]
For $B=QAT_p$, let $P_m$ and $Q_m$ be the first-$m$ coordinate
projections on $\ell^p$ and $\ell^q$.  Let $U_\varepsilon^{(s)}$
multiply the first $m$ coordinates of $\ell^s$ by $\varepsilon_1,\ldots,\varepsilon_m$
and leave the others fixed.  The first $m$ diagonal entries of $B$
are obtained from
\[
 2^{-m}\sum_{\varepsilon\in\{-1,1\}^m}
 Q_mU_\varepsilon^{(q)}BU_\varepsilon^{(p)}P_m.
\]
Each term has $r$-summing norm at most $\pi_r(B)$.
Applying the bound to finite families and letting $m\to\infty$
gives the displayed extraction estimate, exactly as in
\eqref{eq:H-block-extraction}.
Equation \eqref{eq:Hankel-diagonal-factorization} gives
$\mathcal C\mathcal E=I$. In particular,
\[
 \|b\|_{\mathfrak d_{p,q}^{\,r}}
 \le\|\mathcal C\|\,\pi_r(\mathcal E b),
\]
so the range of $\mathcal E$ is closed. Therefore
$\mathcal P=\mathcal E\mathcal C$ is bounded and satisfies
\[
 \mathcal P^2=\mathcal E(\mathcal C\mathcal E)\mathcal C
 =\mathcal P,
 \qquad
 \operatorname{ran}\mathcal P=\mathcal E\mathfrak d_{p,q}^{\,r}.
\]
This proves the assertion about the complemented subspace.

Finally, let $A\in\mathcal L(A^p,L^q)$ have rank less than $N$. Then
$QAT_p$ also has rank less than $N$, and
\[
 \pi_r(D_b-QAT_p)
 =\pi_r\bigl(Q(H_{f_b}-A)T_p\bigr)
 \le\|Q\|\,\|T_p\|\,\pi_r(H_{f_b}-A).
\]
Take the infimum over $A$. This proves
\eqref{eq:Hankel-approximation-lower}.
\end{proof}

The local Taylor degree need not be the same on every symbol cell.
This gives an upper estimate that can be combined with
\eqref{eq:Hankel-approximation-lower}.

\begin{proposition}\label{prop:Hankel-adaptive-approximation}
Let $1<p,q<\infty$, $1\le r<\infty$, and let
$b\in\mathfrak d_{p,q}^{\,r}$. Let $f_b$ be the symbol in
\cref{thm:complemented-Hankel-model}. There are constants $C>0$ and
$0<\vartheta<1$, independent of $b$, such that the following holds.
For every finite set $F\subset\N$ and every choice of integers
$L_j\ge0$, $j\in F$, there is a bounded operator of finite rank
$A_{F,\mathbf L}:A^p\to L^q$ for which
\begin{align}
 \operatorname{rank}A_{F,\mathbf L}
 &\le\sum_{j\in F}\binom{L_j+n}{n},
 \label{eq:Hankel-adaptive-rank}\\
 \pi_r(H_{f_b}-A_{F,\mathbf L})
 &\le C\|d\|_{\mathfrak d_{p,q}^{\,r}},
 \label{eq:Hankel-adaptive-error}
\end{align}
where
\[
 d_j=
 \begin{cases}
 |b_j|(L_j+1)^{n-1}\vartheta^{L_j+1},&j\in F,\\
 |b_j|,&j\notin F.
 \end{cases}
\]
The sequence $d$ belongs to $\mathfrak d_{p,q}^{\,r}$, since the
Taylor factor is uniformly bounded and the ideal is solid.
The constants may depend on $p,q,r$ and the fixed geometric data.
They do not depend on $F$ or the degrees $L_j$.
\end{proposition}

\begin{proof}
Use the maps $\Phi_j$, the sets $U_j$, and the coordinates $\zeta_j$
from \eqref{eq:linear-Hankel-model}. Thus
$U_j=\Phi_j(\theta_1\mathbb D^n)$ with a fixed $\theta_1<1$.
Put
\[
 W_j=\Phi_j(\mathbb D^n)=\Q_{A_*\rho}(a_j),\qquad
 X_j=A^p(W_j),\qquad Y_j=L^q(U_j).
\]
The sets $W_j$ have bounded overlap. Hence the restriction map
\[
 \mathcal R:A^p\longrightarrow(\oplus_jX_j)_{\ell^p},
 \qquad \mathcal Rg=(g|_{W_j})_j,
\]
is bounded. Extension by zero gives an isometry
\[
 \mathcal J:(\oplus_jY_j)_{\ell^q}\longrightarrow L^q(\Omega),
 \qquad\mathcal J(u_j)=\sum_j\mathbf1_{U_j}u_j,
\]
because the sets $U_j$ are disjoint.

Fix $\theta_1<\sigma<1$. If $u\in X_j$, write
\[
 u(\Phi_j(\zeta))=\sum_{\nu\in\Nzero^n}
       \ell_{j,\nu}(u)\zeta^\nu.
\]
The Cauchy and submean estimates used in
\cref{lem:local-restriction} give
\[
 \|\ell_{j,\nu}\|_{X_j^*}
 \le C V_j^{-1/p}\sigma^{-|\nu|}.
\]
Define $B_j:X_j\to Y_j$ by
\[
 B_ju=b_jV_j^{1/p-1/q}\overline{\zeta_{j,1}}u|_{U_j}.
\]
Its Taylor vectors are
\[
 v_{j,\nu}=b_jV_j^{1/p-1/q}
       \overline{\zeta_{j,1}}\zeta_j^\nu|_{U_j}.
\]
Since $v(U_j)\asymp V_j$ and $|\zeta_j|_\infty\le\theta_1$ there,
\[
 \|v_{j,\nu}\|_{Y_j}
 \le C|b_j|V_j^{1/p}\theta_1^{|\nu|}.
\]
Set $\vartheta=\theta_1/\sigma$. The series
$B_j=\sum_\nu\ell_{j,\nu}\otimes v_{j,\nu}$ therefore converges in
nuclear norm, and
\[
 \|\ell_{j,\nu}\|\|v_{j,\nu}\|
 \le C|b_j|\vartheta^{|\nu|}.
\]
Let $B_{j,L}$ be its sum over $|\nu|\le L$. Then
\begin{align*}
 \operatorname{rank}B_{j,L}&\le\binom{L+n}{n},\\
 \pi_1(B_j-B_{j,L})
 &\le C|b_j|\sum_{m>L}\binom{m+n-1}{n-1}\vartheta^m\le C_1|b_j|(L+1)^{n-1}\vartheta^{L+1},\\
 \pi_1(B_j)&\le C_2|b_j|.
\end{align*}
The last tail estimate is the same elementary series bound as in
\eqref{eq:local-Taylor-tail}.

Let $\mathcal B_{F,\mathbf L}$ have block $B_{j,L_j}$ on $F$ and
zero block outside $F$. Set
\[
 A_{F,\mathbf L}=(I-P)\mathcal J
       \mathcal B_{F,\mathbf L}\mathcal R.
\]
This operator is bounded and has the rank bound in
\eqref{eq:Hankel-adaptive-rank}. Indeed, it is a finite sum of
bounded coefficient functionals times the vectors
$(I-P)\mathcal Jv_{j,\nu}$. The estimate $\pi_1(B_j)\le C|b_j|$ and
\cref{lem:block-principle} show that the full block map $\mathcal B$
with blocks $B_j$ is bounded. Hence the following are identities of
bounded operators on $A^p$:
\[
 M_{f_b}=\mathcal J\mathcal B\mathcal R,
 \qquad
 H_{f_b}-A_{F,\mathbf L}
 =(I-P)\mathcal J(\mathcal B-\mathcal B_{F,\mathbf L})\mathcal R.
\]
The residual blocks have $1$-summing norms at most $C_3d_j$.
Normalize each nonzero block by $C_3d_j$, and take the zero block
when $d_j=0$. The block principle and the ideal property give
\[
 \pi_r(H_{f_b}-A_{F,\mathbf L})
 \le C_3\|I-P\|\|\mathcal R\|
       \|d\|_{\mathfrak d_{p,q}^{\,r}}.
\]
This proves \eqref{eq:Hankel-adaptive-error}.
\end{proof}

\begin{corollary}\label{cor:Hankel-Hilbert-approximation-lower}
Let $p=q=2$, $1\le r<\infty$, and let $b\in\ell^2$. For the symbols
in \cref{thm:complemented-Hankel-model},
\begin{equation}\label{eq:Hankel-Hilbert-tail-lower}
 a_N^{(r)}(H_{f_b})
 \gtrsim_r\left(\sum_{j\ge N}(b_j^*)^2\right)^{1/2},
 \qquad N\ge1,
\end{equation}
Here $(b_j^*)$ is the decreasing rearrangement of $(|b_j|)$:
it lists the nonzero values in nonincreasing order, with multiplicity,
and is completed by zeros when necessary.
In particular, if $b_j=j^{-\alpha}$ with $\alpha>1/2$, then
\[
 a_N^{(r)}(H_{f_b})\asymp_{r,\alpha}N^{1/2-\alpha},
 \qquad N\ge1.
\]
The same bounded symbol has this property for every $1\le r<\infty$.
\end{corollary}

\begin{proof}
By \cref{prop:Hilbert-calibration}, the $r$-summing and
Hilbert--Schmidt norms are equivalent on Hilbert spaces. For a
diagonal operator on $\ell^2$, its best Hilbert--Schmidt approximation
error among operators of rank less than $N$ is
\[
 \inf_{\operatorname{rank}A<N}\|D_b-A\|_{\mathcal S_2}
 =\left(\sum_{j\ge N}(b_j^*)^2\right)^{1/2}.
\]
To see the lower bound directly, let $P_A$ be the orthogonal
projection onto the range of $A$, whose dimension is $k<N$.
After rearranging $|b_j|$ in decreasing order,
\begin{align*}
 \|D_b-A\|_{\mathcal S_2}^2
 &\ge\|(I-P_A)D_b\|_{\mathcal S_2}^2\\
 &=\sum_j(b_j^*)^2\bigl(1-\|P_Ae_j\|^2\bigr)
 \ge\sum_{j>k}(b_j^*)^2
 \ge\sum_{j\ge N}(b_j^*)^2.
\end{align*}
Here $0\le\|P_Ae_j\|^2\le1$ and
$\sum_j\|P_Ae_j\|^2=k$. The weighted sum is therefore largest
when the first $k$ weights equal one and all other weights equal zero.
Truncation to the first $N-1$ coordinates gives equality.
Apply \eqref{eq:Hankel-approximation-lower}. For the final assertion,
integral comparison gives
$\sum_{j\ge N}j^{-2\alpha}\asymp_\alpha N^{1-2\alpha}$.

For the reverse power estimate, choose $\tau\in(\vartheta,1)$.
There is $C_0$ such that
\[
 (L+1)^{n-1}\vartheta^{L+1}\le C_0\tau^L
 \qquad(L\ge0).
\]
Choose $c>0$ so that $\kappa=c\log(1/\tau)>\alpha$.
For an integer $M\ge1$, take $F=\{1,\ldots,M\}$ and
\[
 L_j=\left\lceil c\log(M/j)\right\rceil\qquad(1\le j\le M).
\]
Then $\tau^{L_j}\le(j/M)^\kappa$. Since
$\mathfrak d_{2,2}^{\,r}=\ell^2$ with equivalent norms,
\cref{prop:Hankel-adaptive-approximation} gives
\begin{align*}
 \pi_r(H_{f_b}-A_{F,\mathbf L})^2
 &\le C_r\left(
 M^{-2\kappa}\sum_{j=1}^M j^{2\kappa-2\alpha}
 +\sum_{j>M}j^{-2\alpha}\right)\\
 &\le C_{r,\alpha}M^{1-2\alpha}.
\end{align*}
The first sum is at most $C M^{2\kappa-2\alpha+1}$, because
$2\kappa-2\alpha>0$. Integral comparison bounds the second sum.
Moreover,
\begin{align*}
 \operatorname{rank}A_{F,\mathbf L}
 &\le C_{n,c}\sum_{j=1}^M(1+\log(M/j))^n\\
 &\le C_{n,c}\int_0^M(1+\log(M/x))^n\,\mathrm dx\\
 &=C_{n,c}M\int_0^\infty(1+t)^ne^{-t}\,\mathrm dt
 \le C_4M.
\end{align*}
The summand is decreasing as a function of $j$, which proves the
integral bound. Take $M=\lfloor(N-1)/C_4\rfloor$ for all sufficiently
large $N$. Then the rank is less than $N$ and $M\asymp N$. Hence
\[
 a_N^{(r)}(H_{f_b})\lesssim_{r,\alpha}N^{1/2-\alpha}.
\]
Increasing the constant covers the remaining finite set of values
of $N$. The lower estimate already proved gives the reverse bound.
Finally, $|f_b|\le\theta_1\sup_j|b_j|$ when $p=q=2$.
Thus this is one bounded symbol, independent of $r$.
\end{proof}

\begin{corollary}\label{cor:atomic-sharp-approximation}
Let $1<p<\infty$, $1\le q,r<\infty$, and
$b\in\mathfrak d_{p,q}^{\,r}$. On the sublattice chosen in
\cref{thm:complemented-Hankel-model}, put
\[
 \mu_b=\sum_{j\ge1}|b_j|^qK(a_j,a_j)^{-q/p}\delta_{a_j}.
\]
Then $\mu_b$ is finite, $J_{\mu_b}\in\Piabs_r(A^p,L^q(\mu_b))$,
and, uniformly for $N\ge1$,
\begin{equation}\label{eq:atomic-sharp-approximation}
 a_N^{(r)}(J_{\mu_b})\asymp a_N^{(r)}(D_b).
\end{equation}
For $p=q=2$, this gives
\[
 a_N^{(r)}(J_{\mu_b})
 \asymp_r\left(\sum_{j\ge N}(b_j^*)^2\right)^{1/2}.
\]
\end{corollary}

\begin{proof}
The choice of the sublattice and the maps $R_p,T_p$ in
\eqref{eq:thin-lattice-interpolation} depend only on $p$.
They are therefore available also when $q=1$.
Write $w_j=|b_j|K(a_j,a_j)^{-1/p}$. Since $1\in A^p$,
\[
 \mu_b(\Omega)=\|D_{|b|}R_p1\|_{\ell^q}^q<\infty.
\]
Define $U:L^q(\mu_b)\to\ell^q$ by
$(Uu)_j=w_ju(a_j)$, taking the coordinate to be zero when $w_j=0$.
It is an isometry. Define $V:\ell^q\to L^q(\mu_b)$ by assigning
$Vx(a_j)=x_j/w_j$ when $w_j>0$. Values elsewhere have no effect
on its $L^q(\mu_b)$ class. Then $\|V\|\le1$ and $VU=I$.
The definitions give the two exact factorizations
\[
 J_{\mu_b}=VD_{|b|}R_p,
 \qquad D_{|b|}=UJ_{\mu_b}T_p.
\]
The ideal property proves summability. Composition with each bounded
factor preserves the rank bound. Hence
\[
 \|T_p\|^{-1}a_N^{(r)}(D_{|b|})
 \le a_N^{(r)}(J_{\mu_b})
 \le\|R_p\|a_N^{(r)}(D_{|b|}).
\]
Multiplication by the coordinate phases is an isometry, so
$a_N^{(r)}(D_{|b|})=a_N^{(r)}(D_b)$.
The formula on Hilbert spaces follows from the diagonal tail calculation
in \cref{cor:Hankel-Hilbert-approximation-lower}.
\end{proof}

The next corollary follows from Pietsch domination and the reflexivity
of $A^p$. We give the proof for a general target Banach space.

\begin{corollary}\label{cor:H-compact}
Let $1<p,q<\infty$, $1\le r<\infty$, and
$f\in\mathcal D_{\Omega,q}$.  Every absolutely $r$-summing Hankel operator
$H_f:A^p(\Omega)\to L^q(\Omega)$ is compact.  In particular, if
$f\in\mathcal I_{p,q}^{r}(\Omega)$, then $H_f$ is compact.
\end{corollary}

\begin{proof}
Let $Y$ be a Banach space and let $T:A^p(\Omega)\to Y$ be absolutely
$r$-summing.  Pietsch domination
gives a probability measure $\nu$ on $B_{(A^p)^*}$ such that
\begin{equation}\label{eq:pietsch-compactness}
 \|Tg\|_Y^r
 \le \pi_r(T)^r
 \int_{B_{(A^p)^*}}|\phi(g)|^r\dd\nu(\phi).
\end{equation}
Let $(g_m)$ converge weakly to zero in $A^p$.  The sequence is norm
bounded.  Hence
\[
 |\phi(g_m)|^r\le\sup_m\|g_m\|_{A^p}^r,
 \qquad
 \phi(g_m)\longrightarrow0
\]
for every $\phi\in B_{(A^p)^*}$.  Dominated convergence in
\eqref{eq:pietsch-compactness} gives $\|Tg_m\|_Y\to0$.  Thus $T$ is
completely continuous.

Since $1<p<\infty$, $A^p(\Omega)$ is reflexive.  Every bounded sequence
has a weakly convergent subsequence.  Its image under $T$ converges in
norm.  Thus $T$ is compact.  Apply this to $T=H_f$.  The last assertion
follows from \cref{thm:intro-hankel}.
\end{proof}

\begin{corollary}\label{cor:ida-independence}
Let $1<p,q<\infty$ and $1\le r<\infty$.
If $f\in\mathcal D_{\Omega,q}$, any two admissible radius pairs
and lattices give equivalent $\mathbf G$-seminorms.  Hence
$\mathcal I_{p,q}^{r}(\Omega)\cap\mathcal D_{\Omega,q}$ is
independent of the fixed pair of scales.
\end{corollary}

\begin{proof}
Let $\mathbf G^{(1)}f$ and
$\mathbf G^{(2)}f$ be formed from two admissible radius pairs and
lattices.  Suppose first that
\[
 \|\mathbf G^{(1)}f\|_{\mathfrak d_{p,q}^{\,r}}<\infty.
\]
The sufficiency part of \cref{thm:intro-hankel}, applied to the first
data, and its necessity part, applied to the second data, give
\begin{align*}
 \pi_r(H_f)
 &\lesssim
 \|\mathbf G^{(1)}f\|_{\mathfrak d_{p,q}^{\,r}},\\
 \|\mathbf G^{(2)}f\|_{\mathfrak d_{p,q}^{\,r}}
 &\lesssim\pi_r(H_f).
\end{align*}
Hence
\[
 \|\mathbf G^{(2)}f\|_{\mathfrak d_{p,q}^{\,r}}
 \lesssim
 \|\mathbf G^{(1)}f\|_{\mathfrak d_{p,q}^{\,r}}.
\]
Interchange the two choices.  The two seminorms are equivalent.  The constants may
depend on the two fixed geometric choices.  They do not depend on $f$.
\end{proof}
\section{Toeplitz operators}\label{sec:toeplitz}

This section applies Carleson localization to projected
multiplication. Diagonal extraction gives a necessary condition for a
complex Toeplitz symbol. For a positive symbol that satisfies a local
reverse-H\"older estimate, the converse also holds. The joint
$(T_f,H_f)$ criterion follows from the projection decomposition.

Positive symbols and Schatten classes are studied in
\cite{Luecking1987,LiLuecking1995,PauZhao2015}.
For positive measures on convex finite-type domains, geometric
boundedness, compactness, and Schatten criteria are in
\cite{XiaoYangYuan2026}. The absolutely summing case on the unit ball is studied in
\cite{HuWang2025}. Results for complex symbols on weighted Fock spaces
are in \cite{XuDong2026}. The exponents $p,q,r$ below are independent.
Each converse for a single operator requires its stated symbol assumptions.

For $1<q<\infty$ and $f\in\mathcal D_{\Omega,q}$, recall that
\[
 T_fg=P(fg),\qquad g\in\Gamma.
\]
We extract a scalar sequence by the diagonal transference method of
\cite[Proposition~4.4]{FanHeWangZeng2026}.  Here the atoms and local
estimates come from the finite-type geometry.  For the fixed lattice,
put
\begin{equation*}
 \Delta_j(A)=V_j^{1/q}(Ae_{j,p})(a_j),
 \qquad
 e_{j,p}(z)=\frac{K(z,a_j)}{K(a_j,a_j)^{1-1/p}}.
\end{equation*}

For $f\in\mathcal D_{\Omega,q}$, define the Berezin transform by
\[
 \widetilde f(z)=K(z,z)^{-1}
     \int_\Omega f(w)|K(w,z)|^2\dd v(w).
\]
The integral converges by H\"older's inequality and
\cref{lem:kernel-span-density}.

\begin{proposition}\label{prop:operator-diagonal}
Let $1<p,q<\infty$ and $1\le r<\infty$.
Let $A\in\Piabs_r(A^p(\Omega),A^q(\Omega))$.  Then
\begin{equation}\label{eq:operator-diagonal-bound}
 \{\Delta_j(A)\}_j\in\mathfrak d_{p,q}^{\,r},
 \qquad
 \|\{\Delta_j(A)\}\|_{\mathfrak d_{p,q}^{\,r}}
 \lesssim\pi_r(A).
\end{equation}
\end{proposition}

\begin{proof}
First restrict the indices to one sufficiently separated sublattice and
relabel it as $\mathbb N$.  On this sublattice the synthesis estimate
\eqref{eq:kernel-synthesis} holds.
Let $S_p:\ell^p\to A^p$ be the synthesis map in
\eqref{eq:kernel-synthesis}.  The sampling map
\begin{equation}\label{eq:sampling-map}
 R_q:A^q\to\ell^q,
 \qquad R_qh=\{V_j^{1/q}h(a_j)\}_j,
\end{equation}
is bounded.  Indeed, the holomorphic submean inequality and disjointness
give
\[
 \sum_jV_j|h(a_j)|^q
 \lesssim\sum_j\int_{\Q_{c\rho}(a_j)}|h|^q\dd v
 \le\|h\|_{A^q}^q.
\]
Put $B=R_qAS_p:\ell^p\to\ell^q$.  Then
\[
 B\in\Piabs_r(\ell^p,\ell^q),
 \qquad
 \pi_r(B)\lesssim\pi_r(A).
\]
Let $e_j$ be the $j$th unit vector of $\ell^p$, and put
$B_{ij}=(Be_j)_i$.
Let $(r_j(t))$ be the Rademacher functions and define
\[
 U_t^{(s)}(x_j)=(r_j(t)x_j),
 \qquad s=p,q.
\]
Both maps $U_t^{(s)}$ are isometries.  For finitely supported $x$,
\begin{equation}\label{eq:rademacher-diagonal}
 \int_0^1U_t^{(q)}BU_t^{(p)}x\dd t
 =D_{(B_{jj})}x.
\end{equation}
Indeed, the $i$th coordinate of the left side is
\[
 \sum_jB_{ij}x_j\int_0^1r_i(t)r_j(t)\dd t=B_{ii}x_i.
\]
Let $P_m:\ell^p\to\ell^p$ and $Q_m:\ell^q\to\ell^q$ be the
first-$m$ coordinate projections.  Put
\[
 D_m=Q_mD_{(B_{jj})}P_m.
\]
Only the first $m$ Rademacher functions occur after both projections.
For $\varepsilon\in\{-1,1\}^m$, let $U_\varepsilon^{(s)}$ multiply
the first $m$ coordinates by $\varepsilon_j$ and leave all other
coordinates unchanged.  The first $m$ Rademacher signs take every
value $\varepsilon$ on a set of measure $2^{-m}$.  Thus
\eqref{eq:rademacher-diagonal} becomes the finite sum
\[
 D_m
 =\int_0^1Q_mU_t^{(q)}BU_t^{(p)}P_m\dd t
 =2^{-m}\sum_{\varepsilon\in\{-1,1\}^m}
 Q_mU_\varepsilon^{(q)}BU_\varepsilon^{(p)}P_m.
\]
Every coordinate projection is a contraction and every sign map is
an isometry.  The triangle inequality and the ideal property give
\begin{align}
 \pi_r(D_m)
 &\le2^{-m}\sum_{\varepsilon\in\{-1,1\}^m}
 \pi_r(Q_mU_\varepsilon^{(q)}BU_\varepsilon^{(p)}P_m)
 \le\pi_r(B).                                      \label{eq:finite-diagonal-pi}
\end{align}
If $x\in c_{00}$, choose $m$ so that $P_mx=x$.  Then
\[
 \|D_{(B_{jj})}x\|_{\ell^q}
 =\|D_mx\|_{\ell^q}
 \le\pi_r(B)\|x\|_{\ell^p}.
\]
Hence the diagonal map on $c_{00}$ has a unique bounded extension
$D_{(B_{jj})}:\ell^p\to\ell^q$.  Let $x_1,\ldots,x_N\in\ell^p$.
Then
\[
 D_mx_k=D_{(B_{jj})}P_mx_k\longrightarrow D_{(B_{jj})}x_k
 \quad(1\le k\le N).
\]
Moreover,
\begin{align*}
 \sup_{\phi\in B_{(\ell^p)^*}}
 \left(\sum_{k=1}^N|\phi(P_mx_k)|^r\right)^{1/r}
 &\le
 \sup_{\phi\in B_{(\ell^p)^*}}
 \left(\sum_{k=1}^N|\phi(x_k)|^r\right)^{1/r},
\end{align*}
because $\|P_m\|=1$.  In particular,
\[
 \left(\sum_{k=1}^N\|D_mx_k\|_{\ell^q}^r\right)^{1/r}
 \le\pi_r(B)w_r(x_1,\ldots,x_N).
\]
The sum on the left has only $N$ terms. Each term converges to
$\|D_{(B_{jj})}x_k\|_{\ell^q}^r$.  Letting $m\to\infty$ gives
\[
 \left(\sum_{k=1}^N
 \|D_{(B_{jj})}x_k\|_{\ell^q}^r\right)^{1/r}
 \le\pi_r(B)w_r(x_1,\ldots,x_N).
\]
Taking the supremum over finite families proves
$\pi_r(D_{(B_{jj})})\le\pi_r(B)$.
By definition, $B_{jj}=\Delta_j(A)$.  This proves the estimate on one
sufficiently separated sublattice.  Decompose the fixed lattice into
$N_*$ such sublattices by \cref{lem:geometry-package}.  Let
$d^{(\ell)}$ equal $\Delta_j(A)$ on the $\ell$th sublattice and zero
elsewhere.  Coordinate injections and projections have norm one.
Thus zero extension preserves the norm of the diagonal ideal.  The preceding
argument gives
\[
 \|d^{(\ell)}\|_{\mathfrak d_{p,q}^{\,r}}
 \lesssim\pi_r(A).
\]
Since
\[
 (\Delta_j(A))_j=\sum_{\ell=1}^{N_*}d^{(\ell)},
\]
the triangle inequality proves \eqref{eq:operator-diagonal-bound}.
\end{proof}

For $A=T_f$, the diagonal is the weighted Berezin transform.  Indeed,
\begin{align}
 (T_fe_{j,p})(a_j)
 &=\frac1{K(a_j,a_j)^{1-1/p}}
 \int_\Omega f(w)|K(w,a_j)|^2\dd v(w)\notag\\
 &=K(a_j,a_j)^{1/p}\widetilde f(a_j).
 \label{eq:Toeplitz-diagonal-computation}
\end{align}
Together with $K(a_j,a_j)\asymp V_j^{-1}$, this gives
\begin{equation}\label{eq:Toeplitz-diagonal-weight}
 |\Delta_j(T_f)|
 \asymp V_j^{1/q-1/p}|\widetilde f(a_j)|.
\end{equation}

\begin{corollary}
\label{cor:scalar-necessity}
Let $1<p,q<\infty$, $1\le r<\infty$, and
$f\in\mathcal D_{\Omega,q}$.  If
$T_f\in\Piabs_r(A^p(\Omega),A^q(\Omega))$, then
\begin{equation}\label{eq:scalar-necessity}
 \left\|\{V_j^{1/q-1/p}|\widetilde f(a_j)|\}\right\|_{\dideal}
 \lesssim\pi_r(T_f).
\end{equation}
No IDA assumption is used.
\end{corollary}

\begin{proof}
For every $j$, the atom $e_{j,p}$ belongs to $\Gamma$.  Since
$fe_{j,p}\in L^q$ and $K(\cdot,a_j)\in L^{q'}$, the reproducing
formula for $P(fe_{j,p})$ is absolutely convergent.  It gives
\eqref{eq:Toeplitz-diagonal-computation}.  Apply
\cref{prop:operator-diagonal} to the unique $r$-summing extension of
$T_f$.  Equations \eqref{eq:Toeplitz-diagonal-computation} and
\eqref{eq:Toeplitz-diagonal-weight} give
\[
 V_j^{1/q-1/p}|\widetilde f(a_j)|
 \lesssim |\Delta_j(T_f)|.
\]
Solidity of $\mathfrak d_{p,q}^{\,r}$ proves
\eqref{eq:scalar-necessity}.
\end{proof}

For a general complex symbol, the diagonal computation gives a
necessary condition. For nonnegative symbols satisfying the local
reverse-H\"older condition below, the Berezin diagonal also controls
the local $L^q$ mean. This gives the converse under that condition.

For $f\ge0$, write $f\in\mathrm{RH}_{q,\rho}(\Omega)$ when
\begin{equation}\label{eq:intro-local-RH}
 \left(\avg_{\Q_\rho(z)}f(w)^q\dd v(w)\right)^{1/q}
 \le C_{\mathrm{RH}}\avg_{\Q_\rho(z)}f(w)\dd v(w),
 \qquad z\in\Omega.
\end{equation}

\begin{theorem}
\label{thm:positive-Toeplitz}
Let $1<p,q<\infty$, $1\le r<\infty$, and let
$f\in\mathcal D_{\Omega,q}$.  Assume that
$f\in\mathrm{RH}_{q,\rho}(\Omega)$ in the sense of
\eqref{eq:intro-local-RH}.  Then the following conditions are
equivalent:
\begin{enumerate}[label=\textup{(\roman*)}]
\item $T_f\in\Piabs_r(A^p(\Omega),A^q(\Omega))$;
\item $M_f\in\Piabs_r(A^p(\Omega),L^q(\Omega))$;
\item $f\in\mathcal L_{p,q}^{r}(\Omega)$;
\item
\begin{equation}\label{eq:positive-Toeplitz-Berezin}
 \{V_j^{1/q-1/p}\widetilde f(a_j)\}_j
 \in\mathfrak d_{p,q}^{\,r}.
\end{equation}
\end{enumerate}
Moreover,
\begin{equation}\label{eq:positive-Toeplitz-norms}
 \pi_r(T_f)
 \asymp\pi_r(M_f)
 \asymp\|f\|_{\mathcal L_{p,q}^{r}}
 \asymp
 \|\{V_j^{1/q-1/p}\widetilde f(a_j)\}\|_{\mathfrak d_{p,q}^{\,r}}.
\end{equation}
The comparison constants may depend on $C_{\mathrm{RH}}$.  They do not
depend on $f$.
\end{theorem}

\begin{proof}
For $w\in\Q_\rho(a_j)$, the near-diagonal kernel estimate gives
\[
 |K(w,a_j)|\asymp K(a_j,a_j)\asymp V_j^{-1}.
\]
By the definition of $\mathrm{RH}_{q,\rho}(\Omega)$, $f\ge0$.  Hence
\begin{align*}
 \widetilde f(a_j)
 &=\frac1{K(a_j,a_j)}
 \int_\Omega f(w)|K(w,a_j)|^2\dd v(w)\notag\\
 &\ge
 \frac1{K(a_j,a_j)}
 \int_{\Q_\rho(a_j)}f(w)|K(w,a_j)|^2\dd v(w)\notag\\
 &\gtrsim
 \frac1{V_j}\int_{\Q_\rho(a_j)}f(w)\dd v(w).
\end{align*}
The reverse-H\"older condition now gives
\begin{equation*}
 \M_{q,\rho}f(a_j)
 \le C_{\mathrm{RH}}
 \avg_{\Q_\rho(a_j)}f\dd v
 \lesssim C_{\mathrm{RH}}\widetilde f(a_j).
\end{equation*}
Multiply by $V_j^{1/q-1/p}$.  Solidity gives
\begin{equation}\label{eq:positive-sequence-recovery}
 \|f\|_{\mathcal L_{p,q}^{r}}
 \lesssim C_{\mathrm{RH}}
 \|\{V_j^{1/q-1/p}\widetilde f(a_j)\}\|_{\mathfrak d_{p,q}^{\,r}}.
\end{equation}

Assume (i).  By \cref{cor:scalar-necessity}, the sequence in
\eqref{eq:positive-Toeplitz-Berezin} belongs to the diagonal ideal and
\[
 \|\{V_j^{1/q-1/p}\widetilde f(a_j)\}\|_{\mathfrak d_{p,q}^{\,r}}
 \lesssim\pi_r(T_f).
\]
Equations \eqref{eq:positive-sequence-recovery} and
\eqref{eq:multiplication-norm} give
\[
 \pi_r(M_f)\lesssim C_{\mathrm{RH}}\pi_r(T_f).
\]
Thus (i) implies (ii).  Conversely,
\[
 T_f=PM_f,
 \qquad
 \pi_r(T_f)\le\|P\|_{L^q\to L^q}\pi_r(M_f).
\]
Hence (i) and (ii) are equivalent, with comparable norms.
Conditions (ii) and (iii) are equivalent by
\cref{cor:multiplication}.  The preceding estimates show that
\begin{align*}
 \|f\|_{\mathcal L_{p,q}^{r}}
 &\lesssim C_{\mathrm{RH}}\pi_r(T_f),\qquad &&\text{under (i)},\\
 \pi_r(T_f)&\lesssim\pi_r(M_f)
 \asymp\|f\|_{\mathcal L_{p,q}^{r}},\qquad &&\text{under (iii)}.
\end{align*}
They also show that (i) implies (iv), and
\eqref{eq:positive-sequence-recovery} shows that (iv) implies (iii).
Finally, (iii) implies (i), and then \cref{cor:scalar-necessity} gives
\[
 \|\{V_j^{1/q-1/p}\widetilde f(a_j)\}\|_{\mathfrak d_{p,q}^{\,r}}
 \lesssim \pi_r(T_f)
 \lesssim \|f\|_{\mathcal L_{p,q}^{r}}.
\]
Together with \eqref{eq:positive-sequence-recovery}, these inequalities
give all the comparisons in \eqref{eq:positive-Toeplitz-norms}.
\end{proof}

Apply the Carleson theorem to $|f|^qv$ and use the boundedness of $P$.
This gives the following graph criterion.

\begin{corollary}\label{thm:intro-main}
Let $1<p,q<\infty$, $1\le r<\infty$, and
$f\in\mathcal D_{\Omega,q}$.  For Banach spaces $X,Y$, let
$X\oplus_qY$ denote the product with norm
$(\|x\|_X^q+\|y\|_Y^q)^{1/q}$.  Define initially on $\Gamma$
\[
 \mathscr T_f:\Gamma\longrightarrow A^q(\Omega)\oplus_qL^q(\Omega),
 \qquad
 \mathscr T_fg=(T_fg,H_fg).
\]
The following conditions are equivalent:
\begin{enumerate}[label=\textup{(\roman*)}]
\item $M_f\in\Piabs_r(A^p(\Omega),L^q(\Omega))$;
\item $\mathscr T_f\in\Piabs_r(A^p(\Omega),A^q(\Omega)\oplus_qL^q(\Omega))$;
\item $T_f\in\Piabs_r(A^p(\Omega),A^q(\Omega))$ and
$H_f\in\Piabs_r(A^p(\Omega),L^q(\Omega))$;
\item $f\in\mathcal L_{p,q}^{r}(\Omega)$.
\end{enumerate}
Moreover,
\begin{equation}\label{eq:intro-main-norm}
\begin{aligned}
 \pi_r(M_f)
 &\asymp\pi_r(\mathscr T_f)
 \asymp\pi_r(T_f)+\pi_r(H_f)\\
 &\asymp\|f\|_{\mathcal L_{p,q}^{r}(\Omega)}.
\end{aligned}
\end{equation}
Whenever these conditions hold, the four initial operators
$M_f,\mathscr T_f,T_f$, and $H_f$ have unique bounded extensions.  The
constants depend only on $p,q,r$, the finite-type data of $\Omega$, and
the fixed admissible geometric parameters.
\end{corollary}

\begin{proof}[Proof of \cref{thm:intro-main}]
Define
\[
 W:L^q(\Omega)\to A^q(\Omega)\oplus_qL^q(\Omega),
 \qquad Wu=(Pu,(I-P)u),
\]
and
\[
 \Sigma:A^q(\Omega)\oplus_qL^q(\Omega)\to L^q(\Omega),
 \qquad\Sigma(u,v)=u+v.
\]
By \eqref{eq:intro-Pq},
\begin{align*}
 \|Wu\|_{A^q\oplus_qL^q}
 &\le
 \bigl(\|P\|_{L^q\to L^q}^q
       +\|I-P\|_{L^q\to L^q}^q\bigr)^{1/q}\|u\|_{L^q},
 \\
 \|\Sigma(u,v)\|_{L^q}
 &\le \|u\|_{L^q}+\|v\|_{L^q}
 \le 2^{1-1/q}\|(u,v)\|_{A^q\oplus_qL^q}.
\end{align*}
Thus $W$ and $\Sigma$ are bounded.  On $\Gamma$,
\begin{equation*}
 \mathscr T_f=WM_f,
 \qquad
 M_f=\Sigma\mathscr T_f.
\end{equation*}
If $M_f$ extends, the first identity defines an extension of
$\mathscr T_f$.  If $\mathscr T_f$ extends, the second defines an
extension of $M_f$.  The coordinate projections have norm one.  Also,
for finite families $(g_k)$,
\[
 \left(\sum_k\|\mathscr T_fg_k\|^r\right)^{1/r}
 \le
 \left(\sum_k\|T_fg_k\|^r\right)^{1/r}
 +
 \left(\sum_k\|H_fg_k\|^r\right)^{1/r}.
\]
Consequently,
\begin{align*}
 \pi_r(\mathscr T_f)&\le\|W\|\pi_r(M_f),&
 \pi_r(M_f)&\le\|\Sigma\|\pi_r(\mathscr T_f),\\
 \pi_r(T_f)&\le\pi_r(\mathscr T_f),&
 \pi_r(H_f)&\le\pi_r(\mathscr T_f),\\
 \pi_r(\mathscr T_f)&\le\pi_r(T_f)+\pi_r(H_f).
\end{align*}
The ideal property now gives
\begin{equation*}
 \pi_r(M_f)\asymp\pi_r(\mathscr T_f)
 \asymp\pi_r(T_f)+\pi_r(H_f).
\end{equation*}
This proves that (i), (ii), and (iii) are equivalent. It also proves the displayed
norm comparisons.

By \cref{cor:multiplication},
\[
 \pi_r(M_f)\asymp
 \|\mathbf M_{p,q,\rho}f\|_{\dideal}.
\]
Thus (i) and (iv) are equivalent.  This also proves
\eqref{eq:intro-main-norm}.

It remains to identify the extensions.  Suppose $T_f$ and $H_f$ have
bounded extensions. Let $g_m\in\Gamma$ tend to $g$ in $A^p$.  Then
$T_fg_m+H_fg_m$ tends in $L^q$ to some $u$.  On every compact
$K\Subset\Omega$, the submean inequality gives a constant $C_K$ such
that
\[
 \sup_{z\in K}|g_m(z)-g(z)|
 \le C_K\|g_m-g\|_{A^p}\longrightarrow0.
\]
Since $f\in L^q_{\loc}(\Omega)$,
\[
 \|fg_m-fg\|_{L^q(K)}
 \le \sup_{z\in K}|g_m(z)-g(z)|\,\|f\|_{L^q(K)}
 \longrightarrow0.
\]
But $fg_m=T_fg_m+H_fg_m$ on $\Gamma$.  Uniqueness of local $L^q$ limits
gives $u=fg$ almost everywhere on $K$.  Let $K$ increase to $\Omega$
through a sequence of compact sets. Then $u=fg$ almost everywhere on
$\Omega$. Hence the extension
$M_f=T_f+H_f$ is the pointwise multiplication operator.

For $g\in A^p$, use the same sequence $(g_m)\subset\Gamma$.  Since
$M_fg_m\to M_fg$ in $L^q$ and $P$ is bounded on $L^q$,
\begin{align*}
 T_fg
 &=\lim_{m\to\infty}T_fg_m
 =\lim_{m\to\infty}P(M_fg_m)
 =P(M_fg),\\
 H_fg
 &=\lim_{m\to\infty}H_fg_m
 =\lim_{m\to\infty}(I-P)(M_fg_m)
 =(I-P)(M_fg).
\end{align*}
Therefore
\[
 T_f=PM_f,
 \qquad
 H_f=(I-P)M_f
\]
on all of $A^p$.  Conversely, suppose that $M_f$ has a bounded extension.
The bounded operators
\[
 PM_f:A^p\to A^q,
 \qquad
 (I-P)M_f:A^p\to L^q
\]
agree on $\Gamma$ with $T_f$ and $H_f$, respectively.  Hence they are
their extensions, and $WM_f$ is the extension of $\mathscr T_f$.
Density of $\Gamma$ gives uniqueness in every case.
\end{proof}

The matching IDA criterion now follows from the preceding results.

\begin{corollary}
\label{cor:matching-ida}
Let $1<p,q<\infty$, $1\le r<\infty$, and
$f\in\mathcal D_{\Omega,q}$.  Then
\begin{equation}\label{eq:matching-ida-equivalence}
\begin{aligned}
 \bigl[T_f\in\Piabs_r(A^p,A^q)
 \text{ and }f\in\mathcal I_{p,q}^{r}(\Omega)\bigr]
 &\Longleftrightarrow M_f\in\Piabs_r(A^p,L^q)\\
 &\Longleftrightarrow f\in\mathcal L_{p,q}^{r}(\Omega).
\end{aligned}
\end{equation}
Moreover,
\begin{equation}\label{eq:matching-ida-norm}
 \pi_r(T_f)+\|f\|_{\mathcal I_{p,q}^{r}(\Omega)}
 \asymp\pi_r(M_f)
 \asymp\|f\|_{\mathcal L_{p,q}^{r}(\Omega)}.
\end{equation}
The constants depend only on $p,q,r$, the finite-type data, and the
fixed geometry.  They do not depend on $f$.
\end{corollary}

\begin{proof}
By \cref{cor:multiplication}, the second and third conditions are
equivalent and
\begin{equation}\label{eq:matching-M-L}
 \pi_r(M_f)\asymp\|f\|_{\mathcal L_{p,q}^{r}}.
\end{equation}
Assume that $M_f$ is $r$-summing.  On $\Gamma$,
$T_f=PM_f$.  Hence $PM_f$ is the unique extension of $T_f$, and
\[
 \pi_r(T_f)\le\|P\|_{L^q\to L^q}\pi_r(M_f).
\]
Also,
\[
 \G_{q,\rho_+}f(a_j)\le\M_{q,\rho_+}f(a_j).
\]
Solidity and the estimate comparing radii in
\cref{lem:discretization} give
\[
 \|f\|_{\mathcal I_{p,q}^{r}}
 \lesssim\|f\|_{\mathcal L_{p,q}^{r}}
 \lesssim\pi_r(M_f).
\]
Therefore
\begin{equation}\label{eq:matching-upper}
 \pi_r(T_f)+\|f\|_{\mathcal I_{p,q}^{r}}
 \lesssim\pi_r(M_f).
\end{equation}
Thus the second condition implies the first.

Conversely, assume the first condition.  By
\cref{thm:intro-hankel},
\[
 H_f\in\Piabs_r(A^p,L^q),
 \qquad
 \pi_r(H_f)\lesssim\|f\|_{\mathcal I_{p,q}^{r}}.
\]
On $\Gamma$,
\[
 M_f=T_f+H_f,
\]
so the sum of the two extensions is an $r$-summing extension of $M_f$.
Thus
\begin{equation}\label{eq:matching-lower}
 \pi_r(M_f)
 \le\pi_r(T_f)+C\|f\|_{\mathcal I_{p,q}^{r}}.
\end{equation}
This proves that the first condition implies the second.  The second and
third conditions are equivalent by \eqref{eq:matching-M-L}.  Finally,
\eqref{eq:matching-upper}, \eqref{eq:matching-lower}, and
\eqref{eq:matching-M-L} give \eqref{eq:matching-ida-norm}.
\end{proof}

\section{Little Hankel operators}\label{sec:little}

We study the little Hankel operator and its multiplication graph
separately. We give a scalar necessary condition for the little Hankel
operator. The joint criterion follows from
\cref{cor:multiplication} and the boundedness of $\overline P$.
The examples of operators with finite rank below show why these two
statements have different conclusions.

Janson studied Hankel operators between weighted Bergman spaces
\cite{Janson1988}. Generalized Toeplitz and little Hankel operators
on $A^p$ were studied in \cite{TaskinenVirtanen2018}.
The case of finite-type domains is treated in
\cite{BonamiPelosoSymesak2001}. Recent results on bounded
symmetric domains include \cite{YangYuan2025,Yuan2026}.
The absolutely summing ball case is treated in
\cite{FanHeWangZeng2026}.

Put
\[
 \overline{A^q(\Omega)}
 =\{u\in L^q(\Omega):\overline u\in A^q(\Omega)\}.
\]
It is a complex Banach space with the $L^q$ norm.  For
$1<q<\infty$, define
\begin{equation*}
 \overline P:L^q(\Omega)\longrightarrow\overline{A^q(\Omega)},
 \qquad
 \overline Pu=\overline{P(\overline u)}.
\end{equation*}
This map is complex linear.  It is bounded and
$\|\overline P\|_{L^q\to L^q}=\|P\|_{L^q\to L^q}$.
Moreover, $\overline P^2=\overline P$.  Put
\[
 N_q(\Omega)=\ker\overline P,
 \qquad
 L^q(\Omega)=\overline{A^q(\Omega)}\oplus N_q(\Omega).
\]
The space $N_q(\Omega)$ is closed and has the $L^q$ norm.
For $q=2$, it is the orthogonal complement of
$\overline{A^2(\Omega)}$.  For general $q$, the displayed sum is the
Banach space decomposition given by the bounded projection
$\overline P$.
For $f\in\mathcal D_{\Omega,q}$, define on $\Gamma$
\begin{equation*}
 h_fg=\overline P(fg),
 \qquad
 k_fg=(I-\overline P)(fg).
\end{equation*}
The operator $h_f:\Gamma\to\overline{A^q}$ is the little Hankel
operator.  The auxiliary operator $k_f:\Gamma\to N_q$ gives the
component removed by $\overline P$.  Thus
\[
 M_f=h_f+k_f,\qquad \overline Pk_f=0
 \quad\text{on }\Gamma.
\]
All statements about operator ideals below concern the unique bounded
extensions of these initial operators, when they exist.  Summability
of $h_f$ alone does not imply that $fg\in L^q$ for every $g\in A^p$.
Set
\begin{equation*}
 \mathscr H_fg=(h_fg,k_fg)
 \in\overline{A^q(\Omega)}\oplus_qL^q(\Omega).
\end{equation*}
The second coordinate lies in $N_q$, which we regard as a subspace of
$L^q$.  The equivalence between summability of this pair and of $M_f$
uses only the bounded projection decomposition.  Its symbol condition
comes from \cref{cor:multiplication}.

\begin{theorem}\label{thm:little-graph}
Let $1<p,q<\infty$, $1\le r<\infty$, and
$f\in\mathcal D_{\Omega,q}$.  The following conditions are equivalent:
\begin{enumerate}[label=\textup{(\roman*)}]
\item $M_f\in\Piabs_r(A^p,L^q)$;
\item $\mathscr H_f\in
\Piabs_r(A^p,\overline{A^q}\oplus_qL^q)$;
\item $h_f\in\Piabs_r(A^p,\overline{A^q})$ and
$k_f\in\Piabs_r(A^p,L^q)$;
\item $f\in\mathcal L_{p,q}^{r}(\Omega)$.
\end{enumerate}
Moreover,
\begin{equation}\label{eq:little-graph-norm}
 \pi_r(M_f)
 \asymp\pi_r(\mathscr H_f)
 \asymp\pi_r(h_f)+\pi_r(k_f)
 \asymp\|f\|_{\mathcal L_{p,q}^{r}(\Omega)}.
\end{equation}
Whenever these conditions hold, the four initial operators
$M_f,\mathscr H_f,h_f$, and $k_f$ have unique bounded extensions.  The
comparison constants depend only on $p,q,r$, the finite-type data, and
the fixed admissible geometric parameters.
\end{theorem}

\begin{proof}
Define
\[
 W_\ell:L^q\longrightarrow\overline{A^q}\oplus_qL^q,
 \qquad
 W_\ell u=(\overline Pu,(I-\overline P)u),
\]
and
\[
 \Sigma_\ell:\overline{A^q}\oplus_qL^q\longrightarrow L^q,
 \qquad
 \Sigma_\ell(v,w)=v+w.
\]
Since $\overline P$ is bounded on $L^q$,
\begin{align*}
 \|W_\ell u\|_{\overline{A^q}\oplus_qL^q}
 &\le
 \bigl(\|\overline P\|^q+\|I-\overline P\|^q\bigr)^{1/q}
 \|u\|_{L^q},                                      \\
 \|\Sigma_\ell(v,w)\|_{L^q}
 &\le \|v\|_{L^q}+\|w\|_{L^q}
 \le2^{1-1/q}\|(v,w)\|_{\overline{A^q}\oplus_qL^q}.
\end{align*}
Thus both maps are bounded.  On $\Gamma$,
\[
 \mathscr H_f=W_\ell M_f,
 \qquad
 M_f=\Sigma_\ell\mathscr H_f.
\]
If $M_f$ extends, then $W_\ell M_f$ extends $\mathscr H_f$.  If
$\mathscr H_f$ extends, then $\Sigma_\ell\mathscr H_f$ extends $M_f$.
The two coordinate projections have norm one.  Hence
\begin{align*}
 \pi_r(\mathscr H_f)&\le\|W_\ell\|\pi_r(M_f),&
 \pi_r(M_f)&\le\|\Sigma_\ell\|\pi_r(\mathscr H_f),\\
 \pi_r(h_f)&\le\pi_r(\mathscr H_f),&
 \pi_r(k_f)&\le\pi_r(\mathscr H_f).
\end{align*}
Conversely, assume that $h_f$ and $k_f$ have absolutely $r$-summing
extensions.  For
$g_1,\ldots,g_N\in A^p$, Minkowski's inequality gives
\begin{align*}
 \left(\sum_{\nu=1}^N
 \|\mathscr H_fg_\nu\|_{\overline{A^q}\oplus_qL^q}^{\,r}\right)^{1/r}
 &\le
 \left(\sum_{\nu=1}^N\|h_fg_\nu\|_{L^q}^{\,r}\right)^{1/r}+
 \left(\sum_{\nu=1}^N\|k_fg_\nu\|_{L^q}^{\,r}\right)^{1/r}.
\end{align*}
The last expression is at most
\[
 \bigl(\pi_r(h_f)+\pi_r(k_f)\bigr)
 w_r(g_1,\ldots,g_N).
\]
Therefore
\[
 \frac12\bigl(\pi_r(h_f)+\pi_r(k_f)\bigr)
 \le\pi_r(\mathscr H_f)
 \le\pi_r(h_f)+\pi_r(k_f).
\]
Thus (i), (ii), and (iii) are equivalent.  By \cref{cor:multiplication},
\[
 \pi_r(M_f)
 \asymp\|\mathbf M_{p,q,\rho}f\|_{\dideal}.
\]
This proves the equivalence with (iv) and
\eqref{eq:little-graph-norm}.

It remains to identify the extensions.  Suppose first that $h_f$ and
$k_f$ have bounded extensions.  Let $g_m\in\Gamma$ and $g_m\to g$ in $A^p$.  Put
\[
 u=\lim_{m\to\infty}(h_fg_m+k_fg_m)\quad\text{in }L^q.
\]
Fix $K\Subset\Omega$.  The point evaluation estimate gives
\[
 \sup_{z\in K}|g_m(z)-g(z)|
 \le C_K\|g_m-g\|_{A^p}\longrightarrow0.
\]
Since $f\in L^q_{\loc}$,
\[
 \|fg_m-fg\|_{L^q(K)}
 \le
 \sup_{z\in K}|g_m(z)-g(z)|\,\|f\|_{L^q(K)}
 \longrightarrow0.
\]
On $\Gamma$,
\[
 fg_m=h_fg_m+k_fg_m.
\]
Thus $u=fg$ almost everywhere on $K$.  Let $K$ increase to $\Omega$
through a sequence of compact sets. Then $u=fg$ almost everywhere on
$\Omega$. Hence the extension given by $\Sigma_\ell\mathscr H_f$ is the
pointwise multiplication operator.

Conversely, suppose that $M_f$ extends.  The bounded operators
\[
 \overline PM_f:A^p\to\overline{A^q},
 \qquad
 (I-\overline P)M_f:A^p\to L^q
\]
agree on $\Gamma$ with $h_f$ and $k_f$.  They are their extensions.
Moreover, $W_\ell M_f$ extends $\mathscr H_f$.  To identify the
range of the second component, let $g_m\in\Gamma$ tend to $g$ in
$A^p$.  Then
\[
 \overline Pk_fg
 =\lim_{m\to\infty}\overline Pk_fg_m=0.
\]
Thus the extended $k_f$ takes its values in the closed space $N_q$.
Density of $\Gamma$ proves uniqueness of all four extensions.
\end{proof}

The little Hankel operator itself has a different scalar diagonal.
For $f\in\mathcal D_{\Omega,q}$, define
\begin{equation}\label{eq:little-Berezin-transform}
 \mathfrak b f(z)
 =\frac{1}{K(z,z)}
 \int_\Omega f(w)K(w,z)^2\dd v(w).
\end{equation}
The integral is absolutely convergent.  Indeed,
$fK(\cdot,z)\in L^q$ and $K(\cdot,z)\in L^{q'}$.

\begin{theorem}\label{thm:little-necessity}
Let $1<p,q<\infty$, $1\le r<\infty$, and
$f\in\mathcal D_{\Omega,q}$.  If
$h_f\in\Piabs_r(A^p,\overline{A^q})$, then
\begin{equation}\label{eq:little-necessity}
 \{V_j^{1/q-1/p}|\mathfrak b f(a_j)|\}_j\in\dideal,
\end{equation}
and
\begin{equation}\label{eq:little-necessity-norm}
 \|\{V_j^{1/q-1/p}|\mathfrak b f(a_j)|\}\|_{\dideal}
 \lesssim\pi_r(h_f).
\end{equation}
\end{theorem}

\begin{proof}
First restrict the indices to one sufficiently separated sublattice
$J_\ell$.  Relabel $J_\ell$ as $\mathbb N$ when defining the sequence
maps below.  Let
$S_p:\ell^p\to A^p$ be the bounded kernel synthesis map.  For
$u\in\overline{A^q}$, define
\begin{equation*}
 R_q^\ell u
 =\left\{\int_\Omega u(w)e_{j,q'}(w)\dd v(w)\right\}_j.
\end{equation*}
Write $u=\overline h$, where $h\in A^q$.  Reproduction gives
\[
 \int_\Omega u e_{j,q'}\dd v
 =\frac{\overline{h(a_j)}}{K(a_j,a_j)^{1/q}}.
\]
Since $K(a_j,a_j)^{-1/q}\asymp V_j^{1/q}$, the sampling estimate
\eqref{eq:sampling-map} gives
\[
 \|R_q^\ell u\|_{\ell^q}\lesssim\|u\|_{\overline{A^q}}.
\]
Thus $R_q^\ell:\overline{A^q}\to\ell^q$ is bounded.  Consequently,
\[
 B=R_q^\ell h_fS_p:\ell^p\longrightarrow\ell^q
\]
is absolutely $r$-summing and
\[
 \pi_r(B)\lesssim\pi_r(h_f).
\]
Its diagonal is
\begin{align*}
 (Be_j)_j
 &=\int_\Omega h_f(e_{j,p})(w)e_{j,q'}(w)\dd v(w)\notag\\
 &=\int_\Omega f(w)e_{j,p}(w)e_{j,q'}(w)\dd v(w)\notag\\
 &=K(a_j,a_j)^{1/p-1/q}\mathfrak b f(a_j).
\end{align*}
For the second equality, use the projection pairing
\eqref{eq:projection-pairing} with the exponents $q,q'$.
For $u\in L^q$ and $\psi\in A^{q'}$, apply this identity to
$\overline u$ and $\psi$, and use $P\psi=\psi$ to obtain
\begin{align*}
 \int_\Omega\overline Pu\,\psi\dd v
 &=\overline{\int_\Omega P(\overline u)\,\overline\psi\dd v}\\
 &=\overline{\int_\Omega\overline u\,\overline\psi\dd v}
 =\int_\Omega u\psi\dd v.
\end{align*}
Apply this identity with
\[
 u=fe_{j,p},
 \qquad
 \psi=e_{j,q'}.
\]
Both functions lie in the required spaces.  Indeed,
$fe_{j,p}\in L^q$ by the definition of
$\mathcal D_{\Omega,q}$, and $e_{j,q'}\in A^{q'}$.
For the third equality, use
\[
 e_{j,p}(w)e_{j,q'}(w)
 =
 \frac{K(w,a_j)^2}
 {K(a_j,a_j)^{\,1-1/p+1/q}}
\]
and \eqref{eq:little-Berezin-transform}.  Finally,
$K(a_j,a_j)\asymp V_j^{-1}$ gives
\begin{equation}\label{eq:little-diagonal-modulus}
 |(Be_j)_j|
 \asymp V_j^{1/q-1/p}|\mathfrak b f(a_j)|.
\end{equation}

Return now to the original lattice indices.  Define
\[
 d_j^{(\ell)}
 =
 \begin{cases}
 V_j^{1/q-1/p}|\mathfrak b f(a_j)|,&j\in J_\ell,\\
 0,&j\notin J_\ell.
\end{cases}
\]
Coordinate injections and projections have norm one.  Hence zero
extension preserves the norm of the diagonal ideal.  The Rademacher diagonal
extraction in the proof of \cref{prop:operator-diagonal}, followed by
\eqref{eq:little-diagonal-modulus} and solidity, gives
\[
 \|d^{(\ell)}\|_{\mathfrak d_{p,q}^{\,r}}
 \lesssim\pi_r(B)
 \lesssim\pi_r(h_f).
\]
Decompose the full lattice into $N_*$ such sublattices.  Since
\[
 \{V_j^{1/q-1/p}|\mathfrak b f(a_j)|\}_j
 =\sum_{\ell=1}^{N_*}d^{(\ell)},
\]
the triangle inequality proves
\eqref{eq:little-necessity} and
\eqref{eq:little-necessity-norm}.
\end{proof}

\begin{corollary}
\label{cor:little-two-sided}
Let $1<p,q<\infty$, $1\le r<\infty$, and
$f\in\mathcal D_{\Omega,q}\cap\mathcal L_{p,q}^{r}(\Omega)$.  Then
$h_f\in\Piabs_r(A^p,\overline{A^q})$ and
\begin{equation*}
 \left\|\{V_j^{1/q-1/p}|\mathfrak b f(a_j)|\}\right\|_{\dideal}
 \lesssim\pi_r(h_f)
 \lesssim\|f\|_{\mathcal L_{p,q}^{r}(\Omega)}.
\end{equation*}
\end{corollary}

\begin{proof}
By \cref{cor:multiplication},
\[
 M_f\in\Piabs_r(A^p,L^q),
 \qquad
 \pi_r(M_f)\asymp\|f\|_{\mathcal L_{p,q}^{r}}.
\]
On $\Gamma$, $h_f=\overline PM_f$.  Thus $\overline PM_f$ is the
unique extension of $h_f$, and
\[
 \pi_r(h_f)
 \le\|\overline P\|\,\pi_r(M_f)
 \lesssim\|f\|_{\mathcal L_{p,q}^{r}}.
\]
The left inequality follows from
\eqref{eq:little-necessity-norm}.
\end{proof}

The next result assumes that the complementary
operator is summing.  It gives a criterion for $h_f$ under this
assumption. It does not characterize $h_f$ by its symbol alone.

\begin{corollary}
\label{cor:little-matching}
Let $1<p,q<\infty$, $1\le r<\infty$, and
$f\in\mathcal D_{\Omega,q}$.  Assume
$k_f\in\Piabs_r(A^p,L^q)$.  Then
\[
 h_f\in\Piabs_r(A^p,\overline{A^q})
 \quad\Longleftrightarrow\quad
 f\in\mathcal L_{p,q}^{r}(\Omega).
\]
Moreover,
\[
 \pi_r(h_f)+\pi_r(k_f)
 \asymp\|f\|_{\mathcal L_{p,q}^{r}(\Omega)}.
\]
\end{corollary}

\begin{proof}
Assume first that $h_f$ is $r$-summing.  On $\Gamma$,
\[
 M_f=h_f+k_f.
\]
The sum of the two extensions is an $r$-summing extension of $M_f$.
Hence \cref{cor:multiplication} gives
\[
 f\in\mathcal L_{p,q}^{r},
 \qquad
 \|f\|_{\mathcal L_{p,q}^{r}}
 \lesssim\pi_r(h_f)+\pi_r(k_f).
\]
Conversely, assume that $f\in\mathcal L_{p,q}^{r}$.  Then
$M_f$ is $r$-summing and
\[
 h_f=\overline PM_f,
 \qquad
 \pi_r(h_f)\le\|\overline P\|\,\pi_r(M_f)
 \lesssim\|f\|_{\mathcal L_{p,q}^{r}}.
\]
The same factorization gives
\[
 \pi_r(k_f)
 \le\|I-\overline P\|\,\pi_r(M_f)
 \lesssim\|f\|_{\mathcal L_{p,q}^{r}}.
\]
The last three estimates prove the equivalence and
the upper and lower norm estimates.
\end{proof}

\begin{example}\label{ex:little-rank-one}
Let $\Omega=\mathbb B_n=\{z\in\C^n:|z|<1\}$ be the unit ball.
Let $1<p<\infty$, set $q=p$, and let
$f=1$.  For a holomorphic polynomial $G$ and a multiindex $\beta$,
rotational orthogonality gives
\[
 \int_{\mathbb B_n}G(w)w^\beta\dd v(w)
 =
 \begin{cases}
 v(\mathbb B_n)G(0),&\beta=0,\\
 0,&\beta\ne0.
 \end{cases}
\]
The expansion in the anti-holomorphic orthonormal basis gives
$\overline PG=G(0)$.  Holomorphic polynomials are dense in $A^p$.
Both $\overline P:L^p\to L^p$ and evaluation at zero are bounded.
Passing to the $A^p$ limit gives
\[
 h_1g=\overline Pg=g(0),
 \qquad g\in A^p(\mathbb B_n).
\]
Consequently,
\[
 k_1g=g-g(0),\qquad H_1g=(I-P)g=0.
\]
Thus the auxiliary operator $k_f$ differs from the big Hankel operator
$H_f$.
The operator $h_1$ has rank one.  It belongs to $\Piabs_r$ for every
$1\le r<\infty$.  In contrast, $M_1:A^p\to L^p$ is neither compact nor
absolutely $r$-summing.
Indeed, choose $a_j\to\partial\mathbb B_n$.  The normalized kernels
$e_{j,p}$ tend weakly to zero.  To verify this, let
$G\in A^{p'}(\mathbb B_n)$ and choose polynomials $G_m\to G$ in
$A^{p'}$.  Then
\[
 \left|\int_{\mathbb B_n}e_{j,p}\overline G\dd v\right|
 =\frac{|G(a_j)|}{K(a_j,a_j)^{1/p'}}
 \le
 \frac{|G_m(a_j)|}{K(a_j,a_j)^{1/p'}}
 +C\|G-G_m\|_{A^{p'}}.
\]
For fixed $m$, the polynomial $G_m$ is bounded on
$\overline{\mathbb B_n}$.  Also,
\[
 K(a_j,a_j)
 =\frac{n!}{\pi^n}(1-|a_j|^2)^{-n-1}\longrightarrow\infty.
\]
Hence
\[
 \limsup_{j\to\infty}
 \left|\int_{\mathbb B_n}e_{j,p}\overline G\dd v\right|
 \le C\|G-G_m\|_{A^{p'}}.
\]
Let $m\to\infty$.  This proves that $(e_{j,p})$ converges weakly to
zero in $A^p$.
On the other hand,
\[
 \|M_1e_{j,p}\|_{L^p}=\|e_{j,p}\|_{A^p}\asymp1.
\]
Thus $M_1$ is not compact.  It is not absolutely $r$-summing by
\cref{lem:summing-calculus}\textup{(iv)}.  Since $h_1$ has rank one,
the identity $M_1=h_1+k_1$ also shows that $k_1$ is not absolutely
$r$-summing.  By
\cref{cor:multiplication}, $1\notin\mathcal L_{p,p}^r$.  Therefore
neither membership in $\mathcal L_{p,p}^r$ nor summability of $M_f$ is
necessary for summability of $h_f$ alone.
\end{example}

The preceding rank-one example belongs to an explicit
family of operators of finite rank.  It also gives a direct check of the
diagonal condition for the little Hankel operator.

\begin{theorem}
\label{thm:little-monomial}
Let $\alpha=(\alpha_1,\ldots,\alpha_n)\in\Nzero^n$ and
$f_\alpha(z)=\overline{z^\alpha}$ on $\mathbb B_n$.  Let
$1<p,q<\infty$ and $1\le r<\infty$.  Then
\[
 h_{f_\alpha}:A^p(\mathbb B_n)\longrightarrow
 \overline{A^q(\mathbb B_n)}
\]
has finite rank.  More precisely,
\begin{equation}\label{eq:little-monomial-rank}
 \operatorname{rank}h_{f_\alpha}
 =\prod_{k=1}^n(\alpha_k+1).
\end{equation}
For every multiindex $\gamma\in\Nzero^n$,
\begin{equation}\label{eq:little-monomial-action}
 h_{\overline{z^\alpha}}(z^\gamma)
 =
 \begin{cases}
 \displaystyle
 \frac{\|z^\alpha\|_{A^2}^2}
 {\|z^{\alpha-\gamma}\|_{A^2}^2}
 \overline{z^{\alpha-\gamma}},
 &0\le\gamma_k\le\alpha_k\quad(1\le k\le n),\\[9pt]
 0,&\text{otherwise}.
 \end{cases}
\end{equation}
Consequently,
\begin{equation}\label{eq:little-monomial-summing}
 h_{f_\alpha}\in\Piabs_1(A^p,\overline{A^q})
 \subset\Piabs_r(A^p,\overline{A^q}).
\end{equation}
More precisely,
\begin{equation}\label{eq:little-monomial-Berezin}
 \mathfrak b f_\alpha(z)
 =c_\alpha(1-|z|^2)^{n+1}\overline{z^\alpha}.
\end{equation}
Here
\begin{equation}\label{eq:little-monomial-constant}
 c_\alpha
 =
 \frac{n!\,(2n+1+|\alpha|)!}
 {(2n+1)!\,(n+|\alpha|)!}>0.
\end{equation}
In particular, $c_0=1$.
\end{theorem}

\begin{proof}
With Lebesgue volume on $\mathbb B_n$,
\begin{equation}\label{eq:ball-monomial-norm}
 \|z^\beta\|_{A^2(\mathbb B_n)}^2
 =\frac{\pi^n\beta!}{(n+|\beta|)!},
 \qquad
 \beta!=\prod_{k=1}^n\beta_k!.
\end{equation}
The functions
\[
 e_\beta(z)=\frac{z^\beta}{\|z^\beta\|_{A^2}},
 \qquad \beta\in\Nzero^n,
\]
form an orthonormal basis of $A^2(\mathbb B_n)$.  Hence
$(\overline{e_\beta})_\beta$ is an orthonormal basis of
$\overline{A^2(\mathbb B_n)}$.  For two multiindices
$\beta,\gamma$,
\begin{align*}
 \left\langle
 \overline{z^\alpha}z^\gamma,\overline{e_\beta}
 \right\rangle_{L^2}
 &=
 \frac1{\|z^\beta\|_{A^2}}
 \int_{\mathbb B_n}
 \overline{z^\alpha}z^{\gamma+\beta}\dd v(z)\notag\\
 &=
 \begin{cases}
 \displaystyle
 \frac{\|z^\alpha\|_{A^2}^2}{\|z^\beta\|_{A^2}},
 &\alpha=\gamma+\beta,\\[7pt]
 0,&\alpha\ne\gamma+\beta.
 \end{cases}
\end{align*}
The function $\overline{z^\alpha}z^\gamma$ belongs to
$L^2(\mathbb B_n)\cap L^q(\mathbb B_n)$.  On this intersection, the
bounded $L^q$ projection $\overline P$ agrees with the orthogonal
$L^2$ projection.  Hence projection onto the anti-holomorphic basis gives
\eqref{eq:little-monomial-action}.

Let $F(z)=\sum_\gamma\widehat F(\gamma)z^\gamma$ be a polynomial.
Equation \eqref{eq:little-monomial-action} gives
\begin{equation}\label{eq:little-monomial-finite-formula}
 h_{\overline{z^\alpha}}F
 =
 \sum_{0\le\gamma\le\alpha}
 \widehat F(\gamma)
 \frac{\|z^\alpha\|_{A^2}^2}
 {\|z^{\alpha-\gamma}\|_{A^2}^2}
 \overline{z^{\alpha-\gamma}}.
\end{equation}
Denote the right side by $\mathcal R_\alpha F$.
The inequalities between multiindices are coordinatewise.  Each
coefficient functional
\[
 F\longmapsto\widehat F(\gamma)
 =\frac1{\gamma!}\partial^\gamma F(0)
\]
is bounded on $A^p(\mathbb B_n)$.  Indeed, set
$t_0=(4\sqrt n)^{-1}$.  The closed polydisc
$\{z:|z_k|\le t_0\}$ lies in $\frac12\mathbb B_n$.  Cauchy's
formula and the submean estimate give
\begin{align*}
 |\widehat F(\gamma)|
 &\le t_0^{-|\gamma|}
 \sup_{|z_k|\le t_0}|F(z)|\\
 &\le t_0^{-|\gamma|}C_{n,p}\|F\|_{A^p}
 \le C_{\alpha,n,p}\|F\|_{A^p},
 \qquad 0\le\gamma\le\alpha.
\end{align*}
Thus the right side of
\eqref{eq:little-monomial-finite-formula} defines a bounded operator
from $A^p$ into the finite-dimensional space
\[
 \operatorname{span}\{
 \overline{z^\beta}:0\le\beta\le\alpha\}
 \subset\overline{A^q}.
\]
For every polynomial $F$,
\[
 \mathcal R_\alpha F=\overline P(f_\alpha F).
\]
If $G\in\Gamma$, then $G$ extends holomorphically to a neighborhood of
$\overline{\mathbb B_n}$.  Its Taylor polynomials converge uniformly
on $\overline{\mathbb B_n}$.  Let $(G_N)$ be these polynomials.  Since
$f_\alpha$ is bounded,
\[
 \|f_\alpha(G_N-G)\|_{L^q}
 \le\|f_\alpha\|_\infty
 v(\mathbb B_n)^{1/q}\|G_N-G\|_\infty
 \longrightarrow0.
\]
Uniform convergence also gives $G_N\to G$ in $A^p$.  Hence
\begin{align*}
 \mathcal R_\alpha G_N&\longrightarrow\mathcal R_\alpha G
 &&\text{in }\overline{A^q},\\
 \overline P(f_\alpha G_N)&\longrightarrow\overline P(f_\alpha G)
 &&\text{in }\overline{A^q}.
\end{align*}
Since $\mathcal R_\alpha G_N=\overline P(f_\alpha G_N)$,
\[
 \mathcal R_\alpha G=\overline P(f_\alpha G)=h_{f_\alpha}G.
\]
Thus $\mathcal R_\alpha$ agrees with the initial operator on $\Gamma$.
It is its unique bounded extension.

For every $\beta$ with $0\le\beta\le\alpha$, take
$\gamma=\alpha-\beta$ in
\eqref{eq:little-monomial-action}.  The image is a nonzero multiple of
$\overline{z^\beta}$.  These monomials are linearly independent.
Therefore the range has dimension
\[
 \#\{\beta:0\le\beta\le\alpha\}
 =\prod_{k=1}^n(\alpha_k+1).
\]
This proves \eqref{eq:little-monomial-rank}.

The finite formula is a finite sum of rank-one operators:
\[
 h_{\overline{z^\alpha}}
 =\sum_{0\le\gamma\le\alpha}
 \left(
 \widehat{\,\cdot\,}(\gamma)
 \right)
 \otimes
 \left(
 \frac{\|z^\alpha\|_{A^2}^2}
 {\|z^{\alpha-\gamma}\|_{A^2}^2}
 \overline{z^{\alpha-\gamma}}
 \right).
\]
Hence it is nuclear.  Nuclear operators are absolutely $1$-summing.
More explicitly,
\begin{align*}
 \pi_1(h_{\overline{z^\alpha}})
 &\le
 \sum_{0\le\gamma\le\alpha}
 \|\widehat{\,\cdot\,}(\gamma)\|_{(A^p)^*}
 \frac{\|z^\alpha\|_{A^2}^2}
 {\|z^{\alpha-\gamma}\|_{A^2}^2}
 \|z^{\alpha-\gamma}\|_{A^q}\\
 &<\infty.
\end{align*}
The inclusion theorem for summing ideals gives
$\pi_r(h_{\overline{z^\alpha}})
\le\pi_1(h_{\overline{z^\alpha}})$ for $r\ge1$.  This proves
\eqref{eq:little-monomial-summing}.

It remains to calculate the scalar diagonal.  Write
\[
 K(w,z)=\frac{n!}{\pi^n}
 (1-\langle w,z\rangle)^{-n-1}.
\]
The square has the locally uniformly convergent expansion
\begin{equation*}
 K(w,z)^2
 =\sum_{\beta\in\Nzero^n}
 d_\beta w^\beta\overline{z^\beta},
 \qquad
 d_\beta
 =
 \left(\frac{n!}{\pi^n}\right)^2
 \frac{(2n+1+|\beta|)!}{(2n+1)!\,\beta!}>0.
\end{equation*}
For fixed $z\in\mathbb B_n$, group the terms by $|\beta|=k$.
The multinomial formula and Cauchy--Schwarz give
\begin{align*}
 \sum_{|\beta|=k}d_\beta|w^\beta z^\beta|
 &=\left(\frac{n!}{\pi^n}\right)^2
 \frac{(2n+1+k)!}{(2n+1)!\,k!}
 \left(\sum_{j=1}^n|w_jz_j|\right)^k\\
 &\le\left(\frac{n!}{\pi^n}\right)^2
 \frac{(2n+1+k)!}{(2n+1)!\,k!}|z|^k,
 \qquad |w|\le1.
\end{align*}
The sum of the last expression over $k\ge0$ is
$(n!/\pi^n)^2(1-|z|)^{-2n-2}<\infty$.  Hence the series for $K(w,z)^2$
converges absolutely and uniformly for
$w\in\overline{\mathbb B_n}$.  The same bound holds after multiplication by
$\overline{w^\alpha}$.  We may therefore
integrate term by term.  Rotational orthogonality and
\eqref{eq:ball-monomial-norm} give
\begin{align*}
 \int_{\mathbb B_n}\overline{w^\alpha}K(w,z)^2\dd v(w)
 &=d_\alpha\|z^\alpha\|_{A^2}^2\overline{z^\alpha},\\
 K(z,z)&=\frac{n!}{\pi^n}(1-|z|^2)^{-n-1}.
\end{align*}
Put $m=|\alpha|$.  Then
\begin{align*}
 \frac{d_\alpha\|z^\alpha\|_{A^2}^2}
 {n!/\pi^n}
 &=
 \frac{\pi^n}{n!}
 \left(\frac{n!}{\pi^n}\right)^2
 \frac{(2n+1+m)!}{(2n+1)!\,\alpha!}
 \frac{\pi^n\alpha!}{(n+m)!}\\
 &=\frac{n!\,(2n+1+m)!}{(2n+1)!\,(n+m)!}
 =c_\alpha.
\end{align*}
Division by $K(z,z)$ proves
\eqref{eq:little-monomial-Berezin} and
\eqref{eq:little-monomial-constant}.
\end{proof}

\begin{corollary}
\label{cor:little-polynomial}
Let $1<p,q<\infty$ and $1\le r<\infty$.  Let
\[
 f(z)=\sum_{\alpha\in F}\lambda_\alpha\overline{z^\alpha},
\]
where $F\subset\Nzero^n$ is finite.  Then
$h_f:A^p(\mathbb B_n)\to\overline{A^q(\mathbb B_n)}$ is finite rank and
belongs to $\Piabs_r$.  In particular, even when $q=p$, summability of
$h_f$ does not imply $f\in\mathcal L_{p,p}^{r}(\mathbb B_n)$.
\end{corollary}

\begin{proof}
Linearity and \cref{thm:little-monomial} give
\[
 h_f=\sum_{\alpha\in F}\lambda_\alpha
 h_{\overline{z^\alpha}}
\]
and
\[
 \operatorname{ran}h_f
 \subset
 \sum_{\alpha\in F}
 \operatorname{span}\{
 \overline{z^\beta}:0\le\beta\le\alpha\}.
\]
The space on the right is finite dimensional.  Moreover,
\[
 \pi_1(h_f)
 \le\sum_{\alpha\in F}|\lambda_\alpha|
 \pi_1(h_{\overline{z^\alpha}})<\infty.
\]
Thus $h_f\in\Piabs_1\subset\Piabs_r$.  The final assertion follows by
taking $f=1$ and using \cref{ex:little-rank-one}.
\end{proof}

\subsection{Differences of projected operators}

\begin{corollary}\label{thm:projected-perturbation}
\label{cor:projected-closedness}
Let $1<p,q<\infty$, $1\le r<\infty$, and
$f_0,f_1\in\mathcal D_{\Omega,q}$.  Put $d=f_1-f_0$.  Then
\begin{align}
 \pi_r(\mathscr T_{f_1}-\mathscr T_{f_0})
 &\asymp\|d\|_{\mathcal L_{p,q}^{r}},
 \label{eq:Toeplitz-graph-perturbation}\\
 \pi_r(\mathscr H_{f_1}-\mathscr H_{f_0})
 &\asymp\|d\|_{\mathcal L_{p,q}^{r}}.
 \notag
\end{align}
For the big Hankel difference,
\begin{equation}\label{eq:Hankel-perturbation}
 \pi_r(H_{f_1}-H_{f_0})
 \asymp\|d\|_{\mathcal I_{p,q}^{r}}.
\end{equation}
Here the norms are allowed to be infinite.  A finite norm means that the
initial difference on $\Gamma$ has a unique absolutely $r$-summing
extension.  If $d\in\mathcal L_{p,q}^{r}$, each of the four component
differences $T_d,H_d,h_d,k_d$ has summing norm at most
$C\|d\|_{\mathcal L_{p,q}^{r}}$.
Thus convergence of symbols in the displayed seminorms gives
convergence of the corresponding differences in $\pi_r$.
The Hankel statement is unchanged by adding a symbol in
$\Hol(\Omega)\cap\mathcal D_{\Omega,q}$ to $d$.
\end{corollary}

\begin{proof}
For every $z\in\Omega$,
\[
 \|dK(\cdot,z)\|_{L^q}
 \le\|f_1K(\cdot,z)\|_{L^q}
   +\|f_0K(\cdot,z)\|_{L^q}<\infty.
\]
Also $d\in L^q_{\loc}$, so $d\in\mathcal D_{\Omega,q}$.  On the
common domain $\Gamma$, linearity gives
\[
 \mathscr T_{f_1}-\mathscr T_{f_0}=\mathscr T_d,
 \qquad
 \mathscr H_{f_1}-\mathscr H_{f_0}=\mathscr H_d,
 \qquad
 H_{f_1}-H_{f_0}=H_d.
\]
Apply \cref{thm:intro-main,thm:little-graph,thm:intro-hankel} to $d$.
These theorems characterize the summing extensions of the operators
on the right. They also give the stated norm comparisons.  This argument
only requires an extension of the difference.  If both original
operators already have bounded extensions, their difference agrees with this
extension on $\Gamma$. Density gives agreement on all of $A^p$.

If $d\in\mathcal L_{p,q}^{r}$, then
$\pi_r(M_d)\lesssim\|d\|_{\mathcal L_{p,q}^{r}}$.  The identities
\[
 T_d=PM_d,\quad H_d=(I-P)M_d,\quad
 h_d=\overline PM_d,\quad k_d=(I-\overline P)M_d
\]
hold first on $\Gamma$ and then on $A^p$ by continuity.  The ideal
property gives, for example,
\[
 \pi_r(T_d)+\pi_r(H_d)
 \le(\|P\|+\|I-P\|)\pi_r(M_d)
 \lesssim\|d\|_{\mathcal L_{p,q}^{r}}.
\]
The same calculation with $\overline P$ proves the other two component
bounds.  All constants are independent of $f_0$ and $f_1$.
Applying these inequalities to $d_\nu=f_\nu-f$ gives
\begin{align*}
 \|d_\nu\|_{\mathcal L_{p,q}^{r}}\longrightarrow0
 &\quad\Longrightarrow\quad
 \pi_r(\mathscr T_{f_\nu}-\mathscr T_f)
 +\pi_r(\mathscr H_{f_\nu}-\mathscr H_f)\longrightarrow0,\\
 \|d_\nu\|_{\mathcal I_{p,q}^{r}}\longrightarrow0
 &\quad\Longrightarrow\quad
 \pi_r(H_{f_\nu}-H_f)\longrightarrow0.
\end{align*}

Finally, let $h\in\Hol(\Omega)\cap\mathcal D_{\Omega,q}$.
For $g\in\Gamma$, $hg\in A^q$, and therefore
\[
 H_hg=(I-P)(hg)=0,
 \qquad H_{d+h}g=H_dg.
\]
In the definition of $\G_{q,\rho_+}$, the map
$a\mapsto a+h$ is a bijection of the holomorphic approximants on
each enlarged polydisc.  Since
$|(d+h)-(a+h)|=|d-a|$, taking the infimum gives
\[
 \G_{q,\rho_+}(d+h)(z)=\G_{q,\rho_+}d(z),
 \qquad z\in\Omega.
\]
Thus the operators and the IDA seminorms are both unchanged.
\end{proof}

\section{Weighted composition operators}\label{sec:composition}

The results of this section follow from \cref{thm:carleson} through
an exact pullback identity. That identity preserves the norm of every
finite family, so it applies directly to the $r$-summing norm.
The geometry enters through the nonisotropic local masses of the
pullback measure. The argument also covers $q=1$.

Pullback methods for boundedness and compactness are standard
\cite{CowenMacCluer1995,CuckovicZhao2004}.
For weakly pseudoconvex domains, see \cite{Zhang2025,Zhang2026}.
Absolutely summing criteria on the disc and half-plane appear in
\cite{HeJreisLefevreLou2024,ChenDongWang2025}.

Let $\varphi:\Omega\to\Omega$ be holomorphic.  Let
$u\in\Hol(\Omega)$.  The weighted composition operator is
\begin{equation*}
 W_{u,\varphi}F=u(F\circ\varphi).
\end{equation*}
Its natural domain as an operator from $A^p$ to $A^q$ is
\begin{equation}\label{eq:weighted-composition-natural-domain}
 \mathcal D(W_{u,\varphi})
 =\{F\in A^p(\Omega):u(F\circ\varphi)\in A^q(\Omega)\}.
\end{equation}
The pointwise formula is well defined for every $F\in\Hol(\Omega)$.
The notation $W_{u,\varphi}:A^p\to A^q$ means that
\eqref{eq:weighted-composition-natural-domain} equals $A^p$ and that
the resulting operator is bounded.
The composition operator $C_\varphi$ is the case $u=1$.
For $u\in A^q(\Omega)$, define the finite positive Borel measure
\begin{equation}\label{eq:weighted-pullback}
 \mu_{u,\varphi}
 =\varphi_*(|u|^qv),
 \qquad
 \mu_{u,\varphi}(E)
 =\int_{\varphi^{-1}(E)}|u(z)|^q\dd v(z).
\end{equation}
Here $\varphi_*$ denotes pushforward of measures.  Thus
\begin{equation}\label{eq:weighted-composition-isometry}
 \|W_{u,\varphi}F\|_{A^q}^q
 =\int_\Omega|F(w)|^q\dd\mu_{u,\varphi}(w).
\end{equation}
The assumption $u\in A^q$ is necessary when
$W_{u,\varphi}:A^p\to A^q$ is bounded.  Indeed,
$W_{u,\varphi}\mathbf1=u$.

\begin{theorem}
\label{thm:weighted-composition}
Let $1<p<\infty$, $1\le q,r<\infty$,
$u\in A^q(\Omega)$, and
$\varphi:\Omega\to\Omega$ be holomorphic.  The following conditions are
equivalent:
\begin{enumerate}[label=\textup{(\roman*)}]
\item $W_{u,\varphi}\in\Piabs_r(A^p(\Omega),A^q(\Omega))$;
\item $J_{\mu_{u,\varphi}}\in
\Piabs_r(A^p(\Omega),L^q(\mu_{u,\varphi}))$;
\item
\begin{equation}\label{eq:weighted-composition-sequence}
 \left\{
 \frac{\mu_{u,\varphi}(\Q_\rho(a_j))^{1/q}}
 {V_j^{1/p}}
 \right\}_j
 \in\mathfrak d_{p,q}^{\,r}.
\end{equation}
\end{enumerate}
Moreover,
\begin{equation}\label{eq:weighted-composition-norm}
 \pi_r(W_{u,\varphi})
 \asymp
 \left\|
 \left\{\frac{\mu_{u,\varphi}(\Q_\rho(a_j))^{1/q}}
 {V_j^{1/p}}\right\}_j
 \right\|_{\mathfrak d_{p,q}^{\,r}}.
\end{equation}
The constants depend only on $p,q,r$, the finite-type data, and the fixed
geometric parameters.  They do not depend on $u$ or $\varphi$.
\end{theorem}

\begin{proof}
First suppose that one of the operators in (i) and (ii) is bounded.
For $F\in A^p$, \eqref{eq:weighted-composition-isometry} gives
\[
 \|W_{u,\varphi}F\|_{A^q}
 =\|J_{\mu_{u,\varphi}}F\|_{L^q(\mu_{u,\varphi})}.
\]
If the right side is finite, then $u(F\circ\varphi)\in A^q$.
If the left side is finite, then $F\in L^q(\mu_{u,\varphi})$.
Hence one operator is bounded on all of $A^p$ if and only if the other is.
Their operator norms are equal.

For every finite family $F_1,\ldots,F_N\in A^p$, the same identity gives
\begin{equation*}
 \left(\sum_{\nu=1}^N
 \|W_{u,\varphi}F_\nu\|_{A^q}^r\right)^{1/r}
 =
 \left(\sum_{\nu=1}^N
 \|J_{\mu_{u,\varphi}}F_\nu\|_{L^q(\mu_{u,\varphi})}^r
 \right)^{1/r}.
\end{equation*}
The weak $r$-norm on the right side of the summing inequality is the
same for both operators.  Hence
\[
 \pi_r(W_{u,\varphi})
 =\pi_r(J_{\mu_{u,\varphi}}).
\]
Apply \cref{thm:carleson}.  This proves the equivalences and
\eqref{eq:weighted-composition-norm}.  The pointwise formula determines
the extension uniquely.
\end{proof}

\begin{corollary}\label{cor:composition}
Let $1<p<\infty$, $1\le q,r<\infty$, and
$\varphi:\Omega\to\Omega$ be holomorphic.  Put
\[
 \mu_\varphi(E)=v(\varphi^{-1}(E)).
\]
Then
\begin{equation}\label{eq:composition-criterion}
 C_\varphi\in\Piabs_r(A^p,A^q)
 \quad\Longleftrightarrow\quad
 \left\{\frac{v(\varphi^{-1}(\Q_\rho(a_j)))^{1/q}}
 {V_j^{1/p}}\right\}_j
 \in\mathfrak d_{p,q}^{\,r}.
\end{equation}
Moreover,
\begin{equation*}
 \pi_r(C_\varphi)
 \asymp
 \left\|
 \left\{\frac{v(\varphi^{-1}(\Q_\rho(a_j)))^{1/q}}
 {V_j^{1/p}}\right\}_j
 \right\|_{\mathfrak d_{p,q}^{\,r}}.
\end{equation*}
The constants have the same dependence as in
\cref{thm:weighted-composition}.
\end{corollary}

\begin{proof}
Since $\Omega$ is bounded, $1\in A^q(\Omega)$.  For $u=1$,
$\mu_{u,\varphi}=\mu_\varphi$.  Apply
\cref{thm:weighted-composition}.
\end{proof}

The pullback formula is essential.  No regularity of
$\varphi^{-1}(\Q_\rho(a_j))$ is required.  For example, if
$\varphi\equiv b\in\Omega$, then
\[
 C_\varphi F=F(b)\mathbf1,
 \qquad
 \mu_\varphi=v(\Omega)\delta_b.
\]
Here $\delta_b$ is the unit point mass at $b$.
Only finitely many coefficients in \eqref{eq:composition-criterion} are
nonzero.  Thus $C_\varphi$ is absolutely $r$-summing for every
$r\ge1$, as also follows from its rank-one form.

Weighted composition operators are closed under operator composition.
The corresponding measure depends on both weights and both self-maps.

\begin{theorem}
\label{thm:weighted-composition-products}
Let $1<p<\infty$, $1\le q,r<\infty$.  Let
$u,w\in\Hol(\Omega)$ and let
$\varphi,\psi:\Omega\to\Omega$ be holomorphic.  Put
\begin{equation*}
 U=w(u\circ\psi),
 \qquad
 \Phi=\varphi\circ\psi.
\end{equation*}
Assume $U\in A^q(\Omega)$.  Then
\begin{equation}\label{eq:weighted-composition-product-identity}
 W_{w,\psi}W_{u,\varphi}=W_{U,\Phi}
\end{equation}
on $\Hol(\Omega)$, and
\begin{equation}\label{eq:weighted-composition-product-criterion}
 W_{w,\psi}W_{u,\varphi}\in\Piabs_r(A^p,A^q)
 \quad\Longleftrightarrow\quad
 \left\{
 \frac{\mu_{U,\Phi}(\Q_\rho(a_j))^{1/q}}
 {V_j^{1/p}}
 \right\}_j\in\mathfrak d_{p,q}^{\,r}.
\end{equation}
The measure is
\begin{equation}\label{eq:weighted-composition-product-measure}
 \mu_{U,\Phi}(E)
 =
 \int_{(\varphi\circ\psi)^{-1}(E)}
 |w(z)|^q|u(\psi(z))|^q\dd v(z).
\end{equation}
Here the left side of
\eqref{eq:weighted-composition-product-identity} is first defined
algebraically on $\Hol(\Omega)$ and then restricted to
$\mathcal P(\Omega)$.  Since $U\in A^q$, this restriction takes values
in $A^q$.  Membership in $\Piabs_r(A^p,A^q)$ means that the restriction
has an absolutely $r$-summing extension to $A^p$.  Under
\eqref{eq:weighted-composition-product-criterion}, the extension is
unique and is given by $F\mapsto U(F\circ\Phi)$ for every $F\in A^p$.
Moreover,
\begin{equation}\label{eq:weighted-composition-product-norm}
 \pi_r(W_{w,\psi}W_{u,\varphi})
 \asymp
 \left\|
 \left\{
 \frac{\mu_{U,\Phi}(\Q_\rho(a_j))^{1/q}}
 {V_j^{1/p}}
 \right\}_j
 \right\|_{\mathfrak d_{p,q}^{\,r}}.
\end{equation}

For $m\in\N$, define
\begin{equation*}
 \varphi^{\circ0}=\mathrm{id}_\Omega,
 \qquad
 u_{[m]}=\prod_{k=0}^{m-1}u\circ\varphi^{\circ k}.
\end{equation*}
If $u_{[m]}\in A^q(\Omega)$, then
\begin{equation}\label{eq:weighted-composition-iterate-identity}
 W_{u,\varphi}^{\,m}
 =W_{u_{[m]},\varphi^{\circ m}},
\end{equation}
and
\begin{equation}\label{eq:weighted-composition-iterate-criterion}
 W_{u,\varphi}^{\,m}\in\Piabs_r(A^p,A^q)
 \quad\Longleftrightarrow\quad
 \left\{
 \frac{\mu_{u_{[m]},\varphi^{\circ m}}
 (\Q_\rho(a_j))^{1/q}}
 {V_j^{1/p}}
 \right\}_j\in\mathfrak d_{p,q}^{\,r}.
\end{equation}
The same convention applies to the power: its algebraic formula is
restricted to $\mathcal P(\Omega)$.  Under
\eqref{eq:weighted-composition-iterate-criterion}, this restriction has
a unique absolutely $r$-summing extension from $A^p$ to $A^q$, and
\begin{equation}\label{eq:weighted-composition-iterate-norm}
 \pi_r(W_{u,\varphi}^{\,m})
 \asymp
 \left\|
 \left\{
 \frac{\mu_{u_{[m]},\varphi^{\circ m}}(\Q_\rho(a_j))^{1/q}}
 {V_j^{1/p}}
 \right\}_j
 \right\|_{\mathfrak d_{p,q}^{\,r}}.
\end{equation}
All implicit constants depend only on $p,q,r$, the finite-type data,
and the fixed geometric parameters.
\end{theorem}

\begin{proof}
For $F\in\Hol(\Omega)$,
\begin{align*}
 W_{w,\psi}W_{u,\varphi}F(z)
 &=w(z)\bigl(W_{u,\varphi}F\bigr)(\psi(z))\\
 &=w(z)u(\psi(z))F(\varphi(\psi(z)))\\
 &=U(z)F(\Phi(z)).
\end{align*}
This proves \eqref{eq:weighted-composition-product-identity}.  For
$F\in\mathcal P(\Omega)$, the function $F$ is bounded on $\Omega$.  Since
$U\in A^q$, we have $U(F\circ\Phi)\in A^q$.  Taking the $A^q$ norm gives
\begin{align*}
 \|W_{w,\psi}W_{u,\varphi}F\|_{A^q}^q
 &=
 \int_\Omega
 |F(\varphi(\psi(z)))|^q
 |w(z)|^q|u(\psi(z))|^q\dd v(z)\\
 &=\int_\Omega|F(\xi)|^q\dd\mu_{U,\Phi}(\xi).
\end{align*}
This proves \eqref{eq:weighted-composition-product-measure}.  Apply the
same identity to $F_1,\ldots,F_N\in\mathcal P(\Omega)$.  Then
\begin{align*}
 &\left(\sum_{\nu=1}^N
 \|W_{w,\psi}W_{u,\varphi}F_\nu\|_{A^q}^r\right)^{1/r}\\
 &\hspace{25mm}=
 \left(\sum_{\nu=1}^N
 \|J_{\mu_{U,\Phi}}F_\nu\|_{L^q(\mu_{U,\Phi})}^r\right)^{1/r}.
\end{align*}
The weak $r$-norm of $(F_\nu)$ is the same on both sides.  Thus the
two initial maps satisfy the same summing inequalities.  The case of
a family with one element gives boundedness.  Density of
$\mathcal P(\Omega)$ in $A^p$ then gives unique bounded extensions.
For a finite family $(F_\nu)\subset A^p$, choose polynomials
$P_{\nu,k}\to F_\nu$ in $A^p$.  The estimate
\[
 w_r(P_{1,k}-F_1,\ldots,P_{N,k}-F_N)
 \le\left(\sum_{\nu=1}^N
 \|P_{\nu,k}-F_\nu\|_{A^p}^r\right)^{1/r}
 \longrightarrow0
\]
passes the summing inequality to the extensions with the same constant.

We also identify the extension into $L^q(\mu_{U,\Phi})$.
If $P_k\to F$ in $A^p$ and the extended inclusion sends $F$ to $h$,
then $P_k\to h$ in $L^q(\mu_{U,\Phi})$.  A subsequence converges to
$h$ almost everywhere for this measure.  On the other hand,
$P_k\to F$ locally uniformly on $\Omega$, hence at every point.
Therefore $h=F$ almost everywhere.  This argument also applies when
the pullback measure is singular or has atoms.  Hence
\[
 \pi_r(W_{w,\psi}W_{u,\varphi})
 =\pi_r(J_{\mu_{U,\Phi}}).
\]
Now use \cref{thm:carleson}.  We obtain
\eqref{eq:weighted-composition-product-criterion} and
\eqref{eq:weighted-composition-product-norm}.  To identify the
extension, choose $F_k\in\mathcal P(\Omega)$ with $F_k\to F$ in $A^p$.
Then $F_k\to F$ locally uniformly, and
\[
 U(F_k\circ\Phi)\longrightarrow U(F\circ\Phi)
\]
locally uniformly.  The bounded extension sends $F_k$ to a sequence
converging in $A^q$, hence locally uniformly, to its value at $F$.
These two limits agree.  Thus the extension has the stated
pointwise formula.

Formula \eqref{eq:weighted-composition-iterate-identity} follows by
induction.  The induction step is
\begin{align*}
 W_{u,\varphi}^{\,m+1}F(z)
 &=u(z)u_{[m]}(\varphi(z))
 F(\varphi^{\circ(m+1)}(z))\\
 &=\left(\prod_{k=0}^{m}
 u(\varphi^{\circ k}(z))\right)
 F(\varphi^{\circ(m+1)}(z)).
\end{align*}
Apply \cref{thm:weighted-composition} with the weight $u_{[m]}$ and the
self-map $\varphi^{\circ m}$.  This gives
\eqref{eq:weighted-composition-iterate-criterion} and
\eqref{eq:weighted-composition-iterate-norm}.
\end{proof}

\begin{corollary}
\label{cor:weighted-product-ideal}
Let $1<p,s,q<\infty$ and $1\le r<\infty$.
If
\[
 W_{u,\varphi}:A^p\to A^s
 \quad\text{is absolutely }r\text{-summing}
\]
and
\[
 W_{w,\psi}:A^s\to A^q
 \quad\text{is bounded},
\]
then
\begin{equation}\label{eq:weighted-product-ideal-bound}
 W_{w,\psi}W_{u,\varphi}\in\Piabs_r(A^p,A^q),
 \qquad
 \pi_r(W_{w,\psi}W_{u,\varphi})
 \le
 \|W_{w,\psi}\|\pi_r(W_{u,\varphi}).
\end{equation}
If $W_{u,\varphi}:A^p\to A^s$ is bounded and
$W_{w,\psi}:A^s\to A^q$ is absolutely $r$-summing, then
\begin{equation}\label{eq:weighted-product-ideal-bound-reverse}
 \pi_r(W_{w,\psi}W_{u,\varphi})
 \le
 \pi_r(W_{w,\psi})\|W_{u,\varphi}\|.
\end{equation}
\end{corollary}

\begin{proof}
Put $A=W_{u,\varphi}$ and $B=W_{w,\psi}$.  If $A\in\Piabs_r$ and $B$ is
bounded, the ideal property gives
\[
 \pi_r(BA)\le\|B\|\pi_r(A).
\]
If $A$ is bounded and $B\in\Piabs_r$, it gives
\[
 \pi_r(BA)\le\pi_r(B)\|A\|.
\]
These are \eqref{eq:weighted-product-ideal-bound} and
\eqref{eq:weighted-product-ideal-bound-reverse}.  In either case,
$U=w(u\circ\psi)=(BA)\mathbf1$ belongs to $A^q$.  Thus
\cref{thm:weighted-composition-products} also gives the scalar
criterion for $BA$.
\end{proof}

\subsection{Continuous pullback criteria}

For a positive locally finite Borel measure $\mu$, define
\begin{equation}\label{eq:continuous-Carleson-density}
 \mathcal C_{p,q,\rho}\mu(z)
 =
 \frac{\mu(\Q_\rho(z))^{1/q}}{\V_\rho(z)^{1/p}},
 \qquad z\in\Omega.
\end{equation}
This is a nonisotropic averaging density.  It is not the
Radon--Nikodym derivative of $\mu$.
At the logarithmic borderline, $L^{q^-}(d\lambda_\Omega)$ denotes the
Orlicz space defined by $\Psi_q$ in \eqref{eq:qminus}.  Its norm is
\begin{equation*}
 \|F\|_{L^{q^-}(d\lambda_\Omega)}
 =\inf\left\{t>0:
 \int_\Omega\Psi_q\left(\frac{|F(z)|}{t}\right)
 \dd\lambda_\Omega(z)\le1\right\}.
\end{equation*}

\begin{lemma}
\label{lem:measure-discrete-continuous}
Let $1<p<\infty$ and $1\le q,r<\infty$.  Let $\mu$ be a positive
locally finite Borel measure on $\Omega$.  If
$\mathfrak d_{p,q}^{\,r}=\ell^\kappa$, then
\begin{equation}\label{eq:measure-discrete-continuous}
 \|\mathcal C_{p,q,\rho}\mu\|_{L^\kappa(d\lambda_\Omega)}
 \asymp
 \left\|
 \left\{\frac{\mu(\Q_\rho(a_j))^{1/q}}{V_j^{1/p}}\right\}_j
 \right\|_{\ell^\kappa}.
\end{equation}
At the logarithmic borderline
$2<p<\infty$, $q=p'<2$, and $1\le r<q$,
\begin{equation}\label{eq:measure-discrete-continuous-Orlicz}
 \|\mathcal C_{p,q,\rho}\mu\|_{L^{q^-}(d\lambda_\Omega)}
 \asymp
 \left\|
 \left\{\frac{\mu(\Q_\rho(a_j))^{1/q}}{V_j^{1/p}}\right\}_j
 \right\|_{\ell^{q^-}}.
\end{equation}
The constants are independent of $\mu$.
\end{lemma}

\begin{proof}
Let $(E_j)$ be the cells in \eqref{eq:partition-cells}.  If $z\in E_j$,
engulfing and volume comparison give
\begin{equation*}
 \mathcal C_{p,q,\rho}\mu(z)
 \lesssim
 \frac{\mu(\Q_{C_1\rho}(a_j))^{1/q}}{V_j^{1/p}}
 =:\beta_j^{C_1}(\mu).
\end{equation*}
On each cell, $\V_\rho(z)\asymp V_j$ and $v(E_j)\asymp V_j$.
Thus,
\[
 \lambda_\Omega(E_j)
 =\int_{E_j}\frac{\dd v(z)}{\V_\rho(z)}\asymp1,
 \qquad
 \int_{E_j}|\mathcal C_{p,q,\rho}\mu(z)|^\kappa
 \dd\lambda_\Omega(z)
 \lesssim|\beta_j^{C_1}(\mu)|^\kappa.
\]
Summing over the disjoint cells gives
\begin{equation}\label{eq:measure-continuous-upper-norm}
 \|\mathcal C_{p,q,\rho}\mu\|_{L^\kappa(d\lambda_\Omega)}
 \lesssim
 \|\{\beta_j^{C_1}(\mu)\}\|_{\ell^\kappa}.
\end{equation}
The finite-neighbor estimate \eqref{eq:beta-radius} gives
\begin{equation*}
 \|\{\beta_j^{C_1}(\mu)\}\|_{\ell^\kappa}
 \lesssim
 \left\|\left\{
 \frac{\mu(\Q_\rho(a_j))^{1/q}}{V_j^{1/p}}
 \right\}_j\right\|_{\ell^\kappa}.
\end{equation*}

We prove the reverse estimate with a finer lattice.  Choose fixed
$0<\rho_1<\rho_2<\rho$ such that
\begin{equation*}
 z\in\Q_{\rho_2}(b)
 \quad\Longrightarrow\quad
 \Q_{\rho_1}(b)\subset\Q_\rho(z).
\end{equation*}
Let $\Lambda'=\{b_k\}$ be an admissible $\rho_1$-lattice.  Put
\[
 V_k'=v(\Q_{\rho_1}(b_k)),
 \qquad
 \beta_k'(\mu)=
 \frac{\mu(\Q_{\rho_1}(b_k))^{1/q}}{(V_k')^{1/p}}.
\]
If $z\in\Q_{\rho_2}(b_k)$, engulfing and volume comparison give
\begin{equation*}
 \beta_k'(\mu)
 \lesssim\mathcal C_{p,q,\rho}\mu(z).
\end{equation*}
The balls $\Q_{\rho_2}(b_k)$ have bounded overlap and
\[
 \lambda_\Omega(\Q_{\rho_2}(b_k))\asymp1.
\]
Thus, with an overlap bound $N_0$,
\begin{align*}
 \sum_k|\beta_k'(\mu)|^\kappa
 &\lesssim\sum_k\int_{\Q_{\rho_2}(b_k)}
 |\mathcal C_{p,q,\rho}\mu(z)|^\kappa\dd\lambda_\Omega(z)\\
 &\le N_0\int_\Omega
 |\mathcal C_{p,q,\rho}\mu(z)|^\kappa\dd\lambda_\Omega(z).
\end{align*}
Taking the $\kappa$th root gives
\begin{equation*}
 \|\{\beta_k'(\mu)\}\|_{\ell^\kappa}
 \lesssim
 \|\mathcal C_{p,q,\rho}\mu\|_{L^\kappa(d\lambda_\Omega)}.
\end{equation*}
Apply \cref{thm:carleson} first with $(\rho,\Lambda)$ and then with
$(\rho_1,\Lambda')$.  In this proof, set $\pi_r(J_\mu)=+\infty$ if the
inclusion has no absolutely $r$-summing extension.  Since
$\mathfrak d_{p,q}^{\,r}=\ell^\kappa$ with equivalent norms,
\begin{equation}\label{eq:measure-two-lattice-comparison}
 \left\|\left\{
 \frac{\mu(\Q_\rho(a_j))^{1/q}}{V_j^{1/p}}
 \right\}_j\right\|_{\ell^\kappa}
 \asymp \pi_r(J_\mu)
 \asymp \|\{\beta_k'(\mu)\}\|_{\ell^\kappa}.
\end{equation}
Equations \eqref{eq:measure-continuous-upper-norm} through \eqref{eq:measure-two-lattice-comparison} prove
\eqref{eq:measure-discrete-continuous}.

At the logarithmic borderline, apply the increasing function
$\Psi_q$ to the same pointwise estimates.  There are fixed constants
$A,B\ge1$ such that, for every $t>0$,
\begin{align*}
 \int_\Omega\Psi_q\left(
 \frac{\mathcal C_{p,q,\rho}\mu(z)}{At}\right)
 \dd\lambda_\Omega(z)
 &\le B\sum_j\Psi_q\left(\frac{\beta_j^{C_1}(\mu)}t\right),\\
 \sum_k\Psi_q\left(\frac{\beta_k'(\mu)}{At}\right)
 &\le B\int_\Omega\Psi_q\left(
 \frac{\mathcal C_{p,q,\rho}\mu(z)}t\right)
 \dd\lambda_\Omega(z).
\end{align*}
Here the first estimate uses the cell measures.  The second uses the
lower measure bound and finite overlap of the balls in the finer lattice.
Convexity and $\Psi_q(0)=0$ give
\[
 \Psi_q(s/B)\le B^{-1}\Psi_q(s)\qquad(s\ge0).
\]
Thus a modular at most $1$ on either right side gives a modular at most
$1$ on the corresponding left side after replacing $At$ by $ABt$.
Taking infima over $t$ proves
\begin{align*}
 \|\mathcal C_{p,q,\rho}\mu\|_{L^{q^-}(d\lambda_\Omega)}
 &\lesssim
 \|\{\beta_j^{C_1}(\mu)\}\|_{\ell^{q^-}},\\
 \|\{\beta_k'(\mu)\}\|_{\ell^{q^-}}
 &\lesssim
 \|\mathcal C_{p,q,\rho}\mu\|_{L^{q^-}(d\lambda_\Omega)}.
\end{align*}
Now use \eqref{eq:beta-radius}, \cref{lem:finite-neighbor}, and
\cref{thm:carleson} for the two lattices.  Since
$\mathfrak d_{p,q}^{\,r}=\ell^{q^-}$ with equivalent norms, we obtain
\eqref{eq:measure-discrete-continuous-Orlicz}.
\end{proof}

\begin{corollary}
\label{cor:weighted-composition-continuous}
Let the hypotheses of \cref{thm:weighted-composition} hold.
If $\mathfrak d_{p,q}^{\,r}=\ell^\kappa$ in
\eqref{eq:kappa-table}, then
\begin{equation*}
 W_{u,\varphi}\in\Piabs_r(A^p,A^q)
 \quad\Longleftrightarrow\quad
 \mathcal C_{p,q,\rho}\mu_{u,\varphi}
 \in L^\kappa(\Omega,d\lambda_\Omega),
\end{equation*}
and
\[
 \pi_r(W_{u,\varphi})
 \asymp
 \|\mathcal C_{p,q,\rho}\mu_{u,\varphi}
 \|_{L^\kappa(d\lambda_\Omega)}.
\]
At the logarithmic borderline,
\begin{align*}
 W_{u,\varphi}\in\Piabs_r(A^p,A^q)
 &\quad\Longleftrightarrow\quad
 \mathcal C_{p,q,\rho}\mu_{u,\varphi}
 \in L^{q^-}(d\lambda_\Omega),\\
 \pi_r(W_{u,\varphi})
 &\asymp
 \|\mathcal C_{p,q,\rho}\mu_{u,\varphi}
 \|_{L^{q^-}(d\lambda_\Omega)}.
\end{align*}
For $q=1$, let $s_p$ be defined by
\eqref{eq:intro-q1-multiplication}.  Then
\begin{align*}
 W_{u,\varphi}\in\Piabs_r(A^p,A^1)
 &\quad\Longleftrightarrow\quad
 \mathcal C_{p,1,\rho}\mu_{u,\varphi}
 \in L^{s_p}(d\lambda_\Omega),\\
 \pi_r(W_{u,\varphi})
 &\asymp
 \|\mathcal C_{p,1,\rho}\mu_{u,\varphi}
 \|_{L^{s_p}(d\lambda_\Omega)}.
\end{align*}
All constants are independent of $u$ and $\varphi$.
\end{corollary}

\begin{proof}
Combine \cref{thm:weighted-composition} with
\cref{lem:measure-discrete-continuous}.  For $q=1$, use
\eqref{eq:q1-diagonal}.
\end{proof}

We next give a class of weighted composition operators that belong to
every summing ideal.
Recall the nuclear norm $\nu_1$ defined in the introduction. Every nuclear
operator is absolutely $1$-summing and
\begin{equation}\label{eq:nuclear-to-summing}
 \pi_1(T)\le\nu_1(T).
\end{equation}

\begin{theorem}
\label{thm:composition-dilation-nuclear}
Let $\Omega=\mathbb B_n$, $0\le s<1$, and
\[
 \varphi_s(z)=sz.
\]
If $1<p<\infty$, $1\le q<\infty$, and $u\in A^q(\mathbb B_n)$, then
\[
 W_{u,\varphi_s}:A^p(\mathbb B_n)\longrightarrow A^q(\mathbb B_n)
\]
is nuclear.  In particular,
\begin{equation}\label{eq:composition-dilation-all-r}
 W_{u,\varphi_s}\in\Piabs_r(A^p,A^q)
 \qquad(1\le r<\infty).
\end{equation}
More precisely,
\begin{equation}\label{eq:composition-dilation-nuclear-bound}
 \nu_1(W_{u,\varphi_s})
 \le C_{n,p,q}\|u\|_{A^q}
 \sum_{m=0}^\infty
 \binom{n+m-1}{m}(m+1)^{(n+1)/p}s^m.
\end{equation}
\end{theorem}

\begin{proof}
Write the homogeneous expansion
\[
 F=\sum_{m=0}^\infty F_m,
 \qquad
 F_m(e^{it}z)=e^{imt}F_m(z).
\]
The $m$th homogeneous projection is
\begin{equation}\label{eq:homogeneous-projection}
 Q_mF(z)
 =F_m(z)
 =\frac1{2\pi}\int_0^{2\pi}
 F(e^{it}z)e^{-imt}\dd t.
\end{equation}
Rotations preserve Lebesgue volume.  Minkowski's inequality gives
\begin{equation}\label{eq:homogeneous-projection-bound}
 \|Q_mF\|_{A^p}\le\|F\|_{A^p}.
\end{equation}

We need a polynomial bound for the supremum norm of $F_m$.  Put
\[
 r_m=1-\frac1{m+2}.
\]
Since
\[
 r_m^{-m}
 =\left(1+\frac1{m+1}\right)^m\le e,
\]
homogeneity gives
\begin{equation}\label{eq:homogeneous-boundary-reduction}
 \|F_m\|_{L^\infty(\mathbb B_n)}
 \le e\sup_{|z|=r_m}|F_m(z)|.
\end{equation}
The estimate for point evaluations on $A^p(\mathbb B_n)$ is
\[
 |G(z)|
 \le C_{n,p}(1-|z|^2)^{-(n+1)/p}\|G\|_{A^p}.
\]
Apply it to $G=F_m$.  Since
$1-r_m^2\asymp(m+1)^{-1}$, equations
\eqref{eq:homogeneous-projection-bound} and
\eqref{eq:homogeneous-boundary-reduction} yield
\begin{equation}\label{eq:homogeneous-supremum}
 \|F_m\|_\infty
 \le C_{n,p}(m+1)^{(n+1)/p}\|F\|_{A^p}.
\end{equation}

Define
\[
 T_mF=uF_m.
\]
Then
\begin{equation}\label{eq:composition-homogeneous-operator}
 \|T_mF\|_{A^q}
 \le\|u\|_{A^q}\|F_m\|_\infty
 \le C_{n,p}(m+1)^{(n+1)/p}
 \|u\|_{A^q}\|F\|_{A^p}.
\end{equation}
The range of $T_m$ is contained in
\[
 u\mathcal P_m,
\]
where $\mathcal P_m$ is the space of homogeneous polynomials of degree
$m$.  Hence
\begin{equation*}
 \operatorname{rank}T_m
 \le\dim\mathcal P_m
 =\binom{n+m-1}{m}.
\end{equation*}
If a bounded operator $T$ has rank at most $N$, then
\begin{equation}\label{eq:finite-rank-nuclear-bound}
 \nu_1(T)\le N\|T\|.
\end{equation}
Indeed, let $d=\operatorname{rank}T\le N$. Choose an Auerbach basis
$(y_k)_{k=1}^d$ of the range, with coordinate functionals
$(y_k^*)_{k=1}^d$ satisfying $y_k^*(y_\ell)=\delta_{k\ell}$ and
$\|y_k\|=\|y_k^*\|=1$. This gives
\[
 T=\sum_{k=1}^d(y_k^*\circ T)\otimes y_k,
 \qquad
 \|y_k^*\circ T\|\le\|T\|,
 \qquad
 \|y_k\|=1.
\]
Combining \eqref{eq:composition-homogeneous-operator} through \eqref{eq:finite-rank-nuclear-bound} gives
\begin{equation}\label{eq:homogeneous-nuclear-bound}
 \nu_1(T_m)
 \le C_{n,p}
 \binom{n+m-1}{m}(m+1)^{(n+1)/p}\|u\|_{A^q}.
\end{equation}

Since $F_m(sz)=s^mF_m(z)$,
\begin{equation}\label{eq:composition-nuclear-series}
 W_{u,\varphi_s}F
 =uF(s\,\cdot)
 =\sum_{m=0}^\infty s^mT_mF.
\end{equation}
The series of the nuclear norms converges because $s<1$ and the
remaining factors grow polynomially.  To see that it defines a nuclear
operator, choose an Auerbach representation for each $T_m$ as above:
\[
 T_m=\sum_{k=1}^{N_m}x_{m,k}^*\otimes y_{m,k},
 \qquad
 \sum_{k=1}^{N_m}\|x_{m,k}^*\|\,\|y_{m,k}\|
 \le N_m\|T_m\|,
 \quad N_m=\operatorname{rank}T_m.
\]
The double series
$\sum_{m,k}s^m x_{m,k}^*\otimes y_{m,k}$ has a finite sum of the
products of the two norms.  It therefore converges in operator norm
and is a nuclear representation of its limit.  The pointwise
limit below shows that this operator is $W_{u,\varphi_s}$.
Summing
\eqref{eq:homogeneous-nuclear-bound} proves
\eqref{eq:composition-dilation-nuclear-bound}.  More explicitly,
\[
 \sum_{m=0}^{\infty}s^m\nu_1(T_m)
 \le C_{n,p,q}\|u\|_{A^q}
 \sum_{m=0}^{\infty}
 \binom{n+m-1}{m}(m+1)^{(n+1)/p}s^m<\infty.
\]
The partial sums in \eqref{eq:composition-nuclear-series} converge in
$A^q$.  Also,
\[
 \sum_{m=0}^Ns^mF_m(z)\longrightarrow F(sz)
\]
locally uniformly on $\mathbb B_n$.  Since $u$ is locally bounded, the
products converge locally uniformly to $u(z)F(sz)$.  Thus the $A^q$
limit is the pointwise weighted composition function.
Finally, \eqref{eq:nuclear-to-summing} and the inclusion theorem for
absolutely summing ideals give \eqref{eq:composition-dilation-all-r}.
\end{proof}

Strict dilation cannot be replaced by an automorphism when the source
and target exponents agree.

\begin{proposition}\label{prop:automorphism-obstruction}
Let $1<p<\infty$ and let $\varphi$ be an automorphism of
$\mathbb B_n$.  Then
\begin{equation}\label{eq:automorphism-not-summing}
 C_\varphi\notin\Piabs_r(A^p(\mathbb B_n),A^p(\mathbb B_n))
 \qquad(1\le r<\infty).
\end{equation}
More generally, let
\[
 H^\infty(\mathbb B_n)
 =\Hol(\mathbb B_n)\cap L^\infty(\mathbb B_n),
\]
let $u\in H^\infty(\mathbb B_n)$, and assume
\[
 \inf_{z\in\mathbb B_n}|u(z)|>0.
\]
Then
\begin{equation*}
 W_{u,\varphi}\notin\Piabs_r(A^p,A^p).
\end{equation*}
\end{proposition}

\begin{proof}
For a fixed ball automorphism, its real Jacobian is bounded above and
below by positive constants.  Indeed, if $a=\varphi^{-1}(0)$, then
\[
 |J_{\mathbb R}\varphi(z)|
 \asymp
 \frac{(1-|a|^2)^{n+1}}{|1-\langle z,a\rangle|^{2n+2}},
\]
and
\[
 1-|a|\le|1-\langle z,a\rangle|\le1+|a|.
\]
Change of variables gives
\begin{equation*}
 \|F\circ\varphi\|_{A^p}^p
 =\int_{\mathbb B_n}|F(w)|^p
 |J_{\mathbb R}\varphi^{-1}(w)|\dd v(w)
 \asymp\|F\|_{A^p}^p.
\end{equation*}
Thus $C_\varphi$ is a bounded isomorphism, with inverse
$C_{\varphi^{-1}}$.

If $u$ is bounded above and away from zero, multiplication by $u$ is a
bounded isomorphism of $A^p$.  Indeed,
\[
 \|M_u\|\le\|u\|_\infty,
 \qquad
 \|M_u^{-1}\|=\|M_{1/u}\|
 \le\left(\inf_{\mathbb B_n}|u|\right)^{-1}.
\]
Hence $W_{u,\varphi}=M_uC_\varphi$ is a bounded
isomorphism.  Its inverse is
\[
 W_{v,\varphi^{-1}},
 \qquad
 v(z)=\frac1{u(\varphi^{-1}(z))}.
\]
Indeed,
\[
 W_{v,\varphi^{-1}}W_{u,\varphi}=I_{A^p}
 =W_{u,\varphi}W_{v,\varphi^{-1}}.
\]
Suppose that $W_{u,\varphi}$ is absolutely $r$-summing.
By \cref{lem:summing-calculus}\textup{(iv)}, it maps weakly convergent
sequences to norm-convergent sequences.  The space $A^p$ is a closed
subspace of the reflexive space $L^p$, so it is reflexive.
Every bounded sequence in $A^p$ therefore has a weakly convergent
subsequence.  Its image under $W_{u,\varphi}$ converges in norm.  Hence
$W_{u,\varphi}$ is compact.
The identity
\[
 I_{A^p}
 =W_{v,\varphi^{-1}}W_{u,\varphi}
\]
would then be compact.  This is impossible because $A^p(\mathbb B_n)$
is infinite dimensional.  Take $u=1$ to obtain
\eqref{eq:automorphism-not-summing}.
\end{proof}

\section{Volterra operators}\label{sec:volterra}

The Volterra results require a radial Littlewood--Paley formula on the
convex domain. We first prove this formula. The derivative identity
then reduces the problem to \cref{thm:carleson} for a weighted
pullback measure. Both steps
remain valid at $q=1$.

For the one-variable Bergman theory, see
\cite{AlemanSiskakis1997}. Absolute summability on the disc is studied
in \cite{JreisLefevre2024}; related Fock-space results appear in
\cite{ChenHeWang2026}.

\subsection{A Littlewood--Paley formula}

Fix $a\in\Omega$.  For $F\in\Hol(\Omega)$, define
\begin{equation*}
 \R_aF(z)
 =\sum_{m=1}^n(z_m-a_m)\frac{\partial F}{\partial z_m}(z).
\end{equation*}
The field $\R_a$ is the radial derivative with center $a$.  The disc
case is classical \cite{AlemanSiskakis1997}.  We give a proof on
$\Omega$.  The proof also covers $q=1$.

\begin{lemma}
\label{lem:radial-LP}
Let $1\le q<\infty$ and $F\in\Hol(\Omega)$.  Then
\begin{equation}\label{eq:radial-LP-membership}
 F\in A^q(\Omega)
 \quad\Longleftrightarrow\quad
 \delta\R_aF\in L^q(\Omega,v).
\end{equation}
In this case, there is a constant $C=C(\Omega,a,q)\ge1$ such that
\begin{equation}\label{eq:radial-LP}
 C^{-1}\|F\|_{A^q}^q
 \le
 |F(a)|^q+\int_\Omega
 \delta(z)^q|\R_aF(z)|^q\dd v(z)
 \le C\|F\|_{A^q}^q.
\end{equation}
\end{lemma}

\begin{proof}
Put $d=2n$.  Thus $d$ is the real dimension of $\Omega$.
Let $\sigma$ be surface measure on
$\partial\Omega$.  The dot denotes the Euclidean inner product on
$\mathbb R^d$.  Put
\[
 D_\Omega=\operatorname{diam}(\Omega)
 =\sup\{|z-w|:z,w\in\Omega\}.
\]
For $x\in\mathbb R^d$ and $s>0$, let
\[
 B(x,s)=\{y\in\mathbb R^d:|y-x|<s\}.
\]
Choose $\varepsilon_a>0$ such that
\[
 B(a,\varepsilon_a)\subset\Omega.
\]
The supporting hyperplane at $\zeta$ does not meet this ball.  Hence
\begin{equation}\label{eq:radial-angle}
 0<\varepsilon_a
 \le(\zeta-a)\mathbin{\cdot}\nu_\Omega(\zeta)
 \le D_\Omega
 \qquad(\zeta\in\partial\Omega).
\end{equation}
Every $z\ne a$ has a unique representation
\[
 z=\Phi(t,\zeta):=a+t(\zeta-a),
 \qquad 0<t<1,\quad \zeta\in\partial\Omega.
\]
The Jacobian of $\Phi$ is
\begin{equation}\label{eq:radial-Jacobian}
 \dd v(\Phi(t,\zeta))
 =t^{d-1}(\zeta-a)\mathbin{\cdot}\nu_\Omega(\zeta)
 \dd t\,\dd\sigma(\zeta).
\end{equation}
Indeed, each tangent vector is multiplied by $t$.  The normal factor is
$(\zeta-a)\mathbin{\cdot}\nu_\Omega(\zeta)$.

Put $x=\Phi(t,\zeta)$.  Convexity gives
\[
 B(x,(1-t)\varepsilon_a)
 =t\zeta+(1-t)B(a,\varepsilon_a)
 \subset\Omega.
\]
Also,
\[
 \delta(x)\le |x-\zeta|
 =(1-t)|\zeta-a|
 \le(1-t)D_\Omega.
\]
Hence
\begin{equation}\label{eq:radial-distance}
 \varepsilon_a(1-t)
 \le\delta(\Phi(t,\zeta))
 \le D_\Omega(1-t).
\end{equation}

Put $F_\zeta(t)=F(\Phi(t,\zeta))$.  Then
\begin{equation*}
 tF_\zeta'(t)=\R_aF(\Phi(t,\zeta)).
\end{equation*}
For $1/2\le t<1$,
\[
 F_\zeta(t)
 =F_\zeta(1/2)
 +\int_{1/2}^t
 \R_aF(\Phi(s,\zeta))\frac{\dd s}{s}.
\]
Put
\[
 G_\zeta(t)
 =\int_{1/2}^t
 |\R_aF(\Phi(s,\zeta))|\frac{\dd s}{s}.
\]
Suppose first that $q>1$.  Fix $1/2<R<1$ and set
\[
 A_R=\int_{1/2}^R G_\zeta(t)^q\dd t.
\]
Since $G_\zeta(1/2)=0$, integration by parts gives
\begin{align*}
 A_R
 &=q\int_{1/2}^R
 (1-t)G_\zeta(t)^{q-1}G_\zeta'(t)\dd t
 -(1-R)G_\zeta(R)^q\\
 &\le q A_R^{(q-1)/q}
 \left(\int_{1/2}^R
 (1-t)^qG_\zeta'(t)^q\dd t\right)^{1/q}.
\end{align*}
If $A_R>0$, divide by $A_R^{(q-1)/q}$.  The case $A_R=0$ is
clear.  Let $R\uparrow1$.  Since $1/t\le2$ on $[1/2,1)$,
\begin{equation*}
 \int_{1/2}^1G_\zeta(t)^q\dd t
 \le (2q)^q\int_{1/2}^1
 (1-t)^q|\R_aF(\Phi(t,\zeta))|^q\dd t.
\end{equation*}
For $q=1$, Fubini's theorem gives, for $1/2<R<1$,
\begin{align*}
 \int_{1/2}^R G_\zeta(t)\dd t
 &=\int_{1/2}^R
 \frac{R-s}{s}|\R_aF(\Phi(s,\zeta))|\dd s\notag\\
 &\le2\int_{1/2}^R
 (1-s)|\R_aF(\Phi(s,\zeta))|\dd s.
\end{align*}
Letting $R\uparrow1$ proves the same estimate for $q=1$.  It follows
from these two estimates and
$|F_\zeta(t)|\le |F_\zeta(1/2)|+G_\zeta(t)$ that
\begin{align}
 \int_{1/2}^1|F_\zeta(t)|^qt^{d-1}\dd t
 &\lesssim |F_\zeta(1/2)|^q \notag\\
 &\quad+
 \int_{1/2}^1
 (1-t)^q|\R_aF(\Phi(t,\zeta))|^qt^{d-1}\dd t.
 \label{eq:radial-Hardy}
\end{align}
Here $2^{1-d}\le t^{d-1}\le1$ on $[1/2,1]$.

Let
\[
 \Omega_s=a+s(\Omega-a),\qquad 0<s<1.
\]
For $h\in\Hol(\Omega)$ with $h(a)=0$, define
\[
 \mathcal J_ah(z)=\int_0^1h(a+t(z-a))\frac{\dd t}{t},
\]
where the quotient has a removable singularity at $t=0$.  We first prove
\begin{equation}\label{eq:Ja-interior-bound}
 \|\mathcal J_ah\|_{L^q(\Omega_{2/3})}
 \lesssim\|h\|_{L^q(\Omega_{3/4})}.
\end{equation}
Choose $0<t_0<1/4$ so small that
\[
 a+t(\Omega_{2/3}-a)\subset B(a,\varepsilon_a/2)
 \qquad(0<t<t_0).
\]
Since $B(a,3\varepsilon_a/4)\subset\Omega_{3/4}$, Cauchy's estimate
and the submean inequality give
\[
 \sup_{B(a,\varepsilon_a/2)}|\nabla h|
 \lesssim\|h\|_{L^q(\Omega_{3/4})},
 \qquad
 |\nabla h|=\left(\sum_{k=1}^n
 |\partial h/\partial z_k|^2\right)^{1/2}.
\]
Use $h(a)=0$ and integrate along the line segment.  For
$z\in\Omega_{2/3}$ and $0<t<t_0$,
\begin{align*}
 h(a+t(z-a))
 &=t\int_0^1\sum_{k=1}^n(z_k-a_k)
 \frac{\partial h}{\partial z_k}(a+st(z-a))\dd s,\\
 |h(a+t(z-a))|
 &\le tD_\Omega
 \sup_{B(a,\varepsilon_a/2)}|\nabla h|
 \lesssim t\|h\|_{L^q(\Omega_{3/4})}.
\end{align*}
Therefore
\[
 \int_0^{t_0}
 \frac{|h(a+t(z-a))|}{t}\dd t
 \lesssim\|h\|_{L^q(\Omega_{3/4})}.
\]
For $t_0\le t\le1$, Minkowski's inequality and
$w=a+t(z-a)$ give
\begin{align*}
 \left\|\int_{t_0}^1
 h(a+t(\,\cdot-a))\frac{\dd t}{t}
 \right\|_{L^q(\Omega_{2/3})}
 &\le\int_{t_0}^1t^{-1-d/q}
 \|h\|_{L^q(\Omega_{2t/3})}\dd t\\
 &\lesssim\|h\|_{L^q(\Omega_{3/4})}.
\end{align*}
This proves \eqref{eq:Ja-interior-bound}.

Apply this with $h=\R_aF$.  We have
\begin{equation*}
 \mathcal J_a(\R_aF)(z)
 =\int_0^1\frac{\dd}{\dd t}F(a+t(z-a))\dd t
 =F(z)-F(a),
\end{equation*}
so \eqref{eq:Ja-interior-bound} gives the following volume estimate.
Choose $r_*>0$ such that
\[
 B(x,2r_*)\subset\Omega_{2/3}
 \qquad(x\in\partial\Omega_{1/2}).
\]
The submean inequality gives
\[
 |F(x)-F(a)|^q
 \lesssim r_*^{-d}
 \int_{B(x,r_*)}|F(w)-F(a)|^q\dd v(w).
\]
Put $H=F-F(a)$ and $S=\partial\Omega_{1/2}$, and let $\sigma_S$
be surface measure on $S$. Tonelli's theorem gives
\begin{align*}
 \int_S|H(x)|^q\dd\sigma_S(x)
 &\lesssim r_*^{-d}\int_{\Omega_{2/3}}|H(w)|^q
              \sigma_S(S\cap B(w,r_*))\dd v(w)\\
 &\le r_*^{-d}\sigma_S(S)\|H\|_{L^q(\Omega_{2/3})}^q.
\end{align*}
The surface measure $\sigma_S(S)$ is finite, and $r_*$ is fixed
independently of $F$. Thus, also when $q=1$,
\[
 \|F-F(a)\|_{L^q(\partial\Omega_{1/2})}
 \lesssim
 \|F-F(a)\|_{L^q(\Omega_{2/3})}.
\]
Thus
\begin{equation}\label{eq:interior-radial-Poincare}
 \|F-F(a)\|_{L^q(\Omega_{1/2})}
 +\|F-F(a)\|_{L^q(\partial\Omega_{1/2})}
 \lesssim\|\R_aF\|_{L^q(\Omega_{3/4})}.
\end{equation}
The second norm uses surface measure on $\partial\Omega_{1/2}$.
Under the map
\[
 \partial\Omega\ni\zeta
 \longmapsto a+\tfrac12(\zeta-a)\in\partial\Omega_{1/2},
\]
surface measure is multiplied by $2^{1-d}$.
Also, $\delta\asymp1$ on $\Omega_{3/4}$.  Integrate
\eqref{eq:radial-Hardy} over $\partial\Omega$.  Use
\eqref{eq:radial-angle}, \eqref{eq:radial-Jacobian},
\eqref{eq:radial-distance}, and
\eqref{eq:interior-radial-Poincare}.  This gives
\begin{align*}
 \int_{\Omega\setminus\Omega_{1/2}}|F|^q\dd v
 &\lesssim
 \|F\|_{L^q(\partial\Omega_{1/2})}^q
 +\int_{\Omega\setminus\Omega_{1/2}}
 \delta^q|\R_aF|^q\dd v\\
 &\lesssim |F(a)|^q
 +\int_{\Omega_{3/4}}|\R_aF|^q\dd v
 +\int_{\Omega\setminus\Omega_{1/2}}
 \delta^q|\R_aF|^q\dd v.
\end{align*}
On the inner set,
\[
 \int_{\Omega_{1/2}}|F|^q\dd v
 \lesssim |F(a)|^q
 +\int_{\Omega_{3/4}}|\R_aF|^q\dd v.
\]
Since $\delta\asymp1$ on $\Omega_{3/4}$, we obtain
\begin{equation*}
 \|F\|_{A^q}^q
 \lesssim |F(a)|^q+
 \int_\Omega\delta^q|\R_aF|^q\dd v.
\end{equation*}

The integrals up to $t=1$ follow by monotone convergence from upper
limits $R<1$.  Thus
\[
 \delta\R_aF\in L^q(\Omega)
 \quad\Longrightarrow\quad F\in A^q(\Omega).
\]

Conversely, let $F\in A^q(\Omega)$.  Choose $0<c_0<1/2$ so that
$B(z,c_0\delta(z))\subset\Omega$.  The Cauchy estimate and the
holomorphic submean inequality give
\begin{equation*}
 \delta(z)^q|\R_aF(z)|^q
 \lesssim
 \frac1{v(B(z,c_0\delta(z)))}
 \int_{B(z,c_0\delta(z))}|F(w)|^q\dd v(w).
\end{equation*}
If $w\in B(z,c_0\delta(z))$, then
\[
 (1-c_0)\delta(z)\le\delta(w)\le(1+c_0)\delta(z).
\]
For fixed $w$, all such $z$ lie in $B(w,C\delta(w))$, and
$\delta(z)\asymp\delta(w)$ there.  Therefore
\[
 \int_{\{z:w\in B(z,c_0\delta(z))\}}
 \frac{\dd v(z)}{\delta(z)^d}
 \lesssim\delta(w)^{-d}v(B(w,C\delta(w)))\lesssim1.
\]
Tonelli's theorem now gives
\begin{align*}
 \int_\Omega\delta(z)^q|\R_aF(z)|^q\dd v(z)
 &\lesssim
 \int_\Omega|F(w)|^q
 \int_{\{z:w\in B(z,c_0\delta(z))\}}
 \frac{\dd v(z)}{\delta(z)^d}\dd v(w)\\
 &\lesssim\int_\Omega|F(w)|^q\dd v(w).
\end{align*}
The submean inequality on a fixed ball about $a$ gives
$|F(a)|^q\lesssim\|F\|_{A^q}^q$.  This proves
\eqref{eq:radial-LP-membership} and \eqref{eq:radial-LP}.
\end{proof}

\subsection{Volterra composition operators}

Let $g\in\Hol(\Omega)$ and let
$\varphi:\Omega\to\Omega$ be holomorphic.  Define
\begin{equation}\label{eq:Volterra-composition-definition}
 V_{g,\varphi}F(z)
 =\int_0^1F(\varphi(a+t(z-a)))
 \R_ag(a+t(z-a))\frac{\dd t}{t}.
\end{equation}
The factor $1/t$ cancels:
\begin{equation}\label{eq:Volterra-cancelled-integrand}
 \frac{\R_ag(a+t(z-a))}{t}
 =\sum_{k=1}^n(z_k-a_k)
 \frac{\partial g}{\partial z_k}(a+t(z-a)).
\end{equation}
For a compact set $K\subset\Omega$, put
\[
 K_a=\{a+t(z-a):z\in K,\ 0\le t\le1\}.
\]
Convexity implies that $K_a$ is a compact subset of $\Omega$.
Also, $\varphi(K_a)$ is a compact subset of $\Omega$.
The integral in \eqref{eq:Volterra-composition-definition} therefore
converges uniformly on compact sets and defines a holomorphic function.  More
precisely, \eqref{eq:Volterra-cancelled-integrand} gives
\begin{equation}\label{eq:Volterra-compact-bound}
 \sup_{z\in K}|V_{g,\varphi}F(z)|
 \le D_\Omega
 \sup_{w\in\varphi(K_a)}|F(w)|
 \sup_{w\in K_a}|\nabla g(w)|.
\end{equation}
This also shows continuity for local uniform convergence of $F$, or
of $g$ and its first derivatives.  Put
\[
 H=(F\circ\varphi)\R_ag.
\]
Then $H(a)=0$.  For $0<s<1$,
\begin{align*}
 V_{g,\varphi}F(a+s(z-a))
 &=\int_0^1H(a+ts(z-a))\frac{\dd t}{t}\\
 &=\int_0^sH(a+\tau(z-a))\frac{\dd\tau}{\tau}.
\end{align*}
Differentiate in $s$ and multiply by $s$.  This gives
\begin{equation}\label{eq:Volterra-radial-identity}
 V_{g,\varphi}F(a)=0,
 \qquad
 \R_a(V_{g,\varphi}F)
 =(F\circ\varphi)\R_ag.
\end{equation}
The ordinary Volterra operator is
$V_g=V_{g,\mathrm{id}_\Omega}$.
The natural domain of $V_{g,\varphi}$ from $A^p$ to $A^q$ is
\begin{equation*}
 \mathcal D(V_{g,\varphi})
 =\{F\in A^p(\Omega):V_{g,\varphi}F\in A^q(\Omega)\}.
\end{equation*}
As above, the notation $V_{g,\varphi}:A^p\to A^q$ means that this
domain is all of $A^p$ and that the operator is bounded.

For $g\in A^q(\Omega)$, define
\begin{equation}\label{eq:Volterra-pullback}
 \mu_{g,\varphi}
 =\varphi_*\bigl(\delta^q|\R_ag|^qv\bigr),
\end{equation}
or, equivalently,
\[
 \mu_{g,\varphi}(E)
 =\int_{\varphi^{-1}(E)}
 \delta(z)^q|\R_ag(z)|^q\dd v(z).
\]
This measure is finite.  Indeed, the right inequality in
\eqref{eq:radial-LP} applied to $g$ gives
\[
 \mu_{g,\varphi}(\Omega)
 =\int_\Omega\delta^q|\R_ag|^q\dd v
 \lesssim\|g\|_{A^q}^q.
\]
The assumption $g\in A^q$ is also necessary for boundedness, since
\[
 V_{g,\varphi}\mathbf1=g-g(a).
\]

\begin{theorem}
\label{thm:Volterra-composition}
Let $1<p<\infty$, $1\le q,r<\infty$,
$g\in A^q(\Omega)$, and
$\varphi:\Omega\to\Omega$ be holomorphic.  The following conditions are
equivalent:
\begin{enumerate}[label=\textup{(\roman*)}]
\item $V_{g,\varphi}\in\Piabs_r(A^p(\Omega),A^q(\Omega))$;
\item $J_{\mu_{g,\varphi}}\in
\Piabs_r(A^p(\Omega),L^q(\mu_{g,\varphi}))$;
\item
\begin{equation}\label{eq:Volterra-sequence}
 \left\{
 \frac{\mu_{g,\varphi}(\Q_\rho(a_j))^{1/q}}
 {V_j^{1/p}}
 \right\}_j
 \in\mathfrak d_{p,q}^{\,r}.
\end{equation}
\end{enumerate}
Moreover,
\begin{equation}\label{eq:Volterra-norm}
 \pi_r(V_{g,\varphi})
 \asymp
 \left\|
 \left\{\frac{\mu_{g,\varphi}(\Q_\rho(a_j))^{1/q}}
 {V_j^{1/p}}\right\}_j
 \right\|_{\mathfrak d_{p,q}^{\,r}}.
\end{equation}
The constants may depend on $\Omega,a,p,q,r,\rho$.  They do not depend
on $g$ or $\varphi$.
\end{theorem}

\begin{proof}
By \eqref{eq:Volterra-radial-identity} and
\cref{lem:radial-LP},
\begin{equation*}
 V_{g,\varphi}F\in A^q(\Omega)
 \quad\Longleftrightarrow\quad
 F\in L^q(\Omega,\mu_{g,\varphi}).
\end{equation*}
Indeed, the value of $V_{g,\varphi}F$ at $a$ is zero.  The corresponding norm estimate is
\begin{align}
 \|V_{g,\varphi}F\|_{A^q}^q
 &\asymp
 \int_\Omega
 |F(\varphi(z))|^q
 \delta(z)^q|\R_ag(z)|^q\dd v(z)\notag\\
 &=\int_\Omega|F(w)|^q\dd\mu_{g,\varphi}(w)
 =\|J_{\mu_{g,\varphi}}F\|_{L^q(\mu_{g,\varphi})}^q.
 \label{eq:Volterra-Carleson-norm}
\end{align}
Thus, $V_{g,\varphi}$ is defined and bounded on all of $A^p$
if and only if $J_{\mu_{g,\varphi}}$ is.  The pointwise formulas
determine both operators uniquely.

Let $F_1,\ldots,F_N\in A^p$.  Equation
\eqref{eq:Volterra-Carleson-norm} gives
\begin{align*}
 C^{-1/q}
 \left(\sum_{\nu=1}^N
 \|J_{\mu_{g,\varphi}}F_\nu\|_{L^q(\mu_{g,\varphi})}^r\right)^{1/r}
 &\le
 \left(\sum_{\nu=1}^N
 \|V_{g,\varphi}F_\nu\|_{A^q}^r\right)^{1/r}\\
 &\le C^{1/q}
 \left(\sum_{\nu=1}^N
 \|J_{\mu_{g,\varphi}}F_\nu\|_{L^q(\mu_{g,\varphi})}^r\right)^{1/r}.
\end{align*}
The weak $r$-norm of $(F_\nu)$ is the same for both maps.  Thus
\[
 \pi_r(V_{g,\varphi})
 \asymp\pi_r(J_{\mu_{g,\varphi}}).
\]
The claim now follows from \cref{thm:carleson}.  The argument uses no
projection and remains valid when $q=1$.
\end{proof}

\begin{corollary}\label{cor:Volterra}
Let $1<p<\infty$, $1\le q,r<\infty$, and $g\in A^q(\Omega)$.  Define
\[
 \dd\mu_g(z)=\delta(z)^q|\R_ag(z)|^q\dd v(z).
\]
Then
\begin{equation}\label{eq:Volterra-criterion}
 V_g\in\Piabs_r(A^p,A^q)
 \quad\Longleftrightarrow\quad
 \left\{
 V_j^{1/q-1/p}
 \left(\avg_{\Q_\rho(a_j)}
 \delta(z)^q|\R_ag(z)|^q\dd v(z)\right)^{1/q}
 \right\}_j
 \in\mathfrak d_{p,q}^{\,r},
\end{equation}
and
\begin{equation}\label{eq:Volterra-criterion-norm}
 \pi_r(V_g)
 \asymp
 \left\|
 \left\{
 V_j^{1/q-1/p}
 \left(\avg_{\Q_\rho(a_j)}
 \delta(z)^q|\R_ag(z)|^q\dd v(z)\right)^{1/q}
 \right\}_j
 \right\|_{\mathfrak d_{p,q}^{\,r}}.
\end{equation}
The constants depend only on $\Omega,a,p,q,r$ and the fixed geometric
parameters.  They do not depend on $g$.
\end{corollary}

\begin{proof}
Take $\varphi=\mathrm{id}_\Omega$ in
\cref{thm:Volterra-composition}.  Since
\[
 \frac{\mu_g(\Q_\rho(a_j))^{1/q}}{V_j^{1/p}}
 =
 V_j^{1/q-1/p}
 \left(\avg_{\Q_\rho(a_j)}
 \delta^q|\R_ag|^q\dd v\right)^{1/q},
\]
substitution into \eqref{eq:Volterra-sequence} and
\eqref{eq:Volterra-norm} proves
\eqref{eq:Volterra-criterion} and
\eqref{eq:Volterra-criterion-norm}.
\end{proof}

If $\varphi\equiv b\in\Omega$, then
\begin{equation*}
 V_{g,\varphi}F=F(b)(g-g(a)),
 \qquad
 \mu_{g,\varphi}
 =\left(\int_\Omega\delta^q|\R_ag|^q\dd v\right)\delta_b.
\end{equation*}
Thus $V_{g,\varphi}$ has rank at most one.  The sequence in
\eqref{eq:Volterra-sequence} has finite support.  This agrees with
\cref{thm:Volterra-composition}, since every operator of finite rank is
absolutely $r$-summing.

There is also an exact formula for placing a weighted composition
operator before a Volterra composition operator.

\begin{theorem}\label{thm:mixed-Volterra-product}
Let $1<p<\infty$, $1\le q,r<\infty$.  Let
$g\in A^q(\Omega)$ and $u\in\Hol(\Omega)$.  Let
$\varphi,\psi:\Omega\to\Omega$ be holomorphic.  Define
\begin{equation}\label{eq:mixed-generator}
 \widetilde g
 =V_{g,\varphi}u,
 \qquad
 \widetilde g(z)
 =\int_0^1u(\varphi(a+t(z-a)))
 \R_ag(a+t(z-a))\frac{\dd t}{t}.
\end{equation}
Assume $\widetilde g\in A^q(\Omega)$.  Then
\begin{equation}\label{eq:mixed-operator-identity}
 V_{g,\varphi}W_{u,\psi}
 =V_{\widetilde g,\,\psi\circ\varphi}
\end{equation}
on $\Hol(\Omega)$.  The following conditions are equivalent:
\begin{enumerate}[label=\textup{(\roman*)}]
\item the common operator in \eqref{eq:mixed-operator-identity},
restricted to $\mathcal P(\Omega)$, has an extension in
$\Piabs_r(A^p(\Omega),A^q(\Omega))$;
\item
\begin{equation*}
 \left\{
 \frac{\mu_{\mathrm{mix}}(\Q_\rho(a_j))^{1/q}}
 {V_j^{1/p}}
 \right\}_j\in\mathfrak d_{p,q}^{\,r},
\end{equation*}
\end{enumerate}
where
\begin{equation}\label{eq:mixed-pullback-measure}
 \mu_{\mathrm{mix}}
 =
 (\psi\circ\varphi)_*
 \left(
 \delta^q|u\circ\varphi|^q|\R_ag|^qv
 \right).
\end{equation}
Under these conditions, the extension is unique.  For every
$F\in A^p(\Omega)$, its value is given by either side of
\eqref{eq:mixed-operator-identity}.  Moreover,
\begin{equation}\label{eq:mixed-product-norm}
 \pi_r(V_{g,\varphi}W_{u,\psi})
 =\pi_r(V_{\widetilde g,\,\psi\circ\varphi})
 \asymp
 \left\|
 \left\{
 \frac{\mu_{\mathrm{mix}}(\Q_\rho(a_j))^{1/q}}
 {V_j^{1/p}}
 \right\}_j
 \right\|_{\mathfrak d_{p,q}^{\,r}}.
\end{equation}
The constants depend only on $\Omega,a,p,q,r,\rho$ and the fixed
geometric parameters.  They do not depend on $g,u,\varphi$, or $\psi$.
No change is required when $q=1$.
\end{theorem}

\begin{proof}
The radial identity applied to \eqref{eq:mixed-generator} gives
\begin{equation}\label{eq:mixed-generator-radial}
 \widetilde g(a)=0,
 \qquad
 \R_a\widetilde g=(u\circ\varphi)\R_ag.
\end{equation}
For $F\in\Hol(\Omega)$, apply
\eqref{eq:Volterra-radial-identity} twice:
\begin{align*}
 \R_a\left(V_{g,\varphi}W_{u,\psi}F\right)
 &=
 \bigl((W_{u,\psi}F)\circ\varphi\bigr)\R_ag\notag\\
 &=
 (u\circ\varphi)
 (F\circ\psi\circ\varphi)\R_ag\notag\\
 &=
 (F\circ\psi\circ\varphi)\R_a\widetilde g\notag\\
 &=
 \R_a\left(V_{\widetilde g,\,\psi\circ\varphi}F\right).
\end{align*}
Both sides vanish at $a$.  If a holomorphic function $H$ satisfies
$H(a)=0$ and $\R_aH=0$, then
\[
 \frac{\dd}{\dd t}H(a+t(z-a))
 =\frac1t\R_aH(a+t(z-a))=0.
\]
Hence $H(z)=H(a)=0$.  Apply this uniqueness argument to the
difference of the two sides of
\eqref{eq:mixed-operator-identity}.  This proves the operator identity.

By \eqref{eq:mixed-generator-radial}, the pullback measure associated
with the operator on the right side is
\begin{align*}
 \mu_{\widetilde g,\,\psi\circ\varphi}
 &=(\psi\circ\varphi)_*
 \left(\delta^q|\R_a\widetilde g|^qv\right)\\
 &=(\psi\circ\varphi)_*
 \left(
 \delta^q|u\circ\varphi|^q|\R_ag|^qv
 \right)
 =\mu_{\mathrm{mix}}.
\end{align*}
This is a finite measure.  In fact,
\[
 \mu_{\mathrm{mix}}(\Omega)
 =\int_\Omega\delta^q|\R_a\widetilde g|^q\dd v
 \lesssim\|\widetilde g\|_{A^q}^q
\]
by \cref{lem:radial-LP}.

For $F\in\mathcal P(\Omega)$, the norm estimate
\eqref{eq:Volterra-Carleson-norm} gives
\[
 \|V_{\widetilde g,\,\psi\circ\varphi}F\|_{A^q}^q
 \lesssim
 \|F\|_{L^\infty(\Omega)}^q\mu_{\mathrm{mix}}(\Omega)<\infty.
\]
Thus the initial operator in \textup{(i)} takes values in $A^q$.

Suppose first that the restriction in \textup{(i)} has a bounded
extension $S:A^p\to A^q$.  Choose $F_k\in\mathcal P(\Omega)$ with
$F_k\to F$ in $A^p$.  Estimates for point evaluations give local uniform
convergence $F_k\to F$.  For a compact $K\subset\Omega$, apply
\eqref{eq:Volterra-compact-bound} to $F_k-F$ and the fixed generator
$\widetilde g$.  This gives
\begin{align*}
 &\sup_K\left|
 V_{\widetilde g,\,\psi\circ\varphi}F_k
 -V_{\widetilde g,\,\psi\circ\varphi}F\right|\\
 &\qquad\le D_\Omega
 \sup_{(\psi\circ\varphi)(K_a)}|F_k-F|
 \sup_{K_a}|\nabla\widetilde g|
 \longrightarrow0.
\end{align*}
On the other hand, $SF_k\to SF$ in $A^q$, hence
locally uniformly.  The identity on $\mathcal P(\Omega)$ gives
\[
 SF=V_{\widetilde g,\,\psi\circ\varphi}F.
\]
Thus $S$ is the pointwise operator on all of $A^p$.

Now apply \cref{thm:Volterra-composition} to
$(\widetilde g,\psi\circ\varphi)$.  It proves the equivalence and
\eqref{eq:mixed-product-norm}.  Conversely, the scalar condition gives
the required extension by that theorem.  Since $\mathcal P(\Omega)$ is
dense in $A^p$, this extension is unique.  The same proof applies when
$q=1$.
\end{proof}

\begin{corollary}
\label{cor:mixed-product-ideal}
Let $1<p,s<\infty$ and $1\le q,r<\infty$.  Let
$g,u\in\Hol(\Omega)$, and let
$\varphi,\psi:\Omega\to\Omega$ be holomorphic.  If
\[
 W_{u,\psi}:A^p\to A^s
 \quad\text{is bounded}
\]
and
\[
 V_{g,\varphi}:A^s\to A^q
 \quad\text{is absolutely }r\text{-summing},
\]
then
\begin{equation}\label{eq:mixed-product-ideal-bound}
 V_{g,\varphi}W_{u,\psi}\in\Piabs_r(A^p,A^q),
 \qquad
 \pi_r(V_{g,\varphi}W_{u,\psi})
 \le
 \pi_r(V_{g,\varphi})\|W_{u,\psi}\|.
\end{equation}
Moreover, $g\in A^q$, $u\in A^s$, and
$\widetilde g=V_{g,\varphi}u\in A^q$.  Hence the measure in
\eqref{eq:mixed-pullback-measure} is finite and
\begin{equation}\label{eq:mixed-product-ideal-sequence}
 \left\{
 \frac{\mu_{\mathrm{mix}}(\Q_\rho(a_j))^{1/q}}
 {V_j^{1/p}}
 \right\}_j\in\mathfrak d_{p,q}^{\,r}.
\end{equation}
It also satisfies
\begin{equation}\label{eq:mixed-product-ideal-sequence-bound}
 \left\|
 \left\{
 \frac{\mu_{\mathrm{mix}}(\Q_\rho(a_j))^{1/q}}
 {V_j^{1/p}}
 \right\}_j
 \right\|_{\mathfrak d_{p,q}^{\,r}}
 \lesssim
 \pi_r(V_{g,\varphi})\|W_{u,\psi}\|.
\end{equation}
\end{corollary}

\begin{proof}
The ideal property gives \eqref{eq:mixed-product-ideal-bound}.  Since
$\mathbf1\in A^p\cap A^s$,
\[
 u=W_{u,\psi}\mathbf1\in A^s,
 \qquad
 g-g(a)=V_{g,\varphi}\mathbf1\in A^q.
\]
The domain is bounded.  Thus constant functions belong to $A^q$, and
$g\in A^q$.  Also,
\[
 \widetilde g=V_{g,\varphi}u\in A^q.
\]
All hypotheses of \cref{thm:mixed-Volterra-product} now hold.
Equations \eqref{eq:mixed-product-ideal-sequence} and
\eqref{eq:mixed-product-ideal-sequence-bound} follow from
\eqref{eq:mixed-product-norm} and
\eqref{eq:mixed-product-ideal-bound}.
\end{proof}

\begin{corollary}
\label{cor:Volterra-continuous}
Let the hypotheses of \cref{thm:Volterra-composition} hold.
If $\mathfrak d_{p,q}^{\,r}=\ell^\kappa$ in
\eqref{eq:kappa-table}, then
\begin{equation*}
 V_{g,\varphi}\in\Piabs_r(A^p,A^q)
 \quad\Longleftrightarrow\quad
 \mathcal C_{p,q,\rho}\mu_{g,\varphi}
 \in L^\kappa(\Omega,d\lambda_\Omega),
\end{equation*}
and
\begin{equation*}
 \pi_r(V_{g,\varphi})
 \asymp
 \|\mathcal C_{p,q,\rho}\mu_{g,\varphi}
 \|_{L^\kappa(d\lambda_\Omega)}.
\end{equation*}
At the logarithmic borderline, the power space is replaced by the
Orlicz space.  If $2<p<\infty$, $q=p'<2$, and $1\le r<q$, then
\begin{align*}
 V_{g,\varphi}\in\Piabs_r(A^p,A^q)
 &\Longleftrightarrow
 \mathcal C_{p,q,\rho}\mu_{g,\varphi}
 \in L^{q^-}(\Omega,d\lambda_\Omega),
 \\
 \pi_r(V_{g,\varphi})
 &\asymp
 \|\mathcal C_{p,q,\rho}\mu_{g,\varphi}
 \|_{L^{q^-}(d\lambda_\Omega)}.
\end{align*}
For $q=1$, let $s_p$ be defined by
\eqref{eq:intro-q1-multiplication}.  Then
\begin{align*}
 V_{g,\varphi}\in\Piabs_r(A^p,A^1)
 &\Longleftrightarrow
 \mathcal C_{p,1,\rho}\mu_{g,\varphi}
 \in L^{s_p}(\Omega,d\lambda_\Omega),
 \\
 \pi_r(V_{g,\varphi})
 &\asymp
 \|\mathcal C_{p,1,\rho}\mu_{g,\varphi}
 \|_{L^{s_p}(d\lambda_\Omega)}.
\end{align*}
All constants are independent of $g$ and $\varphi$.
\end{corollary}

\begin{proof}
Apply \cref{lem:measure-discrete-continuous} to
$\mu=\mu_{g,\varphi}$.  Then use
\cref{thm:Volterra-composition}.  The first two lines of
\eqref{eq:kappa-table} give
$\mathfrak d_{p,1}^{\,r}=\ell^{s_p}$ for every $r$.
The Orlicz case follows from
\eqref{eq:measure-discrete-continuous-Orlicz}.
\end{proof}

\begin{theorem}
\label{thm:Volterra-dilation-nuclear}
Let $\Omega=\mathbb B_n$, let the base point be $a=0$, and let
$\varphi_s(z)=sz$ with $0\le s<1$.  If
$1<p<\infty$, $1\le q<\infty$, and $g\in A^q(\mathbb B_n)$, then
\[
 V_{g,\varphi_s}:A^p(\mathbb B_n)\longrightarrow A^q(\mathbb B_n)
\]
is nuclear.  In particular,
\begin{equation}\label{eq:Volterra-dilation-all-r}
 V_{g,\varphi_s}\in\Piabs_r(A^p,A^q)
 \qquad(1\le r<\infty).
\end{equation}
Moreover,
\begin{equation}\label{eq:Volterra-dilation-nuclear-bound}
 \pi_r(V_{g,\varphi_s})
 \le\nu_1(V_{g,\varphi_s})
 \le C_{n,p,q}\|g\|_{A^q}
 \sum_{m=0}^\infty
 \binom{n+m-1}{m}(m+1)^{(n+1)/p}s^m.
\end{equation}
\end{theorem}

\begin{proof}
Use the homogeneous projections $Q_m$ from
\eqref{eq:homogeneous-projection}.  For $F\in A^p$, let
\[
 S_mF=V_{g,\mathrm{id}_{\mathbb B_n}}(Q_mF).
\]
The radial identity gives
\begin{equation}\label{eq:Volterra-homogeneous-radial}
 S_mF(0)=0,
 \qquad
 \R_0(S_mF)=(Q_mF)\R_0g.
\end{equation}
By \cref{lem:radial-LP},
\begin{align*}
 \|S_mF\|_{A^q}^q
 &\lesssim
 \int_{\mathbb B_n}
 \delta(z)^q|Q_mF(z)|^q|\R_0g(z)|^q\dd v(z)\notag\\
 &\le
 \|Q_mF\|_\infty^q
 \int_{\mathbb B_n}\delta(z)^q|\R_0g(z)|^q\dd v(z)\notag\\
 &\lesssim
 \|Q_mF\|_\infty^q\|g\|_{A^q}^q.
\end{align*}
Equation \eqref{eq:homogeneous-supremum} therefore yields
\[
 \|S_m\|
 \le C_{n,p,q}(m+1)^{(n+1)/p}\|g\|_{A^q}.
\]
The operator $S_m$ depends only on $Q_mF$.  Thus
\[
 \operatorname{rank}S_m
 \le\dim\mathcal P_m
 =\binom{n+m-1}{m}.
\]
By \eqref{eq:finite-rank-nuclear-bound},
\begin{equation}\label{eq:Volterra-homogeneous-nuclear}
 \nu_1(S_m)
 \le C_{n,p,q}
 \binom{n+m-1}{m}(m+1)^{(n+1)/p}\|g\|_{A^q}.
\end{equation}

For $Q_mF=F_m$, homogeneity gives
\[
 F_m(stz)=s^mt^mF_m(z).
\]
Let
\[
 U_NF=\sum_{m=0}^Ns^mS_mF.
\]
Equation \eqref{eq:Volterra-homogeneous-nuclear} gives
\[
\begin{aligned}
 \sum_{m=0}^\infty\nu_1(s^mS_m)
 &\le C_{n,p,q}\|g\|_{A^q}
 \sum_{m=0}^\infty
 \binom{n+m-1}{m}(m+1)^{(n+1)/p}s^m\\
 &<\infty.
\end{aligned}
\]
Thus $(U_N)$ converges in nuclear norm to a nuclear operator $U$.
For every $F\in A^p$, the homogeneous expansion gives
\[
 \sum_{m=0}^Ns^mQ_mF(z)\longrightarrow F(sz)
\]
locally uniformly on $\mathbb B_n$.  By
\eqref{eq:Volterra-homogeneous-radial},
\[
 U_NF(0)=0,
 \qquad
 \R_0(U_NF)
 =\left(\sum_{m=0}^Ns^mQ_mF\right)\R_0g.
\]
Convergence in $A^q$ implies local uniform convergence of the
holomorphic functions and of their derivatives.  Therefore
\[
 UF(0)=0,
 \qquad
 \R_0(UF)=F(s\,\cdot)\R_0g.
\]
The same identities hold for $V_{g,\varphi_s}F$ by
\eqref{eq:Volterra-radial-identity}.  Radial uniqueness gives
$V_{g,\varphi_s}F=UF$ for every $F\in A^p$.  Hence
\begin{equation*}
 V_{g,\varphi_s}
 =U=\sum_{m=0}^\infty s^mS_m
\end{equation*}
in nuclear norm.  Thus,
\[
 V_{g,\varphi_s}F=\sum_{m=0}^\infty s^mS_mF
\]
in $A^q$ for every $F\in A^p$.  This proves nuclearity and the second
inequality in
\eqref{eq:Volterra-dilation-nuclear-bound}.  Finally,
\[
 \pi_r(V_{g,\varphi_s})
 \le\pi_1(V_{g,\varphi_s})
 \le\nu_1(V_{g,\varphi_s})
\]
by \eqref{eq:nuclear-to-summing} and the inclusion theorem.  This gives
\eqref{eq:Volterra-dilation-all-r} and the first inequality in
\eqref{eq:Volterra-dilation-nuclear-bound}.
\end{proof}

\begin{corollary}\label{cor:composition-Volterra-q1}
Let $1<p<\infty$ and $1\le r<\infty$.  Let
$u,g\in A^1(\Omega)$, and let
$\varphi:\Omega\to\Omega$ be holomorphic.  Form the measures in
\eqref{eq:weighted-pullback} and \eqref{eq:Volterra-pullback} with
$q=1$.  Let $s_p$ be given by
\eqref{eq:intro-q1-multiplication}.  Then
\begin{align}
 W_{u,\varphi}\in\Piabs_r(A^p,A^1)
 &\Longleftrightarrow
 \left\{\frac{\mu_{u,\varphi}(\Q_\rho(a_j))}
 {V_j^{1/p}}\right\}_j\in\ell^{s_p},
 \label{eq:composition-q1}\\
 V_{g,\varphi}\in\Piabs_r(A^p,A^1)
 &\Longleftrightarrow
 \left\{\frac{\mu_{g,\varphi}(\Q_\rho(a_j))}
 {V_j^{1/p}}\right\}_j\in\ell^{s_p},
 \label{eq:Volterra-q1}
\end{align}
Moreover,
\begin{align*}
 \pi_r(W_{u,\varphi})
 &\asymp
 \left\|
 \left\{\frac{\mu_{u,\varphi}(\Q_\rho(a_j))}
 {V_j^{1/p}}\right\}_j
 \right\|_{\ell^{s_p}},
 \\
 \pi_r(V_{g,\varphi})
 &\asymp
 \left\|
 \left\{\frac{\mu_{g,\varphi}(\Q_\rho(a_j))}
 {V_j^{1/p}}\right\}_j
 \right\|_{\ell^{s_p}}.
\end{align*}
The constants are independent of $u,g$, and $\varphi$.
\end{corollary}

\begin{proof}
If $1<p\le2$, the first line of \eqref{eq:kappa-table} gives
\[
 \mathfrak d_{p,1}^{\,r}
 =\ell^{(1/p'+1-1/2)^{-1}}
 =\ell^{2p/(3p-2)}.
\]
If $p>2$, its second line gives
\[
 \mathfrak d_{p,1}^{\,r}=\ell^1.
\]
These are exactly the two cases in
\eqref{eq:intro-q1-multiplication}.  Apply
\cref{thm:weighted-composition,thm:Volterra-composition}, including
their norm formulas.
\end{proof}

\subsection{Comparison and differences}

The pullback formulation gives comparison results without pointwise
control of the self-map.

\begin{proposition}\label{prop:pullback-domination}
Let $1<p<\infty$ and $1\le q,r<\infty$.  Let $\mu_1,\mu_2$ be positive
locally finite Borel measures.  Let $C>0$ and suppose
\[
 \mu_1(E)\le C^q\mu_2(E)
 \qquad\text{for every Borel set }E\subset\Omega.
\]
If $J_{\mu_2}\in\Piabs_r(A^p,L^q(\mu_2))$, then
$J_{\mu_1}\in\Piabs_r(A^p,L^q(\mu_1))$, and
\begin{equation}\label{eq:measure-domination-summing}
 \pi_r(J_{\mu_1})\le C\pi_r(J_{\mu_2}).
\end{equation}
Thus:
\begin{enumerate}[label=\textup{(\roman*)}]
\item let $\varphi:\Omega\to\Omega$ be holomorphic and let
$u_1,u_2\in A^q(\Omega)$.  If $|u_1|\le C|u_2|$ on $\Omega$, then
\begin{equation}\label{eq:weight-domination}
 W_{u_2,\varphi}\in\Piabs_r(A^p,A^q)
 \quad\Longrightarrow\quad
 W_{u_1,\varphi}\in\Piabs_r(A^p,A^q),
 \qquad
 \pi_r(W_{u_1,\varphi})
 \le C\pi_r(W_{u_2,\varphi});
\end{equation}
\item let $\varphi:\Omega\to\Omega$ be holomorphic and let
$g_1,g_2\in A^q(\Omega)$.  If
$|\R_ag_1|\le C|\R_ag_2|$ on $\Omega$, then
\begin{equation}\label{eq:Volterra-symbol-domination}
 V_{g_2,\varphi}\in\Piabs_r(A^p,A^q)
 \quad\Longrightarrow\quad
 V_{g_1,\varphi}\in\Piabs_r(A^p,A^q),
 \qquad
 \pi_r(V_{g_1,\varphi})
 \lesssim C\pi_r(V_{g_2,\varphi}).
\end{equation}
\end{enumerate}
\end{proposition}

\begin{proof}
The measure inequality implies $\mu_1\ll\mu_2$.  The identity map
\[
 \iota_{21}:L^q(\mu_2)\longrightarrow L^q(\mu_1),
 \qquad \iota_{21}h=h,
\]
is well defined and
\[
 \|\iota_{21}h\|_{L^q(\mu_1)}^q
 =\int_\Omega|h|^q\dd\mu_1
 \le C^q\int_\Omega|h|^q\dd\mu_2.
\]
Thus $\|\iota_{21}\|\le C$ and
\[
 J_{\mu_1}=\iota_{21}J_{\mu_2}.
\]
The ideal property proves \eqref{eq:measure-domination-summing}.

If $|u_1|\le C|u_2|$, then
\[
 \|W_{u_1,\varphi}F\|_{A^q}^q
 =\int_\Omega|u_1|^q|F\circ\varphi|^q\dd v
 \le C^q\|W_{u_2,\varphi}F\|_{A^q}^q.
\]
Apply this inequality to every member of a finite family in $A^p$.
This proves \eqref{eq:weight-domination} with constant $C$.

For the Volterra symbols,
\begin{align*}
 \mu_{g_1,\varphi}(E)
 &=\int_{\varphi^{-1}(E)}
 \delta^q|\R_ag_1|^q\dd v\\
 &\le C^q\mu_{g_2,\varphi}(E).
\end{align*}
Use \eqref{eq:measure-domination-summing} and
\cref{thm:Volterra-composition}.  This gives
\eqref{eq:Volterra-symbol-domination}.
\end{proof}

For $1<p<\infty$ and $1\le q,r<\infty$, write
\[
 b^{\mathrm C}(u,\varphi)
 =
 \left\{\frac{\mu_{u,\varphi}(\Q_\rho(a_j))^{1/q}}
 {V_j^{1/p}}\right\}_j,
 \qquad
 b^{\mathrm V}(g,\varphi)
 =
 \left\{\frac{\mu_{g,\varphi}(\Q_\rho(a_j))^{1/q}}
 {V_j^{1/p}}\right\}_j.
\]

\begin{theorem}\label{thm:pullback-perturbation}
Let $1<p<\infty$ and $1\le q,r<\infty$.
Fix a holomorphic self-map $\varphi:\Omega\to\Omega$.
For $u_0,u_1\in A^q(\Omega)$,
\begin{align}
 W_{u_1,\varphi}-W_{u_0,\varphi}
 &=W_{u_1-u_0,\varphi},
 \label{eq:composition-difference}\\
 \pi_r(W_{u_1,\varphi}-W_{u_0,\varphi})
 &\asymp
 \|b^{\mathrm C}(u_1-u_0,\varphi)\|_{\mathfrak d_{p,q}^{\,r}}.
 \label{eq:composition-perturbation-norm}
\end{align}
For $g_0,g_1\in A^q(\Omega)$,
\begin{align}
 V_{g_1,\varphi}-V_{g_0,\varphi}
 &=V_{g_1-g_0,\varphi},
 \label{eq:Volterra-difference}\\
 \pi_r(V_{g_1,\varphi}-V_{g_0,\varphi})
 &\asymp
 \|b^{\mathrm V}(g_1-g_0,\varphi)\|_{\mathfrak d_{p,q}^{\,r}}.
 \label{eq:Volterra-perturbation-norm}
\end{align}
The algebraic identities hold on $\Hol(\Omega)$.  In the norm formulas,
the left sides denote the unique extensions of the common formulas
restricted to $\mathcal P(\Omega)$.  We set $\pi_r(T)=+\infty$ if such
an absolutely $r$-summing extension does not exist.  We also set
$\|b\|_{\mathfrak d_{p,q}^{\,r}}=+\infty$ when
$b\notin\mathfrak d_{p,q}^{\,r}$.  Thus
\eqref{eq:composition-perturbation-norm} and
\eqref{eq:Volterra-perturbation-norm} include both the finite and the
infinite cases.  All comparison constants are independent of the
symbols and of $\varphi$.
\end{theorem}

\begin{proof}
For $F\in\Hol(\Omega)$,
\[
 (W_{u_1,\varphi}-W_{u_0,\varphi})F
 =(u_1-u_0)(F\circ\varphi).
\]
This is \eqref{eq:composition-difference}.  Apply
\cref{thm:weighted-composition} with the weight $u_1-u_0$.

Linearity of the radial derivative gives
\[
 \R_a(g_1-g_0)=\R_ag_1-\R_ag_0.
\]
Subtract the two defining integrals in
\eqref{eq:Volterra-composition-definition}.  This proves
\eqref{eq:Volterra-difference}.  Apply
\cref{thm:Volterra-composition} with the symbol $g_1-g_0$.
\end{proof}

\begin{corollary}
\label{cor:pullback-closedness}
Let $1<p<\infty$ and $1\le q,r<\infty$.  Fix a holomorphic map
$\varphi:\Omega\to\Omega$.  Let $(u_m)\subset A^q(\Omega)$.  Assume
\begin{equation}\label{eq:composition-closedness-membership}
 b^{\mathrm C}(u_m,\varphi)\in\mathfrak d_{p,q}^{\,r}
 \qquad(m\in\N)
\end{equation}
and
\begin{equation}\label{eq:composition-closedness-Cauchy}
 \|b^{\mathrm C}(u_m-u_\ell,\varphi)\|_{\mathfrak d_{p,q}^{\,r}}
 \longrightarrow0
 \qquad(m,\ell\to\infty).
\end{equation}
Then there is $u\in A^q(\Omega)$ such that
\[
 b^{\mathrm C}(u,\varphi)\in\mathfrak d_{p,q}^{\,r},
 \qquad
 W_{u,\varphi}\in\Piabs_r(A^p,A^q),
\]
and
\begin{equation}\label{eq:composition-closedness-conclusion}
 \|u_m-u\|_{A^q}\longrightarrow0,
 \qquad
 \pi_r(W_{u_m,\varphi}-W_{u,\varphi})\longrightarrow0.
\end{equation}

Let $(g_m)\subset A^q(\Omega)$.  Assume
\begin{equation*}
 b^{\mathrm V}(g_m,\varphi)\in\mathfrak d_{p,q}^{\,r}
 \qquad(m\in\N)
\end{equation*}
and
\begin{equation*}
 \|b^{\mathrm V}(g_m-g_\ell,\varphi)\|_{\mathfrak d_{p,q}^{\,r}}
 \longrightarrow0
 \qquad(m,\ell\to\infty).
\end{equation*}
Put $\widehat g_m=g_m-g_m(a)$.  Then there is
$g\in A^q(\Omega)$ with $g(a)=0$ such that
\[
 b^{\mathrm V}(g,\varphi)\in\mathfrak d_{p,q}^{\,r},
 \qquad
 V_{g,\varphi}\in\Piabs_r(A^p,A^q).
\]
Moreover,
\begin{equation}\label{eq:Volterra-closedness-conclusion}
 \|\widehat g_m-g\|_{A^q}\longrightarrow0,
 \qquad
 \pi_r(V_{g_m,\varphi}-V_{g,\varphi})\longrightarrow0.
\end{equation}
\end{corollary}

\begin{proof}
By \cref{thm:weighted-composition} and
\eqref{eq:composition-closedness-membership}, every
$W_{u_m,\varphi}$ belongs to $\Piabs_r(A^p,A^q)$.  Equations
\eqref{eq:composition-perturbation-norm} and
\eqref{eq:composition-closedness-Cauchy} show that this sequence is
Cauchy in $\pi_r$.  Moreover,
\begin{align*}
 \|u_m-u_\ell\|_{A^q}
 &=\|(W_{u_m,\varphi}-W_{u_\ell,\varphi})\mathbf1\|_{A^q}\\
 &\le \|\mathbf1\|_{A^p}
 \pi_r(W_{u_m,\varphi}-W_{u_\ell,\varphi})
 \longrightarrow0.
\end{align*}
Hence $u_m\to u$ in $A^q$ for some $u\in A^q$.

The space $\Piabs_r(A^p,A^q)$ is complete in the norm $\pi_r$.
Thus $W_{u_m,\varphi}\to T$ in $\pi_r$ for some
$T\in\Piabs_r(A^p,A^q)$.  For every $F\in A^p$,
\[
 \|W_{u_m,\varphi}F-TF\|_{A^q}
 \le \pi_r(W_{u_m,\varphi}-T)\|F\|_{A^p}
 \longrightarrow0.
\]
The convergence $u_m\to u$ in $A^q$ is locally uniform.  Therefore
\[
 W_{u_m,\varphi}F=u_m(F\circ\varphi)
 \longrightarrow u(F\circ\varphi)=W_{u,\varphi}F
\]
locally uniformly.  Hence $T=W_{u,\varphi}$.  This proves
\eqref{eq:composition-closedness-conclusion}.  The sequence condition
follows from \cref{thm:weighted-composition}.

The Volterra hypotheses and
\eqref{eq:Volterra-perturbation-norm} show that
$(V_{g_m,\varphi})$ is Cauchy in $\pi_r$.  Since
$V_{g_m,\varphi}=V_{\widehat g_m,\varphi}$,
\begin{align*}
 \|\widehat g_m-\widehat g_\ell\|_{A^q}
 &=\|(V_{g_m,\varphi}-V_{g_\ell,\varphi})\mathbf1\|_{A^q}\\
 &\le \|\mathbf1\|_{A^p}
 \pi_r(V_{g_m,\varphi}-V_{g_\ell,\varphi})
 \longrightarrow0.
\end{align*}
Thus $\widehat g_m\to g$ in $A^q$ for some $g\in A^q$.  Point
evaluation at $a$ gives $g(a)=0$.  Completeness gives
$V_{g_m,\varphi}\to S$ in $\pi_r$ for some
$S\in\Piabs_r(A^p,A^q)$.  For each compact $K\subset\Omega$,
Cauchy's estimate on a fixed neighborhood of $K_a$ gives
\[
 \sup_{K_a}|\nabla(\widehat g_m-g)|\longrightarrow0.
\]
For $F\in A^p$, equation \eqref{eq:Volterra-compact-bound} then gives
\begin{align*}
 \sup_K|V_{g_m,\varphi}F-V_{g,\varphi}F|
 &=\sup_K|V_{\widehat g_m-g,\varphi}F|\\
 &\le D_\Omega\sup_{\varphi(K_a)}|F|
 \sup_{K_a}|\nabla(\widehat g_m-g)|\longrightarrow0.
\end{align*}
Convergence in $\pi_r$ also gives $V_{g_m,\varphi}F\to SF$ in $A^q$,
hence locally uniformly.  Therefore $SF=V_{g,\varphi}F$ for every
$F\in A^p$.  Hence $S=V_{g,\varphi}$.  This proves
\eqref{eq:Volterra-closedness-conclusion}.  The sequence condition
follows from \cref{thm:Volterra-composition}.
\end{proof}

We have now treated the five operator classes.  Next, we give
consequences that apply to all five classes.

\section{Explicit criteria, spectral asymptotics, and the endpoint}\label{sec:common}

We collect the parameter forms of the preceding results and give model
examples.  The multiplication and pullback criteria hold also at $q=1$.
The big Hankel criterion holds for every $1<q<\infty$ by
\cref{prop:FTI-all}.  The individual Toeplitz criteria use their stated
symbol assumptions.
For little Hankel operators, the exact criterion concerns the pair
$(h_f,k_f)$, where $k_f=M_f-h_f$ is the auxiliary complementary
operator.  The result for $h_f$ alone is a necessary condition.

\subsection{Explicit parameter form}

Let $1<p<\infty$, $1\le q,r<\infty$, and
$f\in L^q_{\loc}(\Omega)$.  Put $\alpha=1/q-1/p$ and define the
positive locally finite measure
\[
 \mu_f(E)=\int_E|f|^q\dd v.
\]
The local coefficients satisfy
\begin{align*}
 \beta_j(\mu_f)
 &=\frac{\mu_f(\Q_\rho(a_j))^{1/q}}{V_j^{1/p}}
 =V_j^\alpha\M_{q,\rho}f(a_j),
 \\
 \mathcal C_{p,q,\rho}\mu_f(z)
 &=\V_\rho(z)^\alpha\M_{q,\rho}f(z).
\end{align*}
In a power case $\mathfrak d_{p,q}^{\,r}=\ell^\kappa$ of
\eqref{eq:kappa-table}, \cref{lem:measure-discrete-continuous} gives
\begin{equation}\label{eq:discrete-continuous-explicit}
 \|\mathbf M_{p,q,\rho}f\|_{\ell^\kappa}
 \asymp
 \|\V_\rho^{1/q-1/p}\M_{q,\rho}f
 \|_{L^\kappa(\Omega,d\lambda_\Omega)}.
\end{equation}
At $2<p<\infty$, $q=p'<2$, and $1\le r<q$, the same lemma gives
\begin{equation}\label{eq:density-continuous-Orlicz}
 \|\mathbf M_{p,q,\rho}f\|_{\ell^{q^-}}
 \asymp
 \|\V_\rho^{1/q-1/p}\M_{q,\rho}f
 \|_{L^{q^-}(\Omega,d\lambda_\Omega)}.
\end{equation}
Both statements allow infinite norms.  The constants are independent
of $f$.  The proof of \cref{lem:measure-discrete-continuous} compares
the original lattice with an admissible finer lattice.  Each family
covers $\Omega$ at its own radius.  We do not assume that smaller balls
on the original lattice cover $\Omega$.  In particular,
\eqref{eq:discrete-continuous-explicit} is valid also when $q=1$.

For analytic distance, changing the center also changes the domain on
which the holomorphic approximants are defined.  We use the discrete
condition with two fixed scales
\begin{equation*}
 \mathbf G_{p,q,\rho}f
 =\{V_j^\alpha\G_{q,\rho_+}f(a_j)\}_j\in\mathfrak d_{p,q}^{\,r}.
\end{equation*}
The argument using a measure with a density above is used only for local means.

\begin{corollary}\label{cor:explicit}
Let $1<p,q<\infty$, $1\le r<\infty$, and
$f\in\mathcal D_{\Omega,q}$.  Assume
$\mathfrak d_{p,q}^{\,r}=\ell^\kappa$ in
\eqref{eq:kappa-table}.  Then
\begin{equation}\label{eq:explicit-graph}
 \mathscr T_f\in
 \Piabs_r(A^p,A^q\oplus_qL^q)
 \quad\Longleftrightarrow\quad
 \V_\rho^{1/q-1/p}\M_{q,\rho}f
 \in L^\kappa(\Omega,d\lambda_\Omega).
\end{equation}
Moreover,
\begin{equation}\label{eq:explicit-graph-norm}
 \pi_r(\mathscr T_f)
 \asymp
 \|\V_\rho^{1/q-1/p}\M_{q,\rho}f
 \|_{L^\kappa(d\lambda_\Omega)}.
\end{equation}
If $\mathbf G_{p,q,\rho}f\in\ell^\kappa$,
condition \eqref{eq:explicit-graph} is also equivalent to
$T_f\in\Piabs_r(A^p,A^q)$.
In that case,
\begin{equation}\label{eq:explicit-matching-norm}
 \pi_r(T_f)+\|\mathbf G_{p,q,\rho}f\|_{\ell^\kappa}
 \asymp
 \|\V_\rho^{1/q-1/p}\M_{q,\rho}f
 \|_{L^\kappa(d\lambda_\Omega)}.
\end{equation}
\end{corollary}

\begin{proof}
By \eqref{eq:discrete-continuous-explicit},
\[
 \|f\|_{\mathcal L_{p,q}^{r}(\Omega)}
 \asymp
 \|\V_\rho^{1/q-1/p}\M_{q,\rho}f
 \|_{L^\kappa(d\lambda_\Omega)}.
\]
Equation \eqref{eq:intro-main-norm} now gives
\eqref{eq:explicit-graph} and \eqref{eq:explicit-graph-norm}.
By the definition of $\mathcal I_{p,q}^{r}$ and the radius independence
proved in \cref{thm:intro-hankel},
\[
 \|f\|_{\mathcal I_{p,q}^{r}}
 \asymp\|\mathbf G_{p,q,\rho}f\|_{\ell^\kappa}.
\]
Hence \eqref{eq:matching-ida-equivalence} gives the last equivalence.
Equations \eqref{eq:matching-ida-norm} and
\eqref{eq:discrete-continuous-explicit} give
\eqref{eq:explicit-matching-norm}.
\end{proof}

At the logarithmic borderline, define
\begin{equation*}
 \|F\|_{L^{q^-}(d\lambda_\Omega)}
 =\inf\left\{t>0:
 \int_\Omega\Psi_q\left(\frac{|F(z)|}{t}\right)
 \dd\lambda_\Omega(z)\le1\right\}.
\end{equation*}

\begin{corollary}\label{cor:orlicz-borderline}
Let
\begin{equation*}
 2<p<\infty,
 \qquad
 q=p'<2,
 \qquad
 1\le r<q,
 \qquad
 f\in\mathcal D_{\Omega,q}.
\end{equation*}
Then
\begin{equation}\label{eq:orlicz-graph}
 \mathscr T_f\in
 \Piabs_r(A^p,A^q\oplus_qL^q)
 \quad\Longleftrightarrow\quad
 \V_\rho^{1/q-1/p}\M_{q,\rho}f
 \in L^{q^-}(\Omega,d\lambda_\Omega).
\end{equation}
Moreover,
\begin{equation}\label{eq:orlicz-graph-norm}
 \pi_r(\mathscr T_f)
 \asymp
 \|\V_\rho^{1/q-1/p}\M_{q,\rho}f
 \|_{L^{q^-}(d\lambda_\Omega)}.
\end{equation}
If
\begin{equation}\label{eq:orlicz-ida}
 \mathbf G_{p,q,\rho}f\in\ell^{q^-},
\end{equation}
then \eqref{eq:orlicz-graph} is also equivalent to
$T_f\in\Piabs_r(A^p,A^q)$.
In that case,
\begin{equation}\label{eq:orlicz-matching-norm}
 \pi_r(T_f)+\|\mathbf G_{p,q,\rho}f\|_{\ell^{q^-}}
 \asymp
 \|\V_\rho^{1/q-1/p}\M_{q,\rho}f
 \|_{L^{q^-}(d\lambda_\Omega)}.
\end{equation}
\end{corollary}

\begin{proof}
Put
\[
 F=\V_\rho^{1/q-1/p}\M_{q,\rho}f.
\]
Apply \eqref{eq:density-continuous-Orlicz} to obtain
\begin{equation}\label{eq:orlicz-discretization}
 \|F\|_{L^{q^-}(d\lambda_\Omega)}
 \asymp
 \|\mathbf M_{p,q,\rho}f\|_{\ell^{q^-}}.
\end{equation}
The third line of \eqref{eq:kappa-table} identifies
$\mathfrak d_{p,q}^{\,r}$ with $\ell^{q^-}$, with equivalent norms.
Equations \eqref{eq:intro-main-norm} and
\eqref{eq:orlicz-discretization} give
\eqref{eq:orlicz-graph} and \eqref{eq:orlicz-graph-norm}.
Condition \eqref{eq:orlicz-ida} is exactly
$f\in\mathcal I_{p,q}^{r}$.  Apply
\eqref{eq:matching-ida-equivalence}.  The norm statement
\eqref{eq:matching-ida-norm} gives
\eqref{eq:orlicz-matching-norm}.
\end{proof}

\subsection{Parameter substitution and the Hilbert case}

The exponent table in \cref{prop:garling} applies directly to each
sequence criterion proved above.  In a power case,
\[
 \|b\|_{\mathfrak d_{p,q}^{\,r}}
 \asymp\left(\sum_j|b_j|^\kappa\right)^{1/\kappa}.
\]
On the logarithmic line, this expression is replaced by the Luxemburg
norm in \eqref{eq:luxemburg}.  In the exceptional region, take
$\kappa=\kappa_{\mathcal E}$ from
\eqref{eq:exceptional-index}.  These substitutions preserve the
hypotheses and the direction of each implication.  In particular,
substitution into a necessary condition does not give a converse.

For equal source and target exponents,
\[
 \kappa(p,p,r)=
 \begin{cases}
 2,&1<p\le2,\\
 p',&2<p<\infty,\ 1\le r\le p',\\
 r,&2<p<\infty,\ p'<r<p,\\
 p,&2<p<\infty,\ p\le r<\infty.
 \end{cases}
\]
The formulas agree at their common endpoints.  The case $p=q=2$
gives an independent check of the summing exponent.

\begin{proposition}\label{prop:Hilbert-calibration}
Let $H_1,H_2$ be Hilbert spaces and let $1\le r<\infty$.
For $A\in\mathcal L(H_1,H_2)$,
\begin{equation}\label{eq:Hilbert-summing-S2}
 A\in\Piabs_r(H_1,H_2)
 \quad\Longleftrightarrow\quad
 A\in\mathcal S_2(H_1,H_2).
\end{equation}
Moreover,
\begin{equation}\label{eq:Hilbert-summing-norm}
 c_r\|A\|_{\mathcal S_2}
 \le \pi_r(A)
 \le C_r\|A\|_{\mathcal S_2}.
\end{equation}
For $r=2$,
\begin{equation}\label{eq:Hilbert-two-isometry}
 \pi_2(A)=\|A\|_{\mathcal S_2}.
\end{equation}
Consequently,
\begin{equation}\label{eq:d22r}
 \mathfrak d_{2,2}^{\,r}=\ell^2,
 \qquad
 \|b\|_{\mathfrak d_{2,2}^{\,r}}\asymp_r\|b\|_{\ell^2}.
\end{equation}
Here $\mathcal S_2$ denotes the Hilbert--Schmidt class.
\end{proposition}

\begin{proof}
We first prove the comparison for every $1\le r<\infty$.  This is the
theorem for Hilbert spaces of \cite{Garling1970}. Write $\mathbb E$
and $\mathbb P$ for expectation and probability. Let $g$ be a standard
complex Gaussian with $\mathbb E|g|^2=1$, and put
$m_r=(\mathbb E|g|^r)^{1/r}$.  For an operator of finite rank, choose a
singular value expansion
\[
 A_Nx=\sum_{k=1}^N s_k\langle x,v_k\rangle u_k,
 \qquad Z_N=\sum_{k=1}^Ns_kg_kv_k,
\]
where $(u_k)$ and $(v_k)$ are orthonormal and the $g_k$ are independent
copies of $g$.  For every $x\in H_1$,
\[
 \mathbb E|\langle x,Z_N\rangle|^r=m_r^r\|A_Nx\|^r.
\]
Jensen's inequality for $r\le2$ and Minkowski's inequality in
$L^{r/2}$ for $r>2$ give, respectively,
\[
 (\mathbb E\|Z_N\|^r)^{1/r}
 \le \max\{1,m_r\}\Bigl(\sum_{k=1}^Ns_k^2\Bigr)^{1/2}.
\]
Thus, for any finite family $(x_j)$,
\begin{align*}
 \sum_j\|A_Nx_j\|^r
 &=m_r^{-r}\mathbb E\sum_j|\langle x_j,Z_N\rangle|^r\\
 &\le m_r^{-r}w_r(x_j)^r\mathbb E\|Z_N\|^r
 \lesssim_r\|A_N\|_{\mathcal S_2}^r w_r(x_j)^r.
\end{align*}
For $A\in\mathcal S_2$, its singular value truncations converge in
operator norm.  Passing to the limit for each finite family proves
$\pi_r(A)\lesssim_r\|A\|_{\mathcal S_2}$.

Conversely, let $e_1,\ldots,e_N$ be orthonormal in $H_1$ and, for
$\varepsilon\in\{-1,1\}^N$, put
\[
 x_\varepsilon=2^{-N/r}\sum_{k=1}^N\varepsilon_ke_k.
\]
Write $\mathbb E_\varepsilon$ for the average over the $2^N$ sign
choices. The scalar Khintchine inequality gives
\[
 w_r(x_\varepsilon)
 =\sup_{\|h\|\le1}
   \left(\mathbb E_\varepsilon
   \left|\sum_{k=1}^N\varepsilon_k\langle e_k,h\rangle\right|^r
   \right)^{1/r}\le K_r.
\]
The Hilbert-valued Khintchine lower bound gives
\begin{align*}
 \left(\sum_{k=1}^N\|Ae_k\|^2\right)^{1/2}
 &\lesssim_r\left(\mathbb E_\varepsilon
    \left\|\sum_{k=1}^N\varepsilon_kAe_k\right\|^r\right)^{1/r}\\
 &=\left(\sum_\varepsilon\|Ax_\varepsilon\|^r\right)^{1/r}
 \le K_r\pi_r(A).
\end{align*}
The constant in the lower bound is independent of the dimension.
For $r\ge2$, it follows from Jensen's inequality.  For $r<2$, put
$X=\|\sum\varepsilon_kAe_k\|^2$ and $M=\mathbb EX$.
Expansion and cancellation of the odd sign products give
\[
 M=\sum_k\|Ae_k\|^2,
 \qquad \mathbb EX^2\le3M^2.
\]
If $M>0$, then Cauchy--Schwarz gives
\[
 \frac M2\le\mathbb E\bigl(X\mathbf1_{\{X\ge M/2\}}\bigr)
 \le(\mathbb EX^2)^{1/2}\mathbb P(X\ge M/2)^{1/2},
 \qquad \mathbb P(X\ge M/2)\ge\frac1{12}.
\]
Consequently,
\[
 \mathbb EX^{r/2}\ge\frac1{12}(M/2)^{r/2}.
\]
The case $M=0$ is immediate. This proves the required lower bound.
Taking the supremum over finite orthonormal families proves
$\|A\|_{\mathcal S_2}\lesssim_r\pi_r(A)$.  Hence
\begin{equation}\label{eq:Hilbert-Garling-comparison}
 A\in\Piabs_r(H_1,H_2)
 \quad\Longleftrightarrow\quad
 A\in\Piabs_2(H_1,H_2),
 \qquad
 \pi_r(A)\asymp_r\pi_2(A).
\end{equation}
We now compute the exact norm for $r=2$.

First assume $A\in\mathcal S_2(H_1,H_2)$.  Let
$x_1,\ldots,x_N\in H_1$, and let $(y_m)$ be an orthonormal basis of
$H_2$.  Parseval's identity gives
\begin{align*}
 \sum_{j=1}^N\|Ax_j\|_{H_2}^2
 &=\sum_m\sum_{j=1}^N
 |\langle x_j,A^*y_m\rangle|^2\notag\\
 &\le
 w_2(x_1,\ldots,x_N)^2
 \sum_m\|A^*y_m\|_{H_1}^2\notag\\
 &=\|A\|_{\mathcal S_2}^2
 w_2(x_1,\ldots,x_N)^2.
\end{align*}
Hence
\[
 \pi_2(A)\le\|A\|_{\mathcal S_2}.
\]

Conversely, assume $A\in\Piabs_2(H_1,H_2)$.  For every finite
orthonormal family $e_1,\ldots,e_N$ in $H_1$,
\begin{equation*}
 w_2(e_1,\ldots,e_N)
 =\sup_{\|x\|_{H_1}\le1}
 \left(\sum_{j=1}^N|\langle e_j,x\rangle|^2\right)^{1/2}
 =1.
\end{equation*}
Therefore
\[
 \sum_{j=1}^N\|Ae_j\|_{H_2}^2\le\pi_2(A)^2.
\]
Take the supremum over all finite subsets of an orthonormal basis of
$H_1$.  Then
\[
 \|A\|_{\mathcal S_2}^2
 =\sum_j\|Ae_j\|_{H_2}^2
 \le\pi_2(A)^2.
\]
This proves \eqref{eq:Hilbert-two-isometry}.  Combine it with
\eqref{eq:Hilbert-Garling-comparison}.  We obtain
\eqref{eq:Hilbert-summing-S2} and \eqref{eq:Hilbert-summing-norm}.

For the diagonal map $D_b:\ell^2\to\ell^2$,
\[
 \|D_b\|_{\mathcal S_2}^2
 =\sum_j\|D_be_j\|_{\ell^2}^2
 =\sum_j|b_j|^2.
\]
Thus
\[
 D_b\in\Piabs_r(\ell^2,\ell^2)
 \quad\Longleftrightarrow\quad b\in\ell^2.
\]
Equations \eqref{eq:Hilbert-summing-norm} and
\eqref{eq:Hilbert-two-isometry} give \eqref{eq:d22r}.  The diagonal
statement also follows directly from
\cite[Theorem~9(iii)]{Garling1974}.
\end{proof}

\subsection{Model domains}

The preceding criteria give explicit inequalities when the McNeal
volume and the local pullback mass are comparable to powers of the
boundary distance.  The next proposition gives this calculation.

\begin{proposition}\label{prop:boundary-power-test}
Let $1<p<\infty$ and $1\le q,r<\infty$.  Suppose
$\mathfrak d_{p,q}^{\,r}=\ell^\kappa$.
Assume that there are constants $D>1$, $\varepsilon_0>0$, and $C\ge1$
such that
\begin{equation}\label{eq:boundary-volume-power}
 C^{-1}\delta(z)^D
 \le\V_\rho(z)\le C\delta(z)^D,
 \qquad 0<\delta(z)<\varepsilon_0.
\end{equation}
Let $\mu$ be a positive locally finite Borel measure.  Assume that, for
some $\sigma\in\mathbb R$,
\begin{equation}\label{eq:boundary-measure-power}
 C^{-1}\delta(z)^\sigma
 \le\mu(\Q_\rho(z))\le C\delta(z)^\sigma,
 \qquad 0<\delta(z)<\varepsilon_0.
\end{equation}
Then
\begin{equation}\label{eq:boundary-power-threshold}
 J_\mu\in\Piabs_r(A^p,L^q(\mu))
 \quad\Longleftrightarrow\quad
 \kappa\left(\frac{\sigma}{q}-\frac{D}{p}\right)>D-1.
\end{equation}
The same conclusion holds for $W_{u,\varphi}$ or
$V_{g,\varphi}$ when $u,g,\varphi$ satisfy the hypotheses of
\cref{thm:weighted-composition,thm:Volterra-composition} and $\mu$ is
the corresponding pullback measure.

If only the upper estimate in \eqref{eq:boundary-measure-power} holds,
the strict inequality in \eqref{eq:boundary-power-threshold} is
sufficient.  If only the lower estimate holds, it is necessary.
\end{proposition}

\begin{proof}
Equations \eqref{eq:continuous-Carleson-density},
\eqref{eq:boundary-volume-power}, and
\eqref{eq:boundary-measure-power} give
\begin{equation*}
 \mathcal C_{p,q,\rho}\mu(z)
 \asymp
 \delta(z)^{\sigma/q-D/p}.
\end{equation*}
Also,
\begin{equation*}
 \dd\lambda_\Omega(z)
 =\frac{\dd v(z)}{\V_\rho(z)}
 \asymp\delta(z)^{-D}\dd v(z).
\end{equation*}
Since $\partial\Omega$ is smooth, we can decrease $\varepsilon_0$ so
that the normal coordinate map
\[
 (\zeta,t)\longmapsto\zeta-t\nu_\Omega(\zeta)
\]
has a Jacobian bounded above and below for $0<t<\varepsilon_0$.
Therefore
\begin{align}
 \int_{\{\delta<\varepsilon_0\}}
 |\mathcal C_{p,q,\rho}\mu(z)|^\kappa
 \dd\lambda_\Omega(z)
 &\asymp
 \int_{\partial\Omega}\int_0^{\varepsilon_0}
 t^{\kappa(\sigma/q-D/p)-D}\dd t\,\dd\sigma(\zeta)
 \notag\\
 &<\infty
 \quad\Longleftrightarrow\quad
 \kappa\left(\frac{\sigma}{q}-\frac{D}{p}\right)-D>-1.
 \label{eq:boundary-power-integral}
\end{align}
This is exactly \eqref{eq:boundary-power-threshold}.

It remains to check the compact core.  Put
\[
 K_0=\{z\in\Omega:\delta(z)\ge\varepsilon_0\}.
\]
By compactness and the fixed interior convention for McNeal balls,
there is a compact set $K_1\Subset\Omega$ such that
\[
 \Q_\rho(z)\subset K_1,
 \qquad z\in K_0.
\]
Local finiteness gives $\mu(K_1)<\infty$.  Volume comparison on $K_0$
gives
\[
 c_0:=\inf_{z\in K_0}\V_\rho(z)>0.
\]
Consequently,
\[
 \mathcal C_{p,q,\rho}\mu(z)
 \le \frac{\mu(K_1)^{1/q}}{c_0^{1/p}},
 \qquad z\in K_0.
\]
Also, $\lambda_\Omega(K_0)<\infty$.  Thus the integral over $K_0$ is
finite.  Now apply
\cref{thm:carleson,cor:weighted-composition-continuous,cor:Volterra-continuous}.
If only the upper mass bound holds, the same computation gives
sufficiency.  If only the lower mass bound holds, the boundary integral
in \eqref{eq:boundary-power-integral} gives necessity.
\end{proof}

\begin{corollary}
\label{cor:boundary-power-symbol}
Let $1<p<\infty$ and $1\le q,r<\infty$.  Suppose
$\mathfrak d_{p,q}^{\,r}=\ell^\kappa$.
Under \eqref{eq:boundary-volume-power}, suppose
\[
 f_\gamma(z)\asymp\delta(z)^\gamma,
 \qquad 0<\delta(z)<\varepsilon_0,
\]
where $\gamma\in\mathbb R$ and $f_\gamma:\Omega\to[0,\infty)$.
Put $K_0=\{z\in\Omega:\delta(z)\ge\varepsilon_0\}$.  Suppose
$f_\gamma$ is bounded above and below by positive constants on $K_0$.
If
$f_\gamma\in\mathcal D_{\Omega,q}$, then
\begin{equation}\label{eq:general-boundary-symbol-threshold}
 M_{f_\gamma}\in\Piabs_r(A^p,L^q)
 \quad\Longleftrightarrow\quad
 \kappa\left(
 D\left(\frac1q-\frac1p\right)+\gamma
 \right)>D-1.
\end{equation}
If $1<q<\infty$ and
$f_\gamma\in\mathrm{RH}_{q,\rho}$, the same condition is equivalent to
\[
 T_{f_\gamma}\in\Piabs_r(A^p,A^q).
\]
\end{corollary}

\begin{proof}
Let $\dd\mu_\gamma=|f_\gamma|^q\dd v$. Stability of the boundary distance
gives a constant $C_0\ge1$ such that, if
$\delta(z)<\varepsilon_0/C_0$, then
\[
 C_0^{-1}\delta(z)\le\delta(w)\le C_0\delta(z),
 \qquad w\in\Q_\rho(z).
\]
Hence
\begin{align*}
 \mu_\gamma(\Q_\rho(z))
 &=\int_{\Q_\rho(z)}f_\gamma(w)^q\dd v(w)\\
 &\asymp \delta(z)^{\gamma q}\V_\rho(z)
 \asymp\delta(z)^{D+\gamma q}.
\end{align*}
The bounds on the compact core handle the remaining $z$.  Thus
\[
 \sigma=D+\gamma q.
\]
Apply \cref{prop:boundary-power-test} with
$\mu=\mu_\gamma$.  Since
\[
 \frac{D+\gamma q}{q}-\frac Dp
 =D\left(\frac1q-\frac1p\right)+\gamma,
\]
we obtain \eqref{eq:general-boundary-symbol-threshold}.  If $q>1$ and
$f_\gamma\in\mathrm{RH}_{q,\rho}$, then
\cref{thm:positive-Toeplitz} gives the Toeplitz equivalence.
\end{proof}

At the logarithmic diagonal border, the same strict threshold holds.
Indeed, if $\beta>0$, then
\begin{align*}
 \int_0^{\varepsilon_0}
 \Psi_q(t^\beta)\,t^{-D}\dd t
 &\asymp
 \int_0^{\varepsilon_0}
 t^{q\beta-D}(1+\log(t^{-\beta}))\dd t\notag\\
 &<\infty
 \quad\Longleftrightarrow\quad
 q\beta>D-1.
\end{align*}
At equality the integrand is comparable to
$t^{-1}\log(1/t)$ and the integral diverges.

\begin{example}[The unit ball]\label{ex:unit-ball}
Let $\Omega=\mathbb B_n$.  Fix $1<p<\infty$ and
$1\le q,r<\infty$.  For every fixed admissible nonisotropic
radius, equivalently for every fixed Bergman radius,
\[
 \V_\rho(z)\asymp(1-|z|^2)^{n+1},
 \qquad
 d\lambda_\Omega(z)\asymp
 \frac{\dd v(z)}{(1-|z|^2)^{n+1}}.
\]
In a power case $\mathfrak d_{p,q}^{\,r}=\ell^\kappa$, the multiplication
criterion is
\[
 (1-|z|^2)^{(n+1)(1/q-1/p)}\M_{q,\rho}f(z)
 \in L^\kappa(\mathbb B_n,d\lambda_\Omega).
\]
For $q>1$, this is also the graph criterion in
\eqref{eq:explicit-graph}.  At the Orlicz line, replace $L^\kappa$ by
$L^{q^-}$, as in \eqref{eq:density-continuous-Orlicz}.
For $p=q$, the boundary factor disappears.

This formula gives a concrete threshold.  Let
\[
 f_\gamma(z)=(1-|z|^2)^\gamma,
 \qquad \gamma\ge0.
\]
Then $f_\gamma\in\mathcal D_{\mathbb B_n,q}$ and
\[
 \M_{q,\rho}f_\gamma(z)\asymp(1-|z|^2)^\gamma.
\]
If $\mathfrak d_{p,q}^{\,r}=\ell^\kappa$, then
\begin{equation}\label{eq:ball-radial-threshold}
 M_{f_\gamma}\in\Piabs_r(A^p(\mathbb B_n),L^q(\mathbb B_n))
 \quad\Longleftrightarrow\quad
 \kappa\left((n+1)\left(\frac1q-\frac1p\right)+\gamma\right)>n.
\end{equation}
Indeed, with $t=1-|z|^2$, the integral near the boundary is comparable
to
\[
 \int_0^1
 t^{\kappa((n+1)(1/q-1/p)+\gamma)-(n+1)}\dd t.
\]
For $\gamma=0$, this gives the inclusion criterion in every power case.
On the Orlicz line $q=p'<2$ and $1\le r<q$, put
\[
 \beta_\gamma=(n+1)(1/q-1/p)+\gamma>0.
\]
Since $\Psi_q(t)\asymp t^q(1+\log(1/t))$ as $t\downarrow0$,
\begin{align*}
 \int_0^{1/2}\Psi_q(t^{\beta_\gamma})\frac{\dd t}{t^{n+1}}
 &\asymp\int_0^{1/2}
 t^{q\beta_\gamma-n-1}(1+\log(1/t))\dd t,\\
 M_{f_\gamma}\in\Piabs_r
 &\quad\Longleftrightarrow\quad q\beta_\gamma>n.
\end{align*}
The integral diverges at equality.

These thresholds also hold for the individual Toeplitz operator when
$1<q<\infty$.  Indeed, on every fixed Bergman ball,
\[
 1-|w|^2\asymp1-|z|^2.
\]
Therefore
\begin{align*}
 \left(\avg_{\Q_\rho(z)}f_\gamma(w)^q\dd v(w)\right)^{1/q}
 &\asymp(1-|z|^2)^\gamma,\\
 \avg_{\Q_\rho(z)}f_\gamma(w)\dd v(w)
 &\asymp(1-|z|^2)^\gamma.
\end{align*}
Thus $f_\gamma\in\mathrm{RH}_{q,\rho}(\mathbb B_n)$, with a constant
depending only on $\gamma,q,\rho$.  By
\cref{thm:positive-Toeplitz}, in a power case,
\begin{equation}\label{eq:ball-positive-Toeplitz-threshold}
 T_{f_\gamma}\in\Piabs_r(A^p(\mathbb B_n),A^q(\mathbb B_n))
 \quad\Longleftrightarrow\quad
 \kappa\left((n+1)\left(\frac1q-\frac1p\right)+\gamma\right)>n.
\end{equation}
On the Orlicz line the criterion is $q\beta_\gamma>n$ instead.
These are exact individual Toeplitz criteria.  They use no Hankel or
matching-IDA assumption.

We next compute a Volterra threshold.  Take $a=0$,
$p=q$, and $g(z)=z_1$.  Then
\[
 \R_0g(z)=z_1,
 \qquad
 \dd\mu_g(z)=\delta(z)^p|z_1|^p\dd v(z).
\]
If $\mathfrak d_{p,p}^{\,r}=\ell^\kappa$, then
\begin{equation}\label{eq:ball-Volterra-threshold}
 V_{z_1}\in\Piabs_r(A^p(\mathbb B_n),A^p(\mathbb B_n))
 \quad\Longleftrightarrow\quad \kappa>n.
\end{equation}
Indeed, \cref{cor:Volterra-continuous} gives
\[
 V_{z_1}\in\Piabs_r
 \quad\Longleftrightarrow\quad
 \mathcal C_{p,p,\rho}\mu_g\in
 L^\kappa(\mathbb B_n,d\lambda_\Omega).
\]
For every $z$ near the boundary,
\[
 \mu_g(\Q_\rho(z))
 \lesssim (1-|z|^2)^p\V_\rho(z).
\]
Thus
\[
 \mathcal C_{p,p,\rho}\mu_g(z)
 \lesssim1-|z|^2.
\]
The integral over each compact core is finite.
This proves sufficiency when
\[
 \int_0^1t^{\kappa-(n+1)}\dd t<\infty,
 \quad\text{that is,}\quad \kappa>n.
\]
For necessity, choose a relatively open set
$U\subset\partial\mathbb B_n$ and $c>0$ such that
$|\zeta_1|\ge2c$ on $U$.  On a thin normal collar over a smaller
boundary patch, $|w_1|\ge c$ for every $w\in\Q_\rho(z)$.  Hence
\[
 \mu_g(\Q_\rho(z))
 \gtrsim (1-|z|^2)^p\V_\rho(z),
 \qquad
 \mathcal C_{p,p,\rho}\mu_g(z)
 \gtrsim1-|z|^2.
\]
The collar integral diverges when $\kappa\le n$.  This proves
\eqref{eq:ball-Volterra-threshold}.

More generally, let $\alpha\in\Nzero^n\setminus\{0\}$ and
$g_\alpha(z)=z^\alpha$.  Then
\[
 \R_0g_\alpha(z)=|\alpha|z^\alpha.
\]
Since $|z^\alpha|\le1$, the preceding upper estimate is unchanged.
Choose $\zeta\in\partial\mathbb B_n$ such that
\[
 \zeta_k\ne0\qquad\text{whenever }\alpha_k>0.
\]
On a fixed collar over a smaller boundary patch about $\zeta$,
$|w^\alpha|\ge c_\alpha>0$ for every
$w\in\Q_\rho(z)$.  The same lower integral gives
\begin{equation}\label{eq:ball-Volterra-monomial-threshold}
 V_{z^\alpha}\in
 \Piabs_r(A^p(\mathbb B_n),A^p(\mathbb B_n))
 \quad\Longleftrightarrow\quad
 \kappa(p,p,r)>n.
\end{equation}
For $\alpha=0$, the radial derivative vanishes and
$V_{1}=0$.  Thus the constant symbol is the only monomial exception.
\end{example}

The unweighted identity map gives a sharp test that involves all three
parameters.  It is both a multiplication operator and the weighted
composition operator with $u=1$ and
$\varphi=\mathrm{id}_{\mathbb B_n}$.

\begin{corollary}
\label{cor:ball-inclusion}
Let $1<p<\infty$, $1\le q<\infty$, and $1\le r<\infty$.
Consider
\[
 I_{p,q}:A^p(\mathbb B_n)\longrightarrow A^q(\mathbb B_n),
 \qquad I_{p,q}F=F.
\]
If $q>p$, then $I_{p,q}$ is unbounded.  If $q=p$, it is not absolutely
$r$-summing.  Suppose $1<q<p$ and
$\mathfrak d_{p,q}^{\,r}=\ell^\kappa$ in \eqref{eq:kappa-table}.  Then
\begin{equation}\label{eq:ball-inclusion-threshold}
 I_{p,q}\in\Piabs_r(A^p,A^q)
 \quad\Longleftrightarrow\quad
 \kappa(n+1)\left(\frac1q-\frac1p\right)>n.
\end{equation}
On the logarithmic line $2<p<\infty$, $q=p'<2$, and $1\le r<q$,
\begin{equation}\label{eq:ball-inclusion-Orlicz}
 I_{p,p'}\in\Piabs_r(A^p,A^{p'})
 \quad\Longleftrightarrow\quad p>n+2.
\end{equation}
The last line of \eqref{eq:kappa-table} is not needed when $q<p$.
At the target endpoint,
\begin{equation}\label{eq:ball-inclusion-q1}
 I_{p,1}\in\Piabs_r(A^p,A^1)
 \quad\Longleftrightarrow\quad p>n+1.
\end{equation}
The condition in \eqref{eq:ball-inclusion-q1} is independent of $r$.
\end{corollary}

\begin{proof}
For $q>1$, the normalized kernel atoms satisfy
\begin{equation}\label{eq:ball-inclusion-kernel-size}
 \|e_{a,p}\|_{A^p}\asymp1,
 \qquad
 \|e_{a,p}\|_{A^q}\asymp
 (1-|a|^2)^{(n+1)(1/q-1/p)}.
\end{equation}
At $q=1$, the corresponding formula has a logarithmic factor:
\begin{equation}\label{eq:ball-inclusion-kernel-L1}
 \|e_{a,p}\|_{A^1}
 \asymp(1-|a|^2)^{(n+1)/p'}
 \log\frac{e}{1-|a|^2}
 \qquad(|a|\to1).
\end{equation}
Indeed, writing $a=(s,0,\ldots,0)$ by unitary invariance, integration
over the remaining variables gives
\[
 \int_{\mathbb B_n}|K(z,a)|\dd v(z)
 \asymp\int_{\mathbb D}
 \frac{(1-|\zeta|^2)^{n-1}}{|1-s\zeta|^{n+1}}\dd A(\zeta)
 \asymp\log\frac{e}{1-s^2}.
\]
The last comparison follows by integrating first in angle near
$\zeta=1$, and then using
\[
 \int_0^1\frac{x^{n-1}}{(1-s+x)^n}\dd x
 \asymp\log\frac{e}{1-s}
 \qquad(s\to1).
\]
Multiplication by $K(a,a)^{-1/p'}$ gives
\eqref{eq:ball-inclusion-kernel-L1}.

If $q>p$, the second quantity in
\eqref{eq:ball-inclusion-kernel-size} tends to infinity as $|a|\to1$.
Therefore $I_{p,q}$ is unbounded.  If $q=p$, the operator is the
identity on the infinite-dimensional reflexive space $A^p$.  It is not
compact.  Hence it is not absolutely summing.

If $1<q<p$ and $\mathfrak d_{p,q}^{\,r}=\ell^\kappa$, apply
\eqref{eq:ball-radial-threshold} with $\gamma=0$.  This gives
\eqref{eq:ball-inclusion-threshold}.

On the Orlicz line, the preceding example gives the strict condition
\[
 q(n+1)(1/q-1/p)>n.
\]
Since $q=p/(p-1)$,
\[
 q(n+1)(1/q-1/p)=\frac{(n+1)(p-2)}{p-1}>n
 \quad\Longleftrightarrow\quad p>n+2.
\]
This proves \eqref{eq:ball-inclusion-Orlicz}.  Its equality case is
excluded by logarithmic divergence.  This does not identify the Orlicz
space with $\ell^q$.

For $q=1$, use \eqref{eq:ball-radial-threshold} with $\gamma=0$ and
$\kappa=s_p$ from \eqref{eq:intro-q1-multiplication}.  No kernel $L^1$
estimate is needed for this step.  We now simplify the resulting inequality.
If $1<p\le2$, then
\[
 s_p(n+1)\left(1-\frac1p\right)
 =\frac{2(n+1)(p-1)}{3p-2}.
\]
The inequality
\[
 \frac{2(n+1)(p-1)}{3p-2}>n
\]
is equivalent to $(2-n)p>2$.  It has no solution with
$n\ge1$ and $1<p\le2$.  If $p>2$, then $s_p=1$, and
\[
 (n+1)\left(1-\frac1p\right)>n
 \quad\Longleftrightarrow\quad p>n+1.
\]
For $n=1$, this already forces $p>2$.  Hence the two ranges combine to
give \eqref{eq:ball-inclusion-q1}.
\end{proof}

At the Hilbert point these thresholds have a direct spectral meaning.

\begin{corollary}
\label{cor:Hilbert-ball-thresholds}
Let $1\le r<\infty$.
\begin{enumerate}[label=\textup{(\roman*)}]
\item If $\gamma\ge0$, then
\begin{equation*}
 T_{(1-|z|^2)^\gamma}\in
 \Piabs_r(A^2(\mathbb B_n),A^2(\mathbb B_n))
 \quad\Longleftrightarrow\quad
 \gamma>\frac n2.
\end{equation*}
The same condition characterizes
$M_{(1-|z|^2)^\gamma}:A^2\to L^2$.
\item If $\alpha\in\Nzero^n\setminus\{0\}$, then
\begin{equation*}
 V_{z^\alpha}\in
 \Piabs_r(A^2(\mathbb B_n),A^2(\mathbb B_n))
 \quad\Longleftrightarrow\quad n=1.
\end{equation*}
For $\alpha=0$, $V_1=0$ in every dimension.
\end{enumerate}
Both conclusions are independent of $r$.
\end{corollary}

\begin{proof}
Set $p=q=2$.  By \cref{prop:Hilbert-calibration}, the diagonal
exponent is $\kappa=2$ for every $r$.  Formula
\eqref{eq:ball-positive-Toeplitz-threshold} becomes
\[
 2\gamma>n.
\]
The multiplication assertion follows from
\eqref{eq:ball-radial-threshold}.  Formula
\eqref{eq:ball-Volterra-monomial-threshold} becomes $2>n$.  Since
$n\in\N$, this is equivalent to $n=1$.

For completeness, the Toeplitz threshold can also be checked without
the lattice theorem.  Let
\[
 e_\beta(z)=c_\beta z^\beta,
 \qquad \beta\in\Nzero^n,
\]
be the normalized monomial basis of $A^2(\mathbb B_n)$.
Write $\Gamma(x)=\int_0^\infty t^{x-1}e^{-t}\dd t$ for $x>0$
for Euler's gamma function. Radiality gives
\[
 T_{(1-|z|^2)^\gamma}e_\beta
 =\lambda_{|\beta|,\gamma}e_\beta,
\]
where
\begin{equation*}
 \lambda_{m,\gamma}
 =\frac{\displaystyle
 \int_{\mathbb B_n}|z^\beta|^2(1-|z|^2)^\gamma\dd v(z)}
 {\displaystyle\int_{\mathbb B_n}|z^\beta|^2\dd v(z)}
 =\frac{\Gamma(\gamma+1)\Gamma(n+m+1)}
 {\Gamma(n+m+\gamma+1)}
 \asymp(m+1)^{-\gamma}.
\end{equation*}
The eigenspace of homogeneous degree $m$ has dimension
\[
 d_m=\binom{n+m-1}{m}\asymp(m+1)^{n-1}.
\]
Hence
\begin{align*}
 \|T_{(1-|z|^2)^\gamma}\|_{\mathcal S_2}^2
 &=\sum_{m=0}^\infty d_m|\lambda_{m,\gamma}|^2\notag\\
 &\asymp\sum_{m=0}^\infty(m+1)^{n-1-2\gamma}<\infty
 \quad\Longleftrightarrow\quad 2\gamma>n.
\end{align*}
Now use \eqref{eq:Hilbert-summing-S2}.  This calculation also shows why
the inequality must be strict.
\end{proof}

On a weakly pseudoconvex ellipsoid, the McNeal volume changes with the
boundary point.  The following model in $\C^2$ includes the transition
from weak to strong boundary points.  The kernel formula is classical; see
\cite[Section~2.2]{BoasFuStraube1997}.  We include its derivation to fix
the volume normalization.

\begin{proposition}\label{prop:ellipsoid-global-volume}
Let $m\ge2$ be an integer, and put
\[
 \Omega_m=\{Z=(z,w)\in\C^2:|z|^{2m}+|w|^2<1\},
 \qquad t(Z)=1-|z|^{2m}-|w|^2.
\]
For a fixed sufficiently small $\rho>0$, set
\[
 R(Z)=|z|^2+t(Z)^{1/m}.
\]
Then, uniformly for $Z\in\Omega_m$,
\begin{equation}\label{eq:ellipsoid-global-volume}
 t(Z)\asymp\delta(Z),\qquad
 \V_\rho(Z)\asymp\frac{t(Z)^3}{R(Z)^{m-1}},\qquad
 \dd\lambda_{\Omega_m}(Z)
 \asymp\frac{R(Z)^{m-1}}{t(Z)^3}\dd v(Z).
\end{equation}
Moreover, if $Y\in\Q_\rho(Z)$, then
\begin{equation}\label{eq:ellipsoid-factor-stability}
 t(Y)\asymp t(Z),\qquad R(Y)\asymp R(Z).
\end{equation}
All constants may depend on $m$ and $\rho$.
\end{proposition}

\begin{proof}
The gradient of $|z|^{2m}+|w|^2-1$ does not vanish on the boundary.
Thus $t\asymp\delta$ in a boundary collar.  Both functions are positive
on each compact core, so the comparison holds throughout $\Omega_m$.

The normalized monomials form an orthonormal basis of $A^2(\Omega_m)$.
Orthogonality follows from the two independent rotations.  To prove
completeness, expand a holomorphic function in its Taylor series on this
complete Reinhardt domain.  Then apply the same rotations.
For $j,k\in\Nzero$, write $a_j=(j+1)/m$.  Polar integration gives
\begin{align}
 N_{j,k}:=\|z^jw^k\|_{A^2}^2
 &=\pi^2\int_0^1\int_0^{(1-y)^{1/m}}x^jy^k\dd x\,\dd y\notag\\
 &=\frac{\pi^2}{j+1}
   \int_0^1y^k(1-y)^{(j+1)/m}\dd y\notag\\
 &=\frac{\pi^2}{m}
   \frac{\Gamma(a_j)\Gamma(k+1)}{\Gamma(a_j+k+2)}.
 \label{eq:ellipsoid-monomial-norm}
\end{align}
Put
\[
 A=1-|w|^2,\qquad \xi=\frac{|z|^2}{A^{1/m}}\in[0,1).
\]
All terms of the diagonal series are nonnegative.  Thus Tonelli's
theorem allows either order of summation.  The binomial series gives
\[
 \sum_{k=0}^\infty
 \frac{\Gamma(a+k+2)}{\Gamma(k+1)}y^k
 =\Gamma(a+2)(1-y)^{-a-2},\qquad 0\le y<1.
\]
Also,
\[
 \sum_{j=0}^\infty(j+1)\xi^j=\frac1{(1-\xi)^2},
 \qquad
 \sum_{j=0}^\infty(j+1)^2\xi^j=\frac{1+\xi}{(1-\xi)^3}.
\]
Summing first over $k$ gives
\begin{align*}
 K_m(Z,Z)
 &=\frac m{\pi^2}\sum_{j=0}^\infty
   \frac{|z|^{2j}}{\Gamma(a_j)}
   \sum_{k=0}^\infty
   \frac{\Gamma(a_j+k+2)}{\Gamma(k+1)}|w|^{2k}\\
 &=\frac{A^{-2-1/m}}{m\pi^2}
   \sum_{j=0}^\infty(j+1)(j+1+m)\xi^j\\
 &=\frac{A^{-2-1/m}}{m\pi^2}
   \frac{(m+1)+(1-m)\xi}{(1-\xi)^3}.
\end{align*}
Since
\[
 2\le(m+1)+(1-m)\xi\le m+1,
 \qquad
 1-\xi\le1-\xi^m\le m(1-\xi),
\]
and $t=A(1-\xi^m)$, we obtain
\begin{equation}\label{eq:ellipsoid-diagonal-kernel}
 K_m(Z,Z)\asymp\frac{A^{1-1/m}}{t^3}
 \asymp\frac{R^{m-1}}{t^3}.
\end{equation}
The second comparison uses
\[
 A=t+|z|^{2m}\asymp(t^{1/m}+|z|^2)^m=R^m.
\]
The comparison between the kernel and the volume in \cref{sec:geometry} gives
\eqref{eq:ellipsoid-global-volume}.

For $Y\in\Q_\rho(Z)$, stability of the boundary distance and volume gives
\[
 t(Y)\asymp t(Z),\qquad\V_\rho(Y)\asymp\V_\rho(Z).
\]
Consequently,
\[
 R(Y)^{m-1}\asymp\frac{t(Y)^3}{\V_\rho(Y)}
 \asymp\frac{t(Z)^3}{\V_\rho(Z)}\asymp R(Z)^{m-1}.
\]
Since $m\ge2$, this proves \eqref{eq:ellipsoid-factor-stability}.
\end{proof}

\begin{theorem}\label{thm:ellipsoid-two-thresholds}
Let $m\ge2$ be an integer, $1<p<\infty$, and $1\le q,r<\infty$.
Suppose $\mathfrak d_{p,q}^{\,r}=\ell^\kappa$ with $1\le\kappa<\infty$.
Let $\eta,b\ge0$ and $\gamma\ge\eta/(2m)$, and define
\begin{equation}\label{eq:ellipsoid-symbol-family}
 f_{\gamma,\eta,b}(Z)
 =t(Z)^\gamma R(Z)^{-\eta/2}
   \left(\log\frac e{t(Z)}\right)^{-b}.
\end{equation}
Set $\Delta=1/q-1/p$ and
\begin{align}
 S&=\kappa(\gamma+3\Delta)-2,\notag\\
 W&=\kappa\left(\gamma+\left(2+\frac1m\right)\Delta
                      -\frac\eta{2m}\right)-1,\notag\\
 E&=\min\{S,W\},\qquad
 \varepsilon_*=
 \begin{cases}1,&S=W,\\0,&S\ne W.\end{cases}
 \label{eq:ellipsoid-critical-indices}
\end{align}
Then $0<f_{\gamma,\eta,b}\le1$ and
$f_{\gamma,\eta,b}\in\mathcal D_{\Omega_m,q}$.
Moreover,
\begin{equation}\label{eq:ellipsoid-full-criterion}
 M_{f_{\gamma,\eta,b}}\in\Piabs_r(A^p(\Omega_m),L^q(\Omega_m))
 \quad\Longleftrightarrow\quad
 \begin{cases}
 E>0,\quad\text{or}\\
 E=0\ \text{and}\ \kappa b>1+\varepsilon_*.
 \end{cases}
\end{equation}
If $q>1$, the same condition is equivalent to
$T_{f_{\gamma,\eta,b}}\in\Piabs_r(A^p(\Omega_m),A^q(\Omega_m))$.
In particular, for $b=0$ the criterion is
\begin{equation}\label{eq:ellipsoid-power-two-conditions}
 \kappa(\gamma+3\Delta)>2,
 \qquad
 \kappa\left(\gamma+\left(2+\frac1m\right)\Delta
                  -\frac\eta{2m}\right)>1.
\end{equation}
\end{theorem}

\begin{proof}
Write $L(t)=\log(e/t)$.  Since $R\ge t^{1/m}$ and $0<t\le1$,
\[
 0<f_{\gamma,\eta,b}
 \le t^{\gamma-\eta/(2m)}L(t)^{-b}\le1.
\]
For each fixed $z\in\Omega_m$, the kernel $K_m(\cdot,z)$ is bounded.
This follows from the kernel estimates in \cref{sec:geometry}.  It also
follows by polarizing the formula in the preceding proof.  Thus bounded
symbols belong to $\mathcal D_{\Omega_m,q}$.

Equations \eqref{eq:ellipsoid-factor-stability} and
$L(t(Y))\asymp L(t(Z))$ give
\[
 f_{\gamma,\eta,b}(Y)\asymp f_{\gamma,\eta,b}(Z),
 \qquad Y\in\Q_\rho(Z).
\]
Hence
\begin{equation}\label{eq:ellipsoid-local-means}
 \M_{q,\rho}f_{\gamma,\eta,b}(Z)
 \asymp f_{\gamma,\eta,b}(Z)
 \asymp\avg_{\Q_\rho(Z)}f_{\gamma,\eta,b}(Y)\dd v(Y).
\end{equation}
In particular, these nonnegative symbols satisfy
$\mathrm{RH}_{q,\rho}$ when $q>1$.

By the multiplication criterion and
\eqref{eq:ellipsoid-global-volume}, the summing condition is equivalent
to $I<\infty$, where
\begin{align}
 I&:=\int_{\Omega_m}
      f_{\gamma,\eta,b}(Z)^\kappa
      \V_\rho(Z)^{\kappa\Delta-1}\dd v(Z)\notag\\
 &\asymp\int_{\Omega_m}
      t^{\kappa(\gamma+3\Delta)-3}
      R^{m-1-\kappa((m-1)\Delta+\eta/2)}
      L(t)^{-\kappa b}\dd v(Z).
 \label{eq:ellipsoid-integral-reduction}
\end{align}
Here and below the comparison of integrals allows the value $+\infty$.
Put
\[
 B=m-1-\kappa((m-1)\Delta+\eta/2).
\]
In polar coordinates $x=|z|^2$, $y=|w|^2$, the angular integration
gives $\dd v=\pi^2\dd x\,\dd y$.
The change of variable $t=1-x^m-y$ has Jacobian of absolute value one.
Thus the boundary part of \eqref{eq:ellipsoid-integral-reduction} is
comparable to
\begin{equation}\label{eq:ellipsoid-double-integral}
 \int_0^{1/2}t^{S-1}L(t)^{-\kappa b}
   \left(\int_0^{(1-t)^{1/m}}(x+t^{1/m})^B\dd x\right)\dd t.
\end{equation}
The integral over $\{t\ge1/2\}$ is finite.
Direct integration in $x$ gives
\begin{equation}\label{eq:ellipsoid-tangential-integral}
 \int_0^{(1-t)^{1/m}}(x+t^{1/m})^B\dd x
 \asymp
 \begin{cases}
 1,&B>-1,\\
 L(t),&B=-1,\\
 t^{(B+1)/m},&B<-1.
 \end{cases}
\end{equation}
Indeed, for $B\ne-1$ the integral equals
\[
 \frac{\bigl((1-t)^{1/m}+t^{1/m}\bigr)^{B+1}
              -t^{(B+1)/m}}{B+1}.
\]
For $B=-1$ it equals
\[
 \log\left(1+\frac{(1-t)^{1/m}}{t^{1/m}}\right).
\]
Finally,
\[
 S+\frac{B+1}{m}=W,\qquad B=-1\ \Longleftrightarrow\ S=W.
\]
Therefore \eqref{eq:ellipsoid-double-integral} is finite exactly when
\[
 \int_0^{1/2}t^{E-1}L(t)^{\varepsilon_*-\kappa b}\dd t<\infty.
\]
If $E>0$, this integral converges.  If $E<0$, it diverges.
For $E=0$, the substitution $u=L(t)$ gives
\[
 \int_{\log(2e)}^\infty u^{\varepsilon_*-\kappa b}\dd u<\infty
 \quad\Longleftrightarrow\quad \kappa b>1+\varepsilon_*.
\]
This proves \eqref{eq:ellipsoid-full-criterion} and
\eqref{eq:ellipsoid-power-two-conditions}.
For $q>1$, apply \cref{thm:positive-Toeplitz} using
\eqref{eq:ellipsoid-local-means}.
\end{proof}

The two inequalities come from different boundary regions.  A boundary
patch with $|z|$ bounded away from zero gives the first one.  The region
$|z|^2\lesssim t^{1/m}$ gives the second one.  At their common critical
point, the tangential integral contributes the extra logarithm in
\eqref{eq:ellipsoid-full-criterion}.

\begin{corollary}\label{cor:ellipsoid-Orlicz-border}
Let $m\ge2$ be an integer, $\eta,b\ge0$, and
$\gamma\ge\eta/(2m)$.
Let $2<p<\infty$, $q=p'<2$, and $1\le r<q$.
Use the symbol \eqref{eq:ellipsoid-symbol-family} and the indices
in \eqref{eq:ellipsoid-critical-indices} with $\kappa=q$.
Then
\begin{equation}\label{eq:ellipsoid-Orlicz-criterion}
 M_{f_{\gamma,\eta,b}}\in\Piabs_r
 \quad\Longleftrightarrow\quad
 T_{f_{\gamma,\eta,b}}\in\Piabs_r
 \quad\Longleftrightarrow\quad
 \begin{cases}
 E>0,\quad\text{or}\\
 E=0\ \text{and}\ qb>2+\varepsilon_*.
 \end{cases}
\end{equation}
\end{corollary}

\begin{proof}
Here $\Delta=1/q-1/p>0$.  By
\eqref{eq:ellipsoid-local-means}, the continuous density is comparable to
\[
 a(Z)=t^{\gamma+3\Delta}
 R^{-\eta/2-(m-1)\Delta}L(t)^{-b}.
\]
Since $t^{1/m}\le R\le2$ and $\gamma\ge\eta/(2m)$,
\[
 c\,t^{\gamma+3\Delta}L(t)^{-b}
 \le a(Z)
 \le t^{\gamma+(2+1/m)\Delta-\eta/(2m)}L(t)^{-b}.
\]
Both powers of $t$ have positive exponents.  Consequently, uniformly
as $t\downarrow0$,
\[
 a(Z)\longrightarrow0,\qquad
 1+\log\frac1{a(Z)}\asymp L(t),\qquad
 \Psi_q(a(Z))\asymp a(Z)^qL(t).
\]
Indeed, if $A_0=\gamma+3\Delta$ and
$A_1=\gamma+(2+1/m)\Delta-\eta/(2m)>0$, the displayed bounds imply
\[
 A_1\log(1/t)+b\log L(t)
 \le\log\frac1{a(Z)}
 \le A_0\log(1/t)+b\log L(t)+C.
\]
Since $\log L(t)=o(L(t))$, this proves the logarithmic comparison.
For each fixed $h>0$, the same argument gives
$\Psi_q(a(Z)/h)\asymp_h a(Z)^qL(t)$ near the boundary.
Thus the convergence test is unchanged by the scale in the Luxemburg
norm.  On a compact core the density is bounded and
$\lambda_{\Omega_m}$ has finite mass.
If the modular is finite at one scale, dominated convergence makes it
at most one after increasing that scale.  This proves membership in the
Luxemburg space.
The Orlicz integral is therefore the integral in
\eqref{eq:ellipsoid-double-integral}, with $\kappa=q$ and one additional
factor $L(t)$.  Its critical logarithmic condition is
$qb>2+\varepsilon_*$.  The Toeplitz assertion follows from
\cref{thm:positive-Toeplitz}.
\end{proof}

\begin{example}\label{ex:ellipsoid}
For each integer $m\ge2$, put
\[
 f_m(Z)=t(Z)^{7/4}R(Z)^{-3}.
\]
These are bounded nonnegative symbols.  Take $p=q=2$.
Then $\kappa=2$ for every $1\le r<\infty$, and
\[
 S=\frac32,\qquad W=\frac52-\frac6m.
\]
Thus, for every $1\le r<\infty$,
\begin{equation}\label{eq:ellipsoid-type-dependent-example}
 M_{f_m}\in\Piabs_r(A^2(\Omega_m),L^2(\Omega_m))
 \quad\Longleftrightarrow\quad
 T_{f_m}\in\Piabs_r(A^2(\Omega_m),A^2(\Omega_m))
 \quad\Longleftrightarrow\quad m\ge3.
\end{equation}
Indeed, $\gamma=7/4\ge3/m=\eta/(2m)$ for every $m\ge2$.
The first critical index is always positive.  The second is positive
exactly when $m>12/5$.  Thus the global criterion for these bounded
symbols depends on the type.
\end{example}

There is an independent spectral check of the two power conditions at
the Hilbert point.

\begin{proposition}\label{prop:ellipsoid-spectral-check}
Let $m\ge2$ and $\gamma\ge\beta\ge0$.  On $\Omega_m$, put
\[
 g_{\gamma,\beta}(z,w)
 =(1-|z|^{2m}-|w|^2)^\gamma(1-|w|^2)^{-\beta}.
\]
Then $0<g_{\gamma,\beta}\le1$.  For $0<s<\infty$, let
$\mathcal S_s$ denote the Schatten class of compact operators whose
singular values belong to $\ell^s$.  Then
\begin{equation}\label{eq:ellipsoid-Schatten-threshold}
 T_{g_{\gamma,\beta}}\in\mathcal S_s(A^2(\Omega_m))
 \quad\Longleftrightarrow\quad
 s\gamma>2\ \text{and}\ s(\gamma-\beta)>1.
\end{equation}
In particular, for every $1\le r<\infty$,
\begin{equation}\label{eq:ellipsoid-spectral-threshold}
 T_{g_{\gamma,\beta}}\in\Piabs_r(A^2(\Omega_m),A^2(\Omega_m))
 \quad\Longleftrightarrow\quad
 \gamma>1\ \text{and}\ \gamma-\beta>\frac12.
\end{equation}
\end{proposition}

\begin{proof}
Since $t\le1-|w|^2$, boundedness follows from
$g_{\gamma,\beta}\le t^{\gamma-\beta}\le1$.
Let $e_{j,k}=N_{j,k}^{-1/2}z^jw^k$, with $N_{j,k}$ given by
\eqref{eq:ellipsoid-monomial-norm}.  Rotation invariance gives
\[
 T_{g_{\gamma,\beta}}e_{j,k}=\lambda_{j,k}e_{j,k},\qquad
 \lambda_{j,k}=\frac1{N_{j,k}}
     \int_{\Omega_m}g_{\gamma,\beta}|z|^{2j}|w|^{2k}\dd v.
\]
Using $x=|z|^2$, $y=|w|^2$, and then $u=x^m/(1-y)$, we obtain
\begin{align*}
 &\int_{\Omega_m}g_{\gamma,\beta}|z|^{2j}|w|^{2k}\dd v\\
 &\quad=\pi^2\int_0^1 y^k(1-y)^{-\beta}
     \int_0^{(1-y)^{1/m}}x^j(1-y-x^m)^\gamma\dd x\,\dd y\\
 &\quad=\frac{\pi^2}{m}
    \frac{\Gamma(a_j)\Gamma(\gamma+1)}{\Gamma(a_j+\gamma+1)}
    \int_0^1y^k(1-y)^{a_j+\gamma-\beta}\dd y\\
 &\quad=\frac{\pi^2}{m}
    \frac{\Gamma(a_j)\Gamma(\gamma+1)}{\Gamma(a_j+\gamma+1)}
    \frac{\Gamma(k+1)\Gamma(a_j+\gamma-\beta+1)}
         {\Gamma(a_j+k+\gamma-\beta+2)}.
\end{align*}
Therefore
\begin{equation}\label{eq:ellipsoid-Toeplitz-eigenvalues}
 \lambda_{j,k}
 =\Gamma(\gamma+1)
   \frac{\Gamma(a_j+\gamma-\beta+1)}{\Gamma(a_j+\gamma+1)}
   \frac{\Gamma(a_j+k+2)}{\Gamma(a_j+k+\gamma-\beta+2)}.
\end{equation}
The gamma quotient estimate gives, uniformly for $j,k\in\Nzero$,
\[
 \lambda_{j,k}\asymp
 (j+1)^{-\beta}(j+k+1)^{-(\gamma-\beta)}.
\]
Here the estimate
$\Gamma(x+a)/\Gamma(x+b)\asymp x^{a-b}$ holds uniformly for
$x\ge c>0$.  It follows from Stirling's formula for large $x$ and
continuity on a compact interval.  The comparison
$a_j+k+1\asymp j+k+1$ has constants depending only on $m$.
Put $d=\gamma-\beta$.  If $sd\le1$, the sum over $k$ of
$\lambda_{j,k}^s$ diverges for each fixed $j$.  If $sd>1$, then
the integral test gives
\[
 \sum_{k=0}^\infty(j+k+1)^{-sd}
 \asymp\int_{j+1}^\infty x^{-sd}\dd x
 =\frac{(j+1)^{1-sd}}{sd-1}.
\]
Consequently,
\begin{align*}
 \sum_{j,k\ge0}\lambda_{j,k}^s
 &\asymp\sum_{j=0}^\infty(j+1)^{-s\beta}
                  \sum_{k=0}^\infty(j+k+1)^{-sd}\\
 &\asymp\sum_{j=0}^\infty(j+1)^{1-s\gamma}
 <\infty
 \quad\Longleftrightarrow\quad s\gamma>2.
\end{align*}
This proves \eqref{eq:ellipsoid-Schatten-threshold}, including
$0<s<1$. Indeed, if $\sum_{j,k}\lambda_{j,k}^s<\infty$,
only finitely many eigenvalues exceed any fixed positive number.
The finite diagonal truncations therefore converge in operator norm,
so the operator is compact. Positivity then identifies its singular
values with the $\lambda_{j,k}$. Conversely, membership in
$\mathcal S_s$ implies compactness and requires precisely this sum
to be finite.  At $sd=1$ the inner sum
diverges; at $s\gamma=2$ the remaining outer sum diverges.
Set $s=2$.  Equation
\eqref{eq:Hilbert-summing-S2} gives
\eqref{eq:ellipsoid-spectral-threshold}.
Finally,
\[
 1-|w|^2=t+|z|^{2m}\asymp R^m.
\]
Thus $g_{\gamma,\beta}\asymp f_{\gamma,2m\beta,0}$, and the
spectral criterion agrees with
\eqref{eq:ellipsoid-power-two-conditions} at $p=q=2$.
For $\gamma=7/4$ and $\beta=3/m$, it also proves directly that
$T_{g_{7/4,3/m}}$ is absolutely $r$-summing exactly when $m\ge3$.
\end{proof}

Pushnitski \cite{Pushnitski2017} obtained singular value asymptotics
on the disc for symbols with a regular boundary power and an
angular profile. Here we count the eigenvalues of an explicit
two-parameter family on $\Omega_m$. The calculation gives leading
constants in three regimes. A logarithmic factor occurs when the
two counting exponents agree. It also determines the best
approximation error in the summing norm at a prescribed rank.
For $r\ne2$, we use the classical comparison on Hilbert spaces in
\cref{prop:Hilbert-calibration}.

\begin{theorem}\label{thm:ellipsoid-counting}
Let $m\ge2$ be an integer and let $\gamma>\beta\ge0$.  Use
$g=g_{\gamma,\beta}$ from \cref{prop:ellipsoid-spectral-check}, and put
\[
 d=\gamma-\beta,\qquad G=\Gamma(\gamma+1),\qquad
 \alpha=\max\left\{\frac2\gamma,\frac1d\right\},\qquad
 \varepsilon=\begin{cases}1,&\beta=d,\\0,&\beta\ne d.\end{cases}
\]
The operator $T_g$ is positive and compact.  Write its singular values
as $s_1(T_g)\ge s_2(T_g)\ge\cdots$, with multiplicity, and define
\[
 \mathcal N_g(\tau)=\#\{N\ge1:s_N(T_g)>\tau\},\qquad \tau>0.
\]
Then, as $\tau\downarrow0$,
\begin{equation}\label{eq:ellipsoid-counting-asymptotic}
 \mathcal N_g(\tau)
 \sim C_*\tau^{-\alpha}\bigl(\log(1/\tau)\bigr)^\varepsilon,
\end{equation}
where
\begin{equation}\label{eq:ellipsoid-counting-constant}
 C_*=
 \begin{cases}
 \displaystyle\frac{m\gamma}{2(d-\beta)}G^{2/\gamma},&\beta<d,\\[7pt]
 \displaystyle\frac m\gamma G^{1/d},&\beta=d,\\[7pt]
 \displaystyle\sum_{j=0}^\infty
 \left(G\frac{\Gamma((j+1)/m+d+1)}
 {\Gamma((j+1)/m+\gamma+1)}\right)^{1/d},&\beta>d.
 \end{cases}
\end{equation}
The series in the last case converges.  Set
$A_*=(C_*\alpha^{-\varepsilon})^{1/\alpha}$.  As $N\to\infty$,
\begin{equation}\label{eq:ellipsoid-singular-asymptotic}
 s_N(T_g)\sim
 A_*N^{-1/\alpha}(\log N)^{\varepsilon/\alpha}.
\end{equation}
In particular, if $\mathcal S_{\alpha,\infty}$ denotes the weak
Schatten class defined by $\sup_N N^{1/\alpha}s_N<\infty$, then
\begin{equation}\label{eq:ellipsoid-weak-endpoint}
 T_g\in\mathcal S_{\alpha,\infty}
 \quad\Longleftrightarrow\quad \beta\ne d.
\end{equation}
If $\alpha<2$, then for each $1\le r<\infty$ the approximation numbers
defined before \eqref{eq:summing-rank-budget} satisfy, as $N\to\infty$,
\begin{equation}\label{eq:ellipsoid-optimal-summing-error}
 a_N^{(r)}(T_g)
 \asymp_{m,\gamma,\beta,r} N^{1/2-1/\alpha}(\log N)^{\varepsilon/\alpha}.
\end{equation}
For $r=2$, the precise constant is
\begin{equation}\label{eq:ellipsoid-optimal-two-error}
 a_N^{(2)}(T_g)\sim
 \frac{A_*}{\sqrt{2/\alpha-1}}
 N^{1/2-1/\alpha}(\log N)^{\varepsilon/\alpha}.
\end{equation}
\end{theorem}

\begin{proof}
Put $a_j=(j+1)/m$ and
\[
 c_j=G\frac{\Gamma(a_j+d+1)}{\Gamma(a_j+\gamma+1)}.
\]
Equation \eqref{eq:ellipsoid-Toeplitz-eigenvalues} becomes
\begin{equation}\label{eq:ellipsoid-counting-gamma}
 \lambda_{j,k}=c_j
 \frac{\Gamma(a_j+k+2)}{\Gamma(a_j+k+d+2)}.
\end{equation}
Thus $\lambda_{j,k}>0$ and
\[
 c_j\sim G a_j^{-\beta},\qquad
 \lambda_{j,k}\asymp a_j^{-\beta}(a_j+k)^{-d}.
\]
Since $d>0$, only finitely many $\lambda_{j,k}$ exceed any fixed
positive number.  The diagonal truncations converge in operator norm,
which proves compactness.

First suppose $\beta\le d$.  Consider the model eigenvalues
\[
 \mu_{j,k}=G a_j^{-\beta}(a_j+k)^{-d}.
\]
Put $\vartheta=\beta/d$ and
$B=(G/\tau)^{1/\gamma}$.  The inequality $\mu_{j,k}>\tau$ is
equivalent to
\[
 0\le k<B^{1+\vartheta}a_j^{-\vartheta}-a_j.
\]
Its right side is positive exactly when $a_j<B$.  Counting integers
$k$ gives
\begin{align}
 \#\{(j,k):\mu_{j,k}>\tau\}
 &=B^{1+\vartheta}\sum_{a_j<B}a_j^{-\vartheta}
   -\sum_{a_j<B}a_j+O(B).
 \label{eq:ellipsoid-model-count}
\end{align}
The $O(B)$ term follows because there are at most $mB$ nonempty rows,
and rounding in each row has error at most one.  For $0\le\vartheta<1$,
the integral test gives
\[
 \sum_{a_j<B}a_j^{-\vartheta}
 \sim\frac{m}{1-\vartheta}B^{1-\vartheta},\qquad
 \sum_{a_j<B}a_j\sim\frac m2B^2.
\]
Consequently,
\begin{equation}\label{eq:ellipsoid-model-count-strong}
 \#\{(j,k):\mu_{j,k}>\tau\}
 \sim \frac{m\gamma}{2(d-\beta)}
 G^{2/\gamma}\tau^{-2/\gamma}.
\end{equation}
This also covers $\beta=0$.  If $\vartheta=1$, then
\[
 \sum_{a_j<B}a_j^{-1}=m\log B+O(1),\qquad
 \sum_{a_j<B}a_j=O(B^2).
\]
Since $\gamma=2d$, \eqref{eq:ellipsoid-model-count} now yields
\begin{equation}\label{eq:ellipsoid-model-count-critical}
 \#\{(j,k):\mu_{j,k}>\tau\}
 \sim mB^2\log B
 \sim\frac m\gamma G^{1/d}\tau^{-1/d}\log(1/\tau).
\end{equation}

We next compare the model with the exact eigenvalues. Stirling's
formula gives
\[
 \frac{\Gamma(x+u)}{\Gamma(x+v)}
 =x^{u-v}\bigl(1+O(x^{-1})\bigr)
 \qquad(x\to\infty),
\]
for each fixed pair $u,v$. Hence
\[
\begin{aligned}
 \frac{\lambda_{j,k}}{\mu_{j,k}}
 &=
 a_j^\beta
 \frac{\Gamma(a_j+d+1)}{\Gamma(a_j+\gamma+1)}
 (a_j+k)^d
 \frac{\Gamma(a_j+k+2)}{\Gamma(a_j+k+d+2)}\\
 &=\bigl(1+O(a_j^{-1})\bigr)
   \bigl(1+O((a_j+k)^{-1})\bigr)\\
 &=1+O(a_j^{-1}),
\end{aligned}
\]
uniformly for $k\ge0$ as $j\to\infty$.
For every $0<\eta<1$, there is $J_\eta$ such that
\[
 (1-\eta)\mu_{j,k}\le\lambda_{j,k}
 \le(1+\eta)\mu_{j,k},\qquad j\ge J_\eta,\quad k\ge0.
\]
For each fixed $j$, both eigenvalue sequences have
$O(\tau^{-1/d})$ terms greater than $\tau$.
The finitely many rows $j<J_\eta$ therefore contribute
$O_\eta(\tau^{-1/d})$.  Its ratio to $\tau^{-2/\gamma}$ tends to zero
if $\beta<d$.
Its ratio to $\tau^{-1/d}\log(1/\tau)$ tends to zero if $\beta=d$.
Apply \eqref{eq:ellipsoid-model-count-strong} or
\eqref{eq:ellipsoid-model-count-critical} at the levels
$\tau/(1-\eta)$ and $\tau/(1+\eta)$.
The resulting lower and upper constants differ by the factors
$(1-\eta)^\alpha$ and $(1+\eta)^\alpha$.
Letting $\eta\downarrow0$ proves
\eqref{eq:ellipsoid-counting-asymptotic} in these two cases.

Suppose now that $\beta>d$.  For each fixed $j$, define
\[
 n_j(\tau)=\#\{k\ge0:\lambda_{j,k}>\tau\}.
\]
The asymptotic in \eqref{eq:ellipsoid-counting-gamma} gives
\[
 \lim_{\tau\downarrow0}\tau^{1/d}n_j(\tau)=c_j^{1/d}.
\]
The uniform gamma quotient bound also gives
$\lambda_{j,k}\le Cc_j(k+1)^{-d}$.  Hence
\[
 0\le\tau^{1/d}n_j(\tau)\le C^{1/d}c_j^{1/d},\qquad
 \sum_{j=0}^\infty c_j^{1/d}<\infty,
\]
where convergence follows from $c_j^{1/d}\asymp(j+1)^{-\beta/d}$.
Dominated convergence in the sum over $j$ gives
\[
 \lim_{\tau\downarrow0}\tau^{1/d}\mathcal N_g(\tau)
 =\sum_{j=0}^\infty c_j^{1/d}.
\]
This proves the third case of
\eqref{eq:ellipsoid-counting-asymptotic}.

To recover the singular values, put
$v_N=A_*N^{-1/\alpha}(\log N)^{\varepsilon/\alpha}$.
Then
\[
 \log(1/v_N)\sim\frac1\alpha\log N,\qquad
 \frac{\mathcal N_g(cv_N)}N\longrightarrow c^{-\alpha}
 \quad(c>0).
\]
For fixed $0<c<1<C$ and all sufficiently large $N$, this gives
\[
 \mathcal N_g(cv_N)>N,\qquad \mathcal N_g(Cv_N)<N.
\]
By the definition of the counting function,
\[
 cv_N<s_N(T_g)\le Cv_N.
\]
These inequalities also hold when eigenvalues are repeated.
Let $c$ increase to $1$ and $C$ decrease to $1$. Then
$s_N(T_g)/v_N\to1$. This gives
\eqref{eq:ellipsoid-singular-asymptotic} and
\eqref{eq:ellipsoid-weak-endpoint}.

Finally assume $\alpha<2$.  By \cref{prop:Hilbert-calibration},
$\pi_2=\|\cdot\|_{\mathcal S_2}$ and
$\pi_r\asymp_r\|\cdot\|_{\mathcal S_2}$ on these Hilbert spaces.
Truncating the singular value expansion after $N-1$ terms gives
\begin{equation}\label{eq:ellipsoid-best-Hilbert-tail}
 a_N^{(2)}(T_g)^2=\sum_{\ell=N}^\infty s_\ell(T_g)^2.
\end{equation}
To check the lower bound in this equality, let $\operatorname{rank}A<N$
and let $Q$ be the orthogonal projection onto its range.  If
$(u_\ell)$ is an eigenbasis of $T_g$, then
\begin{align*}
 \|T_g-A\|_{\mathcal S_2}^2
 &\ge\|(I-Q)T_g\|_{\mathcal S_2}^2\\
 &=\sum_{\ell=1}^\infty s_\ell(T_g)^2
       (1-\|Qu_\ell\|^2)
 \ge\sum_{\ell=N}^\infty s_\ell(T_g)^2.
\end{align*}
The last step uses $0\le\|Qu_\ell\|^2\le1$ and
$\sum_\ell\|Qu_\ell\|^2=\operatorname{rank}Q\le N-1$.
Put
\[
 \sigma=\frac2\alpha>1,\qquad
 \nu=\frac{2\varepsilon}{\alpha}\ge0,\qquad
 h(x)=x^{-\sigma}(\log x)^\nu.
\]
The function $h$ is decreasing for all sufficiently large $x$.
Thus
\[
 \int_N^\infty h(x)\dd x
 \le\sum_{\ell=N}^\infty h(\ell)
 \le h(N)+\int_N^\infty h(x)\dd x.
\]
Writing $u=\log x$ and integrating by parts gives
\begin{align*}
 I_N:=\int_N^\infty h(x)\dd x
 &=\frac{N^{1-\sigma}(\log N)^\nu}{\sigma-1}\\
 &\quad+\frac{\nu}{\sigma-1}
 \int_{\log N}^\infty e^{-(\sigma-1)u}u^{\nu-1}\dd u.
\end{align*}
The last term is zero when $\nu=0$. Otherwise it is at most
$\nu I_N/((\sigma-1)\log N)$. Therefore
\[
 I_N\sim\frac{N^{1-\sigma}(\log N)^\nu}{\sigma-1},
 \qquad h(N)=o(I_N).
\]
Equation \eqref{eq:ellipsoid-singular-asymptotic} now gives
\[
 \sum_{\ell=N}^\infty s_\ell(T_g)^2
 \sim\frac{A_*^2}{2/\alpha-1}
 N^{1-2/\alpha}(\log N)^{2\varepsilon/\alpha}.
\]
Taking square roots proves \eqref{eq:ellipsoid-optimal-two-error}.
The equivalence of the summing and Hilbert--Schmidt norms gives
\eqref{eq:ellipsoid-optimal-summing-error}.
\end{proof}

For quasi-radial symbols on these ellipsoids, diagonalization in the
monomial basis is classical; see \cite{QuirogaSanchez2015}.
The next proposition computes the leading counting constant when the
symbol contains a bounded function of the transverse variable.
It also gives precise asymptotics for
the geometric symbol in \eqref{eq:ellipsoid-symbol-family}.

\begin{proposition}\label{prop:ellipsoid-profile}
Let $m\ge2$ be an integer and let $\gamma>\beta\ge0$.
Use $d,\alpha,\varepsilon$ and $C_*$ from
\cref{thm:ellipsoid-counting}. Let
$\psi:[0,1]\to[0,\infty)$ be bounded and measurable. Suppose
\[
 \lim_{u\uparrow1}\psi(u)=L_\psi>0.
\]
For $Z=(z,w)\in\Omega_m$, put
\[
 A(Z)=1-|w|^2,\qquad
 u(Z)=\frac{|z|^{2m}}{A(Z)},\qquad
 h_\psi(Z)=t(Z)^\gamma A(Z)^{-\beta}\psi(u(Z)).
\]
Set $a_j=(j+1)/m$ and
\begin{equation}\label{eq:ellipsoid-profile-moments}
 c_j(\psi)=
 \frac{\Gamma(a_j+d+1)}{\Gamma(a_j)}
 \int_0^1u^{a_j-1}(1-u)^\gamma\psi(u)\dd u.
\end{equation}
Then $T_{h_\psi}$ is positive and compact. Its singular value
counting function satisfies
\begin{equation}\label{eq:ellipsoid-profile-counting}
 \mathcal N_\psi(\tau)
 :=\#\{N\ge1:s_N(T_{h_\psi})>\tau\}
 \sim C_\psi\tau^{-\alpha}
        \bigl(\log(1/\tau)\bigr)^\varepsilon
 \qquad(\tau\downarrow0),
\end{equation}
where
\begin{equation}\label{eq:ellipsoid-profile-constant}
 C_\psi=
 \begin{cases}
 L_\psi^{2/\gamma}C_*,&\beta<d,\\[3pt]
 L_\psi^{1/d}C_*,&\beta=d,\\[3pt]
 \displaystyle\sum_{j=0}^\infty c_j(\psi)^{1/d},&\beta>d.
 \end{cases}
\end{equation}
The series in the last case converges. If
$A_\psi=(C_\psi\alpha^{-\varepsilon})^{1/\alpha}$, then
\begin{equation}\label{eq:ellipsoid-profile-singular}
 s_N(T_{h_\psi})\sim
 A_\psi N^{-1/\alpha}(\log N)^{\varepsilon/\alpha}.
\end{equation}
In particular,
$T_{h_\psi}\in\mathcal S_{\alpha,\infty}$ if and only if
$\beta\ne d$. If $\alpha<2$, then
\begin{equation}\label{eq:ellipsoid-profile-two-error}
 a_N^{(2)}(T_{h_\psi})\sim
 \frac{A_\psi}{\sqrt{2/\alpha-1}}
 N^{1/2-1/\alpha}(\log N)^{\varepsilon/\alpha}.
\end{equation}
If $\alpha<2$ and $1\le r<\infty$, the same rate holds for
$a_N^{(r)}(T_{h_\psi})$ up to positive constants.
\end{proposition}

\begin{proof}
Since $0<t\le A\le1$ and $d>0$, we have
\[
 0\le h_\psi\le\|\psi\|_\infty t^\gamma A^{-\beta}
 \le\|\psi\|_\infty.
\]
Thus $T_{h_\psi}$ is bounded and positive. The two independent
rotations make it diagonal in the normalized monomial basis used in
\eqref{eq:ellipsoid-monomial-norm}. We compute its diagonal entries.
Write $a=a_j$, $s=|z|^{2m}$ and $y=|w|^2$. Polar integration and
the substitution $s=(1-y)u$ give
\begin{align*}
 &\int_{\Omega_m}h_\psi(Z)|z|^{2j}|w|^{2k}\dd v(Z)\\
 &\quad=\frac{\pi^2}{m}\int_0^1y^k(1-y)^{-\beta}
       \int_0^{1-y}s^{a-1}(1-y-s)^\gamma
          \psi\!\left(\frac{s}{1-y}\right)\dd s\,\dd y\\
 &\quad=\frac{\pi^2}{m}
       \left(\int_0^1u^{a-1}(1-u)^\gamma\psi(u)\dd u\right)
       \int_0^1y^k(1-y)^{a+d}\dd y\\
 &\quad=\frac{\pi^2}{m}
       \left(\int_0^1u^{a-1}(1-u)^\gamma\psi(u)\dd u\right)
       \frac{\Gamma(k+1)\Gamma(a+d+1)}
            {\Gamma(a+k+d+2)}.
\end{align*}
Divide by \eqref{eq:ellipsoid-monomial-norm}. The resulting
eigenvalues are
\begin{equation}\label{eq:ellipsoid-profile-eigenvalues}
 \lambda_{j,k}(\psi)=c_j(\psi)
   \frac{\Gamma(a_j+k+2)}{\Gamma(a_j+k+d+2)}.
\end{equation}
For $\psi=1$, the beta integral gives
\[
 c_j(1)=\Gamma(\gamma+1)
       \frac{\Gamma(a_j+d+1)}{\Gamma(a_j+\gamma+1)}.
\]
These are the coefficients in
\eqref{eq:ellipsoid-counting-gamma}. In particular,
\begin{equation}\label{eq:ellipsoid-profile-domination}
 0\le c_j(\psi)\le\|\psi\|_\infty c_j(1),\qquad
 0\le\lambda_{j,k}(\psi)
       \le\|\psi\|_\infty\lambda_{j,k}(1).
\end{equation}
Only finitely many entries on the right exceed any fixed positive
number. The same is true on the left. Diagonal truncation therefore
proves compactness.

We next show that
\begin{equation}\label{eq:ellipsoid-profile-moment-limit}
 \frac{c_j(\psi)}{c_j(1)}\longrightarrow L_\psi.
\end{equation}
For $a>0$, define the probability measure
\[
 \dd\nu_a(u)=
 \frac{u^{a-1}(1-u)^\gamma}{B(a,\gamma+1)}\dd u,
 \qquad
 B(a,b)=\frac{\Gamma(a)\Gamma(b)}{\Gamma(a+b)}.
\]
For every fixed $0<\delta<1$,
\begin{align*}
 \nu_a([0,1-\delta])
 &\le (1-\delta)^{a/2}
       \frac{B(a/2,\gamma+1)}{B(a,\gamma+1)}
 \longrightarrow0\qquad(a\to\infty).
\end{align*}
Indeed, the quotient of beta functions tends to $2^{\gamma+1}$.
Given $\eta>0$, choose $\delta$ so that
$|\psi(u)-L_\psi|\le\eta$ for $1-\delta<u<1$. Then
\[
 \left|\int_0^1\psi(u)\dd\nu_a(u)-L_\psi\right|
 \le\eta+(\|\psi\|_\infty+L_\psi)
          \nu_a([0,1-\delta]).
\]
Let $a\to\infty$ and then $\eta\downarrow0$.
Since $c_j(\psi)/c_j(1)=\int\psi\dd\nu_{a_j}$, this proves
\eqref{eq:ellipsoid-profile-moment-limit}.

Suppose first that $\beta\le d$.
For $0<\eta<L_\psi$, choose $J_\eta$ so that
\[
 (L_\psi-\eta)\lambda_{j,k}(1)
 \le\lambda_{j,k}(\psi)
 \le(L_\psi+\eta)\lambda_{j,k}(1),
 \qquad j\ge J_\eta,\quad k\ge0.
\]
For each fixed $j$, the gamma quotient in
\eqref{eq:ellipsoid-profile-eigenvalues} gives
$\lambda_{j,k}(\psi)=O((k+1)^{-d})$.
The rows with $j<J_\eta$ contribute $O_\eta(\tau^{-1/d})$ to
the counting function. This is of lower order than
$\tau^{-2/\gamma}$ when $\beta<d$ and than
$\tau^{-1/d}\log(1/\tau)$ when $\beta=d$.
Consequently,
\begin{align*}
 \mathcal N_g\!\left(\frac{\tau}{L_\psi-\eta}\right)
       -O_\eta(\tau^{-1/d})
 &\le\mathcal N_\psi(\tau)\\
 &\le
 \mathcal N_g\!\left(\frac{\tau}{L_\psi+\eta}\right)
       +O_\eta(\tau^{-1/d}).
\end{align*}
Apply \eqref{eq:ellipsoid-counting-asymptotic} and let
$\eta\downarrow0$. This proves the first two cases of
\eqref{eq:ellipsoid-profile-constant}.

Now suppose that $\beta>d$, and set
$n_j^\psi(\tau)=\#\{k\ge0:\lambda_{j,k}(\psi)>\tau\}$.
The limit in \eqref{eq:ellipsoid-profile-moment-limit} is positive,
and $\psi$ is positive on an interval ending at $1$.
Thus $c_j(\psi)>0$ for every $j$. For each fixed $j$,
\[
 \lambda_{j,k}(\psi)\sim c_j(\psi)k^{-d},\qquad
 \tau^{1/d}n_j^\psi(\tau)\longrightarrow c_j(\psi)^{1/d}.
\]
The uniform gamma quotient bound also gives
\[
 \lambda_{j,k}(\psi)\le Cc_j(\psi)(k+1)^{-d},\qquad
 \tau^{1/d}n_j^\psi(\tau)
 \le C^{1/d}\|\psi\|_\infty^{1/d}c_j(1)^{1/d}.
\]
Here $C$ depends only on $m$ and $d$.
Since $c_j(1)^{1/d}\asymp(j+1)^{-\beta/d}$ and $\beta/d>1$,
dominated convergence yields
\[
 \lim_{\tau\downarrow0}\tau^{1/d}\mathcal N_\psi(\tau)
 =\sum_{j=0}^\infty c_j(\psi)^{1/d}<\infty.
\]
This proves the remaining case of
\eqref{eq:ellipsoid-profile-counting}.

The monotone inversion of the counting function in the proof of
\cref{thm:ellipsoid-counting} now gives
\eqref{eq:ellipsoid-profile-singular} and the weak Schatten statement.
If $\alpha<2$, the Hilbert--Schmidt tail formula
\eqref{eq:ellipsoid-best-Hilbert-tail} and the integral estimate in
that proof give \eqref{eq:ellipsoid-profile-two-error}.
Finally, \cref{prop:Hilbert-calibration} gives the same rate for every
finite $r\ge1$.
\end{proof}

\begin{corollary}\label{cor:ellipsoid-geometric-spectral}
Let $m\ge2$ be an integer and $\gamma>\beta\ge0$. Use the notation of
\cref{thm:ellipsoid-counting}, and let
$f=f_{\gamma,2m\beta,0}$ be the symbol in
\eqref{eq:ellipsoid-symbol-family}. Define
\[
 \psi_{m,\beta}(u)=
 \bigl(u^{1/m}+(1-u)^{1/m}\bigr)^{-m\beta},\qquad 0\le u\le1.
\]
The conclusions of \cref{prop:ellipsoid-profile} hold for $T_f$,
with counting constant
\begin{equation}\label{eq:ellipsoid-geometric-spectral-constant}
 C_f=
 \begin{cases}
 C_*,&\beta\le d,\\[3pt]
 \displaystyle\sum_{j=0}^\infty
        c_j(\psi_{m,\beta})^{1/d},&\beta>d.
 \end{cases}
\end{equation}
In the second case, $0<C_f<C_*$.
In particular, if $\gamma=2$ and $\beta=1$, then
\[
 s_N(T_f)\sim m\frac{\log N}{N},\qquad
 a_N^{(2)}(T_f)\sim m\frac{\log N}{\sqrt N}.
\]
\end{corollary}

\begin{proof}
With $A$ and $u$ as in \cref{prop:ellipsoid-profile}, we have
\[
 t=A(1-u),\qquad
 R=A^{1/m}\bigl(u^{1/m}+(1-u)^{1/m}\bigr).
\]
It follows that
\[
 f=t^\gamma R^{-m\beta}
   =t^\gamma A^{-\beta}\psi_{m,\beta}(u),\qquad
 2^{-(m-1)\beta}\le\psi_{m,\beta}\le1,
 \qquad \psi_{m,\beta}(1)=1.
\]
Apply \cref{prop:ellipsoid-profile} with $L_\psi=1$.
If $\beta>d$, then $\beta>0$ and
$0<\psi_{m,\beta}(u)<1$ for $0<u<1$.
The integrals in \eqref{eq:ellipsoid-profile-moments} therefore give
$0<c_j(\psi_{m,\beta})<c_j(1)$ for every $j$.
Summing their $1/d$ powers proves $0<C_f<C_*$.
For $\gamma=2$ and $\beta=1$, we have
$d=\alpha=\varepsilon=1$ and $C_*=m$.
The two displayed asymptotics follow.
\end{proof}

Thus the leading constant depends only on the endpoint value of the
profile when $\beta\le d$. When $\beta>d$, it depends on its full
sequence of beta moments. The logarithmic approximation factor is
therefore present for the original geometric symbol as well as for
the explicit gamma-function model.

\begin{theorem}\label{thm:sharp-models}
Let $1<p,q<\infty$ and $1\le r<\infty$.  Let
$\Lambda_0=\{a_j:j\in J_0\}$ be a sufficiently separated sublattice.
Let $b=(b_j)_{j\in J_0}\ge0$ be finitely supported, and extend it by
zero to $\N$ when forming $D_b$.  There is a nonnegative symbol
$f_b\in\mathcal D_{\Omega,q}$, supported in
the disjoint cells of $\Lambda_0$, such that
\begin{equation}\label{eq:sharp-coefficients}
 V_j^{1/q-1/p}\M_{q,\rho}f_b(a_j)\asymp b_j
 \quad(j\in J_0),
\end{equation}
and all coefficients off a fixed enlargement of
$\{a_j:b_j\ne0\}$ vanish.
Moreover,
\begin{equation*}
 \pi_r(M_{f_b}:A^p\to L^q)
 \asymp\pi_r(D_b:\ell^p\to\ell^q).
\end{equation*}
The constants do not depend on the support of $b$.  Thus the exact
diagonal ideal is sharp at every finite level.
\end{theorem}

\begin{proof}
Take the sublattice sufficiently separated that the sets
$\Q_\rho(a_j)$, $j\in J_0$, are pairwise disjoint.  Fix
$0<\theta<1/2$ and use the affine contractions from
\eqref{eq:frozen-polydisc-dilation}:
\[
 E_j=\Q_\rho^{[\theta]}(a_j)\subset\Q_\rho(a_j),
 \qquad v(E_j)=\theta^{2n}V_j.
\]
The extremal basis is fixed at scale $\rho$ in this inclusion.  Put
\[
 f_b(z)=\sum_{j\in J_0}
 b_jV_j^{1/p-1/q}\mathbf1_{E_j}(z).
\]
Here $\mathbf1_E$ denotes the characteristic function of $E$.  Since
$b$ has finite support, the support of $f_b$ lies in a finite union of
relatively compact cells.  For fixed $z\in\Omega$, the kernel section
$K(\cdot,z)$ is bounded on that union.  Hence
$f_bK(\cdot,z)\in L^q$.  Thus $f_b\in\mathcal D_{\Omega,q}$.
No other cell meets $\Q_\rho(a_j)$ for $j\in J_0$.  Hence
\begin{align*}
 \M_{q,\rho}f_b(a_j)^q
 &= V_j^{-1}
 \int_{E_j}b_j^qV_j^{q/p-1}\dd v\\
 &=\theta^{2n}b_j^qV_j^{q/p-1}.
\end{align*}
Hence \eqref{eq:sharp-coefficients} holds for $j\in J_0$.  Finite overlap
shows that $\Q_\rho(a_k)$ misses $\supp f_b$ unless $a_k$ belongs to a
fixed enlargement of $\{a_j:b_j\ne0\}$.  More precisely, let
\[
 \mathcal N(k)=\{j\in J_0:E_j\cap\Q_\rho(a_k)\ne\varnothing\}.
\]
Engulfing and volume stability give $V_j\asymp V_k$ for
$j\in\mathcal N(k)$.  The number of neighbors is uniformly bounded in
both indices.  Thus
\begin{align*}
 \bigl[V_k^{1/q-1/p}\M_{q,\rho}f_b(a_k)\bigr]^q
 &=V_k^{-q/p}\sum_{j\in\mathcal N(k)}
   b_j^qV_j^{q/p-1}v(E_j\cap\Q_\rho(a_k))\\
 &\lesssim\sum_{j\in\mathcal N(k)}b_j^q,
\end{align*}
and taking a $q$th root bounds this coefficient by
$C\sum_{j\in\mathcal N(k)}b_j$.  Conversely, the
coefficients on $J_0$ dominate $b$.  Solidity and
\cref{lem:finite-neighbor} give
\[
 \|\mathbf M_{p,q,\rho}f_b\|_{\mathfrak d_{p,q}^{\,r}}
 \asymp\|b\|_{\mathfrak d_{p,q}^{\,r}}.
\]
Apply \cref{cor:multiplication}.
\end{proof}

\subsection{Compactness and the target endpoint}

\begin{corollary}\label{cor:compactness}
Let $1<p<\infty$, $1\le r<\infty$, and let $Y$ be a Banach space.
Every operator
\[
 A\in\Piabs_r(A^p(\Omega),Y)
\]
is compact.  Consequently, every absolutely summing multiplication,
Toeplitz, big Hankel, little Hankel, weighted composition, or
Volterra composition operator considered above is compact.
If $1<q<\infty$ and $M_f\in\Piabs_r(A^p,L^q)$, then
$T_f,H_f,h_f$, and $k_f$ are compact.
\end{corollary}

\begin{proof}
Absolutely summing operators are completely continuous.  The space
$A^p(\Omega)$ is a closed subspace of the reflexive space $L^p(\Omega)$,
so it is reflexive.  If $(F_m)$ is bounded, pass to a subsequence
$(F_{m_k})$ that converges weakly to $F\in A^p$.  Complete continuity
gives $\|AF_{m_k}-AF\|_Y\to0$.
This also follows directly from Pietsch domination.  For a probability
measure $\nu$ on the dual unit ball,
\[
 \|A(F_{m_k}-F)\|_Y^r
 \le\pi_r(A)^r\int_{B_{(A^p)^*}}
       |\ell(F_{m_k}-F)|^r\dd\nu(\ell)\longrightarrow0.
\]
The integrands tend pointwise to zero and are uniformly bounded by
$\sup_k\|F_{m_k}-F\|_{A^p}^r$.
Thus the image of every bounded sequence has a norm-convergent
subsequence, which proves compactness.  For the final assertion, use
\[
 T_f=PM_f,\qquad H_f=(I-P)M_f,\qquad
 h_f=\overline PM_f,\qquad k_f=(I-\overline P)M_f.
\]
\end{proof}

The next proposition gives the domains of all five classes.  It
distinguishes each initial kernel formula from its bounded extension.

\begin{proposition}
\label{prop:domain-ledger}
Let $1<p<\infty$.  Recall the space $\mathcal P(\Omega)$ defined in
\cref{sec:geometry}.
\begin{enumerate}[label=\textup{(\roman*)}]
\item Both $\Gamma$ and $\mathcal P(\Omega)$ are dense in
$A^p(\Omega)$.
\item Let $1<q<\infty$ and $f\in\mathcal D_{\Omega,q}$.  The four
projected formulas
\[
 T_f,\quad H_f,\quad h_f,\quad k_f
\]
are defined on $\Gamma$.  Any bounded extension to $A^p$ is unique.
\item Let $1\le q<\infty$, $u,g\in A^q(\Omega)$, and let
$\varphi:\Omega\to\Omega$ be holomorphic.  Then
\[
 \mathcal P(\Omega)\subset
 \mathcal D(W_{u,\varphi})\cap
 \mathcal D(V_{g,\varphi}).
\]
If either restriction to $\mathcal P(\Omega)$ has a bounded extension
to $A^p$, that extension is unique. The natural domain of the
corresponding operator is then all of $A^p$, and the extension agrees
with its pointwise formula for every $F\in A^p$.
\item Every occurrence of
\[
 A\in\Piabs_r(A^p,Y)
\]
in the preceding theorems means that the indicated initial operator
has a unique bounded extension to all of $A^p$.  This extension belongs
to $\Piabs_r(A^p,Y)$.
\item At $q=1$, the convention in (iv) applies to
$M_f$, $T_f^{\mathrm{an}}$, $W_{u,\varphi}$, and
$V_{g,\varphi}$.  It is not applied to $T_f,H_f,h_f$, or $k_f$ with a
general nonanalytic symbol.
\end{enumerate}
\end{proposition}

\begin{proof}
The density of $\Gamma$ follows from \cref{lem:kernel-span-density}.
We prove the polynomial assertion.  Fix
$a\in\Omega$.  For $0<t<1$, put
\[
 F_t(z)=F(a+t(z-a)).
\]
Convexity implies
\[
 a+t(\overline\Omega-a)\Subset\Omega.
\]
Hence $F_t$ is holomorphic on a neighborhood of $\overline\Omega$.
Let $\widetilde F$ be the zero extension of $F$ to
$\mathbb R^{2n}$.  For $z\in\Omega$,
\[
 F_t(z)=\widetilde F(a+t(z-a)).
\]
Affine dilations are strongly continuous on $L^p(\mathbb R^{2n})$.
Indeed, their norms are $t^{-2n/p}$, uniformly bounded for
$1/2\le t\le1$.  For a continuous compactly supported function,
strong continuity follows from uniform continuity on a common compact
set.  Approximation in $L^p$ by such functions proves the assertion for
$\widetilde F$.
Therefore
\[
 \|F_t-F\|_{L^p(\Omega)}
 \le
 \|\widetilde F(a+t(\,\cdot-a))-\widetilde F\|_{L^p(\mathbb R^{2n})}
 \longrightarrow0
 \qquad(t\uparrow1).
\]
Thus
\begin{equation}\label{eq:dilation-density-Ap}
 \|F_t-F\|_{A^p}\longrightarrow0
 \qquad(t\uparrow1).
\end{equation}
The compact convex set $\overline\Omega$ is polynomially convex.
Oka--Weil approximation gives polynomials $P_{t,m}$ such that
\[
 \sup_{\overline\Omega}|P_{t,m}-F_t|\longrightarrow0
 \qquad(m\to\infty).
\]
Together with \eqref{eq:dilation-density-Ap}, this proves density of
$\mathcal P(\Omega)$.  Explicitly, choose $t_\nu\uparrow1$ with
$\|F_{t_\nu}-F\|_{A^p}<1/\nu$, and then choose a polynomial $P_\nu$
such that
\[
 \sup_{\overline\Omega}|P_\nu-F_{t_\nu}|
 <\frac1{\nu v(\Omega)^{1/p}},
 \qquad
 \|P_\nu-F\|_{A^p}<\frac2\nu.
\]

Part (ii) follows from the symbol condition
\eqref{eq:symbol-class} and the boundedness of $P$ and $\overline P$ on
$L^q$.  Indeed, for $F\in\Gamma$,
\[
 fF\in L^q,
\]
so all four formulas have values in their displayed target spaces.
Two bounded extensions that agree on the dense space $\Gamma$ agree on
$A^p$.

Let $P_0$ be a holomorphic polynomial.  Since $\Omega$ is bounded,
$P_0$ is bounded on $\Omega$.  Thus
\[
 \|W_{u,\varphi}P_0\|_{A^q}
 \le
 \|P_0\|_{L^\infty(\Omega)}\|u\|_{A^q}.
\]
Hence $P_0\in\mathcal D(W_{u,\varphi})$.  For the Volterra operator,
\eqref{eq:Volterra-Carleson-norm} and the finiteness of
$\mu_{g,\varphi}$ give
\[
 \|V_{g,\varphi}P_0\|_{A^q}^q
 \lesssim
 \|P_0\|_{L^\infty(\Omega)}^q
 \mu_{g,\varphi}(\Omega)<\infty.
\]
Thus $P_0\in\mathcal D(V_{g,\varphi})$.

It remains to identify an extension.  Let $F_m\in\mathcal P(\Omega)$
and $F_m\to F$ in $A^p$.  Then $F_m\to F$ uniformly on compact subsets
of $\Omega$, since for $L\Subset\Omega$ the local mean-value inequality
gives
\[
 \sup_L|F_m-F|\le C_L\|F_m-F\|_{A^p}.
\]
For the weighted composition formula,
\[
 u(F_m\circ\varphi)\longrightarrow u(F\circ\varphi)
\]
locally uniformly.  If the left side converges in $A^q$ by boundedness
of an extension, its $A^q$ limit has the displayed pointwise value.

For the Volterra formula, fix $K\Subset\Omega$ and put
\[
 K_a=\{a+t(z-a):z\in K,\ 0\le t\le1\}.
\]
Then $K_a\Subset\Omega$ and $\varphi(K_a)\Subset\Omega$.  The quotient
has the explicit expression
\[
 \frac{\R_ag(a+t(z-a))}{t}
 =\sum_{\ell=1}^n(z_\ell-a_\ell)
        \partial_\ell g(a+t(z-a)).
\]
It extends continuously to $t=0$ and is bounded uniformly for
$z\in K$, $0\le t\le1$.  Therefore
\[
 \sup_{z\in K}|V_{g,\varphi}(F_m-F)(z)|
 \le C_K\sup_{\varphi(K_a)}|F_m-F|\longrightarrow0.
\]
In particular,
\[
 V_{g,\varphi}F_m\longrightarrow V_{g,\varphi}F
\]
uniformly on $K$.  Any $A^q$ limit has the same pointwise value.  This
proves (iii).  Statements (iv) and (v) are now consequences of density
and of the definitions.

The same local argument verifies the necessity of the weight
assumptions for extensions from a dense initial domain. Start with
holomorphic $u$ or $g$. Suppose the corresponding formula is defined
on a subspace of its natural domain that is dense in $A^p$, and has
a bounded extension $A^p\to A^q$.
Choose $F_k$ in that subspace with $F_k\to\mathbf1$ in $A^p$.
For the respective operators, local uniform convergence gives
\[
 W_{u,\varphi}F_k\longrightarrow u,
 \qquad
 V_{g,\varphi}F_k\longrightarrow g-g(a).
\]
Boundedness of the extension also gives convergence in $A^q$.
The local limits identify the $A^q$ limits. Hence $u\in A^q$,
respectively $g-g(a)\in A^q$. Since $\Omega$ has finite volume,
the latter condition is equivalent to $g\in A^q$.
\end{proof}

Compactness gives a scalar boundary condition for each of the five
classes.  The conclusion below is stronger than boundedness.  It holds
in every case of Garling's classification.

\begin{corollary}
\label{cor:five-boundary-vanishing}
Let $1<p,q<\infty$ and $1\le r<\infty$.
Let $f\in\mathcal D_{\Omega,q}$ and $u,g\in A^q(\Omega)$.
Let $\varphi:\Omega\to\Omega$ be holomorphic.
\begin{enumerate}[label=\textup{(\roman*)}]
\item If $T_f\in\Piabs_r(A^p,A^q)$, then
\begin{equation*}
 V_j^{1/q-1/p}|\widetilde f(a_j)|\longrightarrow0
 \quad(j\to\partial\Omega).
\end{equation*}
\item If $H_f\in\Piabs_r(A^p,L^q)$, then
\begin{equation*}
 V_j^{1/q-1/p}\G_{q,\rho_+}f(a_j)\longrightarrow0
 \quad(j\to\partial\Omega).
\end{equation*}
\item If $h_f\in\Piabs_r(A^p,\overline{A^q})$, then
\begin{equation*}
 V_j^{1/q-1/p}|\mathfrak b f(a_j)|\longrightarrow0
 \quad(j\to\partial\Omega).
\end{equation*}
\item If $W_{u,\varphi}\in\Piabs_r(A^p,A^q)$, then
\begin{equation*}
 \frac{\mu_{u,\varphi}(\Q_\rho(a_j))^{1/q}}
 {V_j^{1/p}}\longrightarrow0
 \quad(j\to\partial\Omega).
\end{equation*}
\item If $V_{g,\varphi}\in\Piabs_r(A^p,A^q)$, then
\begin{equation*}
 \frac{\mu_{g,\varphi}(\Q_\rho(a_j))^{1/q}}
 {V_j^{1/p}}\longrightarrow0
 \quad(j\to\partial\Omega).
\end{equation*}
\end{enumerate}
If either graph $(T_f,H_f)$ or $(h_f,k_f)$ is absolutely summing,
one also has
\begin{equation*}
 V_j^{1/q-1/p}\M_{q,\rho}f(a_j)\longrightarrow0.
\end{equation*}
\end{corollary}

\begin{proof}
Every scalar sequence in $\mathfrak d_{p,q}^{\,r}$ belongs to $c_0$ by
\cref{prop:diagonal-compact-monotone}.  Apply this fact, respectively,
to \cref{cor:scalar-necessity,thm:intro-hankel,thm:little-necessity,thm:weighted-composition,thm:Volterra-composition}.
The graph assertion follows from \cref{thm:intro-main,thm:little-graph}.
\end{proof}

Analytic symbols do not require a projection.

\begin{corollary}
\label{cor:analytic-symbols}
Let $1<p<\infty$, $1\le q<\infty$, $1\le r<\infty$, and
\[
 f\in\Hol(\Omega)\cap\mathcal D_{\Omega,q}.
\]
Define on $\Gamma$
\[
 T_f^{\mathrm{an}}g=fg\in A^q(\Omega).
\]
Then
\begin{equation}\label{eq:analytic-symbol-criterion}
 T_f^{\mathrm{an}}\in\Piabs_r(A^p(\Omega),A^q(\Omega))
 \quad\Longleftrightarrow\quad
 f\in\mathcal L_{p,q}^{r}(\Omega),
\end{equation}
and
\begin{equation*}
 \pi_r(T_f^{\mathrm{an}})
 \asymp\|f\|_{\mathcal L_{p,q}^{r}(\Omega)}.
\end{equation*}
For $q>1$, $T_f^{\mathrm{an}}=T_f$.  For $q=1$,
\eqref{eq:analytic-symbol-criterion} gives a Toeplitz criterion that requires no projection.
\end{corollary}

\begin{proof}
For $g\in\Gamma$, the product $fg$ is holomorphic and belongs to
$L^q$ by the definition of $\mathcal D_{\Omega,q}$.  Thus $fg\in A^q$
and
\[
 \|T_f^{\mathrm{an}}g\|_{A^q}=\|M_fg\|_{L^q}.
\]
The two $r$-summing inequalities on $\Gamma$ are identical.  If either
operator extends, approximate $g\in A^p$ by $g_m\in\Gamma$.  The
$L^q$ limit of the holomorphic functions $fg_m$ is holomorphic in the
sense of distributions: $L^q$ convergence implies local $L^1$
convergence and preserves the equations $\bar\partial(fg_m)=0$.
Also, $g_m\to g$ locally uniformly, so the limit is $fg$.
Hence it belongs to $A^q$.  The two extensions
therefore agree and have the same $r$-summing norm.  Apply
\cref{cor:multiplication}.  If $q>1$, then $P(fg)=fg$ on $\Gamma$, so
$T_f^{\mathrm{an}}=T_f$.
\end{proof}

The diagonal ideal is known at $q=1$.  Let
$1<p<\infty$ and $1\le r<\infty$.  Taking $q=1$ in
\cref{prop:garling} gives
\begin{equation}\label{eq:q1-diagonal}
 \mathfrak d_{p,1}^{\,r}=\ell^{s_p},
 \qquad s_p\text{ is given by \eqref{eq:intro-q1-multiplication}},
\end{equation}
independently of $r$.
For $1<p\le2$, the exponent is
\[
 s_p=\left(\frac1{p'}+1-\frac12\right)^{-1}
 =\frac{2p}{3p-2}.
\]
For $p>2$, it is $s_p=1$.  Thus the two formulas agree at $p=2$.

\begin{proposition}
\label{prop:q1-multiplication}
Let $1<p<\infty$, $1\le r<\infty$, and
$f\in\mathcal D_{\Omega,1}$.  Then
\begin{equation}\label{eq:q1-candidate}
\begin{aligned}
 M_f\in\Piabs_r(A^p(\Omega),L^1(\Omega))
 &\Longleftrightarrow
 \{V_j^{1-1/p}\M_{1,\rho}f(a_j)\}_j\in\ell^{s_p}\\
 &\Longleftrightarrow
 \V_\rho^{1-1/p}\M_{1,\rho}f
 \in L^{s_p}(\Omega,d\lambda_\Omega).
\end{aligned}
\end{equation}
Moreover,
\begin{equation*}
 \pi_r(M_f)
 \asymp
 \|\{V_j^{1-1/p}\M_{1,\rho}f(a_j)\}\|_{\ell^{s_p}}
 \asymp
 \|\V_\rho^{1-1/p}\M_{1,\rho}f
 \|_{L^{s_p}(d\lambda_\Omega)}.
\end{equation*}
The constants depend only on $p,r$ and the fixed finite-type geometry.
They are independent of $f$.
\end{proposition}

\begin{proof}
The proof of \cref{thm:carleson} remains valid when $q=1$.  The local
nuclear decomposition uses only $p\ge1$ and $q\ge1$.  The block proof
uses the triangle inequality in $\ell^1$.  It does not use the Bergman
projection on $L^1$.  Hence \cref{cor:multiplication} with $q=1$ gives
\[
 \pi_r(M_f)\asymp
 \|\{V_j^{1-1/p}\M_{1,\rho}f(a_j)\}\|_{\mathfrak d_{p,1}^{\,r}}.
\]
Use \eqref{eq:q1-diagonal}.  The argument using step functions on the lattice in
\eqref{eq:discrete-continuous-explicit} gives the continuous form and
the two norm comparisons.
\end{proof}

The endpoint exponent $s_p$ is also sharp on infinite separated data.

\begin{proposition}
\label{prop:q1-sharp-realization}
Let $1<p<\infty$ and $1\le r<\infty$.  Let
$\Lambda_0=\{a_j:j\in J_0\}$ be sufficiently separated.  For every
nonnegative $b\in\ell^{s_p}(J_0)$, there is
$f_b\in L^1(\Omega)\cap\mathcal D_{\Omega,1}$ such that
\begin{align}
 \left\|
 \{V_j^{1-1/p}\M_{1,\rho}f_b(a_j)\}_j
 \right\|_{\ell^{s_p}}
 &\asymp\|b\|_{\ell^{s_p}},
 \label{eq:q1-sharp-data}\\
 \pi_r(M_{f_b}:A^p\to L^1)
 &\asymp\|b\|_{\ell^{s_p}}.
 \label{eq:q1-sharp-operator}
\end{align}
The constants do not depend on $b$.
\end{proposition}

\begin{proof}
Take $\Lambda_0$ sufficiently separated that its $\rho$-balls are
pairwise disjoint.  Fix $0<\theta<1/2$ and put
\[
 E_j=\Q_\rho^{[\theta]}(a_j)\subset\Q_\rho(a_j),
 \qquad v(E_j)=\theta^{2n}V_j.
\]
Set
\begin{equation*}
 f_b(z)=\sum_{j\in J_0}
 b_jV_j^{1/p-1}\mathbf1_{E_j}(z).
\end{equation*}
Let $c_j=V_j^{1/p}$.  Separation gives
\[
 \sum_j|c_j|^p=\sum_jV_j\lesssim v(\Omega).
\]
Since $b\in\mathfrak d_{p,1}^{\,r}$ by
\eqref{eq:q1-diagonal},
\begin{equation*}
 \|f_b\|_{L^1}
 \asymp\sum_jb_jV_j^{1/p}
 =\|D_bc\|_{\ell^1}<\infty.
\end{equation*}
The kernel estimate with one variable in a fixed compact subset gives
$K(\cdot,z)\in L^\infty(\Omega)$ for every $z\in\Omega$.  Hence
$f_bK(\cdot,z)\in L^1$ and
$f_b\in\mathcal D_{\Omega,1}$.

For $j\in J_0$, the ball $\Q_\rho(a_j)$ meets only its own cell.
Consequently,
\begin{align*}
 V_j^{1-1/p}\M_{1,\rho}f_b(a_j)
 &=V_j^{-1/p}\int_{\Q_\rho(a_j)}f_b\dd v\\
 &= b_jV_j^{-1}v(E_j)=\theta^{2n}b_j.
\end{align*}
Conversely, finite overlap gives
\begin{align*}
 V_k^{1-1/p}\M_{1,\rho}f_b(a_k)
 &=V_k^{-1/p}\sum_{j\in\mathcal N(k)}
    b_jV_j^{1/p-1}v(E_j\cap\Q_\rho(a_k))\\
 &\lesssim\sum_{j\in\mathcal N(k)}b_j,
\end{align*}
where the number of neighbors is uniformly bounded in both indices.
The ideal $\ell^{s_p}=\mathfrak d_{p,1}^{\,r}$ is solid.  Apply
\cref{lem:finite-neighbor}.  This proves
\eqref{eq:q1-sharp-data}.  Finally use
\cref{prop:q1-multiplication} to obtain
\eqref{eq:q1-sharp-operator}.
\end{proof}

On $\mathbb B_n$, the endpoint criterion is explicit.  For
$f_\gamma(z)=(1-|z|^2)^\gamma$, $\gamma\ge0$,
\begin{equation*}
 M_{f_\gamma}\in\Piabs_r(A^p(\mathbb B_n),L^1(\mathbb B_n))
 \quad\Longleftrightarrow\quad
 s_p\left((n+1)\left(1-\frac1p\right)+\gamma\right)>n.
\end{equation*}
This follows from \eqref{eq:ball-radial-threshold} with
$q=1$ and $\kappa=s_p$.  Thus the $L^1$ multiplication criterion is
proved without an $L^1$ projection estimate.

The next proposition summarizes the endpoint conclusions for all five
classes.  It separates the two classes requiring projections from the
three exact criteria.

\begin{proposition}
\label{prop:q1-five-class-ledger}
Let $1<p<\infty$, $1\le r<\infty$, and let $s_p$ be given by
\eqref{eq:intro-q1-multiplication}.  Then the following assertions hold.
\begin{enumerate}[label=\textup{(\roman*)}]
\item If $f\in\Hol(\Omega)\cap\mathcal D_{\Omega,1}$, then
\begin{equation}\label{eq:q1-ledger-Toeplitz}
 T_f^{\mathrm{an}}\in\Piabs_r(A^p,A^1)
 \quad\Longleftrightarrow\quad
 \sum_j
 \left[V_j^{1-1/p}\M_{1,\rho}f(a_j)\right]^{s_p}<\infty.
\end{equation}
For a nonanalytic symbol, no bounded $L^1$ projection is used and no
individual Toeplitz equivalence is asserted.
\item The big Hankel formula $H_f=(I-P)M_f$ does not yield an endpoint
factorization on $L^1$.  Thus the $q=1$ big Hankel criterion is not a
consequence of the present theorem.
\item The same problem occurs for
$h_f=\overline P M_f$ and $k_f=(I-\overline P)M_f$.  Thus the
graph for little Hankel operators with nonanalytic symbols is not covered
by this endpoint argument.
\item If $u\in A^1(\Omega)$ and
$\varphi:\Omega\to\Omega$ is holomorphic, then
\begin{equation}\label{eq:q1-ledger-composition}
 W_{u,\varphi}\in\Piabs_r(A^p,A^1)
 \quad\Longleftrightarrow\quad
 \sum_j
 \left[
 \frac{\mu_{u,\varphi}(\Q_\rho(a_j))}{V_j^{1/p}}
 \right]^{s_p}<\infty.
\end{equation}
\item If $g\in A^1(\Omega)$ and
$\varphi:\Omega\to\Omega$ is holomorphic, then
\begin{equation}\label{eq:q1-ledger-Volterra}
 V_{g,\varphi}\in\Piabs_r(A^p,A^1)
 \quad\Longleftrightarrow\quad
 \sum_j
 \left[
 \frac{\mu_{g,\varphi}(\Q_\rho(a_j))}{V_j^{1/p}}
 \right]^{s_p}<\infty.
\end{equation}
\end{enumerate}
The three equivalences are independent of $r$ at the level of sets.
Their norm constants may depend on $r$.
\end{proposition}

\begin{proof}
For (i), combine \cref{cor:analytic-symbols,prop:q1-multiplication} and
\eqref{eq:q1-diagonal}.  This gives
\[
 \pi_r(T_f^{\mathrm{an}})
 \asymp
 \left(
 \sum_j[V_j^{1-1/p}\M_{1,\rho}f(a_j)]^{s_p}
 \right)^{1/s_p}.
\]
Assertions (ii) and (iii) state the missing bounded factorizations.
They do not assert endpoint criteria.  Assertion (iv) is
\eqref{eq:composition-q1}.  Assertion (v) is
\eqref{eq:Volterra-q1}.  Expanding the $\ell^{s_p}$ norms gives
\eqref{eq:q1-ledger-composition} and
\eqref{eq:q1-ledger-Volterra}.  Since $s_p$ does not contain $r$, the
set of admissible data is the same for every finite $r$.
\end{proof}

The endpoint also has an explicit boundary test.

\begin{corollary}
\label{cor:q1-boundary-powers}
Let $1<p<\infty$ and $1\le r<\infty$.  Let $s_p$ be defined by
\eqref{eq:intro-q1-multiplication}.  Assume
\eqref{eq:boundary-volume-power}.  Let $\mu$ be a positive locally
finite Borel measure.  Suppose that
\eqref{eq:boundary-measure-power} holds for some $\sigma\in\mathbb R$.
Then
\begin{equation}\label{eq:q1-measure-boundary-threshold}
 J_\mu\in\Piabs_r(A^p,L^1(\mu))
 \quad\Longleftrightarrow\quad
 s_p\left(\sigma-\frac Dp\right)>D-1.
\end{equation}
The same statement holds for $W_{u,\varphi}$ or
$V_{g,\varphi}$ when $u,g\in A^1(\Omega)$,
$\varphi:\Omega\to\Omega$ is holomorphic, and $\mu$ is the
corresponding pullback measure.  Let $\gamma\in\mathbb R$.  If
$f_\gamma\in\mathcal D_{\Omega,1}$ and
\[
 |f_\gamma(z)|\asymp\delta(z)^\gamma,
 \qquad 0<\delta(z)<\varepsilon_0,
\]
then
\begin{equation}\label{eq:q1-symbol-boundary-threshold}
 M_{f_\gamma}\in\Piabs_r(A^p,L^1)
 \quad\Longleftrightarrow\quad
 s_p\left[D\left(1-\frac1p\right)+\gamma\right]>D-1.
\end{equation}
For holomorphic $f_\gamma$, the same condition characterizes
$T_{f_\gamma}^{\mathrm{an}}$.
\end{corollary}

\begin{proof}
By \eqref{eq:q1-diagonal},
\[
 \mathfrak d_{p,1}^{\,r}=\ell^{s_p}.
\]
Apply \cref{prop:boundary-power-test} with $q=1$ and
$\kappa=s_p$.  This gives \eqref{eq:q1-measure-boundary-threshold}.
The weighted-composition and Volterra assertions follow from the same
argument applied to the corresponding pullback measures.  No projection
estimate on $L^1$ is used.

For the symbol assertion, put $\dd\mu_\gamma=|f_\gamma|\dd v$.  The
assumption $f_\gamma\in\mathcal D_{\Omega,1}$ gives local finiteness of
$\mu_\gamma$. Stability of the boundary distance on McNeal balls gives
\[
 \delta(w)\asymp\delta(z),
 \qquad w\in\Q_\rho(z),
 \qquad 0<\delta(z)<\varepsilon_1,
\]
with $\varepsilon_1>0$ fixed.  Hence, near the boundary,
\[
 \mu_\gamma(\Q_\rho(z))
 =\int_{\Q_\rho(z)}|f_\gamma(w)|\dd v(w)
 \asymp\delta(z)^\gamma\V_\rho(z)
 \asymp\delta(z)^{D+\gamma}.
\]
Thus $\sigma=D+\gamma$, which gives
\eqref{eq:q1-symbol-boundary-threshold}.  The final assertion follows
from \eqref{eq:q1-ledger-Toeplitz}.
\end{proof}

At equality in \eqref{eq:q1-measure-boundary-threshold} or
\eqref{eq:q1-symbol-boundary-threshold}, the normal integral is
$\int_0^{\varepsilon_0}t^{-1}\dd t$.  It diverges.  Hence the strict
inequality cannot be weakened.

For nonanalytic symbols, neither projected graph follows from
\cref{prop:q1-multiplication}.  The proofs for $q>1$ use the bounded
projections
\[
 P:L^q(\Omega)\longrightarrow A^q(\Omega),
 \qquad
 \overline P:L^q(\Omega)\longrightarrow\overline{A^q(\Omega)}.
\]
At the endpoint, the same factorization would require bounded maps
\[
 P:L^1(\Omega)\longrightarrow A^1(\Omega),
 \qquad
 \overline P:L^1(\Omega)\longrightarrow\overline{A^1(\Omega)}.
\]
No such boundedness estimates are available here.  A weak-type
estimate would not imply
\[
 \pi_r(PM_f)\lesssim\pi_r(M_f).
\]
Thus the graphs $(T_f,H_f)$ and $(h_f,k_f)$ for nonanalytic symbols are not
covered by the present projection argument.  At this endpoint,
the kernel formulas are understood on any dense initial domain on which
they take values in their target spaces:

\begin{equation*}
\boxed{
 \begin{gathered}
 \text{Characterize }T_f,\ H_f,\ h_f,\text{ and }k_f
 \text{ at target exponent }q=1\\
 \text{in }\Piabs_r\text{ by a replacement for }
 \eqref{eq:q1-candidate}.
 \end{gathered}}
\end{equation*}
Endpoint atomic decompositions, Hardy--Orlicz spaces, or additional
weak-type estimates may help answer this question.

For the summary below, let $\mu_A$ denote $\mu$ when $A=J_\mu$,
$|f|^qv$ when $A=M_f$, $\mu_{u,\varphi}$ when $A=W_{u,\varphi}$,
and $\mu_{g,\varphi}$ when $A=V_{g,\varphi}$.
The criteria proved here can then be summarized by
\[
\begin{array}{c}
 J_\mu,\ M_f,\ W_{u,\varphi},\ V_{g,\varphi}\\[2pt]
 \Downarrow\\[-2pt]
 \beta_j(A)=\mu_A(\Q_\rho(a_j))^{1/q}V_j^{-1/p}\\[2pt]
 \Downarrow\\[-2pt]
 A\in\Piabs_r
 \quad\Longleftrightarrow\quad
 D_{\{\beta_j(A)\}}\in\Piabs_r(\ell^p,\ell^q).
\end{array}
\]
For $q>1$, bounded projection decompositions give the joint
criteria for $(T_f,H_f)$ and $(h_f,k_f)$.
The IDA decomposition at two fixed scales and Ahn's solution estimate give
the big Hankel criterion. The individual Toeplitz converses use
the stated symbol conditions. For the little Hankel operator alone,
we prove a necessary condition.
Diagonal factorizations give approximation lower bounds.
For atomic embeddings, the approximation errors agree with those of
diagonal operators up to constants.
The ellipsoid calculation gives a logarithmic factor in the optimal
approximation rate in the summing norm.
\begingroup
\makeatletter
\let\@tocwrite\@gobbletwo
\small

\medskip
\noindent\textsc{Funding.}
The first author was supported by the National Natural Science Foundation
of China (Grant No.~12601234).  The second author was supported by the
National Natural Science Foundation of China (Grant No.~12401154).

\smallskip
\noindent\textsc{Data availability.}
No data were used for the research described in this article.

\smallskip
\noindent\textsc{Declaration of competing interest.}
The authors declare that they have no known competing financial interests
or personal relationships that could have appeared to influence the work
reported in this paper.

\makeatother
\endgroup

\raggedbottom


\begin{thebibliography}{99}
\small
\urlstyle{same}
\newcommand{\bibdoi}[1]{\mbox{DOI:~\href{https://doi.org/#1}{\nolinkurl{#1}}}}

\bibitem{AbateRaissySaracco2012}
M.~Abate, J.~Raissy, and A.~Saracco,
Toeplitz operators and Carleson measures in strongly pseudoconvex domains,
\emph{J. Funct. Anal.} \textbf{263} (2012), no.~11, 3449--3491.
\bibdoi{10.1016/j.jfa.2012.08.027}.

\bibitem{Ahn2004}
H.~Ahn,
Weighted $L^p$ estimates for the $\overline\partial$-equation on convex
domains of finite type,
\emph{Publ. Mat.} \textbf{48} (2004), no.~1, 139--157.
\bibdoi{10.5565/PUBLMAT_48104_07}.

\bibitem{AhnCho2003}
H.~Ahn and H.~R. Cho,
Optimal Sobolev estimates for $\overline\partial$ on convex domains of
finite type,
\emph{Math. Z.} \textbf{244} (2003), 837--857.
\bibdoi{10.1007/s00209-003-0525-z}.

\bibitem{AlemanSiskakis1997}
A.~Aleman and A.~G. Siskakis,
Integration operators on Bergman spaces,
\emph{Indiana Univ. Math. J.} \textbf{46} (1997), no.~2, 337--356.
\bibdoi{10.1512/iumj.1997.46.1373}.

\bibitem{Alexandre2006}
W.~Alexandre,
$C^k$-estimates for the $\overline\partial$-equation on convex domains
of finite type,
\emph{Math. Z.} \textbf{252} (2006), no.~3, 473--496.
\bibdoi{10.1007/s00209-005-0812-y}.

\bibitem{BoasFuStraube1997}
H.~P. Boas, S.~Fu, and E.~J. Straube,
The Bergman kernel function: explicit formulas and zeroes,
preprint, \href{https://arxiv.org/abs/math/9706202}{arXiv:math/9706202}, 1997.

\bibitem{BonamiPelosoSymesak2001}
A.~Bonami, M.~M. Peloso, and F.~Symesak,
Factorization of Hardy spaces and Hankel operators on convex domains in
$\C^n$,
\emph{J. Geom. Anal.} \textbf{11} (2001), no.~3, 363--397.
\bibdoi{10.1007/bf02922011}.

\bibitem{Catlin1987}
D.~Catlin,
Subelliptic estimates for the $\overline\partial$-Neumann problem on
pseudoconvex domains,
\emph{Ann. of Math. (2)} \textbf{126} (1987), no.~1, 131--191.
\bibdoi{10.2307/1971347}.

\bibitem{CharpentierDupain2018}
P.~Charpentier and Y.~Dupain,
Weighted and boundary $L^p$ estimates for solutions of the
$\overline\partial$-equation on lineally convex domains of finite type and
applications,
\emph{Math. Z.} \textbf{290} (2018), no.~1--2, 195--220.
\bibdoi{10.1007/s00209-017-2015-8}.

\bibitem{CharpentierDupain2024}
P.~Charpentier and Y.~Dupain,
On weighted $L^p$-Sobolev estimates for solutions of the
$\overline\partial$-equation on linearly convex domains of finite type
and application,
\emph{Acta Math. Vietnam.} \textbf{49} (2024), no.~2, 217--240.
\bibdoi{10.1007/s40306-024-00530-y}.

\bibitem{ChenDongWang2025}
J.~Chen, C.~Dong, and M.~Wang,
Absolutely summing Carleson embeddings over the half-plane,
\emph{Complex Anal. Oper. Theory} \textbf{19} (2025), no.~4,
Paper No.~72.
\bibdoi{10.1007/s11785-025-01697-4}.

\bibitem{ChenHeWang2026}
J.~Chen, B.~He, and M.~Wang,
Absolutely summing Carleson embeddings on weighted Fock spaces with
$A_\infty$-type weights,
\emph{J. Operator Theory} \textbf{95} (2026), no.~1, 119--151.
\bibdoi{10.7900/jot.2024apr01.24}.

\bibitem{CimaWogen1982}
J.~A. Cima and W.~R. Wogen,
A Carleson measure theorem for the Bergman space on the ball,
\emph{J. Operator Theory} \textbf{7} (1982), 157--165.

\bibitem{CowenMacCluer1995}
C.~C. Cowen and B.~D. MacCluer,
\emph{Composition Operators on Spaces of Analytic Functions},
Studies in Advanced Mathematics, CRC Press, Boca Raton, FL, 1995.
\bibdoi{10.1201/9781315139920}.

\bibitem{CuckovicMcNeal2006}
Z.~\v{C}u\v{c}kovi\'c and J.~D. McNeal,
Special Toeplitz operators on strongly pseudoconvex domains,
\emph{Rev. Mat. Iberoam.} \textbf{22} (2006), no.~3, 851--866.
\bibdoi{10.4171/rmi/476}.

\bibitem{CuckovicZhao2004}
Z.~\v{C}u\v{c}kovi\'c and R.~Zhao,
Weighted composition operators on the Bergman space,
\emph{J. Lond. Math. Soc. (2)} \textbf{70} (2004), no.~2, 499--511.
\bibdoi{10.1112/s0024610704005605}.

\bibitem{Cumenge2001}
A.~Cumenge,
Sharp estimates for $\overline\partial$ on convex domains of finite type,
\emph{Ark. Mat.} \textbf{39} (2001), no.~1, 1--25.
\bibdoi{10.1007/bf02388789}.

\bibitem{DiestelJarchowTonge1995}
J.~Diestel, H.~Jarchow, and A.~Tonge,
\emph{Absolutely Summing Operators},
Cambridge Studies in Advanced Mathematics, vol.~43,
Cambridge University Press, Cambridge, 1995.
\bibdoi{10.1017/cbo9780511526138}.

\bibitem{FanHeWangZeng2026}
Z.~Fan, B.~He, X.~Wang, and Z.~Zeng,
Absolute summability from diagonal operators to Carleson embeddings and
Hankel operators,
preprint, \href{https://arxiv.org/abs/2511.22165v2}{arXiv:2511.22165v2}, 2026.

\bibitem{FanWangZeng2026}
Z.~Fan, X.~Wang, and Z.~Zeng,
IDA function and asymptotic behavior of singular values of Hankel
operators on weighted Bergman spaces,
\emph{Adv. Math.} \textbf{501} (2026), Paper No.~111113.
\bibdoi{10.1016/j.aim.2026.111113}.

\bibitem{FangXia2018}
Q.~Fang and J.~Xia,
Hankel operators on weighted Bergman spaces and norm ideals,
\emph{Complex Anal. Oper. Theory} \textbf{12} (2018), no.~3, 629--668.
\bibdoi{10.1007/s11785-017-0710-4}.

\bibitem{FaresLefevre2022}
T.~Far\`es and P.~Lef\`evre,
Absolutely summing weighted composition operators on Bloch spaces,
\emph{J. Operator Theory} \textbf{88} (2022), no.~2, 407--443.
\bibdoi{10.7900/jot.2021oct20.2342}.

\bibitem{FollandKohn1972}
G.~B. Folland and J.~J. Kohn,
\emph{The Neumann Problem for the Cauchy--Riemann Complex},
Annals of Mathematics Studies, vol.~75, Princeton University Press,
Princeton, NJ, 1972.
\bibdoi{10.1515/9781400881529}.

\bibitem{GanHuKhan2022}
C.~Gan, B.~Hu, and I.~Khan,
Dyadic decomposition of convex domains of finite type and applications,
\emph{Math. Z.} \textbf{301} (2022), no.~2, 1939--1962.
\bibdoi{10.1007/s00209-022-02984-y}.

\bibitem{Garling1970}
D.~J.~H. Garling,
Absolutely $p$-summing operators in Hilbert space,
\emph{Studia Math.} \textbf{38} (1970), no.~1, 319--331.
\bibdoi{10.4064/sm-38-1-319-331}.

\bibitem{Garling1974}
D.~J.~H. Garling,
Diagonal mappings between sequence spaces,
\emph{Studia Math.} \textbf{51} (1974), no.~2, 129--138.
\bibdoi{10.4064/sm-51-2-129-138}.

\bibitem{Hastings1975}
W.~W. Hastings,
A Carleson measure theorem for Bergman spaces,
\emph{Proc. Amer. Math. Soc.} \textbf{52} (1975), no.~1, 237--241.
\bibdoi{10.1090/s0002-9939-1975-0374886-9}.

\bibitem{HeJreisLefevreLou2024}
B.~He, J.~Jreis, P.~Lef\`evre, and Z.~Lou,
Absolutely summing Carleson embeddings on Bergman spaces,
\emph{Adv. Math.} \textbf{439} (2024), Paper No.~109495, 61 pp.
\bibdoi{10.1016/j.aim.2024.109495}.

\bibitem{HuLv2026}
Z.~Hu and X.~Lv,
Absolutely summing Hankel operators on Fock spaces and the
Berger--Coburn phenomenon,
preprint, \href{https://arxiv.org/abs/2601.01197v1}{arXiv:2601.01197v1}, 2026.

\bibitem{HuLvZhu2016}
Z.~Hu, X.~Lv, and K.~Zhu,
Carleson measures and balayage for Bergman spaces of strongly
pseudoconvex domains,
\emph{Math. Nachr.} \textbf{289} (2016), no.~10, 1237--1254.
\bibdoi{10.1002/mana.201500021}.

\bibitem{HuVirtanen2023}
Z.~Hu and J.~A. Virtanen,
IDA and Hankel operators on Fock spaces,
\emph{Anal. PDE} \textbf{16} (2023), no.~9, 2041--2077.
\bibdoi{10.2140/apde.2023.16.2041}.

\bibitem{HuWang2025}
Z.~Hu and E.~Wang,
Absolutely summing Toeplitz operators on Bergman spaces in the unit ball
of $\C^n$,
preprint, \href{https://arxiv.org/abs/2512.24710v1}{arXiv:2512.24710v1}, 2025.

\bibitem{Isralowitz2013}
J.~Isralowitz,
Schatten $p$ class commutators on the weighted Bergman space
$L_a^2(\mathbb B_n,dv_\gamma)$ for
$2n/(n+1+\gamma)<p<\infty$,
\emph{Indiana Univ. Math. J.} \textbf{62} (2013), no.~1, 201--233.
\bibdoi{10.1512/iumj.2013.62.4767}.

\bibitem{Janson1988}
S.~Janson,
Hankel operators between weighted Bergman spaces,
\emph{Ark. Mat.} \textbf{26} (1988), no.~1--2, 205--219.
\bibdoi{10.1007/bf02386120}.

\bibitem{Jasiczak2010}
M.~Jasiczak,
Carleson embedding theorem on convex finite type domains,
\emph{J. Math. Anal. Appl.} \textbf{362} (2010), no.~1, 167--189.
\bibdoi{10.1016/j.jmaa.2009.09.022}.

\bibitem{JreisLefevre2024}
J.~Jreis and P.~Lef\`evre,
Some operator ideal properties of Volterra operators on Bergman and
Bloch spaces,
\emph{Integral Equations Operator Theory} \textbf{96} (2024), no.~1,
Paper No.~1, 29 pp.
\bibdoi{10.1007/s00020-023-02742-7}.

\bibitem{Krantz2001}
S.~G. Krantz,
\emph{Function Theory of Several Complex Variables}, second ed.,
AMS Chelsea Publishing, vol.~340, American Mathematical Society,
Providence, RI, 2001.
\bibdoi{10.1090/chel/340}.

\bibitem{KrantzLi1995}
S.~G. Krantz and S.-Y. Li,
Duality theorems for Hardy and Bergman spaces on convex domains of finite
type in $\C^n$,
\emph{Ann. Inst. Fourier (Grenoble)} \textbf{45} (1995), no.~5, 1305--1327.
\bibdoi{10.5802/aif.1497}.

\bibitem{Lefevre2026}
P.~Lef\`evre,
Nuclear Carleson embeddings on Hardy and Bergman spaces with exponent~1,
preprint, \href{https://hal.science/hal-05467038}{HAL:hal-05467038}, 2026.

\bibitem{LefevreRodriguezPiazza2018}
P.~Lef\`evre and L.~Rodr\'iguez-Piazza,
Absolutely summing Carleson embeddings on Hardy spaces,
\emph{Adv. Math.} \textbf{340} (2018), 528--587.
\bibdoi{10.1016/j.aim.2018.10.012}.

\bibitem{Li1992}
H.~Li,
BMO, VMO and Hankel operators on the Bergman space of strongly
pseudoconvex domains,
\emph{J. Funct. Anal.} \textbf{106} (1992), no.~2, 375--408.
\bibdoi{10.1016/0022-1236(92)90054-m}.

\bibitem{Li1993}
H.~Li,
Schatten class Hankel operators on the Bergman spaces of strongly
pseudoconvex domains,
\emph{Proc. Amer. Math. Soc.} \textbf{119} (1993), no.~4, 1211--1221.
\bibdoi{10.1090/s0002-9939-1993-1169879-9}.

\bibitem{LiLiuWang2024}
H.~Li, J.~Liu, and H.~Wang,
Carleson measures on convex domains with smooth boundary of finite type,
\emph{Math. Nachr.} \textbf{297} (2024), no.~2, 694--706.
\bibdoi{10.1002/mana.202200325}.

\bibitem{LiLuecking1995}
H.~Li and D.~H. Luecking,
Schatten class of Hankel and Toeplitz operators on the Bergman space of
strongly pseudoconvex domains,
in \emph{Multivariable Operator Theory}, Contemp. Math., vol.~185,
Amer. Math. Soc., Providence, RI, 1995, pp.~237--257.
\bibdoi{10.1090/conm/185/02157}.

\bibitem{Luecking1983}
D.~H. Luecking,
A technique for characterizing Carleson measures on Bergman spaces,
\emph{Proc. Amer. Math. Soc.} \textbf{87} (1983), no.~4, 656--660.
\bibdoi{10.1090/s0002-9939-1983-0687635-6}.

\bibitem{Luecking1987}
D.~H. Luecking,
Trace ideal criteria for Toeplitz operators,
\emph{J. Funct. Anal.} \textbf{73} (1987), no.~2, 345--368.
\bibdoi{10.1016/0022-1236(87)90072-3}.

\bibitem{McNeal1992}
J.~D. McNeal,
Convex domains of finite type,
\emph{J. Funct. Anal.} \textbf{108} (1992), no.~2, 361--373.
\hfill\break \bibdoi{10.1016/0022-1236(92)90029-i}.

\bibitem{McNeal1994}
J.~D. McNeal,
Estimates on the Bergman kernels of convex domains,
\emph{Adv. Math.} \textbf{109} (1994), no.~1, 108--139.
\bibdoi{10.1006/aima.1994.1082}.

\bibitem{McNealSIO1994}
J.~D. McNeal,
The Bergman projection as a singular integral operator,
\emph{J. Geom. Anal.} \textbf{4} (1994), no.~1, 91--103.
\bibdoi{10.1007/bf02921594}.

\bibitem{McNealStein1994}
J.~D. McNeal and E.~M. Stein,
Mapping properties of the Bergman projection on convex domains of finite
type,
\emph{Duke Math. J.} \textbf{73} (1994), no.~1, 177--199.
\bibdoi{10.1215/s0012-7094-94-07307-9}.

\bibitem{NagelRosaySteinWainger1989}
A.~Nagel, J.-P. Rosay, E.~M. Stein, and S.~Wainger,
Estimates for the Bergman and Szeg\H{o} kernels in $\C^2$,
\emph{Ann. of Math. (2)} \textbf{129} (1989), no.~1, 113--149.
\bibdoi{10.2307/1971487}.

\bibitem{NagelSteinWainger1985}
A.~Nagel, E.~M. Stein, and S.~Wainger,
Balls and metrics defined by vector fields. I. Basic properties,
\emph{Acta Math.} \textbf{155} (1985), 103--147.
\bibdoi{10.1007/bf02392539}.

\bibitem{NikolovPflug2003}
N.~Nikolov and P.~Pflug,
Estimates for the Bergman kernel and metric of convex domains in $\C^n$,
\emph{Ann. Polon. Math.} \textbf{81} (2003), no.~1, 73--78.
\bibdoi{10.4064/ap81-1-6}.

\bibitem{NikolovPflugThomas2013}
N.~Nikolov, P.~Pflug, and P.~J. Thomas,
On different extremal bases for $\C$-convex domains,
\emph{Proc. Amer. Math. Soc.} \textbf{141} (2013), no.~9, 3223--3230.
\bibdoi{10.1090/s0002-9939-2013-11584-4}.

\bibitem{Pau2016}
J.~Pau,
Characterization of Schatten-class Hankel operators on weighted Bergman
spaces,
\emph{Duke Math. J.} \textbf{165} (2016), no.~14, 2771--2791.
\bibdoi{10.1215/00127094-3627310}.

\bibitem{PauZhao2015}
J.~Pau and R.~Zhao,
Carleson measures and Toeplitz operators for weighted Bergman spaces on
the unit ball,
\emph{Michigan Math. J.} \textbf{64} (2015), no.~4, 759--796.
\bibdoi{10.1307/mmj/1447878031}.

\bibitem{PelaezPeralaRattya2020}
J.~\'{A}. Pel\'{a}ez, A.~Per\"al\"a, and J.~R\"atty\"a,
Hankel operators induced by radial B\'{e}koll\'{e}--Bonami weights on
Bergman spaces,
\emph{Math. Z.} \textbf{296} (2020), no.~1--2, 211--238.
\bibdoi{10.1007/s00209-019-02412-8}.

\bibitem{Pietsch1967}
A.~Pietsch,
Absolut $p$-summierende Abbildungen in normierten R\"aumen,
\emph{Studia Math.} \textbf{28} (1967), no.~3, 333--353.
\bibdoi{10.4064/sm-28-3-333-353}.

\bibitem{Pietsch1980}
A.~Pietsch,
\emph{Operator Ideals},
North-Holland Mathematical Library, vol.~20, North-Holland, Amsterdam,
1980.

\bibitem{Pushnitski2017}
A.~Pushnitski,
Spectral asymptotics for a class of Toeplitz operators on the Bergman space,
arXiv:1710.02652 (2017).
\url{https://arxiv.org/abs/1710.02652}.

\bibitem{QuirogaSanchez2015}
R.~Quiroga-Barranco and A.~Sanchez-Nungaray,
\emph{Toeplitz operators with quasi-homogeneuos quasi-radial symbols
on some weakly pseudoconvex domains},
Complex Anal. Oper. Theory \textbf{9} (2015), no.~5, 1111--1134.
\url{https://doi.org/10.1007/s11785-014-0407-x}.

\bibitem{Simon2005}
B.~Simon,
\emph{Trace Ideals and Their Applications}, second ed.,
Mathematical Surveys and Monographs, vol.~120,
American Mathematical Society, Providence, RI, 2005.
\bibdoi{10.1090/surv/120}.

\bibitem{TaskinenVirtanen2018}
J.~Taskinen and J.~A. Virtanen,
On generalized Toeplitz and little Hankel operators on Bergman spaces,
\emph{Arch. Math. (Basel)} \textbf{110} (2018), no.~2, 155--166.
\bibdoi{10.1007/s00013-017-1124-2}.

\bibitem{XiaoYangYuan2026}
J.~Xiao, W.~Yang, and C.~Yuan,
Geometrical Toeplitz operators and Carleson embeddings over smoothly
bounded convex domains of finite type in $\C^n$,
\emph{Sci. China Math.} \textbf{69} (2026), no.~4, 985--1010.
\bibdoi{10.1007/s11425-024-2449-5}.

\bibitem{XuDong2026}
C.~Xu and J.~Dong,
Absolutely summing Toeplitz operators on weighted Fock spaces,
preprint, \href{https://arxiv.org/abs/2609.15424v1}{arXiv:2609.15424v1}, 2026.

\bibitem{YangYuan2025}
W.~Yang and C.~Yuan,
Schatten class little Hankel operators on Bergman spaces in bounded
symmetric domains,
\emph{J. Funct. Anal.} \textbf{288} (2025), no.~10,
Paper No.~110875, 32 pp.
\bibdoi{10.1016/j.jfa.2025.110875}.

\bibitem{Yao2024}
L.~Yao,
Sobolev and H\"older estimates for homotopy operators of the
$\overline\partial$-equation on convex domains of finite multitype,
\emph{J. Math. Anal. Appl.} \textbf{538} (2024), no.~2,
Paper No.~128238, 41 pp.
\bibdoi{10.1016/j.jmaa.2024.128238}.

\bibitem{Yuan2026}
C.~Yuan,
Little Hankel operators between Bergman spaces in bounded symmetric
domains,
\emph{Rev. R. Acad. Cienc. Exactas F\'is. Nat. Ser. A Mat. RACSAM}
\textbf{120} (2026), no.~2, Paper No.~32.
\hfill\break \bibdoi{10.1007/s13398-026-01828-2}.

\bibitem{Zhang2025}
S.~Zhang,
Carleson measures on weakly pseudoconvex domains,
\emph{J. Geom. Anal.} \textbf{35} (2025), no.~9, Paper No.~277.
\bibdoi{10.1007/s12220-025-02114-0}.

\bibitem{Zhang2026}
S.~Zhang,
Carleson measures and uniformly discrete sequences in domains with
bounded intrinsic geometry,
\emph{Proc. Amer. Math. Soc.} \textbf{154} (2026), no.~11, 4731--4744.
\bibdoi{10.1090/proc/17743}.

\bibitem{Zhu1991}
K.~Zhu,
Hankel operators on the Bergman space of bounded symmetric domains,
\emph{Trans. Amer. Math. Soc.} \textbf{324} (1991), no.~2, 707--730.
\bibdoi{10.1090/s0002-9947-1991-1093426-6}.

\bibitem{Zhu2007}
K.~Zhu,
\emph{Operator Theory in Function Spaces}, second ed.,
Mathematical Surveys and Monographs, vol.~138,
American Mathematical Society, Providence, RI, 2007.
\bibdoi{10.1090/surv/138}.

\bibitem{Zimmer2021}
A.~Zimmer,
Compactness of the $\bar\partial$-Neumann problem on domains with
bounded intrinsic geometry,
\emph{J. Funct. Anal.} \textbf{281} (2021), no.~1,
Paper No.~108992, 47 pp.
\bibdoi{10.1016/j.jfa.2021.108992}.

\bibitem{Zimmer2023}
A.~Zimmer,
Hankel operators on domains with bounded intrinsic geometry,
\emph{J. Geom. Anal.} \textbf{33} (2023),
Paper No.~176, 29 pp.
\bibdoi{10.1007/s12220-023-01231-y}.

\end{thebibliography}
\end{document}